\documentclass[11pt]{amsart}
\usepackage{amsrefs}
\usepackage[T1]{fontenc}
\usepackage[latin1]{inputenc}
\usepackage{a4wide}
\usepackage{setspace}
\usepackage[all]{xypic}
\usepackage{amssymb}
\usepackage{amsthm}
\usepackage[english]{babel}
\usepackage{latexsym}
\usepackage{amscd}
\usepackage{mathrsfs}
\date{}
\newtheorem{thm}{Theorem}[section]

\newtheorem{defn}[thm]{Definition}
\newtheorem{rem}[thm]{Remark}
\newtheorem{conj}[thm]{Conjecture}

\newtheorem{prop}[thm]{Proposition}

\newtheorem{lem}[thm]{Lemma}
\newtheorem{cor}[thm]{Corollary}
\newtheorem{notation}[thm]{Notation}
\newenvironment{f-proof}[1][\sc D\'emonstration.]{\begin{trivlist}
\item[\hskip \labelsep {\bfseries #1}]}{\hfill{$\square$}\end{trivlist}}

\newcommand{\Coprod}{\displaystyle\coprod}
\newcommand{\Prod}{\displaystyle\prod}

\newcommand{\fonc}[5]{
 \begin{array}{cccc}
 #1: & #2 & \longrightarrow & #3\\
     & #4 & \longmapsto & #5
 \end{array}
}

\newcommand{\Ar}{{\mathrm ar}}

\newcommand{\beil}{\mathbf{B}}

\newcommand{\vol}{\mathrm{vol}}
\newcommand{\coker}{\mathrm{coker}}

\newcommand{\sbar}{\bar{s}}
\newcommand{\etabar}{\bar{\eta}}
\newcommand{\shen}{S^{ur}}
\newcommand{\etahen}{\eta^{ur}}
\newcommand{\qbar}{\bar{\mathbb Q}}

\newcommand{\bq}{\mathbb Q}

\newcommand{\bz}{\mathbb Z}
\newcommand{\br}{\mathbb R}
\newcommand{\bc}{\mathbb C}

\newcommand{\bF}{\mathbb F}

\newcommand{\tr}{\tilde{\mathbb R}}

\newcommand{\co}{\mathcal O}
\newcommand{\cd}{\mathcal D}
\newcommand{\F}{\mathcal F}

\newcommand{\X}{\mathcal X}

\newcommand{\et}{\mathrm{et}}
\newcommand{\mydet}{\text{Det}}

\DeclareMathOperator{\ord}{ord}
\DeclareMathOperator{\Spec}{Spec}
\DeclareMathOperator{\Pic}{Pic}
\DeclareMathOperator{\Cl}{Cl}

\DeclareMathOperator{\Ind}{Ind}

\DeclareMathOperator{\im}{im}
\DeclareMathOperator{\Gal}{Gal}
\DeclareMathOperator{\Det}{Det}
\DeclareMathOperator{\id}{id}
\DeclareMathOperator{\cone}{Cone}
\DeclareMathOperator{\Hom}{Hom}

\DeclareMathOperator{\val}{val}

\begin{document}

\title[Weil-\'etale cohomology and Zeta-values]{Weil-\'etale cohomology and Zeta-values of proper regular arithmetic schemes}
\author{M. Flach and B. Morin}
\subjclass[2010]{14F20 (primary) 14F42 11G40 }
\begin{abstract}{We give a conjectural description of the vanishing order and leading Taylor coefficient of the Zeta function of a proper, regular arithmetic scheme $\X$ at any integer $n$ in terms of Weil-\'etale cohomology complexes. This extends work of Lichtenbaum \cite{Lichtenbaum05} and Geisser \cite{Geisser04b} for $\X$ of characteristic $p$, of Lichtenbaum \cite{li04} for $\X=\Spec(\co_F)$ and $n=0$ where $F$ is a number field, and of the second author for arbitrary $\X$ and $n=0$ \cite{Morin14}. We show that our conjecture is compatible with the Tamagawa number conjecture of Bloch, Kato, Fontaine and Perrin-Riou \cite{fpr91} if $\X$ is smooth over $\Spec(\co_F)$, and hence that it holds in cases where the Tamagawa number conjecture is known.}\end{abstract}
\maketitle

\tableofcontents

\section{Introduction}

In \cite{Lichtenbaum05}, \cite{li04} and \cite{li08} Lichtenbaum introduced
the idea of a Weil-\'etale cohomology theory in order to describe
the vanishing order and leading Taylor coefficient of the
Zeta-function of an arithmetic scheme (i.e. a scheme separated and
of finite type over $\Spec(\bz)$) at any integer argument $s=n$. In
\cite{Lichtenbaum05} he defined a Weil-\'etale topos for any variety over a
finite field and showed that Weil-\'etale cohomology groups have the
expected relationship to the Zeta function of smooth, proper
varieties at $s=0$. Assuming the Tate conjecture, Geisser extended
this to smooth, proper varieties over finite fields and any
$s=n\in\bz$ in \cite{Geisser04b}, and to arbitrary varieties over
finite fields and any $n$ in \cite{Geisser06} (also assuming
resolution of singularities).

For schemes of characteristic 0, Lichtenbaum made the crucial first
step in \cite{li04} where he defined Weil-\'etale cohomology groups
for $\X=\Spec(\co_F)$, the spectrum of the ring of integers in a
number field $F$, and proved the expected relationship to the
Dedekind Zeta-function at $s=0$ modulo a vanishing statement for
higher degree cohomology with $\bz$-coefficients. Unfortunately this
cohomology was then shown to be nonvanishing in \cite{flach06-2}.
However, the cohomology with $\br$-coefficients as defined by
Lichtenbaum is correct in all degrees. Encouraged by this fact we
defined in \cite{Flach-Morin-12} a Weil-\'etale
topos for any regular, proper, flat scheme $\X$ over $\Spec(\bz)$
and showed again that the cohomology with $\br$-coefficients has the
expected relationship to the vanishing order of the Zeta-function at
$s=0$ (provided the Hasse-Weil L-functions of all motives
$h^i(\X_\bq)$ have the expected meromorphic continuation and
functional equation). Next, assuming finite generation of motivic cohomology of $\X$, the second author gave in \cite{Morin14}
a description of the leading coefficient at $s=0$ in terms of Weil-\'etale cohomology groups. The key idea of \cite{Morin14} was to define Weil-etale cohomology complexes with $\bz$-coefficients via (Artin-Verdier) duality rather than as cohomology of a Weil-etale topos associated to $\X$.

In the present article we pursue this idea further and define Weil-etale cohomology complexes with $\bz(n)$-coefficients for any $n\in\bz$, and we give their conjectural relation to the Zeta-function of $\X$ at $s=n$. As in \cite{Morin14} a key assumption in this construction is finite generation of \'etale motivic cohomology of $\X$ in a certain range. If $\X$ is smooth over the ring of integers of a number field, we prove that our description is compatible with
the Tamagawa number conjecture of Bloch-Kato \cite{bk88} and
Fontaine-Perrin-Riou \cite{fpr91}, and hence also with the analytic
class number formula and the conjecture of Birch and
Swinnerton-Dyer. Besides \cite{li04}, the only other work on Weil-etale cohomology for arithmetic schemes of characteristic zero that we are aware of is \cite{li08} where Lichtenbaum gives a description of the value of the Zeta-function of a 1-motive modulo rational factors in terms of two sets of Weil-\'etale cohomology groups ("motivic" and "additive"). Our description is somewhat different from \cite{li08} in cases where both apply although the two descriptions are of course equivalent.

For the remainder of this introduction $\X$ is a regular connected scheme of dimension $d$, proper over $\Spec(\bz)$ and $n\in\bz$ is any integer. We assume that $\X$ and $n$ satisfy Conjectures ${\bf AV}(\overline{\mathcal{X}}_{et},n)$, ${\bf L}(\overline{\mathcal{X}}_{et},n)$, ${\bf L}(\overline{\mathcal{X}}_{et},d-n)$, ${\bf B}(\mathcal{X},n)$ and ${\bf D}_p(\mathcal{X},n)$ below which we shall refer to as the "standard assumptions". We construct two sets of cohomology complexes associated to $\X$ which we call "Weil-\'etale" and "Weil-Arakelov" cohomology, respectively. More precisely, we construct a perfect complex of abelian groups
\[ R\Gamma_{W,c}(\mathcal{X},\mathbb{Z}(n))\]
and an exact triangle
\begin{equation} R\Gamma_{\Ar,c}(\X,\bz(n))\to R\Gamma_{\Ar,c}(\X,\tr(n))\to R\Gamma_{\Ar,c}(\X,\tr/\bz(n))\to\label{tri0}\end{equation}
in the bounded derived category of locally compact abelian groups (see \cite{spitzweck07} for precise definitions) with the following properties.

\begin{itemize}
\item[a)] The groups $H^i_{\Ar,c}(\X,\tr(n))$
are finite dimensional vector spaces over $\br$ for all $i$ and there is an exact sequence
\begin{equation}
\cdots\xrightarrow{\cup\theta}
H^i_{\Ar,c}(\X,\tr(n))\xrightarrow{\cup\theta}H^{i+1}_{\Ar,c}(\X,\tr(n))\xrightarrow{\cup\theta}\cdots\label{thetaseq}\end{equation}
In particular, the complex $R\Gamma_{\Ar,c}(\X,\tr(n))$ has vanishing Euler characteristic:
\[ \sum_{i\in\bz}(-1)^i\dim_\br H^i_{\Ar,c}(\X,\tr(n))=0.\]
\item[b)] The groups $H^i_{\Ar,c}(\X,\tr/\bz(n))$ are compact, commutative Lie groups for all $i$.
\end{itemize}
Note here that the cohomology groups of a complex of locally compact abelian groups carry an induced topology which however need not be locally compact. Indeed, the groups $H^i_{\Ar,c}(\X,\bz(n))$ will not always be locally compact.

The conjectural relation to the Zeta-function of $\X$ is given by the following two statements.
\begin{itemize}
\item[c)] The function $\zeta(\X,s)$ has a meromorphic continuation to $s=n$ and
\[ \ord_{s=n}\zeta(\X,s)=\sum_{i\in\bz}(-1)^i\cdot i\cdot\dim_\br
H^i_{\Ar,c}(\X,\tr(n)).\]
\item[d)]  If $\zeta^*(\X,n)\in\br$ denotes the leading Taylor-coefficient of $\zeta(\X,n)$ at $s=n$
then
\begin{equation}|\zeta^*(\X,n)|^{-1}=\prod_{i\in\bz}\left(\vol(H^i_{\Ar,c}(\X,\tr/\bz(n)))\right)
^{(-1)^{i}}.
\label{tam3}\end{equation}
\end{itemize}

We explain the right hand side. There is no well defined measure on the individual groups $H^i_{\Ar,c}(\X,\tr/\bz(n))$ but only on the entire complex, in the following sense. One has an isomorphism of $\br$-vector spaces
\begin{equation} H^i_{W,c}(\mathcal{X},\mathbb{Z}(n))_\br\cong T_\infty H^i_{\Ar,c}(\X,\tr/\bz(n)),\label{tangentiso}\end{equation}
where $A_\br:=A\otimes_\bz\br$ and $T_\infty G$ denotes the tangent space of a compact commutative Lie group $G$. A Haar measure on $G$ is given by a volume form, i.e. a nonzero section $s\in\det_\br T_\infty G$, up to sign. The volume of $G$ with respect to this measure equals $|\coker(\exp)|\lambda$ where $\exp:T_\infty G\to G$ is the exponential map and $\lambda\in\br^{>0}$ is such that $\det_\bz(\ker(\exp))=\bz\cdot\lambda s$. One can extend the tangent space functor $T_\infty$ to complexes of locally compact abelian groups (see Remark \ref{tangentremark}) and the image of (\ref{tri0}) under the tangent space functor identifies with an exact triangle
\begin{equation} R\Gamma(\X_{Zar},L\Omega_{\X/\bz}^{<n})_\br[-2]\to R\Gamma_{\Ar,c}(\X,\tr(n))\to R\Gamma_{W,c}(\mathcal{X},\mathbb{Z}(n))_\br\to \label{tangenttri}\end{equation}
in the derived category of $\br$-vector spaces. We obtain an isomorphism
\begin{align} \bigotimes_{i\in\bz}{\det}_\br^{(-1)^i}T_\infty H^i_{\Ar,c}(\X,\tr/\bz(n))\cong &{\det}_\br R\Gamma_{\Ar,c}(\X,\tr(n))\otimes_\br {\det}_\br R\Gamma(\X_{Zar},L\Omega_{\X/\bz}^{<n})_\br[-1]\label{tangentcomplex}\\
\cong & {\det}_\br R\Gamma(\X_{Zar},L\Omega_{\X/\bz}^{<n})_\br[-1]\notag\end{align}
where the trivialization ${\det}_\br R\Gamma_{\Ar,c}(\X,\tr(n))\cong\br$ is induced by the exact sequence (\ref{thetaseq}). Here
$R\Gamma(\X_{Zar},L\Omega_{\X/\bz}^{<n})$ is derived deRham cohomology as defined by Illusie \cite{Illusie72} modulo the $n$-th step in its Hodge filtration. A generator of the $\bz$-line ${\det}_\bz R\Gamma(\X_{Zar},L\Omega_{\X/\bz}^{<n})[-1]$, multiplied with a certain correction factor $C(\X,n)\in\bq^\times$, gives a unique section of (\ref{tangentcomplex}), up to sign, which we view as a "volume form" on the complex $R\Gamma_{\Ar,c}(\X,\tr/\bz(n))$. Now it turns out that the isomorphism (\ref{tangentiso}) is induced by an exact triangle
\begin{equation} R\Gamma_{W,c}(\mathcal{X},\mathbb{Z}(n))\to R\Gamma_{W,c}(\mathcal{X},\mathbb{Z}(n))_\br\xrightarrow{\exp} R\Gamma_{\Ar,c}(\X,\tr/\bz(n))\to\label{exptriangle}\end{equation}
and hence the volume, i.e. the right hand side, in (\ref{tam3}) is the unique $\lambda\in\br^{>0}$ so that
\begin{equation} {\det}_\bz R\Gamma_{W,c}(\mathcal{X},\mathbb{Z}(n))=\lambda\cdot C(\X,n) \cdot{\det}_\bz R\Gamma(\X_{Zar},L\Omega_{\X/\bz}^{<n})[-1].\label{tamid}\end{equation}
This is an identity of invertible $\bz$-submodules of the invertible $\br$-module
\[ {\det}_\br R\Gamma_{W,c}(\mathcal{X},\mathbb{Z}(n))_\br\cong \bigotimes_{i\in\bz}{\det}_\br^{(-1)^i}T_\infty H^i_{\Ar,c}(\X,\tr/\bz(n))\cong {\det}_\br R\Gamma(\X_{Zar},L\Omega_{\X/\bz}^{<n})_\br[-1].\]
At this point one can draw the connection to the Tamagawa number conjecture of Fontaine and Perrin Riou \cite{fpr91}. At least if $\X\to\Spec(\co_F)$ is smooth proper over a number ring it turns out that
\[ \Delta(\X_\bq,n):={\det}_\bq R\Gamma_{W,c}(\mathcal{X},\mathbb{Z}(n))_\bq\otimes_\bq {\det}_\bq R\Gamma(\X_{Zar},L\Omega_{\X/\bz}^{<n})_\bq\]
is the fundamental line of Fontaine-Perrin-Riou for the motive $h(\X_\bq)(n)=\bigoplus_{i=0}^{d-1}h^i(\X_\bq)(n)[-i]$ of the generic fibre of $\X$ with trivialization
\begin{equation}\br\xrightarrow{\sim}\Delta(\X_\bq,n)_\br\label{ourtheta}\end{equation} induced by (\ref{tangenttri}) and (\ref{thetaseq}). An element $\lambda\in\br$ maps to $\Delta(\X_\bq,n)$ under this trivialization if and only if it satisfies (\ref{tamid}) up to factors in $\bq^\times$.

We make a few more remarks about this construction.

\begin{itemize}
\item[1)] If $\X\to\Spec(\bF_p)$ is smooth proper over a finite field, and $n\in\bz$ is arbitrary, one has isomorphisms
\[ R\Gamma_{\Ar,c}(\X,\bz(n))\cong R\Gamma_{W,c}(\mathcal{X},\mathbb{Z}(n))\cong R\Gamma(\mathcal{X}_W,\mathbb{Z}(n))\]
where $\mathcal{X}_W$ is the Weil-\'etale topos associated to $\X$ by Lichtenbaum \cite{Lichtenbaum05} and our standard assumptions include the perfectness of these complexes ("Tate conjecture"). Moreover, one has an isomorphism
\[ R\Gamma_{\Ar,c}(\X,\tr(n))\cong R\Gamma_{\Ar,c}(\X,\bz(n))\otimes_\bz\br.\]
Assuming perfectness, properties a)-d) are all true since they are a straightforward reformulation of those proved by Lichtenbaum and Geisser (see \cite{Geisser04b}[Thm. 9.1]). Under a resolution of singularities assumption, one also has an isomorphism
\[ R\Gamma_{\Ar,c}(\X,\bz(n))\cong R\Gamma(\mathcal{X}_W,\mathbb{Z}(n))\cong R\Gamma_c(\mathcal{X}_{\Ar},\mathbb{Z}(n))\]
with the "arithmetic cohomology" groups defined by Geisser in \cite{Geisser06}, i.e. the cohomology groups of the (large) Weil-eh site associated to $\X$ \cite{Geisser06}[Cor. 5.5]. So our notation is consistent with \cite{Geisser06} even though we do not call our groups "arithmetic" for reasons explained below. Also note that the purpose of \cite{Geisser06} was to generalize Weil-etale cohomology from smooth proper to arbitrary arithmetic schemes over $\bF_p$ whereas in this paper we generalize from smooth proper schemes over $\bF_p$ to regular, proper schemes over $\bz$. We will have nothing to say about arithmetic schemes that are not regular or not proper.
\item[2)] We have denoted our complexes $R\Gamma_{\Ar,c}(\X,\bz(n))$, resp. $R\Gamma_{W,c}(\mathcal{X},\mathbb{Z}(n))$, rather than $R\Gamma_c(\X_{\Ar},\bz(n))$, resp. $R\Gamma_{c}(\mathcal{X}_W,\mathbb{Z}(n))$, since we do not define a topos (or an $\infty$-topos, or even an $\infty$-category) $\X_{\Ar}$, resp. $\X_W$, associated to the scheme $\X$ whose compact support cohomology with appropriately defined $\bz(n)$-coefficients gives rise to those complexes. In fact, even for $\X=\Spec(\co_F)$ and $n=0$ this remains a major open problem and we do not know whether to expect the existence of such a topos. It seems somewhat more likely that the groups $R\Gamma_{\Ar,c}(\X,\bz(n))$ will be associated to some geometric object $\X_{\Ar}$. For example, one expects the existence of a class $\theta\in H^1(\X_{\Ar},\tr)$ so that cup-product with $\theta$ produces the exact sequence (\ref{thetaseq}). Our construction of (\ref{thetaseq}) will be entirely formal.  Our choice of notation
 $R\Gamma_{W,c}(\mathcal{X},\mathbb{Z}(n))$ was motivated by the idea that "Weil-\'etale" cohomology complexes should always be perfect complexes of abelian groups.
%\item[3)] If $\X\to\Spec(\bz)$ is flat and $n\geq 1$ the complex $R\Gamma_{\Ar,c}(\X,\bz(n))$ may fail to be perfect. In fact the groups $H^i_{\Ar,c}(\X,\bz(n))$ may fail to carry a natural locally compact topology. Somewhat similarly to Deligne cohomology groups $H^i_\cd(\X_{/\br},\bz(n))$ for $i>2n$ they may contain a copy of a group $\br/(\bz+\bz\alpha)$ with $\alpha$ irrational. This made it all the more surprising to us that the groups $H^i_{\Ar,c}(\X,\tr/\bz(n))$ always carry a natural compact topology for any $i$ and $n$ although, in the absence of a geometric theory of $\X_\Ar$, this compactness will be an entirely formal consequence of our definition of $H^i_{\Ar,c}(\X,\tr/\bz(n))$.
\item[3)] The use of derived rather than naive de Rham cohomology tends to simplify the correction factor $C(\X,n)$. One has $C(\X,n)=1$ for $n\leq 0$ or if $\X$ is smooth proper over a finite field \cite{Morin15}, and $C(\Spec(\co_F),n)=(n-1)!^{-[F:\bq]}$ for a number field $F$ and $n\geq 1$ (see Prop. \ref{Ccompute} below). The correction factor would be $|D_F|^{1-n}\cdot(n-1)!^{-[F:\bq]}$ for naive deRham cohomology where $D_F$ is the discriminant. In general $C(\X,n)$ is defined using $p$-adic Hodge theory.
\end{itemize}

For $\X$ regular, proper and flat over $\Spec(\bz)$, any $n\in\bz$ and $\F=\bz(n),\tr(n),\tr/\bz(n)$ we shall also construct an exact triangle
\[ R\Gamma_{\Ar,c}(\X,\F)\to R\Gamma_{\Ar}(\overline{\X},\F)\to R\Gamma_{\Ar}(\X_\infty,\F)\to\]
as well as an exact triangle of perfect complexes of abelian groups
\[ R\Gamma_{W,c}(\X,\bz(n))\to R\Gamma_{W}(\overline{\X},\bz(n))\to R\Gamma_{W}(\X_\infty,\bz(n))\to.\]
Here one should think of $\overline{\X}$ as the Artin-Verdier or Arakelov compactification of $\X$ and $\X_\infty$ as the fibre at infinity but, as in remark 2) above, we shall have nothing to say about actual geometric objects $\overline{\X}_{\Ar}$, $\overline{\X}_{W}$, $\X_{\infty,\Ar}$ or $\X_{\infty,W}$. In the construction of
$R\Gamma_{W}(\overline{\X},\bz(n))$ we do however make use of the classical Artin-Verdier etale topos $\overline{\X}_\et$ \cite{av} associated to $\X$ since it has the right duality properties needed in this construction. Under our standard assumptions on $\X$, the complexes associated to $\overline{\X}$ satisfy some remarkable duality properties.
There is an isomorphism of perfect complexes of abelian groups
$$R\Gamma_W(\overline{\mathcal{X}},\mathbb{Z}(n))\rightarrow  R\mathrm{Hom}_{\mathbb{Z}}(R\Gamma_W(\overline{\mathcal{X}},\mathbb{Z}(d-n)),\mathbb{Z}[-2d-1])$$
a perfect duality of finite dimensional $\br$-vector spaces
\begin{equation}H^{i}_{\Ar}(\overline{\mathcal{X}},\tr(n))\times H^{2d+1-i}_{\Ar}(\overline{\mathcal{X}},\tr(d-n))\rightarrow H^{2d+1}_{\mathrm{\Ar}}(\overline{\mathcal{X}},\tr(d))\simeq \mathbb{R}\label{xbarduality}\end{equation}
for any $i,n\in\mathbb{Z}$
and a Pontryagin duality of locally compact abelian groups
$$H^{i}_{\Ar}(\overline{\mathcal{X}},\mathbb{Z}(n))\times H^{2d+1-i}_{\Ar}(\overline{\mathcal{X}},\tr/\mathbb{Z}(d-n))\rightarrow H^{2d+1}_{\Ar}(\overline{\mathcal{X}},\tr/\mathbb{Z}(d))\simeq \mathbb{R}/\mathbb{Z}$$
for any $i,n\in\mathbb{Z}$.
One has an isomorphism $$H^{2n}_{\Ar}(\overline{\mathcal{X}},\tr(n))\cong CH^n(\overline{\X})_\br$$ with the Arakelov Chow groups defined by Gillet and Soule \cite{gs94}[3.3.3] and there is also a close relation between $H^{2n}_{\Ar}(\overline{\mathcal{X}},\bz(n))$ and $CH^n(\overline{\X})$ defined in \cite{gs90}[5.1] (note that $CH^n(\overline{\X})_\br$ does {\em not} denote $CH^n(\overline{\X})\otimes_\bz\br$. The two groups $CH^n(\overline{\X})$ and $CH^n(\overline{\X})_\br$ are rather Arakelov modifications of classical Chow groups with $\bz$ and $\br$-coefficients, respectively).  The $\br$-vector spaces $H^{i}_{\Ar}(\overline{\mathcal{X}},\tr(n))$ vanish for $i\neq 2n,2n+1$ and one has an isomorphism
\[ H^{2n}_{\Ar}(\overline{\mathcal{X}},\tr(n))\xrightarrow{\cup\theta}H^{2n+1}_{\Ar}(\overline{\mathcal{X}},\tr(n)).\]
In this regard the spaces $H^{i}_{\Ar}(\overline{\mathcal{X}},\tr(n))$ behave completely analogous to the Weil-etale cohomology spaces
\[ H^i(Y_W,\br(n))\cong H^i(Y_W,\bz(n))\otimes_\bz\br\]
for a smooth projective variety $Y$ over a finite field. In fact, there are also analogues of Grothendieck's standard conjectures for $H^{2n}_{\Ar}(\overline{\mathcal{X}},\tr(n))\cong CH^n(\overline{\X})_\br$ in the literature \cite{kuennemann95}.
Because of this relation to Arakelov theory we call our groups Weil-Arakelov cohomology. Using the terminology "arithmetic" would be confusing since arithmetic Chow groups as defined in \cite{gs90} differ from Arakelov Chow groups.

It is fairly easy to prove the analogue of c)
\begin{equation} \ord_{s=n}\zeta(\overline{\X},s)=\sum_{i\in\bz}(-1)^i\cdot i\cdot\dim_\br H^{i}_{\Ar}(\overline{\mathcal{X}},\tr(n))
\notag\end{equation}
for the {\em completed} Zeta-function
$$\zeta(\overline{\X},s)=\zeta(\X,s)\zeta(\X_\infty,s)$$
of $\X$ provided that c) holds for $\zeta(\X,s)$. Here $\zeta(\X_\infty,s)$ is the usual product of $\Gamma$-functions.  As far as we know this conjectural relation between $CH^n(\overline{\X})_\br$ and the pole order of the completed Zeta-function has not been noticed in the literature. However, there is no statement d) for $\zeta(\overline{\X},s)$ as the groups $H^i_{\Ar}(\overline{\X},\tr/\bz(n))$ are in general only locally compact. This is somewhat consistent with the fact that there are no special value conjectures for the completed Zeta-function in the literature.

We refer the reader to section \ref{sec:examples} below for computations of Weil-Arakelov and Weil-\'etale cohomology groups in the case $\X=\Spec(\co_F)$.

We now give a brief outline of the paper. In section \ref{beilsec} we give a formulation of Beilinson's conjectures for arithmetic schemes rather than motives over $\bq$. This also has been done independently by Scholbach \cite{schol12}. Our definition of motivic cohomology throughout the paper will be via (etale hypercohomology of) Bloch's higher Chow complex \cite{bloch86}. Our formulation of Beilinson's conjectures will be a simple duality statement, Conjecture ${\bf B}(\mathcal{X},n)$, which includes finite dimensionality of motivic cohomology tensored with $\br$.

In section \ref{sec:weiletale} we construct the Weil-etale cohomology complexes following the model of \cite{Morin14}. We state a finite generation assumption on etale motivic cohomology, Conjecture ${\bf L}(\overline{\X}_{et},n)$, which will play a key role in the rest of the paper. We need one further assumption,  Conjecture ${\bf AV}(\overline{\mathcal{X}}_{et},n)$, which concerns Artin-Verdier-duality for motivic cohomology with torsion coefficients and is known in many more cases than either  ${\bf L}(\overline{\X}_{et},n)$ or ${\bf B}(\mathcal{X},n)$.

In section \ref{sec:arithmetic} we construct the Weil-Arakelov cohomology complexes without any further assumption.

In section \ref{sec:zeta} we define the correction factor $C(\X,n)$ and formulate our conjectures on the vanishing order and leading Taylor coefficient of the Zeta function. The rational number $C(\X,n)$ is defined as a product of its $p$-primary parts and for each prime $p$ we need one further assumption, Conjecture ${\bf D}_p(\mathcal{X},n)$, which relates $p$-adically completed motivic cohomology of $\X_{\bz_p}$ with deRham cohomology and $p$-adically completed motivic cohomology of $\X_{\bF_p}$. For smooth schemes and $n<p-1$ such a description follows from the relation between syntomic and motivic cohomology proved by Geisser \cite{Geisser04a}. For general regular $\X$ we expect a similar relationship and we isolate in App. B the results in $p$-adic Hodge theory which would be needed to prove ${\bf D}_p(\mathcal{X},n)$ in general. In view of recent progress in the theory of syntomic cohomology (\cite{nekniz13}, \cite{colniz15}, \cite{ertlniziol16}) these results hold for smooth schemes and might be within reach for semistable schemes.
In subsection \ref{sec:compatibility} we prove that our conjecture is equivalent to the Tamagawa number conjecture of Bloch, Kato, Fontaine, Perrin-Riou for all primes $p$ if $\X$ is smooth over the ring of integers in a number field. This proof also draws on the results of App. B.  In subsection \ref{sec:examples} we discuss in detail the case $\X=\Spec(\co_F)$ for a number field $F$.

In App. A we discuss in detail the Artin-Verdier etale topos $\overline{\X}_{et}$ associated to $\X$, we construct motivic complexes
$\bz(n)^{\overline{\X}}$ on $\overline{\X}_{et}$ for any $n\in\bz$ and we prove conjecture ${\bf AV}(\overline{\mathcal{X}}_{et},n)$ in many cases. The main novelty is a complete discussion of $2$-primary parts.

In App. B. we outline the expected relation between $p$-adically completed motivic cohomology and syntomic cohomology for regular $\X$ and we discuss the motivic decomposition of $p$-adically completed motivic cohomology.
\medskip

{\em Acknowledgements:} We would like to thank W. Niziol and S. Lichtenbaum for helpful discussions related to this paper.

\section{Motivic cohomology of proper regular schemes and the Beilinson conjectures}
\label{beilsec}

Throughout this section $\X$ denotes a regular scheme of dimension $d$, proper over $\Spec(\bz)$. For any complex of abelian groups $A$ we set $A_\br:=A\otimes_\bz\br$.

\subsection{The Beilinson regulator on the level of complexes} We consider Bloch's higher Chow complex \cite{bloch86}
$$\mathbb{Z}(n):=z^n(-,2n-*)$$
which is in fact a complex of sheaves in the etale topology on $\X$.
The first construction of a map of complexes
\begin{equation} z^n(\X,2n-*)\to R\Gamma_\cd(\X_{/\br},\br(n)) \label{gonregulator}\end{equation}
inducing the Beilinson regulator map
\[ CH^n(\X,i)\to H^{2n-i}_\cd(\X_{/\br},\br(n)) \]
was given by Goncharov in \cite{Goncharov95} and \cite{Goncharov05}. This was refined to a map of complexes
\begin{equation} z^n(\X,2n-*)\to R\Gamma_\cd(\X_{/\br},\bz(n)) \label{kmsregulator}\end{equation}
in \cite{kms06} and, building on this construction, the thesis of Fan \cite{fan15} gives a map of complexes
\begin{equation} R\Gamma(\X_{et},\bz(n)) \to R\Gamma_\cd(\X_{/\br},\bz(n)).\label{etaleregulator}\end{equation}
The mapping fibre of this map will play a role in the construction of $R\Gamma_{\Ar,c}(\X,\bz(n))$ in section \ref{sect-arcompactsupport}
and the mapping fibre of the composite map
\begin{equation} R\Gamma(\X_{et},\bz(n)) \to R\Gamma_\cd(\X_{/\br},\bz(n))\to R\Gamma(G_\br,\X(\bc),(2\pi i)^n\bz)\label{bettiregulator}\end{equation}
in the construction of $R\Gamma_{W,c}(\X,\bz(n))$  in section \ref{sect-compactsupport}.
For the remainder of this section we shall only consider  the hypercohomology of $\bz(n)$
$$R\Gamma(\X,\br(n)):=R\Gamma(\X_{et},\bz(n))_\br\cong R\Gamma(\X_{Zar},\bz(n))_\br$$
tensored with $\br$. By definition $R\Gamma(\X,\br(n))=0$ for $n<0$.
We denote by $R\Gamma_c(\X,\br(n))$ the mapping fibre of the Beilinson regulator map so that there is an exact triangle of complexes of $\br$-vector spaces
\begin{equation} R\Gamma_c(\X,\br(n))\to R\Gamma(\X,\br(n)) \to R\Gamma_\cd(\X_{/\br},\br(n))\to .\label{rgc}\end{equation}
Recall that for any $n\in\bz$ Deligne cohomology is defined as the $G_\br$-equivariant cohomology of the complex
\[ \br(n)_\cd: \br(n)\to \Omega_{\X(\bc)/\bc}^\bullet/F^n\]
of $G_\br$-equivariant sheaves on the $G_\br$-space $\X(\bc)$, where $\br(n)$ is the constant $G_\br$-equivariant sheaf $\br(n):=(2\pi i)^n\cdot\br$. So we have
\[ R\Gamma_\cd(\X_{/\br},\br(n)):=R\Gamma(G_\br,\X(\bc), \br(n)_\cd)\cong R\Gamma(\X(\bc), \br(n)_\cd)^{G_\br}.\]
For $n<0$ we have $R\Gamma(\X,\br(n)))=0$ and
\[R\Gamma_c(\X,\br(n)))\cong R\Gamma_\cd(\X_{/\br},\br(n))[-1]\cong R\Gamma(\X(\bc),\br(n))^{G_\br}[-1].\]
The mapping fibre of the Beilinson regulator to real Deligne cohomology (usually without tensoring the source with $\br$) has been denoted "Arakelov motivic cohomology" in \cite{Goncharov05}, \cite{holmstrom-scholbach15}, and a slightly modified mapping fibre yields the "higher arithmetic Chow groups" of \cite{burgosf08} which generalize the arithmetic Chow groups of \cite{gs90}.

\subsection{The Beilinson conjectures and arithmetic duality with $\br$-coefficients}

The purpose of this subsection is to give a uniform statement (for the
central, the near central and the other points) of Beilinson's
conjectures relating motivic to Deligne cohomology, including non-degeneracy of a height pairing. Our statement, Conjecture ${\bf B}(\mathcal{X},n)$ below,
has the form of a simple duality between motivic
cohomology with $\br$-coefficients and Arakelov motivic cohomology
with $\br$-coefficients. Such a formulation of Beilinson's conjectures is implicit in the six term sequence of Fontaine and Perrin-Riou \cite{fpr91}[Prop. 3.2.5].
However, both the original Beilinson conjectures \cite{schneider88} and \cite{fpr91} work with motives over $\bq$ rather than arithmetic schemes. The origin of his section is an unpublished note of the first author from the early 1990's transposing the ideas of \cite{fpr91} to arithmetic schemes. Meanwhile, a formulation of Beilinson's conjectures as a duality theorem for objects in the triangulated category of motives $DM(\Spec(\bz))$ (which includes arithmetic schemes but also the intermediate extension of motives over $\bq$ if one assumes the existence of a motivic $t$-structure) has been given by Scholbach in \cite{schol12}.

\begin{prop} For $n,m\in\bz$ there is a product
\begin{equation} R\Gamma(\X,\br(n))\otimes^L
R\Gamma_c(\X,\br(m))\rightarrow
R\Gamma_c(\X,\br(n+m))\label{rpairing}\end{equation}in the derived category of $\br$-vector spaces.
\label{productprop}\end{prop}

\begin{proof} For the (regular) arithmetic scheme $f:\X\to\Spec(\bz)$, a spectral sequence from $H^\bullet(\Spec(\bz)_{Zar},f_*z_{d-n}(-,\bullet))$ to algebraic K-groups was constructed by Levine in \cite{Levine01}[(8.8)] and it was shown to degenerate after $\otimes\bq$ in \cite{Levine99}[Thm. 11.8]. This gives isomorphisms
\[ H^i(\X_\et,\bz(n))_\bq\cong H^i(\X_{Zar},\bz(n))_\bq\cong H^i(\Spec(\bz)_{Zar},f_*z_{d-n}(-,\bullet))_\bq\cong K_{2n-i}(\X)_\bq^{(n)}\]
where the third isomorphism is \cite{Geisser04a}[Cor.3.3 b)]. By \cite{cisdeg09}[Cor. 14.2.14] there is an isomorphism
\[K_{2n-i}(\X)_\bq^{(n)} \cong \Hom_{DM_\beil(\X)}(\bq(0),\bq(n)[i])\cong \Hom_{SH(\X)_\bq}(S^0,H_{\beil,\X}(n)[i]) \]
where $DM_\beil(\X)$, resp. $SH(\X)_\bq$, is the triangulated category of mixed motives, resp. $\bq$-linear stable homotopy category defined by Cisinski and Deglise, resp. Morel-Voevodsky. Now note that $SH(\X)_\bq$ is naturally enriched over the derived category of $\bq$-vector spaces and hence we get an isomorphism
\[R\Gamma(\X_{et},\bz(n))_\bq\cong R\Hom_{SH(\X)_\bq}(S^0,H_{\beil,\X}(n))\]
and a similar $\br$-linear variant
\[R\Gamma(\X,\br(n))\cong R\Hom_{SH(\X)_\br}(S^0,H_{\beil,\X,\br}(n)).\]
The spectrum $H_{\beil,\X,\br}$ is a strict ring spectrum and in \cite{holmstrom-scholbach15}[Def. 4.1, Rem. 4.6] an exact triangle
\[ \hat{H}_{\beil,\X,\br} \to H_{\beil,\X,\br} \xrightarrow{\rho}  H_\cd\to\]
in $SH(\X)_\br$ was constructed, where $H_\cd$ is a ring spectrum representing real Deligne cohomology and $\rho$ induces the Beilinson regulator \cite{scholbach15}[Thm. 5.7]. This gives an isomorphism
\[R\Gamma_c(\X,\br(n))\cong R\Hom_{SH(\X)_\br}(S^0,\hat{H}_{\beil,\X,\br}(n)).\]
The map $\rho$ is a map of ring spectra which implies that $\hat{H}_{\beil,\X,\br}$ acquires a structure of $H_{\beil,\X,\br}$-module spectrum \cite{holmstrom-scholbach15}[Thm. 4.2 (ii)]. The product map
\[ H_{\beil,\X,\br}(n)\wedge \hat{H}_{\beil,\X,\br}(m)\to \hat{H}_{\beil,\X,\br}(n+m)\]
induces a map
\[ [S^0,H_{\beil,\X,\br}(n)]\otimes^L [S^0,\hat{H}_{\beil,\X,\br}(m)] \to [S^0\wedge S^0, \hat{H}_{\beil,\X,\br}(n+m)]\cong [S^0,\hat{H}_{\beil,\X,\br}(n+m)]\]
where we have written
\[[S^0,A]:=R\Hom_{SH(\X)_\br}(S^0,A)\]
for brevity. In view of the isomorphisms $R\Gamma(\X,\br(n))\cong [S^0,H_{\beil,\X,\br}(n)]$ and
$R\Gamma_c(\X,\br(m)))\cong [S^0,\hat{H}_{\beil,\X,\br}(m)]$ we obtain the product (\ref{rpairing}).
\end{proof}

\begin{rem} The construction of the product on $R\Gamma(X,\br(n))$ makes use of the elaborate formalism of \cite{cisdeg09}, in particular the representability of algebraic K-theory by a ring spectrum in $SH$. We are not aware of a direct construction of a product on higher Chow complexes (even tensored with $\bq$), unless $\X$ is smooth over a number ring or a finite field \cite{Levine99}.  In order to appreciate the amount of detail hidden in the short proof of Prop. \ref{productprop} the reader may want to look at the construction of a product structure in \cite{burgosf08} on the mapping fibre of the Beilinson regulator from the higher Chow complex of the (smooth, proper) generic fibre $\X_\bq$ to real Deligne cohomology. Both \cite{burgosf08} and
\cite{holmstrom-scholbach15} use a representative of the real Deligne complex by differential forms due to Burgos Gil \cite{burgos97} which is quite different from the complex in terms of currents used by Goncharov in (\ref{gonregulator}), and also in (\ref{kmsregulator}) and (\ref{etaleregulator}). It remains an open problem to construct a  $R\Gamma(\X_{et},\bz(n))$-module structure on the mapping fibre of (\ref{etaleregulator}).
\end{rem}

\begin{lem} a) One has
$$H^{i-1}_\mathcal D(\X_{/\br},\br(d))=H^i_c(\X,\br(d))=0$$ for $i>2d$
and there is a commutative diagram
\[\begin{CD} H^{2d-1}_\mathcal D(\X_{/\br},\br(d)) @>>> H^{2d}_c(\X,\br(d))\\
@VVV @VVV\\
\br @= \br
\end{CD}\]

b) For any $i,n\in\bz$ the product on Deligne cohomology
\[ H^i_\cd (\X_{/\br},\br(n)) \times H^{2d-1-i}_\cd(\X_{/\br},\br(d-n))\to H^{2d-1}_\cd(\X_{/\br},\br(d))\to\br\]
induces a perfect duality, i.e. an isomorphism
\begin{equation}
H^i_\cd(\X_{/\br},\br(n))\cong H^{2d-1-i}_\cd(\X_{/\br},\br(d-n))^* \label{deldual}\end{equation}
where $(-)^*$ denotes the dual $\mathbb{R}$-vector space.
\label{deligneduality}\end{lem}

\begin{proof} In this proof we write $H^i_\mathcal D(\X_{/\bc},\br(n))$ for $H^i(\X(\bc),\br(n)_\cd)$. From the long exact sequence
\begin{equation} H^{i-1}(\X(\bc),\bc)/F^d\rightarrow H^i_\mathcal
D(\X_{/\bc},\br(d))\xrightarrow{\alpha_i} H^i(\X(\bc),\br(d))
\notag\end{equation} we find $H^i_\mathcal D(\X_{/\bc},\br(d))=0$ and
hence $H^i_\mathcal D(\X_{/\br},\br(d))=0$ for $i>2d-1$. From the long
exact sequence
\[ H^{i-1}_\mathcal D(\X_{/\br},\br(d)) \rightarrow H^i_c(\X,\br(d))
\rightarrow H^i(\X,\br(d))\] and $H^i(\X,\br(d))\subseteq
K_{2d-i}(\X)_\br^{(d)}=0$ for $i>2d$ we get $H^i_c(\X,\br(d))=0$ for
$i>2d$. For $i=2d$ we find
$H^i(\X,\br(d))\cong\text{CH}^d(\X)_\br=0$, using the finiteness of $\text{CH}^d(\X)$ due to Kato and Saito \cite{ks86}. So $H^{2d}_c(\X,\br(d))$
is the cokernel of the regulator map
\begin{align*} K_1(\X)_\br^{(d)}\rightarrow
H^{2d-1}_\mathcal D(\X_{/\br},\br(d))\cong
&(H^{2d-2}(\X(\bc),\bc)/F^d)^+
\cong &(H^{2d-2}(\X(\bc),\bc))^+\cong\br^{S_\infty}
\end{align*}
where $S_\infty$ is the set of archimedean places of the \'etale
$\bq$-algebra $L:=H^0(\X,\mathcal O_\X)_\bq$ and the last isomorphism
is induced by the cycle classes of points in
$H^{2(d-1)}(\X(\bc),\bc)$. In particular, we get a canonical map to $\br$ given by the sum on
$\br^{S_\infty}$. On the other hand from the
Gersten-Quillen spectral sequence with weights \cite{soule85}[Th.
4(iii))] the group $K_1(\X)^{(d)}$ is generated by
\[ E_2^{d-1,-d}= \text{ker}\left(\coprod_{x\in\X^{d-1}}k(x)^\times \rightarrow \coprod_{x\in\X^{d}}\bz\right)\]
and the regulator is induced by the usual Dirichlet unit regulator
on $L$ composed with $N_{k(x)/L}$ (if $x\in\X$ survives in the
generic fibre of $\X$). So the image of the regulator on
$K_1(\X)^{(d)}$  lies in the subspace $\{(x_v)\in\br^{S_\infty}\vert
\sum_v x_v=0 \}$ and the sum map induces a map form the cokernel $H^{2d}_c(\X,\br(d))$ to $\br$.
\bigskip

For part b) we first prove the following Lemma.

\begin{lem} For any $n\in\bz$ the natural map
\begin{equation}H^i(\X(\bc),\br(n))\to H^i(\X(\bc),\bc)/F^n \label{filmap}\end{equation}
is injective for $i\leq 2n-1$ and surjective for $i\geq 2n-1$.
\end{lem}

\begin{proof} (see also \cite{schneider88}[\S 2] for the injectivity). We can write
\begin{equation} H^i(\X(\bc),\bc)\cong H^i(\X(\bc),\br(n))\oplus H^i(\X(\bc),\br(n-1))\label{taudec}\end{equation} and this is the decomposition of
$H^i(\X(\bc),\bc)$ into the $(-1)^n$ and $(-1)^{n-1}$ eigenspaces
for the involution $\tau$ which is induced by complex conjugation on
the coefficients $\bc$. For $i\leq 2n-1$ we show that $F^n$ contains
no eigenvector for $\tau$. Indeed
\[ F^n\cap\tau F^n=\bigoplus_{n\leq s}H^{s,t}\cap\bigoplus_{n\leq
t}H^{s,t}=\bigoplus_{n\leq s,t}H^{s,t} =0\] for $i=s+t\leq 2n-1$.
For $i\geq 2n-1$ we show
\[ \bigoplus_{s+t=i}H^{s,t} =\bigoplus_{n\leq s}H^{s,t} + \text{any
$\tau$ eigenspace $V^\pm$.}\] Given $s,t$ with $s+t=i\geq 2n-1$ we
have either $n\leq s$ or $n\leq t$, so $F^n=\bigoplus_{n\leq
s'}H^{s',t'}$ contains either $H^{s,t}$ or $H^{t,s}$. We can assume
$s\neq t$ and $F^n$ contains $H^{s,t}$. Since
$\tau(H^{s,t})=H^{t,s}$ we have $V^\pm\cap
(H^{s,t}+H^{t,s})=\{v\pm\tau(v)\vert v\in H^{s,t}\}$ and
$H^{s,t}+H^{t,s}\subseteq H^{s,t}+V$. This finishes the proof of the
lemma.\end{proof}

This Lemma gives short exact sequences
\begin{equation} 0\rightarrow H^{i-1}(\X(\bc),\br(n))\rightarrow
H^{i-1}(\X(\bc),\bc)/F^n\rightarrow H^i_\mathcal
D(\X_{/\bc},\br(n))\rightarrow 0\label{delseq<2n}\end{equation} for $i\leq 2n-1$
and
\begin{equation}0\rightarrow H^i_\mathcal D(\X_{/\bc},\br(n))\rightarrow
H^i(\X(\bc),\br(n))\rightarrow H^i(\X(\bc),\bc)/F^n\rightarrow 0\label{delseq>2n}\end{equation}
for $i\geq 2n$. Taking the $\br$-dual of (\ref{delseq>2n}) and using Poincare
duality we obtain (noting that $\text{dim}(X)=d-1$ and that the
orthogonal complement of $F^n$ is $F^{d-n}$)
\[0\leftarrow H^i_\mathcal D(\X_{/\bc},\br(n))^*\leftarrow
H^{2d-2-i}(\X(\bc),\br(d-1-n))\leftarrow
F^{d-n}H^{2d-2-i}(\X(\bc),\bc)\leftarrow 0.\]
Using (\ref{taudec}) this sequence can also be written as
\[0\leftarrow H^i_\mathcal D(\X_{/\bc},\br(n))^*\leftarrow
H^{2d-2-i}(\X(\bc),\bc)/F^{d-n}\leftarrow
H^{2d-2-i}(\X(\bc),\br(d-n))\leftarrow 0\] and comparing this to
(\ref{delseq<2n}) with $(i,n)$ replaced by $(2d-1-i,d-n)$ we obtain
\begin{equation}
H^i_\mathcal D(\X_{/\bc},\br(n))\cong H^{2d-i-1}_\mathcal
D(\X_{/\bc},\br(d-n))^* .\label{deldual2}\end{equation}
Taking $G_\br$-invariants gives (\ref{deldual}).
\end{proof}
\bigskip

We can now state the main conjecture of this section.

\begin{conj} ${\bf B}(\mathcal{X},n)$ For any $i\in\mathbb{Z}$, the pairing  (\ref{rpairing}) $$H^i_c(\mathcal{X},\mathbb{R}(n))\times H^{2d-i}(\mathcal{X},\mathbb{R}(d-n))\rightarrow H^{2d}_c(\mathcal{X},\mathbb{R}(d))\rightarrow\mathbb{R}$$
is a perfect pairing of finite dimensional $\mathbb{R}$-vector spaces.
\label{ardual}\end{conj}

\begin{rem} By Lemma \ref{deligneduality} a) one has a morphism of long exact sequences
\[ \xymatrix{
\ar[r]& H^i_c(\mathcal{X},\mathbb{R}(n))\ar[d]\ar[r]
&H^i(\mathcal{X},\mathbb{R}(n))\ar[d]_{}^{}\ar[r]
&H^i_{\mathcal{D}}(\mathcal{X}_{/\mathbb{R}},\mathbb{R}(n))\ar[d]\ar[r]&
\\
\ar[r]&  H^{2d-i}(\mathcal{X},\mathbb{R}(d-n))^*\ar[r]
&H^{2d-i}_c(\mathcal{X},\mathbb{R}(d-n))^*\ar[r]
&H^{2d-1-i}_{\mathcal{D}}(\mathcal{X}_{/\mathbb{R}},\mathbb{R}(d-n))^*\ar[r]&
}
\]
and the right hand vertical maps are isomorphism by Lemma \ref{deligneduality} b). Hence the Five Lemma implies
$${\bf B}(\mathcal{X},n)\Leftrightarrow {\bf B}(\mathcal{X},d-n).$$
\end{rem}

\begin{rem} If $\X\to\Spec(\bF_p)$ is smooth proper over a finite field, then
$$H^i_c(\X,\br(n))=H^i(\X,\br(n))$$
and it is expected that $H^i(\X,\br(n))=0$ for $i\neq 2n$ (Parshin's conjecture).
By definition, there is an isomorphism $H^{2n}(\X,\br(n))\cong CH^n(\X)_\br$ and Conjecture ${\bf B}(\mathcal{X},n)$ reduces to perfectness of the intersection pairing
\[ CH^n(\X)_\br\times CH^{d-n}(\X)_\br\to CH^d(\X)_\br\to\br.\]
This is also a conjecture of Beilinson (numerical and rational equivalence coincide).
\end{rem}

\begin{rem} The conjectures of Beilinson \cite{schneider88} concern an "integral
motivic cohomology" of the generic fibre $X:=\X_\bq$ whereas we
work directly with the arithmetic scheme $\X$. More
precisely, there is a long exact localization sequence
\[\cdots \to\bigoplus_{l}CH^{n-1}(\X_l,2n-i)_\bq\to H^i(\X,\bq(n))\xrightarrow{\rho^i(n)}
H^i(X,\bq(n))\to\cdots \] where $\X_l:=\X\otimes\bF_l$, and
Beilinson's conjectures concern \begin{equation}
H^i(X_{/\bz},\bq(n)):=\im(\rho^i(n)) \label{intr}\end{equation}
which is not naturally the cohomology of any complex but which one can show to be independent of the choice of a regular model. However, one expects
\begin{conj} The map $\rho^i(n)$ is injective for $i\neq 2n.$ \label{inj}\end{conj}
This means that the discussion below also applies to the groups
$$H^i(X_{/\bz},\br(n)):=H^i(X_{/\bz},\bq(n))_\br$$ instead of
$H^i(\X,\br(n))$ provided $i\neq 2n$. We refer to Prop. \ref{injprop} and Remark \ref{injrem} below for some evidence for Conjecture \ref{inj}.
\end{rem}
\bigskip

For the rest of this subsection we assume that $\X$ is flat over $\bz$ and we indicate how Conjecture ${\bf B}(\mathcal{X},n)$ is equivalent to Beilinson's conjectures, tacitly assuming Conjecture \ref{inj}. For $i<2n$ we have
\[ H^{2d-i}(\X,\br(d-n))\subseteq K_{i-2n}(\X)_\br^{(d-n)}=0\]
and so ${\bf B}(\mathcal{X},n)$ is equivalent to
\begin{equation}H^i_c(\X,\br(n))=0\quad \text{for $i<2n$.}\label{van1}\end{equation} From the long
exact sequence
\begin{equation} H^i_c(\X,\br(n))\rightarrow
H^i(\X,\br(n))\rightarrow H^i_\mathcal D(\X_{/\br},\br(n))\rightarrow
H^{i+1}_c(\X,\br(n))\label{les}\end{equation} induced by (\ref{rgc})
this is equivalent to
\begin{equation} H^i(\X,\br(n))\cong H^i_\mathcal D(\X_{/\br},\br(n))
\quad\quad\quad\text{for $i<2n-1$.}\label{beil}\end{equation}
For $n\geq 0$ this is Beilinson's conjecture away from the central and near
central point, including the Beilinson-Soul\'e conjecture for $i<0$. For $n<0$ both sides are zero since then also $i<2n-1<0$.
The central and near central point are accounted for by the exact sequence
\begin{multline} 0 \rightarrow H^{2n-1}(\X,\br(n))\xrightarrow{r^n}
H^{2n-1}_\cd(\X_{/\br},\br(n))\xrightarrow{z^{d-n,*}}
H^{2d-2n}(\X,\br(d-n))^*\xrightarrow{h}\\H^{2n}(\X,\br(n))
\xrightarrow{z^n} H^{2n}_\mathcal D(\X_{/\br},\br(n))\rightarrow
H^{2d-2n-1}(\X,\br(d-n))^*\rightarrow 0\label{central}\end{multline}
where we have rewritten $H^{2n}_c$ and $H^{2n+1}_c$ in terms of the dual of $H^i$ using Conjecture ${\bf B}(\mathcal{X},n)$ and we also used $H^{2n+1}(\X,\br(n))=0$. For $n<0$ this sequence is exact since all terms are zero (for $$H^{2d-2n-i}(\X,\br(d-n))\cong K_i(\X)_\br^{(d-n)}$$ and $i=0,1$ this follows for dimension reasons from the Gersten-Quillen spectral sequence with weights \cite{soule85}[Th. 4(iii))]). If $n\geq 0$ the exactness of (\ref{central}), i.e. Conjecture ${\bf B}(\mathcal{X},n)$, is equivalent to
nondegeneracy of the height pairing $h$
\[H^{2n}(\X,\br(n))^0\times H^{2d-2n}(\X,\br(d-n))^0 \to \br\]
on the space
\[H^{2n}(\X,\br(n))^0:=\ker(z^n)\]
together with a decomposition
\begin{equation} H^{2n-1}_\cd(\X_{/\br},\br(n))\cong\im(r^n)\oplus\im(z^{d-n})^*\cong H^{2n-1}(\X,\br(n)) \oplus\im(z^{d-n})^*.\label{nearcentral}\end{equation}
This decomposition is Beilinson's conjecture for the near central point if one assumes the standard conjecture "numerical equals homological equivalence" for the generic fibre $X=\X_\bq$ in which case the dual of
\[ \im(z^{d-n})=CH^{d-n}(\X)_\br/\text{hom}\cong CH^{d-n}(X)_\br/\text{hom}\]
can be computed as
\begin{equation}\im(z^{d-n})^*\cong CH^{d-1-(d-n)}(X)_\br/\text{hom}\cong CH^{n-1}(X)_\br/\text{hom}.
\notag\end{equation}
Beilinson's conjecture at the central point asserts non-degeneracy of the Bloch-Beilinson height pairing on the space
\[ CH^n(X)_\br^0\cong H^{2n}(X,\br(n))^0\cong H^{2n}(X_{/\bz},\br(n))^0\]
of homologically trivial cycles on the generic fibre. Beilinson also conjectures that there is a commutative diagram of pairings
\[\begin{CD}H^{2n}(\X,\br(n))^{00}\times H^{2d-2n}(\X,\br(d-n))^{00} @>>> \br\\
@VVV \Vert@.\\
 H^{2n}(X_{/\bz},\br(n))^0\times H^{2d-2n}(X_{/\bz},\br(d-n))^0 @>>> \br \end{CD}\]
with surjective vertical maps where
\[H^{2n}(\X,\br(n))^{00}\subseteq H^{2n}(\X,\br(n))^{0}\]
is the subgroup of classes homologically trivial on all fibres $\X_{\bF_p}$. This implies that the Bloch-Beilinson height pairing can be defined in terms of the pairing $h$ on $\X$ but we have not tried to investigate whether nondegeneracy of $h$ is equivalent to nondegeneracy of the Bloch-Beilinson pairing.
Finally, in the index range $i>2n$ we have
\[ H^i(\X,\br(n))\subseteq K_{2n-i}(\X)_\br^{(n)}=0\]
and therefore
\[ H^i_\mathcal D(\X_{/\br},\br(n))\cong H^{i+1}_c(\X,\br(n)).\]
Using the duality (\ref{deldual}) we see that ${\bf B}(\mathcal{X},n)$ is equivalent to Beilinson's conjecture (\ref{beil}) with $(i,n)$ replaced by $(2d-1-i,d-n)$, i.e. the second map in the sequence
\[ H^i_\mathcal D(\X_{/\br},\br(n))\cong H^{i+1}_c(\X,\br(n))\cong H^{2d-i-1}(\X,\br(d-n))^*\cong
H^{2d-1-i}_\mathcal D(\X_{/\br},\br(d-n))^*\]
is an isomorphism if and only if the third is.

\subsection{Motivic cohomology of the Artin-Verdier compactification}\label{sec:av}

For any $n\in\bz$ we shall now extend the exact triangle (\ref{rgc}) to a
diagram in the derived category of $\br$-vector
spaces\begin{equation}\begin{CD} R\Gamma_c(\X,\br(n)) @>>>
R\Gamma(\X,\br(n)) @>>>
R\Gamma_\cd(\X_{/\br},\br(n)) @>>>\\
\Vert@. @AAA @AAA\\
R\Gamma_c(\X,\br(n)) @>>> R\Gamma(\overline{\X},\br(n)) @>>>
R\Gamma(\X_\infty,\br(n)) @>>>\\
@. @AAA @AA 0 A\\ {} @. R\Gamma_{\X_\infty}(\overline{\X},\br(n)) @=
R\Gamma_{\X_\infty}(\overline{\X},\br(n)) @.{}
\end{CD}\label{rgc2}\end{equation} with exact rows and columns. Here one should think of $\overline{\X}$ as the
"Artin-Verdier" or "Arakelov" compactification of $\X$ and $\X_\infty$ as the "fibre at infinity".  We choose a splitting $\sigma$ of the inclusion
\begin{equation}\tau^{\leq{2n-1}}R\Gamma_\cd(\X_{/\br},\br(n))\to R\Gamma_\cd(\X_{/\br},\br(n))
\notag\end{equation}
and define $R\Gamma(\overline{\X},\br(n))$ as the mapping fibre of the Beilinson regulator composed with $\sigma$
\begin{equation} R\Gamma(\X,\br(n))\to R\Gamma_\cd(\X_{/\br},\br(n)) \xrightarrow{\sigma} \tau^{\leq{2n-1}}R\Gamma_\cd(\X_{/\br},\br(n)).\label{regsigma}\end{equation}
We then set
\[R\Gamma_{\X_\infty}(\overline{\X},\br(n))[1]:=\tau^{\leq{2n-1}}R\Gamma_\cd(\X_{/\br},\br(n))\]
and
\begin{equation}R\Gamma(\X_\infty,\br(n)):=\tau^{\geq
2n}R\Gamma_\cd(\X_{/\br},\br(n)).
\label{rgxinftydef}
\end{equation}

With these definitions it is clear that the diagram (\ref{rgc2}) has exact rows and columns and that the right hand column is a split exact triangle, i.e. the morphism $0$ is indeed the zero map. The middle horizontal and the middle vertical triangle are analogous to localization triangles in sheaf theory. The right vertical triangle only becomes a "local" localization triangle in the presence of an isomorphism between $R\Gamma(\X_\infty,\br(n))$ and the cohomology of a suitable "tubular neighborhood" of $\X_\infty$.

\begin{prop} Conjecture ${\bf B}(\mathcal{X},n)$ implies that there is a perfect pairing
\begin{equation} H^i(\overline{\X},\br(n))\times
H^{2d-i}(\overline{\X},\br(d-n))\rightarrow
H^{2d}(\overline{\X},\br(d))\to\br\label{dual2}\end{equation} of finite dimensional $\br$-vector spaces
for all $i$ and $n$ and that $H^i(\overline{\X},\br(n))=0$ for $i\neq
2n$.
\label{xbardual}\end{prop}

\begin{proof} By definition of $R\Gamma(\overline{\X},\br(n))$ we have an isomorphism
\begin{equation} H^i_c(\X,\br(n))\cong H^i(\overline{\X},\br(n))\label{eq2} \end{equation}
for $i<2n$ and an isomorphism
\begin{equation} H^i(\overline{\X},\br(n))\cong H^i(\X,\br(n))\label{eq1}\end{equation}
for $i>2n$. Since $i<2n$ implies $2d-i>2d-2n=2(d-n)$ the duality (\ref{dual2})
is an immediate consequence of Conjecture ${\bf B}(\mathcal{X},n)$. Actually
both groups are zero in this case by (\ref{van1}). For $i=2n$ we
have a diagram with exact rows and columns
\begin{equation}\begin{CD}\minCDarrowwidth1em{}@.{}@.0@.{}\\@.@.@AAA\\ { }@>\alpha >>H^{2n}_c(\X,\br(n)) @>>>
H^{2n}(\X,\br(n)) @> z^n>>
H^{2n}_\cd(\X_{/\br},\br(n))\\
@.\Vert@. @AAA \Vert@.\\
0 @>>> H^{2n}_c(\X,\br(n)) @>\iota >> H^{2n}(\overline{\X},\br(n)) @>>>
H^{2n}(\X_\infty,\br(n))\\
@. @A\alpha AA @AAA @.\\{}@.  H^{2n-1}_\cd(\X_{/\br},\br(n)) @= H^{2n}_{\X_\infty}(\overline{\X},\br(n))
@.  \\@. @AA r^n A @AAA @.\\
{}@. H^{2n-1}(\X,\br(n)) @= H^{2n-1}(\X,\br(n)) @. {}
\end{CD}\notag\end{equation}
where $z^n$, $r^n$ and $\alpha\cong (z^{d-n})^*$ are the maps in (\ref{central}) and the injectivity of $\iota$ follows by an easy diagram chase. We obtain
exact sequences
\begin{equation} 0\to H^{2n}_c(\X,\br(n))\to
H^{2n}(\overline{\X},\br(n))\to\im(z^n)\to 0\label{cent1}\end{equation}
and \begin{equation} 0\to \text{coker}(r^n) \to
H^{2n}(\overline{\X},\br(n))\to H^{2n}(\X,\br(n))\to
0.\label{cent2}\end{equation}
In view of the isomorphism $\coker(r^n)^*\cong \im(z^{d-n})$ of (\ref{nearcentral}), Conjecture ${\bf B}(\mathcal{X},d-n)$ implies that there is an isomorphism between the dual of (\ref{cent2}) and
(\ref{cent1}) with $n$ replaced by $d-n$.
\end{proof}

\begin{prop} There is an isomorphism
\begin{equation} H^{2n}(\overline{\X},\br(n))\cong CH^n(\overline{\X})_\br\label{gscomparison}\end{equation}
where $CH^n(\bar{X})_\br$ is the Arakelov Chow group with real coefficients defined by Gillet and Soule \cite{gs94}[3.3.3].
\end{prop}

\begin{proof} One first checks that the proof of the exactness of the sequence
\begin{equation} CH^{n,n-1}(\X)\xrightarrow{\rho} H^{n-1,n-1}(\X_{\br})\xrightarrow{a} CH^n(\overline{\X})\to CH^n(\X)\to 0\label{gssequence}\end{equation}
in \cite{gs90}[Thm. 5.1.2] equally works for (Arakelov) Chow groups made from cycles with real coefficients to give an exact sequence
\[ CH^{n,n-1}(\X)_\br\xrightarrow{\rho} H^{n-1,n-1}(\X_{/\br})\xrightarrow{a} CH^n(\overline{\X})_\br\to CH^n(\X)_\br\to 0.\]
Here
\[ CH^n(X)_\br:=CH^n(\X)\otimes_\bz\br\cong K_0(\X)_\br^{(n)}\cong H^{2n}(\X,\br(n))\]
and
\[CH^{n,n-1}(\X)_\br:=CH^{n,n-1}(\X)\otimes_\bz\br\cong K_1(\X)_\br^{(n)}\cong H^{2n-1}(\X,\br(n)) \]
by the Gersten-Quillen spectral sequence with weights \cite{soule85}[Th. 4(iii))],
and $H^{p,p}(\X_{\br})$ is the space of real differential $(p,p)$ forms $\eta$ on $\X(\bc)$ which are harmonic for the choice of a K\"ahler metric $\omega$ on $\X(\bc)$ and satisfy $F^*_\infty\eta=(-1)^p\eta$ where $F_\infty$ is complex conjugation on $\X(\bc)$.
It is remarked in \cite{gs90}[3.5.3 1)] that
\[H^{n-1,n-1}(\X_{\br})\cong H^{2n-1}_\cd(\X_{/\br},\br(n))\]
and one can see this as follows. One has the exact sequence (\ref{delseq<2n}) for $i=2n-1$
$$ 0\to H^{2n-2}(\X(\bc),\br(n))\to H^{2n-2}(\X(\bc),\bc)/F^n\to H^{2n-1}_\cd(\X_{/\bc},\br(n))\to 0$$
and
$$H^{2n-2}(\X(\bc),\bc)/F^n\cong H^{n-1,n-1}\oplus \bigoplus_{p<n-1}H^{p,2n-2-p}.$$
Denoting by $\tau$ the complex conjugation on coefficients, given $x\in H^{p,2n-2-p}$ with $p<n-1$, we have $x+(-1)^n\tau(x)\in H^{2n-2}(\X(\bc),\br(n))$ and $\tau(x)\in F^n$.  So
$$ H^{2n-2}(\X(\bc),\bc)/F^n = H^{2n-2}(\X(\bc),\br(n)) + H^{n-1,n-1}$$
whereas
$$H^{2n-2}(\X(\bc),\br(n)) \cap H^{n-1,n-1}$$
consists of harmonic $(n-1,n-1)$-forms in the $(-1)^n$eigenspace of $\tau$. So
$$H^{2n-1}_\cd(\X_{/\bc},\br(n))\cong H^{n-1,n-1}/
(H^{2n-2}(\X(\bc),\br(n)) \cap H^{n-1,n-1})$$
consists of harmonic $(n-1,n-1)$-forms in the $(-1)^{n-1}$-eigenspace of $\tau$ and
$$H^{2n-1}_\cd(\X_{/\br},\br(n))=H^{2n-1}_\cd(\X_{/\bc},\br(n))^{G_\br}=H^{n-1,n-1}(\X_{\br})$$
consists of forms $\eta$ satisfying $(F_\infty^*\otimes\tau)\eta=\eta$ and $\tau\eta=(-1)^{n-1}\eta$, i.e. real forms satisfying $F^*_\infty\eta=(-1)^{n-1}\eta$.

It is proved in \cite{gs90}[Thm. 3.5.4] that the map $\rho$ coincides with the Beilinson regulator (up to a constant factor $-2$), hence we obtain an exact sequence
\[ 0\to\coker(\rho)\to CH^n(\overline{\X})_\br\to CH^n(\X)_\br\to 0\]
whose outer terms are isomorphic to the outer terms of (\ref{cent2}), hence there exists an isomorphism on the middle terms.
\end{proof}

\begin{rem} We have "constructed" the pairing (\ref{dual2}) and the isomorphism (\ref{gscomparison}) in an ad hoc way by choosing an isomorphism of middle terms in exact sequences whose outer terms are isomorphic. We expect that there is an isomorphism (\ref{gscomparison}) so that the pairing (\ref{dual2}) is the Arakelov intersection pairing of \cite{gs90}[5.1.4]. By \cite{kuennemann94}[Eq. (18)] the space $H^i_c(\X,\br(n))$ is orthogonal to $\coker(r^{d-n})$ under the Arakelov intersection pairing, and it remains to show that the induced pairing coincides with the pairing (\ref{rpairing}). Assuming finite-dimensionality of $H^{2n}(\X,\br(n))$, the non-degeneracy of the Arakelov intersection pairing is a consequence of the standard conjectures for Arakelov Chow groups \cite{kuennemann95}[Prop. 3.1].
\end{rem}

\begin{rem} The definition of Arakelov Chow groups depends on the choice of a K\"ahler metric on $\X(\bc)$ even though any two choices yield isomorphic groups \cite{gs90}[Thm. 5.1.6]. Recall that the Deligne complex $R\Gamma_\cd(\X_{/\bc},\br(n))$ has a representative which in degrees $2n-1$ and $2n$  looks like \cite{burgos97}[Thm. 2.6]
\[ \cdots \to\cd_\br^{n-1,n-1}(n-1) \xrightarrow{(2\pi i)dd^c} \cd_\br^{n,n}(n)\to\cdots\]
where $\cd_\br^{p,q}(n)$ is the space of $(p,q)$-currents on $\X(\bc)$ tensored by $\br(n)$. A choice of K\"ahler metric also yields a harmonic projection \cite{gs90}[5.1.1]
\[ H: \cd_\br^{n-1,n-1}(n-1)\to H^{n-1,n-1}(\X)(n-1)=\ker(dd^c) \]
and hence a splitting $\sigma$ of the Deligne complex as above.
\end{rem}

\section{Weil-\'etale cohomology of proper regular schemes}\label{sec:weiletale}
Throughout this section, $\mathcal{X}$ denotes a regular scheme of pure dimension $d$, proper over $\mathrm{Spec}(\mathbb{Z})$, and satisfying Conjectures ${\bf L}(\overline{\mathcal{X}}_{et},n)$, ${\bf L}(\overline{\mathcal{X}}_{et},d-n)$ and ${\bf AV}(\overline{\mathcal{X}}_{et},n)$ stated in Section \ref{section-assumptions}.

\subsection{Notations}\label{sect-emc}
For any $n\geq0$, we consider Bloch's cycle complex $$\mathbb{Z}(n):=z^n(-,2n-*)$$ as a complex of sheaves
on the small \'etale topos $\mathcal{X}_{et}$ of the scheme $\mathcal{X}$ (see \cite{Levine01}, \cite{Levine99}, \cite{Geisser04a} and Section \ref{sectAVD} for more details). We write
$\mathbb{Z}/m\mathbb{Z}(n):=\mathbb{Z}(n)\otimes^L\mathbb{Z}/m\mathbb{Z}$
and
$\mathbb{Q}/\mathbb{Z}(n):=\underrightarrow{ \mathrm{lim}}\,\mathbb{Z}/m\mathbb{Z}(n)$. For $n<0$,
we have $\mathbb{Q}(n)=0$ hence
$\mathbb{Z}(n)=\mathbb{Q}/\mathbb{Z}(n)[-1]$. Proper base change and the projective bundle formula suggest
$\mathbb{Z}/p^{r}\mathbb{Z}(n)\simeq j_{p,!}(\mu_{p^{r}}^{\otimes n})$,
where $j_p$ is the open immersion
$j_p:\mathcal{X}[1/p]\rightarrow \mathcal{X}$,
 $j_{p,!}$ is the extension by zero functor and $\mu_{p^{r}}$ is the \'etale sheaf of $p^r$-th roots of unity. This leads to the following definition. For $n<0$ we define the complex $\mathbb{Z}(n)$ on $\mathcal{X}_{et}$ as follows (see also \cite{Geisser04b}):
$$\mathbb{Z}(n):=\bigoplus_{p}j_{p,!}(\mu_{p^{\infty}}^{\otimes n})[-1].$$

The  complexes $\mathbb{Z}(n)^{\overline{\mathcal{X}}}$ and $R\widehat{\phi}_!\mathbb{Z}(n)$ over the Artin-Verdier \'etale topos $\overline{\mathcal{X}}_{et}$ are defined in Appendix A. Recall that there is a  canonical open embedding  $\phi:\mathcal{X}_{et}\rightarrow \overline{\mathcal{X}}_{et}$, where $\mathcal{X}_{et}$ is the usual small \'etale topos of the scheme $\mathcal{X}$. We simply denote by $R\Gamma(\overline{\mathcal{X}}_{et},\mathbb{Z}(n))$ the hypercohomology of the complex $\mathbb{Z}(n)^{\overline{\mathcal{X}}}$ over $\overline{\mathcal{X}}_{et}$. If $\mathcal{X}(\mathbb{R})=\emptyset$, or if one is willing to ignore $2$-torsion issues, one has quasi-isomorphisms  (see Proposition \ref{prop-comp})
$$R\widehat{\phi}_!\mathbb{Z}(n)\stackrel{\sim}{\rightarrow} \mathbb{Z}(n)^{\overline{\mathcal{X}}}\stackrel{\sim}{\rightarrow} R\phi_*\mathbb{Z}(n),$$ hence one may simply define $$R\Gamma(\overline{\mathcal{X}}_{et},\mathbb{Z}(n)):= R\Gamma(\mathcal{X}_{et},\mathbb{Z}(n))$$ where $R\Gamma(\mathcal{X}_{et},\mathbb{Z}(n))$ denotes the hypercohomology of the complex $\mathbb{Z}(n)$ of sheaves over $\mathcal{X}_{et}$. We also denote
$$R\Gamma(\overline{\mathcal{X}}_{et},\widehat{\mathbb{Z}}(n)):=\mathrm{holim} \,R\Gamma(\overline{\mathcal{X}}_{et},\mathbb{Z}/m\mathbb{Z}(n)).$$

If $A$ is an abelian group, we denote by $A_{tor}$ (resp. $A_{div}$) its maximal torsion (resp. divisible) subgroup, and by  $A_{cotor}$ (resp. $A_{codiv}$)  the cokernel of the inclusion $A_{tor}\rightarrow A$ (resp. $A_{div}\rightarrow A$). We denote by $_{n}A$ (resp. $A_n$) the kernel (resp. the cokernel) of the multiplication map  $n:A\rightarrow A$, and by $TA:=\underleftarrow{ \mathrm{lim}}\, _{n}A$ the Tate module of $A$. If $A$ is torsion or profinite, $A^D$ denotes its Pontryagin dual. We say that $A$ is of \emph{cofinite type} if $A$ is of the form $\mathrm{Hom}_{\mathbb{Z}}(B,\mathbb{Q}/\mathbb{Z})$ where $B$ is finitely generated. We denote by $\mathcal{D}$ the derived category of abelian group. More generally, if $T$ is a topos, we denote by $\mathcal{D}(T)$ the derived category of abelian sheaves on $T$. If $C$ is an object of $\mathcal{D}(T)$ then we denote by $C_{\leq n}$ or by $\tau^{\leq n}C$ the good truncation of $C$ in degrees $\leq n$.

\subsection{Assumptions}\label{section-assumptions} The definition of Weil-\'etale cohomology requires the following conjectures.
\begin{conj}\label{etaleduality-modn} ${\bf AV}(\overline{\mathcal{X}}_{et},n)$ There are compatible product maps $\mathbb{Z}(n)^{\overline{\mathcal{X}}}\otimes^L \mathbb{Z}(d-n)^{\overline{\mathcal{X}}}\rightarrow \mathbb{Z}(d)^{\overline{\mathcal{X}}}$
and $R\widehat{\phi}_!\mathbb{Z}(n)\otimes^L R\phi_*\mathbb{Z}(d-n)\rightarrow \mathbb{Z}(d)^{\overline{\mathcal{X}}}$
 in $\mathcal{D}(\overline{\mathcal{X}}_{et})$
inducing a perfect pairing of finite groups
$$H^i(\overline{\mathcal{X}}_{et},\mathbb{Z}/m\mathbb{Z}(n))\times H^{2d+1-i}(\overline{\mathcal{X}}_{et},\mathbb{Z}/m\mathbb{Z}(d-n))\rightarrow
H^{2d+1}(\overline{\mathcal{X}}_{et},\mathbb{Z}/m\mathbb{Z}(d))\rightarrow \mathbb{Q}/\mathbb{Z}.$$
for any $i\in\mathbb{Z}$ and any $m>0$.
\end{conj}
Conjecture ${\bf AV}(\overline{\mathcal{X}}_{et},n)$ holds in the following cases:
\begin{itemize}
\item for any $n$ and $\mathcal{X}$ smooth over a finite field;
\item for any $n\leq 0$ or $n\geq d$ and $\mathcal{X}$ regular;
\item for any $n$ and  $\mathcal{X}$ smooth over a number ring.
\end{itemize}
Indeed, the second case (respectively the third) is Corollary \ref{corAVn=0} (respectively Corollary \ref{cor-AVsmooth}). Assume that $\mathcal{X}$ is smooth over a finite field of characteristic $p$. The result for $m$ prime to $p$ is well known. For $m=p^r$, it follows from $\mathbb{Z}/p^r\mathbb{Z}(n)\simeq \nu^n_r[-n]$ (see \cite{Geisser-Levine-00} Theorem 8.5) and from (\cite{Milne86} Theorem 1.14).

\begin{conj}${\bf L}(\overline{\mathcal{X}}_{et},n)$
The group $H^i(\overline{\mathcal{X}}_{et},\mathbb{Z}(n))$ is finitely generated for $i\leq 2n+1$ and vanishes for $i<<0$.
\end{conj}
Conjecture ${\bf L}(\overline{\mathcal{X}}_{et},n)$ holds in the following cases:
\begin{itemize}
\item for $ \mathrm{dim}(\mathcal{X})\leq 1$;
\item for $n=1$ and $\mathcal{X}$ an arithmetic surface (or a surface over a finite field) with finite Brauer group;
\item for $n\geq d-1$ or $n\leq 1$ and $\mathcal{X}$ in the category $A(\mathbb{F}_q)$ (see Section \ref{subsect-Licht});
\item for $n\geq d$  or $n\leq 0$ and $\mathcal{X}$ a regular cellular scheme over a number ring (more generally for $n\geq d$  or $n\leq 0$ and $\mathcal{X}$ regular in the class $\mathcal{L}(\mathbb{Z})$, see \cite{Morin14}).
\end{itemize}

We have the following slight reformulation.

\begin{lem} Conjecture ${\bf L}(\overline{\mathcal{X}}_{et},n)$ is equivalent to
\smallskip

${\bf L}(\mathcal{X}_{et},n)$: The group $H^i(\mathcal{X}_{et},\mathbb{Z}(n))$ is finitely generated for $i\leq 2n+1$
\smallskip

\end{lem}

\begin{proof} By Corollary \ref{cor-u^!} there is an exact triangle
\[ R\Gamma_{\X_\infty}(\overline{\X},\bz(n))\to R\Gamma(\overline{\X},\bz(n))\to R\Gamma(\X,\bz(n)) \]
where $R\Gamma_{\X_\infty}(\overline{\X},\bz(n))$ has finite 2-torsion cohomology and is bounded below.
Hence finite generation of $H^i(\mathcal{X}_{et},\mathbb{Z}(n))$ and $H^i(\overline{\mathcal{X}}_{et},\mathbb{Z}(n))$ for $i\leq 2n+1$ are equivalent.  If $H^i(\overline{\mathcal{X}}_{et},\mathbb{Z}(n))$ is finitely generated its vanishing is implied by the vanishing of $H^i(\overline{\mathcal{X}}_{et},\mathbb{Z}(n)/p)$ for all primes $p$ and, if $i<<0$, by the vanishing of $H^i(\overline{\mathcal{X}}_{et},\tau^{\leq n}\mathbb{Z}(n)/p)$.
By Prop. \ref{propstalk} and \cite{Zhong14}[Thm. 2.6] there is an isomorphism
\[j^*\phi^*\tau^{\leq n}\mathbb{Z}(n)^{\overline{\X}}/p\cong \tau^{\leq n}\mathbb{Z}(n)/p\cong\mu_p^{\otimes n}\]
where
\[ \mathcal{X}[1/p]_{et}\xrightarrow{j}\X_{et}\xrightarrow{\phi} \overline{\X}_{et}\]
are the natural open immersions. Following the proof of Lemma \ref{locdiagram} below we obtain an exact triangle
\[\minCDarrowwidth1em\begin{CD}R\Gamma_c(\X[1/p],\mu_p^{\otimes n}) @>>> R\Gamma(\overline{\X},\tau^{\leq n}\bz_p(n)/p) @>>> R\Gamma(\X_\br,\mu_p^{\otimes n} )\oplus R\Gamma(\X_{\bz_p},\tau^{\leq n}\bz(n)/p)
\end{CD}\]
where the outer terms have vanishing cohomology for $i<<0$. For the complex $\tau^{\leq n}\bz(n)/p$ on $\X_{\bz_p}$ this follows from the proof of Lemma \ref{loclem} which shows that the mapping fibre of $\tau^{\leq n}\bz(n)/p\to\tau^{\leq n}Rj_{p,*}\mu_p^{\otimes n}$ is quasi-isomorphic to $i_*\bz(n-1)/p[-2]$ in degrees $\leq n$, even without assuming Conj. \ref{zhongconj}. But $\bz(n-1)/p$ on $\X_{\bF_p}$ is cohomologically bounded below by \cite{Zhong14}[Thm. 1.1].
\end{proof}

\subsection{The complex $R\Gamma_W(\overline{\mathcal{X}},\mathbb{Z}(n))$}

\begin{prop}\label{prop-etaleduality}
For any $i\geq 2n+2$ there is an isomorphism of cofinite type groups
$$H^{i}(\overline{\mathcal{X}}_{et},\mathbb{Z}(n))\stackrel{\sim}{\longrightarrow}  \mathrm{\emph{Hom}}(H^{2d+2-i}(\overline{\mathcal{X}}_{et},\mathbb{Z}(d-n)),\mathbb{Q}/\mathbb{Z}).$$
For $i=2n+1$, there is an isomorphism of finite groups
$$H^{2n+1}(\overline{\mathcal{X}}_{et},\mathbb{Z}(n))\stackrel{\sim}{\longrightarrow} H^{2(d-n)+1}(\overline{\mathcal{X}}_{et},\mathbb{Z}(d-n))^D.$$
For any $i\leq 2n$ there is an isomorphism of profinite groups
$$H^{i}(\overline{\mathcal{X}}_{et},\mathbb{Z}(n))^{\wedge}
\stackrel{\sim}{\longrightarrow}  \mathrm{\emph{Hom}}(H^{2d+2-i}(\overline{\mathcal{X}}_{et},\mathbb{Z}(d-n)),\mathbb{Q}/\mathbb{Z}).$$
where $(-)^{\wedge}$ is the profinite completion.
\end{prop}
\begin{proof}
The distinguished triangle
$$\mathbb{Z}(n)\rightarrow \mathbb{Q}(n) \rightarrow \mathbb{Q}/\mathbb{Z}(n)$$
and the fact that $H^i(\overline{\mathcal{X}}_{et},\mathbb{Q}(n))=H^i(\mathcal{X},\mathbb{Q}(n))=0$ for $i\geq 2n+1$ imply
\begin{eqnarray*}
H^i(\overline{\mathcal{X}}_{et},\mathbb{Z}(n))&=&H^{i-1}(\overline{\mathcal{X}}_{et},\mathbb{Q}/\mathbb{Z}(n))\\
&\simeq&\underrightarrow{ \mathrm{lim}}\,H^{i-1}(\overline{\mathcal{X}}_{et},\mathbb{Z}/m\mathbb{Z}(n))\\
&\simeq&\underrightarrow{ \mathrm{lim}}\,H^{2d+1-(i-1)}(\overline{\mathcal{X}}_{et},\mathbb{Z}/m\mathbb{Z}(d-n))^D\\
&\simeq&\left(\underleftarrow{ \mathrm{lim}}\,H^{2d+2-i}(\overline{\mathcal{X}}_{et},\mathbb{Z}/m\mathbb{Z}(d-n))\right)^D\\
&\simeq&H^{2d+2-i}(\overline{\mathcal{X}}_{et},\mathbb{Z}(d-n))^D
\end{eqnarray*}
for $i\geq2n+2$. Indeed, to show the last isomorphism we consider the exact sequence
$$0\rightarrow H^{2d+2-i}(\overline{\mathcal{X}}_{et},\mathbb{Z}(d-n)))_n\rightarrow H^{2d+2-i}(\overline{\mathcal{X}}_{et},\mathbb{Z}/m\mathbb{Z}(d-n)))
\rightarrow {_n}H^{2d+2-i+1}(\overline{\mathcal{X}}_{et},\mathbb{Z}(d-n))\rightarrow 0$$
Passing to the limit we get
$$0\rightarrow \underleftarrow{ \mathrm{lim}}\, H^{2d+2-i}(\overline{\mathcal{X}}_{et},\mathbb{Z}(d-n))_n\rightarrow \underleftarrow{ \mathrm{lim}}\, H^{2d+2-i}(\overline{\mathcal{X}}_{et},\mathbb{Z}/m\mathbb{Z}(d-n))\rightarrow
TH^{2d+2-i+1}(\overline{\mathcal{X}}_{et},\mathbb{Z}(d-n))=0$$
since $H^{2d+2-i+1}(\overline{\mathcal{X}}_{et},\mathbb{Z}(d-n))$ is finitely generated for $i\geq 2n+2$. We obtain isomorphisms
$$(\underleftarrow{ \mathrm{lim}}\,H^{2d+2-i}(\overline{\mathcal{X}}_{et},\mathbb{Z}/m\mathbb{Z}(d-n)))^D
\stackrel{\sim}{\rightarrow}(\underleftarrow{ \mathrm{lim}}\, H^{2d+2-i}(\overline{\mathcal{X}}_{et},\mathbb{Z}(d-n))_n)^D
\stackrel{\sim}{\rightarrow}H^{2d+2-i}(\overline{\mathcal{X}}_{et},\mathbb{Z}(d-n))^D$$
again using finite generation of $H^{2d+2-i}(\overline{\mathcal{X}}_{et},\mathbb{Z}(d-n))$.

It remains to treat the case $i=2n+1$. The exact sequence
$$H^{2n}(\overline{\mathcal{X}}_{et},\mathbb{Q}/\mathbb{Z}(n))
\rightarrow H^{2n+1}(\overline{\mathcal{X}}_{et},\mathbb{Z}(n))
\rightarrow H^{2n+1}(\overline{\mathcal{X}}_{et},\mathbb{Q}(n))=0$$
implies that $H^{2n+1}(\overline{\mathcal{X}}_{et},\mathbb{Z}(n))$ is torsion hence finite. We have an exact sequence
$$0\rightarrow \underleftarrow{ \mathrm{lim}}\, H^{2n+1}(\overline{\mathcal{X}}_{et},\mathbb{Z}(n))_n\rightarrow
\underleftarrow{ \mathrm{lim}}\, H^{2n+1}(\overline{\mathcal{X}}_{et},\mathbb{Z}/m\mathbb{Z}(n))\rightarrow
TH^{2n+2}(\overline{\mathcal{X}}_{et},\mathbb{Z}(n))$$
But $H^{2n+2}(\overline{\mathcal{X}}_{et},\mathbb{Z}(n))$ is of  cofinite type (thanks to the case $i\geq2n+2$ treated above) hence its Tate module is torsion-free. We obtain
\begin{eqnarray*}
H^{2n+1}(\overline{\mathcal{X}}_{et},\mathbb{Z}(n))&=&\underleftarrow{ \mathrm{lim}}\, H^{2n+1}(\overline{\mathcal{X}}_{et},\mathbb{Z}(n))_n\\
&\stackrel{\sim}{\rightarrow}&\left(\underleftarrow{ \mathrm{lim}}\, H^{2n+1}(\overline{\mathcal{X}}_{et},\mathbb{Z}/m\mathbb{Z}(n))\right)_{tors}\\
&\simeq&\left(\underleftarrow{ \mathrm{lim}}\, H^{2d+1-(2n+1)}(\overline{\mathcal{X}}_{et},\mathbb{Z}/m\mathbb{Z}(d-n))^D\right)_{tors}\\
&\simeq&\left(H^{2(d-n)}(\overline{\mathcal{X}}_{et},\mathbb{Q}/\mathbb{Z}(d-n))^D\right)_{tors}\\
&\simeq&H^{2(d-n)+1}(\overline{\mathcal{X}}_{et},\mathbb{Z}(d-n))^D.
\end{eqnarray*}

\end{proof}

\begin{thm}\label{thm-alpha} There is a canonical morphism in $\mathcal{D}$:
$$\alpha_{\mathcal{X},n}:R\mathrm{Hom}( R\Gamma(\mathcal{X},\mathbb{Q}(d-n)),\mathbb{Q}[-2d-2])\rightarrow  R\Gamma(\overline{\mathcal{X}}_{et},\mathbb{Z}(n))$$
functorial in $\mathcal{X}$ and such that $H^i(\alpha_{\mathcal{X},n})$ factors as follows
$$\mathrm{Hom}(H^{2d+2-i}(\mathcal{X},\mathbb{Q}(d-n)),\mathbb{Q})\twoheadrightarrow H^i(\overline{\mathcal{X}}_{et},\mathbb{Z}(n))_{div}\hookrightarrow H^i(\overline{\mathcal{X}}_{et},\mathbb{Z}(n))$$
where $H^i(\overline{\mathcal{X}}_{et},\mathbb{Z})_{div}$ denotes the maximal divisible subgroup of $H^i(\overline{\mathcal{X}}_{et},\mathbb{Z})$.
\end{thm}
\begin{proof}

We set $D_{\mathcal{X},n}:= R\mathrm{Hom}(R\Gamma(\mathcal{X},\mathbb{Q}(d-n)),\mathbb{Q}[-2d-2])$. Using Proposition \ref{prop-etaleduality}, it is easy to see that the spectral sequence
$$\prod_{i\in\mathbb{Z}} \mathrm{Ext}^p(H^{i}(D_{\mathcal{X},n}),H^{q+i}(\overline{\mathcal{X}}_{et},\mathbb{Z}(n)))\Rightarrow
H^{p+q}(R\mathrm{Hom}(D_{\mathcal{X},n}, R\Gamma(\overline{\mathcal{X}}_{et},\mathbb{Z}(n))))$$
yields a canonical isomorphism
$$\prod_{i\in\mathbb{Z}} \mathrm{Ext}^0(H^{i}(D_{\mathcal{X},n}),H^{i}(\overline{\mathcal{X}}_{et},\mathbb{Z}(n)))\simeq
H^{0}(R\mathrm{Hom}(D_{\mathcal{X},n}, R\Gamma(\overline{\mathcal{X}}_{et},\mathbb{Z}(n)))).$$
For $i\leq 2n+1$ any map $H^i(\alpha_{\mathcal{X},n}):H^{i}(D_{\mathcal{X},n})\rightarrow H^i(\overline{\mathcal{X}}_{et},\mathbb{Z}(n))$ must be trivial since $H^{i}(D_{\mathcal{X},n})=0$.
For any $i\geq 2n+2$ there is a canonical map
$$H^i(\alpha_{\mathcal{X},n}):H^{i}(D_{\mathcal{X},n})= \mathrm{Hom}(H^{2d+2-i}(\mathcal{X},\mathbb{Q}(d-n)),\mathbb{Q})\stackrel{\sim}{\rightarrow}
 \mathrm{Hom}(H^{2d+2-i}(\overline{\mathcal{X}}_{et},\mathbb{Z}(d-n)),\mathbb{Q})$$
$$\rightarrow
 \mathrm{Hom}(H^{2d+2-i}(\overline{\mathcal{X}}_{et},\mathbb{Z}(d-n)),\mathbb{Q}/\mathbb{Z})
\stackrel{\sim}{\leftarrow}H^i(\overline{\mathcal{X}}_{et},\mathbb{Z}(n)).$$
Hence there exists a unique map
$$\alpha_{\mathcal{X},n}:R\mathrm{Hom}( R\Gamma(\mathcal{X},\mathbb{Q}(d-n)),\mathbb{Q}[-2d-2])\rightarrow R\Gamma(\overline{\mathcal{X}}_{et},\mathbb{Z}(n))$$
inducing $H^i(\alpha_{\mathcal{X},n})$ on cohomology. The fact that $\alpha_{\mathcal{X},n}$ is functorial is shown in Theorem \ref{cor-functoriality}.

\end{proof}

\begin{defn}\label{defn-fg-cohomology}
We define
$R\Gamma_W(\overline{\mathcal{X}},\mathbb{Z}(n))$, up to non-canonical isomorphism,
such that there is an exact triangle
$$R\mathrm{Hom}(R\Gamma(\mathcal{X},\mathbb{Q}(d-n)),\mathbb{Q}[-\delta])\rightarrow
R\Gamma(\overline{\mathcal{X}}_{et},\mathbb{Z}(n))\rightarrow R\Gamma_W(\overline{\mathcal{X}},\mathbb{Z}(n)).$$
\end{defn}
We shall see below that $R\Gamma_W(\overline{\mathcal{X}},\mathbb{Z}(n))$ is in fact defined up to a canonical isomorphism in the derived category (see Corollary \ref{cor-welldefined}). The long exact sequence of cohomology groups associated to the exact triangle of Definition \ref{defn-fg-cohomology} together with Proposition \ref{prop-etaleduality} yields the following
\begin{lem}\label{lemstructureWetgrps}
There is an exact sequence
$$0\rightarrow H^i(\overline{\mathcal{X}}_{et},\mathbb{Z}(n))_{codiv}\rightarrow
H_W^i(\overline{\mathcal{X}},\mathbb{Z}(n))\rightarrow \mathrm{Hom}(H^{2d+1-i}(\overline{\mathcal{X}}_{et},\mathbb{Z}(d-n)),\mathbb{Z})\rightarrow 0$$
for any $i\in\mathbb{Z}$.
\end{lem}

\begin{prop}\label{finitelygenerated-cohomology}
The group $H_W^i(\overline{\mathcal{X}},\mathbb{Z}(n))$ is finitely generated for any $i\in\mathbb{Z}$. Moreover one has
$H_W^i(\overline{\mathcal{X}},\mathbb{Z}(n))=0$ for almost all $i\in\mathbb{Z}$.
\end{prop}
\begin{proof}
By Proposition \ref{prop-etaleduality}, the group $H^{2d+2-(i+1)}(\overline{\mathcal{X}}_{et},\mathbb{Z}(d-n))$ is either finitely generated or of cofinite type and $H^i(\overline{\mathcal{X}}_{et},\mathbb{Z}(n))_{codiv}$ is finitely generated.
Hence the exact sequence
$$0\rightarrow H^i(\overline{\mathcal{X}}_{et},\mathbb{Z}(n))_{codiv} \rightarrow H_{W}^i(\overline{\mathcal{X}},\mathbb{Z}(n))\rightarrow \mbox{Hom}_{\mathbb{Z}}(H^{2d+2-(i+1)}(\overline{\mathcal{X}}_{et},\mathbb{Z}(d-n)),\mathbb{Z})\rightarrow 0.$$
shows that $H_W^i(\overline{\mathcal{X}},\mathbb{Z}(n))$ is finitely generated. Moreover,
$H^i(\overline{\mathcal{X}}_{et},\mathbb{Z}(n))=0$ for $i<<0$ and $H^{j}(\overline{\mathcal{X}}_{et},\mathbb{Z}(d-n))$ is torsion for $j>2(d-n)$, hence $H_W^i(\overline{\mathcal{X}},\mathbb{Z}(n))=0$ for $i<<0$. For $i$ large, we have
$$H^{i+1}(\overline{\mathcal{X}}_{et},\mathbb{Z}(n))_{codiv}=H^{i}(\overline{\mathcal{X}}_{et},\mathbb{Q}/\mathbb{Z}(n))_{codiv}\simeq \left(H^{2d+1-i}(\overline{\mathcal{X}}_{et},\widehat{\mathbb{Z}}(d-n))^D\right)_{codiv}
=0$$
by $\mathbf{L}(\overline{\mathcal{X}}_{et},d-n)$.

\end{proof}

\begin{thm}\label{cor-functoriality}
Let $f:\mathcal{X}\rightarrow\mathcal{Y}$ be a flat morphism between proper regular arithmetic schemes of pure dimensions $d_{\mathcal{X}}$ and $d_{\mathcal{Y}}$ respectively. Assume that ${\bf L}(\overline{\mathcal{X}}_{et},n)$, ${\bf L}(\overline{\mathcal{X}}_{et},d_{\mathcal{X}}-n)$, ${\bf AV}(\overline{\mathcal{X}}_{et},n)$, ${\bf L}(\overline{\mathcal{Y}}_{et},n)$, ${\bf L}(\overline{\mathcal{Y}}_{et},d_{\mathcal{Y}}-n)$, ${\bf AV}(\overline{\mathcal{Y}}_{et},n)$, and ${\bf AV}(f,n)$ hold (see Section \ref{sectionAVf}). We choose complexes $ R\Gamma_W(\overline{\mathcal{X}},\mathbb{Z}(n))$ and $ R\Gamma_W(\overline{\mathcal{Y}},\mathbb{Z}(n))$ as in Definition \ref{defn-fg-cohomology}. Then there exists
a \emph{unique} map in $\mathcal{D}$
$$f^*_{W}: R\Gamma_W(\overline{\mathcal{Y}},\mathbb{Z}(n))\rightarrow  R\Gamma_W(\overline{\mathcal{X}},\mathbb{Z}(n))$$
which sits in the morphism of exact triangles:
\[ \xymatrix{
R\mathrm{Hom}( R\Gamma(\mathcal{Y},\mathbb{Q}(d_{\mathcal{Y}}-n)),\mathbb{Q}[-\delta_{\mathcal{Y}}])\ar[r]\ar[d]&  R\Gamma(\overline{\mathcal{Y}}_{et},\mathbb{Z}(n))\ar[r]\ar[d]&  R\Gamma_W(\overline{\mathcal{Y}},\mathbb{Z}(n))\ar[r]\ar[d]& \\
R\mathrm{Hom}( R\Gamma(\mathcal{X},\mathbb{Q}(d_{\mathcal{X}}-n)),\mathbb{Q}[-\delta_{\mathcal{X}}])\ar[r]&  R\Gamma(\overline{\mathcal{X}}_{et},\mathbb{Z}(n))\ar[r]&  R\Gamma_W(\overline{\mathcal{X}},\mathbb{Z}(n))\ar[r]&
}
\]
\end{thm}
\begin{proof} We may assume $\mathcal{X}$ and $\mathcal{Y}$ connected; thus $\mathcal{X}$ and $\mathcal{Y}$ are proper regular connected arithmetic schemes of dimension $d_{\mathcal{X}}$ and $d_{\mathcal{Y}}$ respectively. We set
$\delta_{\mathcal{X}}=2d_{\mathcal{X}}+2$ and $\delta_{\mathcal{Y}}=2d_{\mathcal{Y}}+2$. We choose complexes $R\Gamma_W(\overline{\mathcal{X}},\mathbb{Z})$ and $R\Gamma_W(\overline{\mathcal{Y}},\mathbb{Z})$ as in Definition \ref{defn-fg-cohomology}. Let $f:\mathcal{X}\rightarrow\mathcal{Y}$ be a morphism of relative dimension $c=d_{\mathcal{X}}-d_{\mathcal{Y}}$. Proper push-forward of cycles along the proper morphism $f$
$$z^m(\mathcal{X},*)\rightarrow z^{m-c}(\mathcal{Y},*)$$
induces a morphism
$$f_*\mathbb{Q}(d_{\mathcal{X}}-n)\rightarrow \mathbb{Q}(d_{\mathcal{Y}}-n))[-2c]$$
of complexes of abelian Zariski sheaves on $\mathcal{Y}$. We have $f_*\mathbb{Q}(d_{\mathcal{X}}-n)\simeq Rf_*\mathbb{Q}(d_{\mathcal{X}}-n)$. Indeed, for a scheme over a discrete valuation ring, the cohomology of cycle complex coincides with its Zariski hypercohomology. We obtain a morphism
$$R\Gamma(\mathcal{X},\mathbb{Q}(d_{\mathcal{X}}-n))
\simeq R\Gamma(\mathcal{Y},f_*\mathbb{Q}(d_{\mathcal{X}}-n))\rightarrow
R\Gamma(\mathcal{Y},\mathbb{Q}(d_{\mathcal{Y}}-n))[-2c]$$
hence
$$R\mathrm{Hom}(R\Gamma(\mathcal{Y},\mathbb{Q}(d_{\mathcal{Y}}-n)),
\mathbb{Q}[-\delta_{\mathcal{Y}}])
\rightarrow R\mathrm{Hom}(R\Gamma(\mathcal{X},\mathbb{Q}(d_{\mathcal{X}}-n)),
\mathbb{Q}[-\delta_{\mathcal{X}}]).$$
On the other hand, by Proposition \ref{prop-funct-barX}, we have a pull-back map
$$
R\Gamma(\overline{\mathcal{Y}}_{et},\mathbb{Z}(n))\rightarrow R\Gamma(\overline{\mathcal{X}}_{et},\mathbb{Z}(n))
$$
We need to show that the following square
\[ \xymatrix{
R\mathrm{Hom}(R\Gamma(\mathcal{Y},\mathbb{Q}(d_{\mathcal{Y}}-n)),
\mathbb{Q}[-\delta_{\mathcal{Y}}])\ar[r]^{\,\,\,\,\,\,\,\,\,\,\,\,\,\,\,\,\,\,\,\,\,\,\,\,\,\,\,\,\,\,\,\,\,\,\,\,\alpha_{\mathcal{Y},n}}\ar[d]
&R\Gamma(\overline{\mathcal{Y}}_{et},\mathbb{Z}(n))\ar[d]_{}  \\
 R\mathrm{Hom}(R\Gamma(\mathcal{X},\mathbb{Q}(d_{\mathcal{X}}-n)),
\mathbb{Q}[-\delta_{\mathcal{X}}])\ar[r]^{\,\,\,\,\,\,\,\,\,\,\,\,\,\,\,\,\,\,\,\,\,\,\,\,\,\,\,\,\,\,\,\,\,\,\,\,\alpha_{\mathcal{X},n}}
&R\Gamma(\overline{\mathcal{X}}_{et},\mathbb{Z}(n))
}
\]
commutes. It is enough to showing that the induced diagrams of cohomology groups commute. One may assume $i\geq 2n+2$. Then the map $H^i(\alpha_{\mathcal{Y},n})$ coincides with the following composite map
$$\mathrm{Hom}(H^{\delta_{\mathcal{Y}}-i}(\mathcal{Y},\mathbb{Q}(d_{\mathcal{Y}}-n)),\mathbb{Q})\stackrel{\sim}{\rightarrow}
 \mathrm{Hom}(H^{\delta_{\mathcal{Y}}-i}(\mathcal{Y}_{et},\mathbb{Z}(d_{\mathcal{Y}}-n)),\mathbb{Q})$$
$$\rightarrow
 \mathrm{Hom}(H^{\delta_{\mathcal{Y}}-i}(\mathcal{Y}_{et},\mathbb{Z}(d_{\mathcal{Y}}-n)),\mathbb{Q}/\mathbb{Z})
\stackrel{\sim}{\leftarrow} \mathrm{Hom}(H^{\delta_{\mathcal{Y}}-i}(\mathcal{Y}_{et},\widehat{\mathbb{Z}}(d-n)),\mathbb{Q}/\mathbb{Z})$$
$$\stackrel{\sim}{\leftarrow}\widehat{H}^{i-1}_c(\mathcal{Y}_{et},\mathbb{Q}/\mathbb{Z}(n))\rightarrow \widehat{H}^{i-1}(\overline{\mathcal{Y}}_{et},\mathbb{Q}/\mathbb{Z}(n))\rightarrow
\widehat{H}^{i}(\overline{\mathcal{Y}}_{et},\mathbb{Z}(n)).$$
It follows from $\mathbf{AV}(f,n)$ that this map is functorial in $\mathcal{Y}$. Hence there exists a morphism
$$f_W^*:R\Gamma_W(\overline{\mathcal{Y}},\mathbb{Z}(n))\rightarrow
R\Gamma_W(\overline{\mathcal{X}},\mathbb{Z}(n))$$
inducing a morphism of exact triangles.

We claim that such a morphism $f_{W}^*$ is unique. In order to ease the notations, we set
$$D_{\mathcal{X},n}:=R\mathrm{Hom}( R\Gamma(\mathcal{X},\mathbb{Q}(d_{\mathcal{X}}-n)),
\mathbb{Q}[-\delta_{\mathcal{X}}])\mbox{ and }D_{\mathcal{Y},n}:=R\mathrm{Hom}( R\Gamma(\mathcal{Y},\mathbb{Q}(d_{\mathcal{Y}}-n)),
\mathbb{Q}[-\delta_{\mathcal{Y}}]).$$
The complexes $R\Gamma_W(\overline{\mathcal{X}},\mathbb{Z}(n))$ and $R\Gamma_W(\overline{\mathcal{Y}},\mathbb{Z}(n))$ are both perfect complexes of abelian groups. Applying the functor $\mbox{Hom}_{\mathcal{D}}(-,R\Gamma_{W}(\overline{\mathcal{X}},\mathbb{Z}(n)))$ to the exact triangle
$$D_{\mathcal{Y},n}\rightarrow R\Gamma(\overline{\mathcal{Y}}_{et},\mathbb{Z}(n))\rightarrow R\Gamma_W(\overline{\mathcal{Y}},\mathbb{Z}(n))\rightarrow D_{\mathcal{Y},n}[1]$$
we obtain an exact sequence of abelian groups:
$$\mbox{Hom}_{\mathcal{D}}(D_{\mathcal{Y},n}[1],R\Gamma_W(\overline{\mathcal{X}},\mathbb{Z}(n)))\rightarrow
\mbox{Hom}_{\mathcal{D}}(R\Gamma_W(\overline{\mathcal{Y}},\mathbb{Z}(n)),R\Gamma_W(\overline{\mathcal{X}},\mathbb{Z}(n)))$$
$$\rightarrow\mbox{Hom}_{\mathcal{D}}
(R\Gamma(\overline{\mathcal{Y}}_{et},\mathbb{Z}(n)),R\Gamma_W(\overline{\mathcal{X}},\mathbb{Z}(n))).$$
On the one hand, $\mbox{Hom}_{\mathcal{D}}(D_{\mathcal{Y},n}[1],R\Gamma_W(\overline{\mathcal{X}},\mathbb{Z}(n)))$ is uniquely divisible since $D_{\mathcal{Y},n}[1]$ is a complex of $\mathbb{Q}$-vector spaces. On the other hand, the abelian group $$\mbox{Hom}_{\mathcal{D}}(R\Gamma_W(\overline{\mathcal{Y}},\mathbb{Z}(n)),
R\Gamma_W(\overline{\mathcal{X}},\mathbb{Z}(n)))$$ is finitely generated
as it follows from the spectral sequence
$$\prod_{i\in\mathbb{Z}}\mbox{Ext}^p(H^i_W(\overline{\mathcal{Y}},\mathbb{Z}(n)),H^{q+i}_W(\overline{\mathcal{X}},\mathbb{Z}(n)))
\Rightarrow H^{p+q}(R\mathrm{Hom}(R\Gamma_W(\overline{\mathcal{X}},\mathbb{Z}(n)),R\Gamma_W(\overline{\mathcal{X}},\mathbb{Z}(n))))$$
since $R\Gamma_W(\overline{\mathcal{X}},\mathbb{Z}(n))$ and $R\Gamma_W(\overline{\mathcal{Y}},\mathbb{Z}(n))$ are both perfect. Hence the morphism
$$\mbox{Hom}_{\mathcal{D}}(R\Gamma_W(\overline{\mathcal{Y}},\mathbb{Z}(n)),
R\Gamma_W(\overline{\mathcal{X}},\mathbb{Z}(n)))\rightarrow
\mbox{Hom}_{\mathcal{D}}
(R\Gamma(\overline{\mathcal{Y}}_{et},\mathbb{Z}(n)),R\Gamma_W(\overline{\mathcal{X}},\mathbb{Z}(n)))$$
is injective. It follows that there exists a unique morphism $f_W^*$ which renders the square
\[ \xymatrix{
 R\Gamma(\overline{\mathcal{Y}}_{et},\mathbb{Z}(n))\ar[r]\ar[d]_{f^*_{et}}& R\Gamma_W(\overline{\mathcal{Y}},\mathbb{Z}(n))\ar[d]_{f^*_{W}} \\ R\Gamma(\overline{\mathcal{X}}_{et},\mathbb{Z}(n))\ar[r]& R\Gamma_W(\overline{\mathcal{X}},\mathbb{Z}(n))
}
\]
commutative.
\end{proof}

\begin{cor}\label{cor-unicity}
$R\Gamma_W(\overline{\mathcal{X}},\mathbb{Z}(n))$ is well defined up to a unique isomorphism in $\mathcal{D}$.
\end{cor}
\begin{proof}
Let $ R\Gamma_W(\overline{\mathcal{X}},\mathbb{Z}(n))$ and $ R\Gamma_{W}(\overline{\mathcal{X}},\mathbb{Z}(n))'$
be two complexes as above. By Theorem \ref{cor-functoriality}, the identity map $Id:\mathcal{X}\rightarrow\mathcal{X}$ induces a unique isomorphism $R\Gamma_W(\overline{\mathcal{X}},\mathbb{Z}(n))
\simeq R\Gamma_W(\overline{\mathcal{X}},\mathbb{Z}(n))'$ in $\mathcal{D}$.

\end{proof}

\begin{rem}
Following \cite{Geisser10} we denote by $\mathbb{Z}^c(n)$ the cycle complex with homological indexing and we keep $\mathbb{Z}(n)$ for the cycle complex with cohomological indexing, so that $\mathbb{Z}^c(n)=\mathbb{Z}(d-n)[2d]$ over a regular scheme of pure dimension $d$. We set $\mathbb{Z}^c(n)=\mathbb{Q}^c(n)$. One should think of $R\Gamma(\X,\mathbb{Q}^c(n))$ (resp. of $R\Gamma(\overline{\mathcal{X}}_{et},\mathbb{Z}(n))$) as Borel-Moore homology (resp. as \'etale motivic cohomology). Then, for any regular proper arithmetic scheme $\X$ (not necessarily connected nor pure dimensional) satisfying our standard assumptions, one may define the Weil-\'etale complex $R\Gamma_W(\overline{\X},\mathbb{Z}(n))$ by the exact triangle
$$R\mathrm{Hom}(R\Gamma(\mathcal{X},\mathbb{Q}^c(n)),\mathbb{Q}[-2])\rightarrow
R\Gamma(\overline{\mathcal{X}}_{et},\mathbb{Z}(n))\rightarrow R\Gamma_W(\overline{\mathcal{X}},\mathbb{Z}(n))\rightarrow$$
which is somewhat more natural. However, in order to avoid confusion, we use exclusively cohomological indexing for the cycle complex throughout this paper. Accordingly, we use the triangle of Definition \ref{defn-fg-cohomology} in order to define the Weil-\'etale complexes.

\end{rem}

\subsection{Rational coefficients}

\begin{cor}\label{prop-rational-decompo}
There is a canonical direct sum decomposition
$$R\Gamma_W(\overline{\mathcal{X}},\mathbb{Z}(n))\otimes\mathbb{Q}\stackrel{\sim}{\longrightarrow} R\Gamma(\mathcal{X},\mathbb{Q}(n))
\oplus R\mathrm{Hom}(R\Gamma(\mathcal{X},\mathbb{Q}(d-n)),\mathbb{Q}[-2d-1])$$
which is functorial with respect to flat morphisms of proper regular arithmetic schemes.
\end{cor}
\begin{proof} Assume that $\mathcal{X}$ is connected of dimension $d$ and write $\delta=2d+2$. Applying $(-)\otimes\mathbb{Q}$ to the exact triangle of Definition \ref{defn-fg-cohomology}, we obtain an exact triangle
$$R\mathrm{Hom}(R\Gamma(\mathcal{X},\mathbb{Q}(d-n)),\mathbb{Q}[-\delta])\rightarrow
R\Gamma(\mathcal{X},\mathbb{Q}(n))
\rightarrow R\Gamma_W(\overline{\mathcal{X}},\mathbb{Z}(n))_{\mathbb{Q}}$$
We write $D:=R\mathrm{Hom}(R\Gamma(\mathcal{X},\mathbb{Q}(d-n)),\mathbb{Q}[-\delta])$ for brevity. The exact sequence
$$\mbox{Hom}_{\mathcal{D}}(D[1],R\Gamma(\mathcal{X},\mathbb{Q}(n)))\rightarrow \mbox{Hom}_{\mathcal{D}}(R\Gamma_W(\overline{\mathcal{X}},\mathbb{Z}(n))_{\mathbb{Q}},
R\Gamma(\mathcal{X},\mathbb{Q}(n)))$$
$$\rightarrow \mbox{Hom}_{\mathcal{D}}(R\Gamma(\mathcal{X},\mathbb{Q}(n)),R\Gamma(\mathcal{X},\mathbb{Q}(n)))\rightarrow \mbox{Hom}_{\mathcal{D}}(D,R\Gamma(\mathcal{X},\mathbb{Q}(n)))$$
yields an isomorphism $$\mbox{Hom}_{\mathcal{D}}(R\Gamma_W(\overline{\mathcal{X}},\mathbb{Z})_{\mathbb{Q}}
,R\Gamma(\mathcal{X},\mathbb{Q}(n)))\stackrel{\sim}{\rightarrow}
\mbox{Hom}_{\mathcal{D}}(R\Gamma(\mathcal{X},\mathbb{Q}(n)),R\Gamma(\mathcal{X},\mathbb{Q}(n))),$$
since $R\Gamma(\mathcal{X},\mathbb{Q}(n))$ a complex of $\mathbb{Q}$-vector spaces concentrated in degrees $\leq 2n$ and the complex $R\mathrm{Hom}(R\Gamma(\mathcal{X},\mathbb{Q}(d-n))_{\geq0},\mathbb{Q}[-\delta])$ is acyclic in degrees $\leq 2n+1$. This yields the canonical direct sum decomposition.

It remains to show that this direct sum decomposition is functorial. Let $\mathcal{X}\rightarrow\mathcal{Y}$ be a flat map between regular proper schemes. One may assume $\mathcal{X}$ and $\mathcal{Y}$ connected of dimension $d_\mathcal{X}$ and $d_\mathcal{Y}$ respectively. We set $\delta_{\mathcal{X}}=2d_\mathcal{X}+2$ and $\delta_{\mathcal{Y}}=2d_\mathcal{Y}+2$. Flat pull-back of cycles yields a map
$$u:R\Gamma(\mathcal{Y},\mathbb{Q}(n))\rightarrow R\Gamma(\mathcal{X},\mathbb{Q}(n))$$
while proper push-forward of cycles yields a map
$$v:R\mathrm{Hom}(R\Gamma(\mathcal{Y},\mathbb{Q}(d_{\mathcal{Y}}-n)),\mathbb{Q}[-\delta_{\mathcal{Y}}])\rightarrow R\mathrm{Hom}(R\Gamma(\mathcal{X},\mathbb{Q}(d_\mathcal{X}-n)),\mathbb{Q}[-\delta_{\mathcal{X}}]).$$
In order to show that the diagram of complexes of $\mathbb{Q}$-vector spaces
\[ \xymatrix{
R\Gamma_W(\overline{\mathcal{Y}},\mathbb{Z}(n))\otimes\mathbb{Q}\ar[r]\ar[d]^{f^*_W}&
R\Gamma(\mathcal{Y},\mathbb{Q}(n))
\oplus R\mathrm{Hom}(R\Gamma(\mathcal{Y},\mathbb{Q}(d_{\mathcal{Y}}-n)),\mathbb{Q}[-\delta_{\mathcal{Y}}])[1] \ar[d]^{(u,v)}\\
R\Gamma_W(\overline{\mathcal{X}},\mathbb{Z}(n))\otimes\mathbb{Q}\ar[r]&R\Gamma(\mathcal{X},\mathbb{Q}(n))
\oplus R\mathrm{Hom}(R\Gamma(\mathcal{X},\mathbb{Q}(d_{\mathcal{X}}-n)),\mathbb{Q}[-\delta_{\mathcal{X}}])[1]
}
\]
commutes in $\mathcal{D}$, it is enough to show that the following square
\[ \xymatrix{
H^i_W(\overline{\mathcal{Y}},\mathbb{Z}(n))\otimes\mathbb{Q}\ar[r]\ar[d]^{f^*_W}&
H^i(\mathcal{Y},\mathbb{Q}(n))
\oplus H^{\delta_{\mathcal{Y}}-(i+1)}(\mathcal{Y},\mathbb{Q}(d_{\mathcal{Y}}-n))^* \ar[d]^{(u,v)}\\
H^i_W(\overline{\mathcal{X}},\mathbb{Z}(n))\otimes\mathbb{Q}\ar[r]&H^i(\mathcal{X},\mathbb{Q}(n))
\oplus H^{\delta_{\mathcal{X}}-(i+1)}(\mathcal{X},\mathbb{Q}(d_{\mathcal{X}}-n))^*
}
\]
commutes for any $i\in\mathbb{Z}$, where $(-)^*$ denotes the dual $\mathbb{Q}$-vector space. The result is obvious for $i\geq 2n+1$ since we then have
$$H^i(\mathcal{Y},\mathbb{Q}(n))=H^i(\mathcal{X},\mathbb{Q}(n))=0.$$
Hence the diagram is commutative for $i\geq 2n+1$ by Theorem \ref{cor-functoriality}. For $i< 2n+1$, we have
$$H^{\delta_{\mathcal{Y}}-(i+1)}(\mathcal{Y},\mathbb{Q}(d_{\mathcal{Y}}-n))^*
=H^{\delta_{\mathcal{X}}-(i+1)}(\mathcal{X},\mathbb{Q}(d_{\mathcal{X}}-n))^*=0$$
and the horizontal maps in the square above are inverse isomorphisms to
$$H^{i}(\mathcal{Y},\mathbb{Q}(n))
\stackrel{\sim}{\longrightarrow}H^i_W(\overline{\mathcal{Y}},\mathbb{Z}(n))_\mathbb{Q}\mbox{ and }H^{i}(\mathcal{X},\mathbb{Q}(n))\stackrel{\sim}{\longrightarrow}
H^i_W(\overline{\mathcal{X}},\mathbb{Z}(n))_\mathbb{Q}
$$
respectively. Hence the result follows from the fact that
\[ \xymatrix{
H^i_W(\overline{\mathcal{Y}},\mathbb{Z}(n))\otimes\mathbb{Q}\ar[d]^{f^*_W}&
H^{i}(\mathcal{Y},\mathbb{Q}(n)) \ar[d]^{u}\ar[l]\\
H^i_W(\overline{\mathcal{X}},\mathbb{Z}(n))\otimes\mathbb{Q}& H^{i}(\mathcal{X},\mathbb{Q}(n))\ar[l]
}
\]
commutes, again by Theorem \ref{cor-functoriality}.

\end{proof}

\subsection{Torsion coefficients}
\begin{lem}\label{lem-finitecoef}
There is a canonical isomorphism
$$R\Gamma_W(\overline{\mathcal{X}},\mathbb{Z}(n))\otimes_{\mathbb{Z}}^L\mathbb{Z}/m\mathbb{Z}\stackrel{\sim}{\longrightarrow} R\Gamma(\overline{\mathcal{X}}_{et},\mathbb{Z}/m\mathbb{Z}(n)).
$$
which is functorial with respect to flat morphisms of proper regular arithmetic schemes.
\end{lem}
\begin{proof}
Consider the exact triangle
$$R\mathrm{Hom}(R\Gamma(\mathcal{X},\mathbb{Q}(d-n)),\mathbb{Q}[-\delta])\rightarrow R\Gamma(\overline{\mathcal{X}}_{et},\mathbb{Z}(n))\rightarrow R\Gamma_W(\overline{\mathcal{X}},\mathbb{Z}(n))$$
Taking derived tensor product $-\otimes^L_{\mathbb{Z}}\mathbb{Z}/m\mathbb{Z}$
we obtain the following diagram
\[ \xymatrix{
R\mathrm{Hom}(R\Gamma(\mathcal{X},\mathbb{Q}(d-n)),\mathbb{Q}[-\delta])\ar[d]^n\ar[r]& R\Gamma(\overline{\mathcal{X}}_{et},\mathbb{Z}(n))\ar[r]\ar[d]^n& R\Gamma_W(\overline{\mathcal{X}},\mathbb{Z}(n))\ar[r]\ar[d]^n&\\
R\mathrm{Hom}(R\Gamma(\mathcal{X},\mathbb{Q}(d-n)),\mathbb{Q}[-\delta])\ar[r]\ar[d]& R\Gamma(\overline{\mathcal{X}}_{et},\mathbb{Z}(n))\ar[r]\ar[d]& R\Gamma_W(\overline{\mathcal{X}},\mathbb{Z}(n))\ar[r]\ar[d]&\\
0\ar[r]& R\Gamma(\overline{\mathcal{X}}_{et},\mathbb{Z}/m\mathbb{Z}(n))\ar[r]& R\Gamma_W(\overline{\mathcal{X}},\mathbb{Z}(n))\otimes^L_{\mathbb{Z}}\mathbb{Z}/m\mathbb{Z}\ar[r]&
}
\]
The upper and middle rows are both exact triangles and the colons are all exact as well. It follows that the bottom row is exact, hence the map
$$R\Gamma(\overline{\mathcal{X}}_{et},\mathbb{Z}/m\mathbb{Z})
\simeq R\Gamma(\overline{\mathcal{X}}_{et},\mathbb{Z}(n))\otimes^L_{\mathbb{Z}}\mathbb{Z}/m\mathbb{Z}\longrightarrow R\Gamma_W(\overline{\mathcal{X}},\mathbb{Z}(n))\otimes^L_{\mathbb{Z}}\mathbb{Z}/m\mathbb{Z},$$
which is functorial by Theorem \ref{cor-functoriality},
is an isomorphism in $\mathcal{D}$.
\end{proof}

\subsection{Relationship with the Lichtenbaum-Geisser definition over finite fields}\label{subsect-Licht}

Let $Y$ be a smooth proper scheme over a finite field $k$. We may assume that $Y$ is connected and $d$-dimensional. In this section, we show that the Weil-\'etale complex $R\Gamma_{W}(Y,\mathbb{Z}(n))$ defined in this paper is (expected to be) canonically isomorphic in the derived category to the Weil-\'etale complex defined in \cite{Geisser04b}, and we describe the relationship between the conjecture $\textbf{L}(Y_{et},n)$ stated in Section \ref{section-assumptions} and the conjecture $\textbf{L}(Y_{W},n)$ stated in \cite{Geisser04b}.

 We denote by $W_k$ the Weil group of the finite field $k$.
The Weil-\'etale topos $Y_W$ is the category of $W_k$-equivariant
sheaves of sets on the \'etale site of $Y\otimes_{k}\overline{k}$, where $\overline{k}/k$ is an algebraic closure.
By \cite{Geisser04b} one has an exact triangle in the derived category of $\mathcal{D}(Y_{et})$
$$\mathbb{Z}(n)\rightarrow R\gamma_*\mathbb{Z}(n)\rightarrow \mathbb{Q}(n)[-1]\rightarrow\mathbb{Z}(n)[1]$$
where $\gamma:Y_W\rightarrow Y_{et}$ is the canonical map. Applying $R\Gamma(Y_{et},-)$ and rotating, we get
$$R\Gamma(Y_{et},\mathbb{Q}(n)[-2])\rightarrow R\Gamma(Y_{et},\mathbb{Z}(n))\rightarrow R\Gamma(Y_W,\mathbb{Z}(n))\rightarrow R\Gamma(Y_{et},\mathbb{Q}(n)[-2])[1].$$ The following conjecture is due to T. Geisser and S. Lichtenbaum.
\begin{conj}
$\textbf{\emph{L}}(Y_W,n)$ For every $i$, the group $H^i(Y_W,\mathbb{Z}(n))$ is finitely generated.
\end{conj}

The following conjecture is due to T. Geisser and B. Kahn (see \cite{Kahn03} and \cite{Geisser04b}). Consider the map $c_l:R\gamma_*\mathbb{Z}(n)\otimes \mathbb{Z}_l\rightarrow R\mathrm{lim}\,\mathbb{Z}/l^r\mathbb{Z}(n)$ in $\mathcal{D}(Y_{et})$ \cite{Geisser04b}.
\begin{conj}
$\textbf{\emph{K}}(Y_W,n)$ For every prime $l$ and any $i\in\mathbb{Z}$, the map $c_l$ induces an
isomorphism
$$H^i (Y_W,\mathbb{Z}(n))\otimes\mathbb{Z}_l\simeq H^i_{cont}(Y,\mathbb{Z}_l(n)).$$
\end{conj}

Consider the full subcategory $A(k)$ of the category of smooth projective varieties over the finite field $k$ generated by product of curves and the following operations:

(1) If $X$ and $Y$ are in $A(k)$ then $X\coprod Y$ is in $A(k)$.

(2) If $Y$ is in $A(k)$ and there are morphisms $c:X\rightarrow Y$ and $c':Y\rightarrow X$ in the category of Chow motives, such that $c'\circ c:X\rightarrow X$ is multiplication by a constant, then $X$ is in $A(k)$.

(3) If $k'/k$ is a finite extension and $X\times_kk'$ is in $A(k)$, then $X$ is in $A(k)$.

(4) If $Y$ is a closed subscheme of $X$ with $X$ and $Y$ in $A(k)$, then the blow-up $X'$ of $X$ along $Y$ is in $A(k)$.

The following result is due to T. Geisser \cite{Geisser04b}.
\begin{thm}\label{GeisserTHM} (Geisser)
Let $Y$ be a smooth projective variety of dimension $d$.

- One has $\textbf{\emph{K}}(Y_W,n)+\textbf{\emph{K}}(Y_W,d-n)\Rightarrow \textbf{\emph{L}}(Y_W,n)\Rightarrow \textbf{\emph{K}}(Y_W,n)$.

- Moreover, if $Y$ belongs to $A(k)$ then $\textbf{\emph{L}}(Y_W,n)$ holds for $n\leq1$ and $n\geq d-1$.
\end{thm}

\begin{prop}\label{prop-equiconjectures-char-p}
Let $Y$ be a connected smooth projective scheme over a finite field of dimension $d$. Then we have
$$\textbf{\emph{L}}(Y_W,n)\Rightarrow \textbf{\emph{L}}(Y_{et},n).$$
\end{prop}
\begin{proof}
By \cite{Geisser04b}, one has an exact sequence
\begin{equation}\label{geisser-sequence}
...\rightarrow H^i(Y_{et},\mathbb{Z}(n))\rightarrow H^i(Y_W,\mathbb{Z}(n))\rightarrow H^{i-1}(Y_{et},\mathbb{Q}(n))\rightarrow H^{i+1}(Y_{et},\mathbb{Z}(n))\rightarrow...
\end{equation}
With rational coefficients, this exact sequence yields isomorphisms
\begin{equation}\label{geisser-splitting}
H^i(Y_W,\mathbb{Z}(n))\otimes\mathbb{Q}\simeq H^i(Y_W,\mathbb{Q}(n))\simeq H^i(Y_{et},\mathbb{Q}(n))\oplus H^{i-1}(Y_{et},\mathbb{Q}(n)).
\end{equation}

Assume now that Conjecture $\textbf{L}(Y_W,n)$ holds. Let us first show that $H^i(Y_W,\mathbb{Z}(n))$ is finite for $i\neq 2n,2n+1$. By Theorem \ref{GeisserTHM}, Conjecture $\textbf{K}(Y_W,n)$ holds, i.e. we have an isomorphism
$$H^i (Y_W,\mathbb{Z}(n))\otimes\mathbb{Z}_l\simeq H^i_{cont}(Y,\mathbb{Z}_l(n))$$
for any $l$ and any $i$. But for $i\neq 2n, 2n+1$, the group $H^i_{cont}(Y,\mathbb{Z}_l(n))$ is finite for any $l$ and zero for almost all $l$ \cite{Gabber83} (see also \cite{Kahn03} Proof of Corollaire 3.8). Hence $H^i(Y_W,\mathbb{Z}(n))$ is finite for $i\neq 2n,2n+1$. Then (\ref{geisser-splitting}) gives $H^i(Y_{et},\mathbb{Q}(n))=0$ for $i\neq2n$. The exact sequence (\ref{geisser-sequence}) then shows that $H^i(Y_{et},\mathbb{Z}(n))\rightarrow H^i(Y_W,\mathbb{Z}(n))$ is injective for $i\leq 2n+1$, hence that $H^i(Y_{et},\mathbb{Z}(n))$ is finitely generated for $i\leq 2n+1$.

\end{proof}

\begin{cor}\label{corAk}
Any variety $Y$ in $A(k)$ satisfies $\textbf{\emph{L}}(Y_{et},n)$ and $\textbf{\emph{L}}(Y_{et},d-n)$ for $n\leq 1$.
\end{cor}
\begin{proof}
 This follows from Theorem \ref{GeisserTHM} and Proposition \ref{prop-equiconjectures-char-p}.
\end{proof}

\begin{conj} \textbf{\emph{P}}(Y,n)
The intersection product induces a perfect pairing:
 $$CH^n(Y)_{\mathbb{Q}}\times CH^{d-n}(Y)_{\mathbb{Q}}\rightarrow CH^{d}(Y)_{\mathbb{Q}}\stackrel{\mathrm{deg}}{\longrightarrow} \mathbb{Q}.$$
\end{conj}
Note that $CH^{d}(Y)_{\mathbb{Q}}\stackrel{\mathrm{deg}}{\longrightarrow} \mathbb{Q}$ is known to be an isomorphism by class field theory.

\begin{thm}\label{thm-comparison-char-p}
If $Y$ satisfies $\textbf{\emph{L}}(Y_{et},n)$, $\textbf{\emph{L}}(Y_{et},d-n)$ and $\textbf{\emph{P}}(Y,n)$ then there is an isomorphism in $\mathcal{D}$
$$ R\Gamma(Y_W,\mathbb{Z}(n))\stackrel{\sim}{\longrightarrow} R\Gamma_W(Y,\mathbb{Z}(n))$$
where the left hand side is the cohomology of the Weil-\'etale topos and the right hand side is the complex defined in this paper.
\end{thm}
\begin{proof} We shall show that there is a commutative diagram in $\mathcal{D}$:
\[ \xymatrix{
 R\Gamma(Y_{et},\mathbb{Q}(n)[-2])\ar[d]_{\simeq}\ar[r]& R\Gamma(Y_{et},\mathbb{Z}(n))\ar[d]_{Id}\\
R\mathrm{Hom}( R\Gamma(Y,\mathbb{Q}(d-n))_{\geq0},\mathbb{Q}[-2d-2])
\ar[r]^{\,\,\,\,\,\,\,\,\,\,\,\,\,\,\,\,\,\,\,\,\,\,\,\,\,\,\,\,\,\,\,\,\,\,\,\,\alpha_{Y,n}} & R\Gamma(Y_{et},\mathbb{Z}(n))
}
\]
where the vertical maps are isomorphisms. This will imply the existence of an isomorphism $ R\Gamma(Y_W,\mathbb{Z}(n))\stackrel{\sim}{\longrightarrow} R\Gamma_W(Y,\mathbb{Z}(n))$, whose uniqueness follows from the argument given in the proof of Theorem \ref{cor-functoriality}.

One has $H^{2d}(Y,\mathbb{Q}(d))=CH^d(Y)_{\mathbb{Q}}=\mathbb{Q}$ and $H^{i}(Y,\mathbb{Q}(d))=0$ for $i>2d$. This yields a map $R\Gamma(Y,\mathbb{Q}(d)))
\rightarrow\mathbb{Q}[-2d]$. Moreover, since $Y$ is smooth over the field $k$, we have a product map
$$R\Gamma(Y,\mathbb{Q}(n))\otimes R\Gamma(Y,\mathbb{Q}(d-n))\rightarrow R\Gamma(Y,\mathbb{Q}(d)).$$
We obtain a morphism
\begin{equation}\label{unemapdeplus+}
R\Gamma(Y_{et},\mathbb{Q}(n))\simeq\Gamma(Y,\mathbb{Q}(n))\longrightarrow R\mathrm{Hom}(\Gamma(Y,\mathbb{Q}(d-n)),\mathbb{Q}[-2d]).
\end{equation}
The conjunction of Conjectures $\textbf{L}(Y_{et},n)$, $\textbf{L}(Y_{et},d-n)$ and  $\textbf{P}(Y,n)$ implies that $R\Gamma(Y_{et},\mathbb{Q}(n))$ is concentrated in degree $2n$ and that the morphism (\ref{unemapdeplus+}) is a quasi-isomorphism. Indeed, assuming $\textbf{L}(Y_{et},n)$ and $\textbf{L}(Y_{et},d-n)$ we get, by Proposition \ref{prop-etaleduality}, the finiteness of $H^{2n+1}(Y_{et},\mathbb{Z}(n))$ and an isomorphism
$$H^{i}(Y_{et},\mathbb{Z}(n))\stackrel{\sim}{\longrightarrow}\mathrm{Hom}(H^{2d+2-i}(Y_{et},\mathbb{Z}(d-n),\mathbb{Q}/\mathbb{Z})$$
for $i\geq2n+2$. Hence $H^{i}(Y,\mathbb{Q}(n))=0$ for $i\geq2n+1$. For $i\leq2n$ we have
$$H^{i}(Y_{et},\mathbb{Z}(n))^{\wedge}\stackrel{\sim}{\longrightarrow}\mathrm{Hom}(H^{2d+2-i}(Y_{et},\mathbb{Z}(d-n),\mathbb{Q}/\mathbb{Z}).$$
But $H^{2d+2-i}(Y_{et},\mathbb{Z}(d-n)$ is finite for $i\leq2n-1$. Indeed, writing  $j=2d+2-i$ and $t=d-n$, we need $H^{j}(Y_{et},\mathbb{Z}(t))$ finite for $j\geq2t+3$. But in this range we have
$$H^{j}(Y_{et},\mathbb{Z}(t))=H^{j-1}(Y_{et},\mathbb{Q}/\mathbb{Z}(t))$$
since $H^{j}(Y_{et},\mathbb{Q}(t))=0$ for $j>2t$. Moreover $Y$ satisfies Artin-Verdier duality with
mod-$m$ coefficients (see Conjecture \ref{etaleduality-modn}), since $\mathbb{Z}/m\mathbb{Z}(n)=\mu_m^{\otimes n}$ for
$p$ not dividing $m$ and $\mathbb{Z}/p^r\mathbb{Z}(n)\simeq \nu_r^n[-n]$. Passing to the limit we get an isomorphism of profinite groups
$$H^{j-1}(Y_{et},\mathbb{Q}/\mathbb{Z}(t))^D=H^{2d+1-(j-1)}(Y_{et},\widehat{\mathbb{Z}}(n))
=\prod_l\,H^{2d+1-(j-1)}_{cont}(Y_{et},\mathbb{Z}_l(n)).$$
For $j-1\geq 2t+2$, this group is finite (indeed, for $i\neq 2t,2t+1$ the $l$-adic cohomology group $H^{i}_{cont}(Y_{et},\mathbb{Z}_l(t))$ is finite for all $l$ and trivial for almost all $l$ \cite{Gabber83}). We obtain that $H^{i}(Y,\mathbb{Q}(n))=0$ for $i\neq2n$. Hence we have canonical isomorphisms $$R\mathrm{Hom}(R\Gamma(Y,\mathbb{Q}(d-n)),\mathbb{Q}[-2d])\simeq\mbox{Hom}(H^{2d-2n}(Y,\mathbb{Q}(d-n)),\mathbb{Q})[-2n]\simeq CH^{d-n}(Y)_{\mathbb{Q}}^*[-2n],$$
and
$$R\Gamma(Y,\mathbb{Q}(n))\simeq H^{2n}(Y,\mathbb{Q}(n))[-2n]\simeq CH^r(Y)_{\mathbb{Q}}[-2n].$$
Moreover the map $$H^{2n}(Y,\mathbb{Q}(n))\rightarrow \mbox{Hom}(H^{2d-2n}(Y,\mathbb{Q}(d-n)),\mathbb{Q})$$
is given by the intersection pairing
$$CH^{n}(Y)_{\mathbb{Q}}\times CH^{d-n}(Y)_{\mathbb{Q}}\rightarrow CH^{d}(Y)_{\mathbb{Q}}\rightarrow\mathbb{Q}.$$
Therefore, it follows from the conjunction of Conjectures $\textbf{L}(Y_{et},n)$, $\textbf{L}(Y_{et},d-n)$ and  $\textbf{P}(Y,n)$ that
the map (\ref{unemapdeplus+}) is an isomorphism.

It remains to check the commutativity of the  above diagram. The complex $$D_{Y,n}:= R\mathrm{Hom}(R\Gamma(Y,\mathbb{Q}(d-n)),\mathbb{Q}[-2d-2])\simeq R\Gamma(Y_{et},\mathbb{Q}(n)[-2])$$
is concentrated in degree $2n+2$, in particular acyclic in degrees $>2n+2$. The spectral sequence
$$\prod_{i\in\mathbb{Z}}\mbox{Ext}^p(H^i(D_{Y,n}),H^{q+i}(Y_{et},\mathbb{Z}(n)))
\Rightarrow H^{p+q}(R\mathrm{Hom}(D_{Y,n},R\Gamma(Y_{et},\mathbb{Z}(n))))$$
degenerates at $E_2$ and yields an isomorphism
$$\mbox{Hom}_{\mathcal{D}}(D_{Y,n},R\Gamma(Y_{et},\mathbb{Z}(n))_{\leq2n+2})\simeq \mbox{Hom}(H^{2n+2}(D_{Y,n}),H^{2n+2}(Y_{et},\mathbb{Z}(n)))$$
since $D_{Y,n}$ is quasi-isomorphic to a $\mathbb{Q}$-vector space put in degree $2n+2$ while $H^i(Y_{et},\mathbb{Z}(n))$ is finitely generated for $i\leq2n$ and finite for $i=2n+1$. One is therefore reduced to show the commutativity of the following square (of abelian groups):
\[ \xymatrix{
CH^{n}(Y)_{\mathbb{Q}}\ar[d]\ar[rr]^{d_2^{2n,1}}&&H^{2n+2}(Y_{et},\mathbb{Z}(n))\ar[d]_{Id}\\
CH^{d-n}(Y)_{\mathbb{Q}}^*
\ar[rr]^{H^{2n+2}(\alpha_{Y,n})}& &H^{2n+2}(Y_{et},\mathbb{Z}(n))
}
\]
where the left vertical map is given by the intersection pairing. The fact that this square commutes follows from Geisser's description of the differential map $d_2^{2n,1}$, see \cite{Geisser04b}.
\end{proof}

\begin{cor}\label{cor-comparison-conject-Let-vs-LW}
Let $Y$ be a connected smooth projective scheme over a finite field of dimension $d$. Then we have
$$\textbf{\emph{L}}(Y_W,n)+\textbf{\emph{L}}(Y_W,d-n)\Leftrightarrow \textbf{\emph{L}}(Y_{et},n)+\textbf{\emph{L}}(Y_{et},d-n)+\textbf{\emph{P}}(Y,n).$$
\end{cor}
\begin{proof}
Assume $\textbf{L}(Y_{et},n)$, $\textbf{L}(Y_{et},d-n)$ and $\textbf{P}(Y,n)$. Then $ R\Gamma_W(Y,\mathbb{Z}(n))$ and $ R\Gamma_W(Y,\mathbb{Z}(d-n))$ are well defined and have finitely generated cohomology groups. Hence Conjecture $\textbf{L}(Y_W,n)$ and $\textbf{L}(Y_W,d-n)$ follow from Theorem \ref{thm-comparison-char-p}.

Assume $\textbf{L}(Y_W,n)$ and $\textbf{L}(Y_W,d-n)$. Conjectures $\textbf{L}(Y_{et},n)$ and $\textbf{L}(Y_{et},d-n)$ then hold by Proposition \ref{prop-equiconjectures-char-p}. The fact that $\textbf{L}(Y_W,n)$ and $\textbf{L}(Y_W,d-n)$ imply $\textbf{P}(Y,n)$ is proven in \cite{Geisser04b}.
\end{proof}

\subsection{Weil-\'etale duality}

\begin{thm}\label{thm-duality-fg}
There is a canonical product map
$$R\Gamma_W(\overline{\mathcal{X}},\mathbb{Z}(n))\otimes_{\mathbb{Z}}^LR\Gamma_W(\overline{\mathcal{X}},\mathbb{Z}(d-n))
\longrightarrow\mathbb{Z}[-2d-1]$$
such that the induced map
$$R\Gamma_W(\overline{\mathcal{X}},\mathbb{Z}(n))\rightarrow  R\mathrm{Hom}(R\Gamma_W(\overline{\mathcal{X}},\mathbb{Z}(d-n)),\mathbb{Z}[-2d-1])$$
is an isomorphism of perfect complexes of abelian groups.
\end{thm}

\begin{proof}
First we consider the map
$$R\Gamma_W(\overline{\mathcal{X}},\mathbb{Z}(n))_{\mathbb{Q}}\otimes_{\mathbb{Q}}^L
R\Gamma_W(\overline{\mathcal{X}},\mathbb{Z}(d-n))_{\mathbb{Q}}
\longrightarrow\mathbb{Q}[-2d-1]$$
given, thanks to Corollary \ref{prop-rational-decompo}, by the sum of the following tautological maps
$$R\Gamma(\mathcal{X},\mathbb{Q}(n))\otimes_{\mathbb{Q}}^L
R\mathrm{Hom}(R\Gamma(\mathcal{X},\mathbb{Q}(n)),\mathbb{Q}[-\delta])[1]\rightarrow \mathbb{Q}[-2d-1],$$
$$R\mathrm{Hom}(R\Gamma(\mathcal{X},\mathbb{Q}(d-n)),\mathbb{Q}[-\delta])[1]
\otimes_{\mathbb{Q}}^LR\Gamma(\mathcal{X},\mathbb{Q}(d-n))\rightarrow \mathbb{Q}[-2d-1].$$
Recall that we denote continuous \'etale cohomology with $\widehat{\mathbb{Z}}(n)$-coefficients by $$R\Gamma(\overline{\mathcal{X}}_{et},\widehat{\mathbb{Z}}(n)):=\mathrm{holim}\, R\Gamma(\overline{\mathcal{X}}_{et},\mathbb{Z}/m\mathbb{Z}(n)).$$
Conjecture ${\bf AV}(\overline{\mathcal{X}}_{et},n)$ provides us with a map $\mathbb{Z}/m\mathbb{Z}(n)^{\overline{\mathcal{X}}}\otimes^L \mathbb{Z}/m\mathbb{Z}(d-n)^{\overline{\mathcal{X}}}\rightarrow \mathbb{Z}/m\mathbb{Z}(d)^{\overline{\mathcal{X}}}$ inducing a morphism
$$R\Gamma(\overline{\mathcal{X}}_{et},\mathbb{Z}/m\mathbb{Z}(n))\otimes^L_{\mathbb{Z}}
R\Gamma(\overline{\mathcal{X}}_{et},\mathbb{Z}/m\mathbb{Z}(d-n))\rightarrow R\Gamma(\overline{\mathcal{X}}_{et},\mathbb{Z}/m\mathbb{Z}(d))$$
$$\rightarrow \mathbb{Z}/m\mathbb{Z}[-2d-1]\rightarrow \mathbb{Q}/\mathbb{Z}[-2d-1]$$
By adjunction, we obtain
$$R\Gamma(\overline{\mathcal{X}}_{et},\mathbb{Z}/m\mathbb{Z}(d-n))\rightarrow
R\mathrm{Hom}(R\Gamma(\overline{\mathcal{X}}_{et},\mathbb{Z}/m\mathbb{Z}(n)),\mathbb{Q}/\mathbb{Z}[-2d-1]).$$
and
$$R\Gamma(\overline{\mathcal{X}}_{et},\widehat{\mathbb{Z}}(d-n))\rightarrow
R\mathrm{Hom}(R\Gamma(\overline{\mathcal{X}}_{et},\mathbb{Q}/\mathbb{Z}(n)),\mathbb{Q}/\mathbb{Z}[-2d-1]).$$
Moreover, by adjunction the maps
$$R\Gamma_W(\overline{\mathcal{X}},\mathbb{Z}(n))\rightarrow R\Gamma_W(\overline{\mathcal{X}},\mathbb{Z}(n))\otimes_{\mathbb{Z}}^L\mathbb{Z}/m\mathbb{Z}\simeq R\Gamma(\overline{\mathcal{X}}_{et},\mathbb{Z}/m\mathbb{Z}(n))$$
yield a map
$$R\Gamma_W(\overline{\mathcal{X}},\mathbb{Z}(n))\longrightarrow R\Gamma(\overline{\mathcal{X}}_{et},\widehat{\mathbb{Z}}(n))$$
inducing the following morphism on cohomology:
$$H^i_W(\overline{\mathcal{X}},\mathbb{Z}(n))\longrightarrow H^i_W(\overline{\mathcal{X}},\mathbb{Z}(n))^{\wedge}\stackrel{\sim}{\longrightarrow}H^i(\overline{\mathcal{X}}_{et},\widehat{\mathbb{Z}}(n)).$$
Then we consider the maps
$$R\Gamma_W(\overline{\mathcal{X}},\mathbb{Z}(n))\otimes_{\mathbb{Z}}^L
R\Gamma_W(\overline{\mathcal{X}},\mathbb{Z}(d-n))\otimes_{\mathbb{Z}}^L\mathbb{Q}/\mathbb{Z}
\stackrel{\sim}{\longrightarrow} R\Gamma_W(\overline{\mathcal{X}},\mathbb{Z}(n))\otimes_{\mathbb{Z}}^L
R\Gamma(\overline{\mathcal{X}}_{et},\mathbb{Q}/\mathbb{Z}(d-n))$$
$$\longrightarrow R\Gamma(\overline{\mathcal{X}}_{et},\widehat{\mathbb{Z}}(n))\otimes_{\mathbb{Z}}^L
R\Gamma(\overline{\mathcal{X}}_{et},\mathbb{Q}/\mathbb{Z}(d-n))\longrightarrow \mathbb{Q}/\mathbb{Z}[-2d-1].$$
and
$$R\Gamma_W(\overline{\mathcal{X}},\mathbb{Z}(n))\otimes_{\mathbb{Z}}^L
R\Gamma_W(\overline{\mathcal{X}},\mathbb{Z}(d-n))\otimes_{\mathbb{Z}}^L\mathbb{Q}/\mathbb{Z}
\stackrel{\sim}{\longrightarrow} R\Gamma(\overline{\mathcal{X}}_{et},\mathbb{Q}/\mathbb{Z}(n))\otimes_{\mathbb{Z}}^L
R\Gamma_W(\overline{\mathcal{X}},\mathbb{Z}(d-n))$$
$$\longrightarrow R\Gamma(\overline{\mathcal{X}}_{et},\mathbb{Q}/\mathbb{Z}(n))\otimes_{\mathbb{Z}}^L
R\Gamma(\overline{\mathcal{X}}_{et},\widehat{\mathbb{Z}}(d-n))\longrightarrow \mathbb{Q}/\mathbb{Z}[-2d-1].$$
We need to see that these two maps give the same element in
\begin{eqnarray*}
&&\mathrm{Hom}(R\Gamma_W(\overline{\mathcal{X}},\mathbb{Z}(n))\otimes_{\mathbb{Z}}^L
R\Gamma_W(\overline{\mathcal{X}},\mathbb{Z}(d-n))\otimes_{\mathbb{Z}}^L\mathbb{Q}/\mathbb{Z},\mathbb{Q}/\mathbb{Z}[-2d-1])\\
&\simeq &\mathrm{Hom}(R\Gamma_W(\overline{\mathcal{X}},\mathbb{Z}(n))\otimes_{\mathbb{Z}}^L
R\Gamma_W(\overline{\mathcal{X}},\mathbb{Z}(d-n)),\widehat{\mathbb{Z}}[-2d-1]).
\end{eqnarray*}
But these maps are both induced by the limit of
$$R\Gamma_W(\overline{\mathcal{X}},\mathbb{Z}(n))\otimes_{\mathbb{Z}}^L
R\Gamma_W(\overline{\mathcal{X}},\mathbb{Z}(d-n))\rightarrow R\Gamma(\overline{\mathcal{X}}_{et},\mathbb{Z}/m(n))\otimes_{\mathbb{Z}}^L
R\Gamma(\overline{\mathcal{X}}_{et},\mathbb{Z}/m(d-n))$$
$$\rightarrow \mathbb{Z}/m[-2d-1],$$
hence they coincide. We obtain a canonical map
\begin{equation}\label{onecoolmap}
R\Gamma_W(\overline{\mathcal{X}},\mathbb{Z}(n))\otimes_{\mathbb{Z}}^L
R\Gamma_W(\overline{\mathcal{X}},\mathbb{Z}(d-n))\otimes_{\mathbb{Z}}^L\mathbb{Q}/\mathbb{Z}
\longrightarrow \mathbb{Q}/\mathbb{Z}[-2d-1].
\end{equation}
We now consider the diagram
\[ \xymatrix{
R\Gamma_W(\overline{\mathcal{X}},\mathbb{Z}(n))\otimes_{\mathbb{Z}}^L
R\Gamma_W(\overline{\mathcal{X}},\mathbb{Z}(d-n))\ar[d]\ar[r]^{\hspace{2 cm}\exists\,!\, p_{n,d-n}}&\mathbb{Z}[-2d-1]\ar[d]  \\
R\Gamma_W(\overline{\mathcal{X}},\mathbb{Z}(n))\otimes_{\mathbb{Z}}^L
R\Gamma_W(\overline{\mathcal{X}},\mathbb{Z}(d-n))\otimes_{\mathbb{Z}}^L\mathbb{Q}\ar[d]\ar[r]&\mathbb{Q}[-2d-1]\ar[d]  \\
R\Gamma_W(\overline{\mathcal{X}},\mathbb{Z}(n))\otimes_{\mathbb{Z}}^L
R\Gamma_W(\overline{\mathcal{X}},\mathbb{Z}(d-n))\otimes_{\mathbb{Z}}^L\mathbb{Q}/\mathbb{Z}\ar[r]&\mathbb{Q}/\mathbb{Z}[-2d-1]
}
\]
We explain why the lower square is commutative. Unwinding the definitions, we see that the following square
\[ \xymatrix{
R\Gamma(\overline{\mathcal{X}}_{et},\mathbb{Z}(n))\otimes_{\mathbb{Z}}^L
R\mathrm{Hom}(R\Gamma(\overline{\mathcal{X}}_{et},\mathbb{Z}(n),\mathbb{Q}[-2d-1])\ar[d]^{a\otimes b}\ar[r]&\mathbb{Q}[-2d-1]\ar[d]  \\
R\Gamma(\overline{\mathcal{X}}_{et},\widehat{\mathbb{Z}}(n))\otimes_{\mathbb{Z}}^L
R\Gamma(\overline{\mathcal{X}}_{et},\mathbb{Q}/\mathbb{Z}(d-n))\ar[r]&\mathbb{Q}/\mathbb{Z}[-2d-1]
}
\]
commutes, where  $R\Gamma(\overline{\mathcal{X}}_{et},\mathbb{Z}(n))\stackrel{a}{\rightarrow} R\Gamma(\overline{\mathcal{X}}_{et},\widehat{\mathbb{Z}}(n))$ is the obvious map and
$$R\mathrm{Hom}(R\Gamma(\overline{\mathcal{X}}_{et},\mathbb{Z}(n),\mathbb{Q}[-2d-1])\stackrel{b}{\rightarrow} R\Gamma(\overline{\mathcal{X}}_{et},\mathbb{Q}/\mathbb{Z}(d-n))\rightarrow R\Gamma(\overline{\mathcal{X}}_{et},\mathbb{Z}(d-n))[1]$$
is $\alpha_{\mathcal{X},d-n}[1]$. Moreover, the square \[ \xymatrix{
R\Gamma(\overline{\mathcal{X}}_{et},\mathbb{Z}(n))\otimes_{\mathbb{Z}}^L
R\Gamma(\overline{\mathcal{X}}_{et},\mathbb{Q}(d-n))\ar[d]\ar[r]^{\hspace{2cm}0}&\mathbb{Q}[-2d-1]\ar[d]  \\
R\Gamma(\overline{\mathcal{X}}_{et},\widehat{\mathbb{Z}}(n))\otimes_{\mathbb{Z}}^L
R\Gamma(\overline{\mathcal{X}}_{et},\mathbb{Q}/\mathbb{Z}(d-n))\ar[r]&\mathbb{Q}/\mathbb{Z}[-2d-1]
}
\]
commutes as well, where the top horizontal arrow is the zero map. It follows that the lower square of the diagram above commutes.

Then the existence of the upper horizontal map $p_{n,d-n}$ follows from the fact that the colons of the diagram above are exact triangles. Its uniqueness follows from the argument given in the proof of Theorem \ref{cor-functoriality}. By adjunction we obtain
\begin{equation}\label{RGamma-map}
R\Gamma_W(\overline{\mathcal{X}},\mathbb{Z}(n))\rightarrow  R\mathrm{Hom}(R\Gamma_W(\overline{\mathcal{X}},\mathbb{Z}(d-n)),\mathbb{Z}[-2d-1]).
\end{equation}
Applying the functor $(-)\otimes^L_{\mathbb{Z}}\mathbb{Z}/m\mathbb{Z}$ to (\ref{RGamma-map}) we obtain the map
\begin{eqnarray}
\label{she}R\Gamma(\overline{\mathcal{X}}_{et},\mathbb{Z}/m(n))&\rightarrow &  R\mathrm{Hom}(R\Gamma(\overline{\mathcal{X}}_{et},\mathbb{Z}/m(d-n))[-1],\mathbb{Z}[-2d-1])\\
&\simeq &  R\mathrm{Hom}(R\Gamma(\overline{\mathcal{X}}_{et},\mathbb{Z}/m(d-n)),\mathbb{Q}/\mathbb{Z}[-2d-1])
\end{eqnarray}
where we identify $R\Gamma_W(\overline{\mathcal{X}},\mathbb{Z}(n))\otimes^L \mathbb{Z}/m$ with $R\Gamma(\overline{\mathcal{X}}_{et},\mathbb{Z}/m(n))$. By construction, (\ref{she}) is the map induced by ${\bf AV}(\overline{\mathcal{X}}_{et},n)$, which is an isomorphism by assumption. So (\ref{RGamma-map}) is a morphism in $\mathcal{D}$ of perfect complexes of abelian groups such that
$$R\Gamma_W(\overline{\mathcal{X}},\mathbb{Z}(n))\otimes^L\mathbb{Z}/m\mathbb{Z}\stackrel{\sim}{\longrightarrow}  R\mathrm{Hom}(R\Gamma_W(\overline{\mathcal{X}},\mathbb{Z}(d-n)),\mathbb{Z}[-2d-1])\otimes^L\mathbb{Z}/m\mathbb{Z}
$$
is an isomorphism for any $m$. It follows that (\ref{RGamma-map}) is an isomorphism as well.

\end{proof}

\subsection{The complex $R\Gamma_{W,c}(\mathcal{X},\mathbb{Z}(n))$}\label{sect-compactsupport}

We continue to use the notations of Section \ref{sectAVD}.
In particular, we denote by $\mathcal{X}(\mathbb{C})$ the set of complex points of $\mathcal{X}$ endowed with the complex topology. Complex conjugation gives a continuous action of $G_{\mathbb{R}}$ on $\mathcal{X}(\mathbb{C})$, and we denote by $\mathcal{X}_{\infty}:=\mathcal{X}(\mathbb{C})/G_{\mathbb{R}}$ the quotient topological space.
We also denote by $Sh(G_{\mathbb{R}},\mathcal{X}(\mathbb{C}))$ the topos of $G_{\mathbb{R}}$-equivariant sheaves on  $\mathcal{X}(\mathbb{C})$, and by $R\Gamma(G_{\mathbb{R}},\mathcal{X}(\mathbb{C}),-)$ the cohomology of the topos $Sh(G_{\mathbb{R}},\mathcal{X}(\mathbb{C}))$. We consider the $G_{\mathbb{R}}$-equivariant sheaves given by the $G_{\mathbb{R}}$-modules $(2\pi i)^n\mathbb{Z}$. We denote by $\pi:Sh(G_{\mathbb{R}},\mathcal{X}(\mathbb{C}))\rightarrow Sh(\mathcal{X}_{\infty})$ the canonical morphism of topoi. Recall that there is a natural transformation $R\pi_*\rightarrow R\widehat{\pi}_*$, where $R\widehat{\pi}_*$ is the functor defined in Section \ref{sectTate}.

\begin{defn}\label{iinftydef}
For any $n\in\mathbb{Z}$, we define the complex of sheaves on $\mathcal{X}_{\infty}$:
$$i^*_{\infty}\mathbb{Z}(n):=\mathrm{Cone}(R\pi_*(2\pi i)^n\mathbb{Z}\longrightarrow \tau^{>n}R\widehat{\pi}_*(2\pi i)^n\mathbb{Z})[-1]$$
and we set $$R\Gamma_W(\mathcal{X}_{\infty},\mathbb{Z}(n)):=R\Gamma(\mathcal{X}_{\infty},i^*_{\infty}\mathbb{Z}(n)).$$
\end{defn}
For $n\geq 0$, the canonical map $\tau^{>n}R\pi_*(2\pi i)^n\mathbb{Z}\rightarrow \tau^{>n}R\widehat{\pi}_*(2\pi i)^n\mathbb{Z}$ is a quasi-isomorphism (see the proof of Lemma \ref{lemcomp}), so that we have a quasi-isomorphism $$i^*_{\infty}\mathbb{Z}(n)\stackrel{\sim}{\rightarrow}\tau^{\leq n}R\pi_*(2i\pi)^n\mathbb{Z}$$
for $n\geq 0$. There is an exact triangle
\begin{equation}\label{test}
R\Gamma_W(\mathcal{X}_{\infty},\mathbb{Z}(n))\longrightarrow R\Gamma(G_{\mathbb{R}},\mathcal{X}(\mathbb{C}),(2\pi i)^n\mathbb{Z})
\stackrel{t}{\longrightarrow} R\Gamma(\mathcal{X}(\mathbb{R}),\tau^{>n}R\widehat{\pi}_*(2\pi i)^n\mathbb{Z}).
\end{equation}
The projective bundle formula
$$R\Gamma_W(\mathbb{P}^N_{\mathcal{X},\infty},\mathbb{Z}(n))\simeq
\bigoplus^{i=N}_{i=0} R\Gamma_W(\mathcal{X}_{\infty},\mathbb{Z}(n-i))[-2i]$$
can be obtained using  (\ref{test}) and an argument similar to the one given in the proof of Proposition \ref{prop-pbf}. We consider the representative of $R\Gamma_W(\mathcal{X}_{\infty},\mathbb{Z}(n))$ given by the mapping fibre of the second map $t$ in the triangle (\ref{test}).
Consider the map of complexes (\ref{bettiregulator}) $$R\Gamma(\mathcal{X}_{et},\mathbb{Z}(n))\longrightarrow R\Gamma(G_{\mathbb{R}},\mathcal{X}(\mathbb{C}),(2\pi i)^n\mathbb{Z}).$$
We may redefine the object $R\Gamma(\overline{\mathcal{X}}_{et},\mathbb{Z}(n))\in\mathcal{D}$ as the one given by the mapping fibre of the map of complexes
$$R\Gamma(\mathcal{X}_{et},\mathbb{Z}(n))\longrightarrow R\Gamma(G_{\mathbb{R}},\mathcal{X}(\mathbb{C}),(2\pi i)^n\mathbb{Z})
\longrightarrow R\Gamma(\mathcal{X}(\mathbb{R}),\tau^{>n}R\widehat{\pi}_*(2\pi i)^n\mathbb{Z}).$$
The square of complexes
\[ \xymatrix{
R\Gamma(\mathcal{X}_{et},\mathbb{Z}(n))\ar[r]\ar[d]&R\Gamma(\mathcal{X}(\mathbb{R}),\tau^{>n}R\widehat{\pi}_*(2\pi i)^n\mathbb{Z})\ar[d] \\
R\Gamma(G_{\mathbb{R}},\mathcal{X}(\mathbb{C}),(2i\pi)^n\mathbb{Z})\ar[r]& R\Gamma(\mathcal{X}(\mathbb{R}),\tau^{>n}R\widehat{\pi}_*(2\pi i)^n\mathbb{Z})
}
\]
commutes. By functoriality of the cone,
we obtain a canonical map
$$u^*_{\infty}:R\Gamma(\overline{\mathcal{X}}_{et},\mathbb{Z}(n))\longrightarrow R\Gamma_W(\mathcal{X}_{\infty},\mathbb{Z}(n)).$$

\begin{prop}\label{lem-com2}
There exists a unique map $$i^*_{\infty}:R\Gamma_W(\overline{\mathcal{X}},\mathbb{Z}(n))\longrightarrow R\Gamma_W(\mathcal{X}_{\infty},\mathbb{Z}(n))$$ which renders the following square commutative:
\[ \xymatrix{
R\Gamma(\overline{\mathcal{X}}_{et},\mathbb{Z}(n))\ar[r]\ar[d]^{u^*_{\infty}}&R\Gamma_W(\overline{\mathcal{X}},\mathbb{Z}(n))\ar[d]^{i^*_{\infty}} \\
R\Gamma_W(\mathcal{X}_{\infty},\mathbb{Z}(n))\ar[r]^{\mathrm{Id}}& R\Gamma_W(\mathcal{X}_{\infty},\mathbb{Z}(n))
}
\]
Moreover, the square
\[ \xymatrix{
R\Gamma_W(\overline{\mathcal{X}},\mathbb{Z}(n))_{\mathbb{Q}}\ar[r]^{\sim\hspace{3cm}}\ar[d]^{i^*_{\infty}\otimes\mathbb{Q}}&R\Gamma(\mathcal{X},\mathbb{Q}(n))
\oplus R\mathrm{Hom}(R\Gamma(\mathcal{X},\mathbb{Q}(d-n)),\mathbb{Q}[-2d-1])\ar[d]^{(u^*_{\infty}\otimes\mathbb{Q},\,0)} \\
R\Gamma_W(\mathcal{X}_{\infty},\mathbb{Z}(n))_{\mathbb{Q}}\ar[r]^{\mathrm{Id}}& R\Gamma_W(\mathcal{X}_{\infty},\mathbb{Z}(n))_{\mathbb{Q}}
}
\]
commutes, where the top horizontal map is the isomorphism of Corollary \ref{prop-rational-decompo}.

\end{prop}
\begin{proof}
We set $D_{\mathcal{X},n}:=R\mathrm{Hom}(R\Gamma(\mathcal{X},\mathbb{Q}(d-n)),\mathbb{Q}[-2d-2])$.
It follows from Lemma \ref{lem-Matthias} below and from the fact that $D_{\mathcal{X},n}$ is acyclic in degrees $<2n+2$ that
the composite map
\begin{equation}\label{hastob0}
D_{\mathcal{X},n}\stackrel{\alpha_{\mathcal{X},n}}{\longrightarrow}
R\Gamma(\overline{\mathcal{X}}_{et},\mathbb{Z}(n))\longrightarrow R\Gamma_W(\mathcal{X}_{\infty},\mathbb{Z}(n))
\end{equation}
is the zero map. The existence of $i^*_{\infty}$ follows. Its uniqueness can then be obtained as in Theorem \ref{cor-functoriality}, using the fact that both $R\Gamma_W(\overline{\mathcal{X}},\mathbb{Z}(n))$ and $R\Gamma_W(\mathcal{X}_{\infty},\mathbb{Z}(n))$ are perfect complexes.

We now show the second statement of the proposition. The exact triangle
$$D_{\mathcal{X},n}\rightarrow \tau^{\geq 2n+1}R\Gamma(\overline{\mathcal{X}}_{et},\mathbb{Z}(n))\rightarrow \tau^{\geq 2n+1}R\Gamma_W(\overline{\mathcal{X}},\mathbb{Z}(n))$$ and the argument given in the proof of Theorem \ref{cor-functoriality} show that the map from
$$\mathrm{Hom}_{\mathcal{D}}(\tau^{\geq 2n+1}R\Gamma_W(\overline{\mathcal{X}},\mathbb{Z}(n)),\tau^{\geq 2n+1}R\Gamma_W(\mathcal{X}_{\infty},\mathbb{Z}(n)))$$ to $$\mathrm{Hom}_{\mathcal{D}}(\tau^{\geq 2n+1}R\Gamma(\overline{\mathcal{X}}_{et},\mathbb{Z}(n)),\tau^{\geq 2n+1}R\Gamma_W(\mathcal{X}_{\infty},\mathbb{Z}(n)))$$
is injective. Hence the fact that $\tau^{\geq 2n+1}u^*_{\infty}$ is torsion (see Lemma \ref{lem-Matthias} below) implies that   $\tau^{\geq 2n+1}i^*_{\infty}$ is torsion as well. It follows that
$$\tau^{\geq 2n+1}(i^*_{\infty}\otimes \mathbb{Q}):D_{\mathcal{X},n}[-1]\simeq \tau^{\geq 2n+1}R\Gamma_W(\overline{\mathcal{X}},\mathbb{Z}(n))_{\mathbb{Q}}\rightarrow \tau^{\geq 2n+1}R\Gamma_W(\mathcal{X}_{\infty},\mathbb{Z}(n))_{\mathbb{Q}}$$
is the zero map. The fact that $\tau^{\leq 2n}(i^*_{\infty}\otimes \mathbb{Q})$  may be identified with $u^*_{\infty}\otimes \mathbb{Q}$ follows from the commutativity of the first square of the proposition.
\end{proof}

\begin{lem}\label{lem-Matthias}
The map $$\tau^{\geq 2n+1}(u^*_{\infty}):\tau^{\geq 2n+1} R\Gamma(\overline{\mathcal{X}}_{et},\mathbb{Z}(n))\rightarrow \tau^{\geq 2n+1} R\Gamma_W(\mathcal{X}_{\infty},\mathbb{Z}(n))$$
is torsion.
\end{lem}
\begin{proof}

In view of the exact triangle (\ref{test}), one is reduced to showing that the composite map
$$v:\tau^{\geq 2n+1} R\Gamma(\overline{\mathcal{X}}_{et},\mathbb{Z}(n))\rightarrow \tau^{\geq 2n+1} R\Gamma(G_{\mathbb{R}},\mathcal{X}(\mathbb{C}),(2\pi i)^n\mathbb{Z})$$
is torsion. Denote by $\tau^{\geq 2n}R\Gamma(\overline{\mathcal{X}}_{et},\mathbb{Q}/\mathbb{Z}(n))'$ the cokernel of the morphism of complexes
$$H^{2n}(\overline{\mathcal{X}}_{et},\mathbb{Q}/\mathbb{Z}(n))_{div}[-2n]\rightarrow \tau^{\leq 2n}\tau^{\geq 2n}R\Gamma(\overline{\mathcal{X}}_{et},\mathbb{Q}/\mathbb{Z}(n))\rightarrow \tau^{\geq 2n}R\Gamma(\overline{\mathcal{X}}_{et},\mathbb{Q}/\mathbb{Z}(n)).$$
Similarly, let $\tau^{\geq 2n} R\Gamma(G_{\mathbb{R}},\mathcal{X}(\mathbb{C}),(2\pi i)^n\mathbb{Q}/\mathbb{Z})'$
be the cokernel of the morphism of complexes
$$H^{2n}(G_{\mathbb{R}},\mathcal{X}(\mathbb{C}),(2\pi i)^n\mathbb{Q}/\mathbb{Z})_{div}[-2n]\rightarrow \tau^{\geq 2n}R\Gamma(G_{\mathbb{R}},\mathcal{X}(\mathbb{C}),(2\pi i)^n\mathbb{Q}/\mathbb{Z}).$$
Then we have a commutative square
\[ \xymatrix{
\tau^{\geq 2n}R\Gamma(\overline{\mathcal{X}}_{et},\mathbb{Q}/\mathbb{Z}(n))'[-1]\ar[r]^{\sim}\ar[d]^{\tilde{v}}& \tau^{\geq 2n+1} R\Gamma(\overline{\mathcal{X}}_{et},\mathbb{Z}(n))\ar[d]^{v} \\
\tau^{\geq 2n} R\Gamma(G_{\mathbb{R}},\mathcal{X}(\mathbb{C}),(2\pi i)^n\mathbb{Q}/\mathbb{Z})'[-1]\ar[r]& \tau^{\geq 2n+1} R\Gamma(G_{\mathbb{R}},\mathcal{X}(\mathbb{C}),(2\pi i)^n\mathbb{Z})
}
\]
where the upper horizontal map is a quasi-isomorphism. One is therefore reduced to showing that $\tilde{v}$ is torsion. Note also that $H^{2n}(\tau^{\geq 2n}R\Gamma(\overline{\mathcal{X}}_{et},\mathbb{Q}/\mathbb{Z}(n))')$ is finite. Verdier's spectral sequence (tensored with $\mathbb{Q}$) shows that $\tilde{v}$ is torsion if and only if the induced map
$$H^i(\overline{\mathcal{X}}_{et},\mathbb{Q}/\mathbb{Z}(n))_{div}\rightarrow H^i(G_{\mathbb{R}},\mathcal{X}(\mathbb{C}),(2\pi i)^n\mathbb{Q}/\mathbb{Z})_{div}
$$
is the zero map for any $i\geq 2n+1$. Since the kernel of the map $$H^i(G_{\mathbb{R}},\mathcal{X}(\mathbb{C}),(2\pi i)^n\mathbb{Q}/\mathbb{Z})_{div}\rightarrow H^i(\mathcal{X}(\mathbb{C}),(2\pi i)^n\mathbb{Q}/\mathbb{Z})_{div}$$ is killed by a power of $2$, it suffices to show that
\begin{equation}\label{zer}
H^i(\overline{\mathcal{X}}_{et},\mathbb{Q}/\mathbb{Z}(n))_{div}\rightarrow H^i(\mathcal{X}(\mathbb{C}),(2\pi i)^n\mathbb{Q}/\mathbb{Z})_{div}
\end{equation}
is the zero map for any $i\geq 2n+1$.

The map (\ref{zer}) factors through
$\left(H^{i}(\mathcal{X}_{\overline{\mathbb{Q}},et},\mu^{\otimes n})^{G_{\mathbb{Q}}}\right)_{div}$
where $\overline{\mathbb{Q}}/\mathbb{Q}$ is an algebraic closure and $\mu$ is the \'etale sheaf on $\mathcal{X}_{\overline{\mathbb{Q}}}$ of all roots of unity. It is therefore enough to showing that
\begin{equation}\label{eq-cequilfaut}
\left(H^{i}(\mathcal{X}_{\overline{\mathbb{Q}},\,et},\mu^{\otimes n})^{G_{\mathbb{Q}}}\right)_{div}=
\bigoplus_{l}\left(H^{i}(\mathcal{X}_{\overline{\mathbb{Q}},\,et},\mathbb{Q}_l/\mathbb{Z}_l(n))^{G_{\mathbb{Q}}}\right)_{div}=0
\end{equation}
for any $i\geq 2n+1$. Let $l$ be a fixed prime number. Let $U\subseteq \mathrm{Spec}(\mathbb{Z})$ be an open subscheme on which $l$ is invertible and such that $\mathcal{X}_U\rightarrow U$ is smooth, and let $p\in U$. By smooth and proper base change we have:
$$H^{i}(\mathcal{X}_{\overline{\mathbb{Q}},\,et},\mathbb{Q}_l/\mathbb{Z}_l(n))^{I_{p}}\simeq H^{i}(\mathcal{X}_{\overline{\mathbb{F}}_p,\,et},\mathbb{Q}_l/\mathbb{Z}_l(n))$$
where $I_p$ denotes an inertia subgroup at $p$. Recall that $H^{i}(\mathcal{X}_{\overline{\mathbb{F}}_p,\,et},\mathbb{Z}_l(n))$ is a finitely generated
$\mathbb{Z}_l$-module. We have an exact sequence
$$0\rightarrow H^{i}(\mathcal{X}_{\overline{\mathbb{F}}_p,\,et},\mathbb{Z}_l(n))_{cotor}\rightarrow H^{i}(\mathcal{X}_{\overline{\mathbb{F}}_p,\,et},\mathbb{Q}_l(n))\rightarrow
H^{i}(\mathcal{X}_{\overline{\mathbb{F}}_p,\,et},\mathbb{Q}_l/\mathbb{Z}_l(n))_{div}\rightarrow 0.$$
We get
$$0\rightarrow (H^{i}(\mathcal{X}_{\overline{\mathbb{F}}_p,\,et},\mathbb{Z}_l(n))_{cotor})^{G_{\mathbb{F}_p}}\rightarrow H^{i}(\mathcal{X}_{\overline{\mathbb{F}}_p,\,et},\mathbb{Q}_l(n))^{G_{\mathbb{F}_p}}$$
$$\rightarrow
(H^{i}(\mathcal{X}_{\overline{\mathbb{F}}_p,\,et},\mathbb{Q}_l/\mathbb{Z}_l(n))_{div})^{G_{\mathbb{F}_p}}
\rightarrow H^1(G_{\mathbb{F}_p},H^{i}(\mathcal{X}_{\overline{\mathbb{F}}_p,\,et},\mathbb{Z}_l(n))_{cotor}).$$
Again, $ H^1(G_{\mathbb{F}_p},H^{i}(\mathcal{X}_{\overline{\mathbb{F}}_p,\,et},\mathbb{Z}_l(n))_{cotor})$ is a finitely generated
$\mathbb{Z}_l$-module, hence we get a surjective map
$$H^{i}(\mathcal{X}_{\overline{\mathbb{F}}_p,\,et},\mathbb{Q}_l(n))^{G_{\mathbb{F}_p}}\rightarrow
((H^{i}(\mathcal{X}_{\overline{\mathbb{F}}_p,\,et},\mathbb{Q}_l/\mathbb{Z}_l(n))_{div})^{G_{\mathbb{F}_p}})_{div}
\rightarrow0.$$
Note that
$$\left((H^{i}(\mathcal{X}_{\overline{\mathbb{F}}_p,\,et},\mathbb{Q}_l/\mathbb{Z}_l(n))_{div})^{G_{\mathbb{F}_p}}\right)_{div}
= \left(H^{i}(\mathcal{X}_{\overline{\mathbb{F}}_p,\,et},\mathbb{Q}_l/\mathbb{Z}_l(n))^{G_{\mathbb{F}_p}}\right)_{div}.$$
By the Weil Conjectures, $H^{i}(\mathcal{X}_{\overline{\mathbb{F}}_p,\,et},\mathbb{Q}_l(n))$ is pure of weight $i-2n$. For $i\geq2n+1$, we have $i-2n>0$, hence there is no non-trivial element in $H^{i}(\mathcal{X}_{\overline{\mathbb{F}}_p,\,et},\mathbb{Q}_l(n))$ fixed by the Frobenius. This shows that
$$
\left(H^{i}(\mathcal{X}_{\overline{\mathbb{F}}_p,\,et},\mathbb{Q}_l/\mathbb{Z}_l(n))^{G_{\mathbb{F}_p}}\right)_{div}
=H^{i}(\mathcal{X}_{\overline{\mathbb{F}}_p,\,et},\mathbb{Q}_l(n))^{G_{\mathbb{F}_p}}=0$$
hence that
$$\left(H^{i}(\mathcal{X}_{\overline{\mathbb{Q}},\,et},\mathbb{Q}_l/\mathbb{Z}_l(n))^{G_{\mathbb{Q}_p}}\right)_{div}
\simeq \left(H^{i}(\mathcal{X}_{\overline{\mathbb{F}}_p,\,et},\mathbb{Q}_l/\mathbb{Z}_l(n))^{G_{\mathbb{F}_p}}\right)_{div}=0.$$
A fortiori, one has
$$\left(H^{i}(\mathcal{X}_{\overline{\mathbb{Q}},\,et},\mathbb{Q}_l/\mathbb{Z}_l(n))^{G_{\mathbb{Q}}}\right)_{div}=0$$
for $i\geq2n+1$, and the result follows.

\end{proof}

\begin{defn}
We define
$R\Gamma_{W,c}(\mathcal{X},\mathbb{Z}(n))$, up to a non-canonical isomorphism in $\mathcal{D}$, such that we have an exact triangle
\begin{equation}\label{triangle-cpctsupp-fgcoh}
R\Gamma_{W,c}(\mathcal{X},\mathbb{Z}(n))\longrightarrow R\Gamma_{W}(\overline{\mathcal{X}},\mathbb{Z}(n))
\stackrel{i_{\infty}^*}{\longrightarrow}R\Gamma_W(\mathcal{X}_{\infty},\mathbb{Z}(n)).
\end{equation}
The determinant $\mathrm{det}_{\mathbb{Z}} R\Gamma_{W,c}(\mathcal{X},\mathbb{Z}(n))$ is well defined up to a \emph{canonical} isomorphism.
\label{z-rgc-def}\end{defn}
To see that $\mathrm{det}_{\mathbb{Z}} R\Gamma_{W,c}(\mathcal{X},\mathbb{Z}(n))$ is indeed well defined, consider another object $R\Gamma_{W,c}(\mathcal{X},\mathbb{Z}(n))'$ of $\mathcal{D}$ endowed with an exact triangle (\ref{triangle-cpctsupp-fgcoh}). There exists a (non-unique) morphism $$u: R\Gamma_{W,c}(\mathcal{X},\mathbb{Z}(n))\rightarrow R\Gamma_{W,c}(\mathcal{X},\mathbb{Z}(n))'$$ lying in a morphism of exact triangles
\[ \xymatrix{
R\Gamma_W(\mathcal{X}_{\infty},\mathbb{Z}(n))[-1]\ar[d]_{Id}\ar[r]
&R\Gamma_{W,c}(\mathcal{X},\mathbb{Z}(n))\ar[d]_{\exists\,u}^{\simeq}\ar[r]
&R\Gamma_{W}(\overline{\mathcal{X}},\mathbb{Z}(n))\ar[d]_{Id}\ar[r]
&R\Gamma_W(\mathcal{X}_{\infty},\mathbb{Z}(n))\ar[d]_{Id}\\
R\Gamma_W(\mathcal{X}_{\infty},\mathbb{Z}(n))[-1]\ar[r]
&R\Gamma_{W,c}(\mathcal{X},\mathbb{Z}(n))'\ar[r]
&R\Gamma_{W}(\overline{\mathcal{X}},\mathbb{Z}(n))\ar[r]
&R\Gamma_W(\mathcal{X}_{\infty},\mathbb{Z}(n))
}
\]
The map $u$ induces
$$\mathrm{det}_{\mathbb{Z}}(u):\mathrm{det}_{\mathbb{Z}} R\Gamma_{W,c}(\mathcal{X},\mathbb{Z}(n))\stackrel{\sim}{\longrightarrow}\mathrm{det}_{\mathbb{Z}} R\Gamma_{W,c}(\mathcal{X},\mathbb{Z}(n))'$$
which does not depend on the choice of $u$ \cite{Knudsen-Mumford-76}.

\section{Weil-Arakelov cohomology of proper regular schemes}\label{sec:arithmetic}

In this section  $\mathcal{X}$ denotes a regular scheme proper over $\mathbb{Z}$ of pure dimension $d$  which satisfies ${\bf AV}(\overline{\mathcal{X}}_{et},n)$, ${\bf L}(\overline{\mathcal{X}}_{et},n)$ and ${\bf L}(\overline{\mathcal{X}}_{et},d-n)$ and ${\bf B}(\mathcal{X},n)$. The Weil-Arakelov complexes we introduce in this section will only play a minor role in subsequent sections (in Conjecture \ref{conj-vanishingorder} which does not really need them for its formulation), and we mainly discuss them to make precise the picture outlined in the introduction. The Weil-Arakelov complexes defined in subsections \ref{sec:w-ar-xbar} and \ref{sect-arcompactsupport} below will only be specified up to a {\em noncanonical} isomorphism because they are defined as mapping fibres or mapping cones in the derived category of abelian groups. We certainly do expect a canonical construction of these objects when the geometry underlying Arakelov theory is better understood but we do not have better definitions at this point. There are more complexes than those discussed below for which we have definitions, for example $R\Gamma_{\Ar}(X,\bz(n))$, but these definitions are of the same preliminary nature and we do not include them.

\subsection{Weil-Arakelov cohomology with $\tr(n)$-coefficients} For any $n\in\bz$ recall the diagram (\ref{rgc2}) of (perfect) complexes of $\br$-vector spaces.

\begin{defn}
For each  complex $R\Gamma_?(Y,\br(n))$ in diagram (\ref{rgc2}) set
\begin{equation} R\Gamma_{\Ar,?}(Y,\tr(n)):=R\Gamma_?(Y,\br(n))\oplus R\Gamma_?(Y,\br(n))[-1].\label{ardef}\end{equation}
\label{r-rgc-def}\end{defn}

We define a map $\xrightarrow{\cup\theta}$ by commutativity of the diagram
\[\begin{CD} H^i_{\Ar,?}(Y,\tr(n))@>{\cup\theta}>> H^{i+1}_{\Ar,?}(Y,\tr(n))\\
\Vert@. \Vert@.\\
H^i_?(Y,\br(n))\oplus H^{i-1}_?(Y,\br(n))@>{\left(\begin{smallmatrix} 0 & \id\\0 & 0\end{smallmatrix}\right)} >> H^{i+1}_?(Y,\br(n))\oplus H^{i}_?(Y,\br(n))\end{CD}\]
so that there is a long exact sequence
\begin{equation}\cdots\xrightarrow{\cup\theta}H^i_{\Ar,?}(Y,\tr(n))\xrightarrow{\cup\theta} H^{i+1}_{\Ar,?}(Y,\tr(n))\xrightarrow{\cup\theta}\cdots\label{cupseq}\end{equation}
The motivation for this definition is its compatibility with previous work on Weil-etale cohomology, with the Weil-Arakelov groups with $\bz(n)$-coefficients defined below, and possibly also with the conjectural picture of Deninger \cite{deninger01}.

The dualities (\ref{deldual}), ${\bf B}(\mathcal{X},n)$ and (\ref{dual2}) imply corresponding dualities for the Weil-Arakelov groups where the top degree is increased by one. We record the duality implied by (\ref{dual2}) in the following proposition.

\begin{prop} There is a canonical homomorphism
$H^{2d+1}_\Ar(\overline{\X},\tr(d))\to\br$ and a perfect duality
\[ H^i_\Ar(\overline{\X},\tr(n))\times H^{2d+1-i}_\Ar(\overline{\X},\tr(d-n)) \to H^{2d+1}_\Ar(\overline{\X},\tr(d)) \to\br  \]
of finite-dimensional $\br$-vector spaces for
all $i,n\in\bz$. Moreover we have $$H^i_\Ar(\overline{\X},\tr(n))=0\quad\text{for $i\neq 2n,2n+1$.}$$
\label{arxbardual}\end{prop}

\begin{proof} This is immediate from Prop. \ref{xbardual}. \end{proof}

\subsection{Weil-Arakelov cohomology of $\overline{\X}$} \label{sec:w-ar-xbar}
Recall the direct sum decomposition
$$R\Gamma_W(\overline{\mathcal{X}},\mathbb{Z}(n))_\mathbb{Q}\stackrel{\sim}{\longrightarrow} R\Gamma(\mathcal{X},\mathbb{Q}(n))
\oplus R\mathrm{Hom}(R\Gamma(\mathcal{X},\mathbb{Q}(d-n)),\mathbb{Q}[-2d-1])$$
of Corollary \ref{prop-rational-decompo} which induces a decomposition
\begin{equation}R\Gamma_W(\overline{\mathcal{X}},\mathbb{Z}(n))_\mathbb{R}\stackrel{\sim}{\longrightarrow} R\Gamma(\mathcal{X},\mathbb{R}(n))
\oplus R\mathrm{Hom}(R\Gamma(\mathcal{X},\mathbb{R}(d-n)),\mathbb{R}[-2d-1]).\label{rgwdecomp}\end{equation}
Also recall the map (\ref{regsigma})
\begin{equation} R\Gamma(\X,\br(n))\xrightarrow{\rho} R\Gamma_\cd(\X_{/\br},\br(n)) \xrightarrow{\sigma} \tau^{\leq{2n-1}}R\Gamma_\cd(\X_{/\br},\br(n))\notag\end{equation}
where $\rho$ is the Beilinson regulator and $\sigma$ a splitting of the natural inclusion
\begin{equation}\tau^{\leq{2n-1}}R\Gamma_\cd(\X_{/\br},\br(n))\to R\Gamma_\cd(\X_{/\br},\br(n)).
\notag\end{equation}
We obtain a composite map
\begin{equation} R\Gamma_W(\overline{\mathcal{X}},\mathbb{Z}(n))\xrightarrow{\otimes 1} R\Gamma_W(\overline{\mathcal{X}},\mathbb{Z}(n))_\br\xrightarrow{\pi_1}R\Gamma(\mathcal{X},\mathbb{R}(n))\xrightarrow{\sigma\circ\rho}\tau^{\leq{2n-1}}R\Gamma_\cd(\X_{/\br},\br(n))
\label{ardefmap}\end{equation}
where $\pi_1$ is the first projection in (\ref{rgwdecomp}).
\begin{defn}
Define $R\Gamma_{\Ar}(\overline{\mathcal{X}},\mathbb{Z}(n))$ as a mapping fibre of the map (\ref{ardefmap}).
\label{xbarardef}\end{defn}
By definition there is an exact triangle
\begin{equation}R\Gamma_{\Ar}(\overline{\mathcal{X}},\mathbb{Z}(n))\rightarrow R\Gamma_W(\overline{\mathcal{X}},\mathbb{Z}(n))\rightarrow
\tau^{\leq{2n-1}}R\Gamma_\cd(\X_{/\br},\br(n))\to.\label{z-ar-w-triangle}\end{equation}

\begin{prop}\label{prop-map-ZtoR} There is a map in the derived category
\begin{equation} R\Gamma_{\Ar}(\overline{\mathcal{X}},\mathbb{Z}(n))\to R\Gamma_{\Ar}(\overline{\mathcal{X}},\tr(n))\label{z-to-r}\end{equation}
where $R\Gamma_{\Ar}(\bar{\mathcal{X}},\tr(n))$ was defined in (\ref{ardef}).
\end{prop}

\begin{proof} By definition
\[R\Gamma_{\Ar}(\bar{\mathcal{X}},\tr(n))\cong R\Gamma(\overline{\X},\br(n))\oplus R\Gamma(\overline{\X},\br(n))[-1],\]
 so to define (\ref{z-to-r}) we need to define its two components. Recall from section \ref{sec:av} that  $R\Gamma(\overline{\X},\br(n))$ was defined as the mapping fibre of $\sigma\circ\rho$. Hence we obtain an induced map on mapping fibres
\[ R\Gamma_{\Ar}(\overline{\mathcal{X}},\mathbb{Z}(n))\to R\Gamma(\overline{\X},\br(n))\]
which is the first component of (\ref{z-to-r}). The second component is the composite
\begin{multline} R\Gamma_{\Ar}(\overline{\mathcal{X}},\mathbb{Z}(n))\to R\Gamma_W(\overline{\mathcal{X}},\mathbb{Z}(n))\to R\Gamma_W(\overline{\mathcal{X}},\mathbb{Z}(n))_\mathbb{R}\to\\\xrightarrow{\pi_2}R\Gamma(\mathcal{X},\mathbb{R}(d-n))^*[-2d-1]
\xleftarrow{B}R\Gamma_c(\mathcal{X},\mathbb{R}(n))[-1]\xrightarrow{\iota}R\Gamma(\bar{\mathcal{X}},\mathbb{R}(n))[-1]\notag\end{multline}
where $\pi_2$ is the second projection in (\ref{rgwdecomp}), $B$ is the duality isomorphism of conjecture ${\bf B}(\mathcal{X},n)$ and $\iota$ is the natural map arising from the respective mapping fibre definitions of its source and target.
\end{proof}

\begin{defn}\label{def-coh-R/Z} Define $R\Gamma_{\Ar}(\overline{\mathcal{X}},\tr/\mathbb{Z}(n))$ as a mapping cone of the map (\ref{z-to-r}).
\end{defn}
By definition there is an exact triangle
$$R\Gamma_{\Ar}(\overline{\mathcal{X}},\mathbb{Z}(n))\rightarrow R\Gamma_{\Ar}(\overline{\mathcal{X}},\tr(n))\rightarrow R\Gamma_{\Ar}(\overline{\mathcal{X}},\tr/\mathbb{Z}(n))
\rightarrow.$$
The following proposition gives an analogue of (\ref{z-ar-w-triangle}) for $\tr/\bz(n)$-coefficients. Define the complex $R\Gamma_W(\overline{\mathcal{X}},\mathbb{R}/\mathbb{Z}(n))$ as the mapping cone of the map $\otimes 1$ in (\ref{ardefmap}) so that there is an exact triangle
\begin{equation} R\Gamma_W(\overline{\mathcal{X}},\mathbb{Z}(n))\xrightarrow{\otimes 1} R\Gamma_W(\overline{\mathcal{X}},\mathbb{Z}(n))_\br\xrightarrow{}R\Gamma_W(\overline{\mathcal{X}},\mathbb{R}/\mathbb{Z}(n))\to.
\label{w-z-to-r}\end{equation}

\begin{prop} There is an exact triangle
\begin{equation} \left(\tau^{\geq{2n}}R\Gamma_\cd(\X_{/\br},\br(n))\right)[-2]\to R\Gamma_W(\overline{\mathcal{X}},\mathbb{R}/\mathbb{Z}(n))\to R\Gamma_{\Ar}(\overline{\mathcal{X}},\tr/\mathbb{Z}(n))\to  . \label{r/z-ar-w-triangle}\end{equation}
\label{r/z-ar-w-prop}\end{prop}

\begin{proof} The definition of $R\Gamma_{\Ar}(\overline{\mathcal{X}},\mathbb{Z}(n))$ as a mapping fibre of the composite map $(\sigma\rho\pi_1)\circ\otimes 1$ gives an exact triangle
\[R\Gamma_W(\overline{\mathcal{X}},\mathbb{R}/\mathbb{Z}(n))[-1]\to R\Gamma_{\Ar}(\overline{\mathcal{X}},\mathbb{Z}(n))\xrightarrow{\alpha} \cone(\sigma\rho\pi_1)[-1]\to\]
and one has an isomorphism
\[ \cone(\sigma\rho\pi_1)[-1]\stackrel{\sim}{\longrightarrow} R\Gamma(\overline{\mathcal{X}},\mathbb{R}(n))
\oplus R\mathrm{Hom}(R\Gamma(\mathcal{X},\mathbb{R}(d-n)),\mathbb{R}[-2d-1]).\]
The definition of $R\Gamma_{\Ar}(\overline{\mathcal{X}},\tr/\mathbb{Z}(n))$ as a mapping cone of the composite map $(\id\oplus\iota B)\circ\alpha$ gives an exact triangle
\[R\Gamma_W(\overline{\mathcal{X}},\mathbb{R}/\mathbb{Z}(n))\to R\Gamma_{\Ar}(\overline{\mathcal{X}},\tr/\mathbb{Z}(n))\to\cone(\id\oplus\iota B)\to.\]
Since $\cone(\id\oplus\iota B)\cong\cone(\iota)$ and one has an exact triangle
\[ R\Gamma_c(\mathcal{X},\mathbb{R}(n))\xrightarrow{\iota[1]}R\Gamma(\bar{\mathcal{X}},\mathbb{R}(n))\to \tau^{\geq{2n}}R\Gamma_\cd(\X_{/\br},\br(n))\to\]
the proposition follows.
\end{proof}

\begin{rem} One has two exact sequences
\[\begin{CD}  {} @>>> H^{n-1,n-1}(\X_{\br})@>>> CH^n(\overline{\X})@>>> CH^n(\X) @>>> 0\\
@. @VV\cong V @VV\bar{\epsilon} V @VV\epsilon V @.\\
{}@>>> H^{2n-1}_\cd(\X_{/\br},\br(n)) @>>> H^{2n}_{\Ar}(\overline{\X},\bz(n)) @>>> H^{2n}(\X_\et,\bz(n)) @>>> 0
\end{CD}\]
where the top row is the exact sequence (\ref{gssequence}) of Gillet and Soule, the bottom sequence is induced by the exact triangle (\ref{z-ar-w-triangle}), $\epsilon$ is the natural map from the higher Chow complex to its etale hypercohomology and $\bar{\epsilon}$ we only expect to exist. However, even if $\bar{\epsilon}$ does exist it will not in general be an isomorphism because $\epsilon$ may not be an isomorphism. In general we expect the maps $\epsilon$ and $\bar{\epsilon}$ to have finite kernel and cokernel. We also remark that $H^{i}_{\Ar}(\overline{\X},\bz(n))$ can be nonzero both for $i<2n$ and $i>2n$ and satisfies the duality in Theorem \ref{thm:w-ar-duality} below.
\end{rem}

Deligne cohomology $R\Gamma_{\mathcal{D}}(\mathcal{X}_{/\mathbb{R}},\mathbb{R}(n))$ is contravariantly functorial and so is its truncation $\tau^{\leq 2n-1}R\Gamma_{\mathcal{D}}(\mathcal{X}_{/\mathbb{R}},\mathbb{R}(n))$ whereas the complex $R\Gamma_W(\overline{\mathcal{X}},\mathbb{Z}(n))$ is contravariantly functorial for flat morphisms.

\begin{lem}
Let $f:\mathcal{Y}\rightarrow\mathcal{X}$ be a flat morphism. Then there exists a (nonunique) map $$f^*:R\Gamma_{\Ar}(\overline{\mathcal{X}},\mathbb{Z}(n))\rightarrow R\Gamma_{\Ar}(\overline{\mathcal{Y}},\mathbb{Z}(n))$$
sitting in a morphism of exact triangles
\[\xymatrix{
R\Gamma_{\Ar}(\overline{\mathcal{X}},\mathbb{Z}(n))\ar[r]\ar[d]^{f^*} &R\Gamma_W(\overline{\mathcal{X}},\mathbb{Z}(n))\ar[r]\ar[d] & \tau^{\leq 2n-1}R\Gamma_{\mathcal{D}}(\mathcal{X}_{/\mathbb{R}},\mathbb{R}(n))\ar[d]\ar[r] &\\
R\Gamma_{\Ar}(\overline{\mathcal{Y}},\mathbb{Z}(n))\ar[r] &R\Gamma_W(\overline{\mathcal{Y}},\mathbb{Z}(n))\ar[r] & \tau^{\leq 2n-1}R\Gamma_{\mathcal{D}}(\mathcal{Y}_{/\mathbb{R}},\mathbb{R}(n))\ar[r] &
}
\]
\end{lem}
\begin{proof} We need to show that the outer square in the diagram
\[\xymatrix{
R\Gamma_W(\overline{\mathcal{X}},\mathbb{Z}(n))\ar[r]\ar[d] & R\Gamma_W(\overline{\mathcal{X}},\mathbb{Z}(n))_{\mathbb{Q}}\ar[r]\ar[d] & R\Gamma(\mathcal{X},\mathbb{Q}(n))\ar[r]\ar[d]&\tau^{\leq 2n-1}R\Gamma_{\mathcal{D}}(\mathcal{X}_{/\mathbb{R}},\mathbb{R}(n))\ar[d]\\
R\Gamma_W(\overline{\mathcal{Y}},\mathbb{Z}(n))\ar[r] &R\Gamma_W(\overline{\mathcal{Y}},\mathbb{Z}(n))_{\mathbb{Q}}\ar[r] & R\Gamma(\mathcal{Y},\mathbb{Q}(n))\ar[r] & \tau^{\leq 2n-1}R\Gamma_{\mathcal{D}}(\mathcal{Y}_{/\mathbb{R}},\mathbb{R}(n))
}
\]
commutes in the derived category. But the left square clearly commutes, the middle square commutes by Corollary \ref{prop-rational-decompo} and the right square commutes by functoriality of the Beilinson regulator.
\end{proof}

\subsection{Weil-Arakelov duality for $\overline{\X}$}\label{sec:w-ar-duality}

We have already noted in Prop. \ref{arxbardual} a duality for Weil-Arakelov cohomology with $\tr(n)$-coefficients. In this section we establish a Pontryagin duality between Weil-Arakelov cohomology with $\bz(n)$ and $\tr/\bz(d-n)$ coefficients.

For homological algebra of locally compact abelian groups we refer to  \cite{spitzweck07}. A continuous homomorphism $f:A\to B$ of locally compact abelian groups is called {\em strict} if $A/\overline{ker(f)}\to B$ is a closed embedding, and a complex of locally compact abelian groups is called {\em strictly acyclic} if all differentials are strict and the complex is acyclic in the usual sense. The bounded derived category of locally compact abelian groups is defined in \cite{spitzweck07} by inverting all maps of complexes whose mapping cone is strictly acyclic.

We denote by $G^D$ the Pontryagin dual of a locally compact abelian group $G$. The functor $(-)^D$ preserves strict morphisms and strictly acyclic complexes and extends to the bounded derived category of \cite{spitzweck07}. Examples of objects in this category are bounded complexes $P^\bullet$ (resp. $V^\bullet$) of finitely generated free abelian groups (resp. $\br$-vector spaces) as well as the complexes $R\Gamma_{\Ar}(\overline{\mathcal{X}},\mathbb{Z}(n))$ and $R\Gamma_{\Ar}(\overline{\mathcal{X}},\tr/\mathbb{Z}(n))$ defined in Def. \ref{xbarardef} and \ref{def-coh-R/Z}, respectively. To see this note that any map $P^\bullet\to V^\bullet$ in the derived category of abelian groups can be realized by a map of complexes which is automatically continuous since the $P^i$ carry the discrete topology. There is a natural isomorphism
\[ (V^\bullet)^D\cong (V^\bullet)^*\]
and a short exact sequence of complexes
\begin{equation} 0\to \Hom_\bz(P^\bullet,\bz)\xrightarrow{\iota} \Hom_\bz(P^\bullet,\br)\to (P^\bullet)^D\to 0. \label{Dcompute}\end{equation}
Finally note that the cohomology groups of a complex of locally compact abelian groups (taken in the category of abelian groups) carry an induced topology which however need not be locally compact.

\begin{thm} For $n\in\bz$ there is a quasi-isomorphism
\[ R\Gamma_{\Ar}(\overline{\mathcal{X}},\mathbb{Z}(n))^D\cong R\Gamma_{\Ar}(\overline{\mathcal{X}},\tr/\mathbb{Z}(d-n))[2d+1]\]
and the cohomology groups of both complexes are locally compact. The isomorphism
$$H^{2d+1}_{\Ar}(\overline{\mathcal{X}},\tr/\mathbb{Z}(d))\simeq \mathbb{R}/\mathbb{Z}$$
is canonical and hence one obtains a Pontryagin duality
$$H^{i}_{\Ar}(\overline{\mathcal{X}},\mathbb{Z}(n))\times H^{2d+1-i}_{\Ar}(\overline{\mathcal{X}},\tr/\mathbb{Z}(d-n))\rightarrow H^{2d+1}_{\Ar}(\overline{\mathcal{X}},\tr/\mathbb{Z}(d))\simeq \mathbb{R}/\mathbb{Z}.$$
\label{thm:w-ar-duality}\end{thm}

\begin{proof} One has an isomorphism of exact triangles with $\delta=2d+1$
\begin{equation}\minCDarrowwidth1em\begin{CD} {}@<<<R\Gamma_W(\overline{\mathcal{X}},\mathbb{Z}(n))^D @<<< R\Gamma_W(\overline{\mathcal{X}},\mathbb{Z}(n))^*_\br @<<< R\Gamma_W(\overline{\mathcal{X}},\mathbb{R}/\mathbb{Z}(n))^D\\
@.@A\beta A\sim A @AA\sim A @AA\sim A\\
{}@<<<R\Gamma_W(\overline{\mathcal{X}},\br/\mathbb{Z}(d-n))[\delta] @<<< R\Gamma_W(\overline{\mathcal{X}},\mathbb{Z}(d-n))_\br[\delta] @<<< R\Gamma_W(\overline{\mathcal{X}},\mathbb{Z}(d-n))[\delta]
\end{CD}\notag\end{equation}
where the top row is the Pontryagin dual of (\ref{w-z-to-r}) and the middle isomorphism is clear from (\ref{rgwdecomp}). More concretely, if $P^\bullet$ denotes a representative of $R\Gamma_W(\overline{\mathcal{X}},\mathbb{Z}(n))$ and
$\iota$ denotes the map in (\ref{Dcompute}) then $\beta$ is the map from the cone of $\iota$ to the quotient complex of $\iota$ (see \cite{weibel}[1.5.8]), combined with the duality isomorphism of Theorem \ref{thm-duality-fg}. Furthermore, one has an isomorphism of exact triangles
\begin{equation}\minCDarrowwidth1em\begin{CD}{}@<<< R\Gamma_{\Ar}(\overline{\mathcal{X}},\mathbb{Z}(n))^D@<<< R\Gamma_W(\overline{\mathcal{X}},\mathbb{Z}(n))^D @<<<
\left(\tau^{\leq{2n-1}}R\Gamma_\cd(\X_{/\br},\br(n))\right)^D\\
@.@A\beta' A\sim A @A\beta A\sim A @A\beta'' A\sim A\\
 {}@<<< R\Gamma_{\Ar}(\overline{\mathcal{X}},\tr/\mathbb{Z}(d-n))[\delta] @<<< R\Gamma_W(\overline{\mathcal{X}},\br/\mathbb{Z}(d-n))[\delta] @<<< \left(\tau^{\geq{2(d-n)}}R\Gamma_\cd(\X_{/\br},\br(d-n))\right)[\delta-2]
\end{CD}\label{}\end{equation}
where the top row is the Pontryagin dual of (\ref{z-ar-w-triangle}) and $\beta''$ arises from the duality (\ref{deldual}) for Deligne cohomology
\begin{align*} \left(\tau^{\leq{2n-1}}R\Gamma_\cd(\X_{/\br},\br(n))\right)^*\cong&\tau^{\geq{-2n+1}}\bigl(R\Gamma_\cd(\X_{/\br},\br(d-n))[2d-1]\bigr)\\
\cong &\left(\tau^{\geq{2(d-n)}}R\Gamma_\cd(\X_{/\br},\br(d-n))\right)[\delta-2].\end{align*}
The bottom row is (\ref{r/z-ar-w-triangle}) with $n$ replaced by $d-n$. The quasi-isomorphism $\beta'$ is non-canonical, quite like our pairing in Prop. \ref{xbardual}. However, in degree $2d+1$, one has a canonical isomorphism
\[ H^{2d+1}_{\Ar}(\overline{\mathcal{X}},\tr(d))\cong\br\]
arising from Lemma \ref{deligneduality} and Definition \ref{r-rgc-def} and a canonical isomorphism
\[ H^{2d+1}_{\Ar}(\overline{\mathcal{X}},\bz(d))\cong H^{2d+1}_{W}(\overline{\mathcal{X}},\bz(d))\cong\bz\]
arising from Theorem \ref{thm-duality-fg} and the map defined in Prop. \ref{prop-map-ZtoR} is in fact the inclusion.
\bigskip

The long exact sequence induced by (\ref{z-ar-w-triangle}) gives an isomorphism
\[H^{i}_{\Ar}(\overline{\mathcal{X}},\mathbb{Z}(n))\cong H^{i}_{W}(\overline{\mathcal{X}},\mathbb{Z}(n))\]
for $i\geq 2n+1$ and an exact sequence
\[ \cdots\to H^{i-1}_{W}(\overline{\mathcal{X}},\mathbb{Z}(n))\xrightarrow{\rho} H^{i-1}_\cd(\X_{/\br},\br(n)) \to H^{i}_{\Ar}(\overline{\mathcal{X}},\mathbb{Z}(n))\to H^{i}_{W}(\overline{\mathcal{X}},\mathbb{Z}(n))\to\cdots \]
for $i\leq 2n$. Since $H^{i}_{W}(\overline{\mathcal{X}},\mathbb{Z}(n))$ is finitely generated and the image of the Beilinson regulator $\rho$ is a lattice, the natural topology on $H^{i}_{\Ar}(\overline{\mathcal{X}},\mathbb{Z}(n))$ is locally compact.
\end{proof}

In the proof we have obtained the following more precise result. For a locally compact abelian group $G$ denote by $G^0$ the connected component of the identity and by $G_c$ a maximal compact subgroup.

\begin{cor}\label{cor-welldefined} The locally compact group $H^{i}_{\Ar}(\overline{\mathcal{X}},\mathbb{Z}(n))$ is finitely generated for $i\geq 2n+1$ and a compact Lie group for $i\leq 2n-1$. Dually, the locally compact group $H^{i}_{\Ar}(\overline{\mathcal{X}},\tr/\mathbb{Z}(n))$ is finitely generated for $i\geq 2n+2$ and a compact Lie group for $i\leq2n$. For
$G=H^{2n}_{\Ar}(\overline{\mathcal{X}},\mathbb{Z}(n))$ and $G=H^{2n+1}_{\Ar}(\overline{\mathcal{X}},\tr/\mathbb{Z}(n))$
the group $G/G^0$ is finitely generated, $G_c$ is a compact Lie group and $G^0/G_c\cap G^0$ is a finite dimensional real vector space.
\end{cor}

\begin{rem}
One may use the previous result to redefine, up to a canonical isomorphism, the groups $H^{i}_{\Ar}(\overline{\mathcal{X}},\mathbb{Z}(r))$ and $H^{i}_{\Ar}(\overline{\mathcal{X}},\tr/\mathbb{Z}(t))$ in terms of Weil-\'etale cohomology groups, for $i\neq 2r$ and $i\neq 2t+1$ respectively.
\end{rem}

\subsection{Weil-Arakelov cohomology with compact support}\label{sect-arcompactsupport} For any $n\in\bz$ recall that the complex
\begin{equation} R\Gamma_{\Ar,c}(\X,\tr(n)):=R\Gamma_c(X,\br(n))\oplus R\Gamma_c(\X,\br(n))[-1]\notag\end{equation}
was already defined in Definition \ref{r-rgc-def}, the complex $R\Gamma_{W,c}(\mathcal{X},\mathbb{Z}(n))$
in Definition \ref{z-rgc-def} and that there is an exact triangle (\ref{triangle-cpctsupp-fgcoh}) of perfect complexes of abelian groups
\begin{equation}\label{triangle-cpctsupp-fgcoh2}
R\Gamma_{W,c}(\mathcal{X},\mathbb{Z}(n))\to R\Gamma_{W}(\overline{\mathcal{X}},\mathbb{Z}(n))
\xrightarrow{i_{\infty}^*}R\Gamma_W(\mathcal{X}_{\infty},\mathbb{Z}(n))\to.
\end{equation}
We define versions of these complexes with $\br/\bz(n)$-coefficients as in (\ref{w-z-to-r}).

\begin{lem} The composite
\[ \left(\tau^{\geq{2n}}R\Gamma_\cd(\X_{/\br},\br(n))\right)[-2]\to R\Gamma_W(\overline{\mathcal{X}},\br/\mathbb{Z}(n))\xrightarrow{i_{\infty}^*\otimes\br/\bz} R\Gamma_W(\mathcal{X}_{\infty},\br/\mathbb{Z}(n)),\]
where the first map is the one in (\ref{r/z-ar-w-triangle}), is the zero map.
\label{easysplit}\end{lem}

\begin{proof} From the proof of Prop. \ref{r/z-ar-w-prop} we obtain a commutative square
\[\begin{CD} \tau^{\geq{2n}}R\Gamma_\cd(\X_{/\br},\br(n))[-2] @>>>\cone(\sigma\rho\pi_1)[-1]\cong R\Gamma(\overline{\mathcal{X}},\mathbb{R}(n))\oplus R\Gamma(\mathcal{X},\mathbb{R}(d-n))^*[-\delta]\\
\Vert@. @VVV\\
\tau^{\geq{2n}}R\Gamma_\cd(\X_{/\br},\br(n))[-2] @>>> R\Gamma_W(\overline{\mathcal{X}},\mathbb{R}/\mathbb{Z}(n))\end{CD}\]
where the upper horizontal map factors through the second summand. We have another commutative diagram
\[\begin{CD} R\Gamma_W(\overline{\mathcal{X}},\mathbb{Z}(n))_\br\cong R\Gamma(\mathcal{X},\mathbb{R}(n))\oplus R\Gamma(\mathcal{X},\mathbb{R}(d-n))^*[-\delta]@>i_{\infty}^*\otimes\br>>R\Gamma_W(\mathcal{X}_{\infty},\mathbb{Z}(n))_\br \\
@VVV @VVV\\
R\Gamma_W(\overline{\mathcal{X}},\mathbb{R}/\mathbb{Z}(n))@>>>R\Gamma_W(\mathcal{X}_{\infty},\br/\mathbb{Z}(n)) \end{CD}\]
where the upper horizonal map is zero on the second summand by Prop. \ref{lem-com2}. This implies the Lemma.
\end{proof}

\begin{defn} Define
\[ R\Gamma_{\Ar,c}(\mathcal{X},\tr/\mathbb{Z}(n)):=R\Gamma_{W,c}(\mathcal{X},\br/\mathbb{Z}(n))\]
and
\[ R\Gamma_\Ar(\mathcal{X}_{\infty},\tr/\mathbb{Z}(n)):=R\Gamma_W(\mathcal{X}_{\infty},\br/\mathbb{Z}(n))\oplus \left(\tau^{\geq{2n}}R\Gamma_\cd(\X_{/\br},\br(n))\right)[-1].\]
\end{defn}

In view of Lemma \ref{easysplit} the exact triangle (\ref{r/z-ar-w-triangle}) then extends to a commutative diagram with exact rows and columns
\begin{equation}\minCDarrowwidth1em\begin{CD}{}@. R\Gamma_{W,c}(\mathcal{X},\br/\mathbb{Z}(n)) @>\sim >>R\Gamma_{\Ar,c}(\mathcal{X},\tr/\mathbb{Z}(n))@>>>\\
@. @VVV @VVV \\
\left(\tau^{\geq{2n}}R\Gamma_\cd(\X_{/\br},\br(n))\right)[-2]@>>>  R\Gamma_W(\overline{\mathcal{X}},\mathbb{R}/\mathbb{Z}(n))@>>> R\Gamma_{\Ar}(\overline{\mathcal{X}},\tr/\mathbb{Z}(n))@>>>\\
\Vert@. @VV{i_{\infty}^*\otimes\br/\bz}V @VVV \\
\left(\tau^{\geq{2n}}R\Gamma_\cd(\X_{/\br},\br(n))\right)[-2]@>0>> R\Gamma_W(\mathcal{X}_{\infty},\br/\mathbb{Z}(n))@>>>R\Gamma_\Ar(\mathcal{X}_{\infty},\tr/\mathbb{Z}(n))@>>> \end{CD}\label{r/z-ar-w-dia}
\end{equation}
and it is also clear that the cohomology groups $H^i_{\Ar,c}(\mathcal{X},\tr/\mathbb{Z}(n))$ are compact for all $i,n\in\bz$.
The exact triangle (\ref{exptriangle}) in the introduction is just the defining triangle of
$R\Gamma_{W,c}(\mathcal{X},\br/\mathbb{Z}(n))$ and the exact triangle (\ref{tangenttri})
in the introduction is given by the following proposition. Recall the definition of algebraic deRham cohomology
$$R\Gamma_{dR}(\mathcal{X}_F/F):=R\Gamma(\mathcal{X}_{F,Zar},\Omega^*_{\mathcal{X}_F/F})$$
for any field $F$ of characteristic zero. For $F=\bc$ one has an isomorphism
$$R\Gamma_{dR}(\mathcal{X}_\bc/\bc)\cong R\Gamma(\X(\bc),\Omega_{\X(\bc)/\bc}^\bullet)$$
and for $F=\br$ an isomorphism
$$R\Gamma_{dR}(\mathcal{X}_\br/\br)\cong R\Gamma_{dR}(\mathcal{X}_\bc/\bc)^{G_\br}\cong R\Gamma(G_\br,\X(\bc), \Omega_{\X(\bc)/\bc}^\bullet).$$

\begin{prop} There is an exact triangle of perfect complexes of $\br$-vector spaces
\begin{equation}R\Gamma_{dR}(\mathcal{X}_\br/\br)/\mathrm{Fil}^n [-2]\to R\Gamma_{\Ar,c}(\X,\tr(n))\to R\Gamma_{W,c}(\mathcal{X},\mathbb{Z}(n))_\br\to \label{tangenttri2}\end{equation}
and hence a map
\begin{equation} R\Gamma_{\Ar,c}(\X,\tr(n))\to R\Gamma_{\Ar,c}(\mathcal{X},\tr/\mathbb{Z}(n)).\label{r-to-rmodz}\end{equation}
\label{tangentprop}\end{prop}

\begin{proof} Recall that by definition
$$R\Gamma_W(\mathcal{X}_{\infty},\mathbb{Z}(n))=R\Gamma(\mathcal{X}_{\infty},i^*_{\infty}\mathbb{Z}(n))$$
where $i^*_{\infty}\mathbb{Z}(n)$ is the complex of sheaves
$$i^*_{\infty}\mathbb{Z}(n):=\mathrm{Cone}(R\pi_*(2\pi i)^n\mathbb{Z}\longrightarrow \tau^{>n}R\widehat{\pi}_*(2\pi i)^n\mathbb{Z})[-1]$$
on $\mathcal{X}_{\infty}=\X(\bc)/G_\br$. So we have
$$R\Gamma_W(\mathcal{X}_{\infty},\mathbb{Z}(n))_\br=R\Gamma(\mathcal{X}_{\infty},R\pi_*(2\pi i)^n\mathbb{R})=
R\Gamma(G_\br,\X(\bc),(2\pi i)^n\mathbb{R} )$$
and the exact triangle
\[ \Omega_{\X(\bc)/\bc}^\bullet/F^n\to \br(n)_\cd\to (2\pi i)^n\mathbb{R}\to \]
in $Sh(G_{\mathbb{R}},\mathcal{X}(\mathbb{C}))$ induces an exact triangle
\[ R\Gamma_{dR}(\mathcal{X}_\br/\br)/\mathrm{Fil}^n[-1]\to R\Gamma_\cd(\X_{/\br},\br(n))\to R\Gamma_W(\mathcal{X}_{\infty},\mathbb{Z}(n))_\br\to.\]
One then has a commutative diagram with exact rows and columns
\begin{equation}\minCDarrowwidth1em\begin{CD}R\Gamma_{dR}(\mathcal{X}_\br/\br)/\mathrm{Fil}^n[-2]@>>>R\Gamma_c(\X,\br(n))\oplus R\Gamma_c(\X,\br(n))[-1]@>>> R\Gamma_{W,c}(\mathcal{X},\mathbb{Z}(n))_\br @>>>\\
@. @VV\beta_2V @VVV\\
{} @. R\Gamma(\X,\br(n))\oplus R\Gamma(\X,\br(d-n))^*[-\delta] @>\beta_1>\sim > R\Gamma_{W}(\overline{\mathcal{X}},\mathbb{Z}(n))_\br @>>>\\
@. @VVV @VVV\\
R\Gamma_{dR}(\mathcal{X}_\br/\br)/\mathrm{Fil}^n[-1]@>>> R\Gamma_\cd(\X_{/\br},\br(n))@>\beta_3>> R\Gamma_W(\mathcal{X}_{\infty},\mathbb{Z}(n))_\br@>>>\end{CD}
\notag\end{equation}
where $\beta_1$ is the isomorphism (\ref{rgwdecomp}), the middle column is the sum of the triangle (\ref{rgc}) with the duality isomorphism of conjecture ${\bf B}(\mathcal{X},n)$ and the bottom triangle is induced by the triangle
\[ \Omega_{\X(\bc)/\bc}^\bullet/F^n[-1]\to \br(n)_\cd\to \br(n)\to \]
in $Sh(G_\br,\X(\bc))$. The top row then gives (\ref{tangenttri2}).
\end{proof}

\begin{defn} Define $R\Gamma_{\Ar,c}(\X,\bz(n))$ to be a mapping fibre of the map (\ref{r-to-rmodz}) so that there is an exact triangle
\begin{equation} R\Gamma_{\Ar,c}(\X,\bz(n))\to R\Gamma_{\Ar,c}(\X,\tr(n))\to R\Gamma_{\Ar,c}(\mathcal{X},\tr/\mathbb{Z}(n))\to.\label{z-r-rmodz}\end{equation}  \end{defn}

One then has a diagram with exact rows and columns
\begin{equation}\minCDarrowwidth1em\begin{CD}R\Gamma_{\Ar,c}(\X,\bz(n))@>>>R\Gamma_c(\X,\br(n))\oplus R\Gamma_c(\X,\br(n))[-1]@>>> R\Gamma_{W,c}(\mathcal{X},\br/\mathbb{Z}(n)) @>>>\\
@VVV @VV\beta_2V @VVV\\
 R\Gamma_{W}(\overline{\X},\bz(n))@>>> R\Gamma(\X,\br(n))\oplus R\Gamma(\X,\br(d-n))^*[-\delta] @>>> R\Gamma_{W}(\overline{\mathcal{X}},\br/\mathbb{Z}(n))@>>>\\
@VVV @VVV @VV{i_{\infty}^*\otimes\br/\bz}V\\
\tilde{R\Gamma}_\cd(\X_{/\br},\bz(n))@>\beta_5>> R\Gamma_\cd(\X_{/\br},\br(n))@>\beta_4>> R\Gamma_W(\mathcal{X}_{\infty},\br/\mathbb{Z}(n))@>>>\end{CD}
\label{rgcdia}\end{equation}
where $\tilde{R\Gamma}_\cd(\X_{/\br},\bz(n))$ is the hypercohomology of the complex of sheaves
\[\mathrm{Cone}(R\pi_*\mathbb{Z}(n)_\cd\longrightarrow \tau^{>n}R\widehat{\pi}_*(2\pi i)^n\mathbb{Z})[-1]\]
on $\X_\infty$ and $\beta_4$ is the composite of $\beta_3$ with the natural map
\[ R\Gamma_W(\mathcal{X}_{\infty},\mathbb{Z}(n))_\br\to R\Gamma_W(\mathcal{X}_{\infty},\br/\mathbb{Z}(n)).\]

Alternatively, one can follow the construction of $R\Gamma_{W,c}(\mathcal{X},\mathbb{Z}(n))$ in section \ref{sect-compactsupport}, starting with the \'etale Beilinson regulator on the level of complexes (\ref{etaleregulator})
\begin{equation} R\Gamma(\X_{et},\bz(n)) \to R\Gamma_\cd(\X_{/\br},\bz(n))\notag\end{equation}
and using the left column in (\ref{rgcdia}) as the defining triangle of $R\Gamma_{\Ar,c}(\X,\bz(n))$.
As we already remarked in the introduction to this section, neither construction gives $R\Gamma_{\Ar,c}(\X,\bz(n))$ or $R\Gamma_{\Ar,c}(\mathcal{X},\tr/\mathbb{Z}(n))$ up to a unique isomorphism in the derived category.

\begin{rem}\label{tangentremark} All complexes in (\ref{z-r-rmodz}) can be represented by bounded complexes of locally compact abelian groups  (see the considerations at the beginning of subsection \ref{sec:w-ar-duality}) and (\ref{z-r-rmodz}) is in fact an exact triangle in the bounded derived category of locally compact abelian groups defined in \cite{spitzweck07}. For a locally compact abelian group $G$ we define the {\em tangent space} to be the $\br$-vector space
\[ T_\infty G:=\Hom_{cts}(G^D,\br)=\Hom_{cts}(\Hom_{cts}(G,\br/\bz),\br)\]
where $\Hom_{cts}(-,-)$ is endowed with the compact open topology. While the topology on $G^D$ is always locally compact, this is not in general true for $\Hom_{cts}(G,\br)$. However, by \cite{spitzweck07}[Prop. 3.12] it is true if $G$ has finite ranks in the sense of \cite{spitzweck07}[Def. 2.6]  and all the complexes in  (\ref{z-r-rmodz}) are easily seen to consist of groups of  finite ranks.  By \cite{spitzweck07}[Prop. 4.14 vii)] the functor $\Hom_{cts}(-,\br)$ is exact and of course so is $(-)^D$. We conclude that $T_\infty$ is an exact covariant functor (with values in finite dimensional real vector spaces if the argument has finite ranks) and extends to the bounded derived category of locally compact abelian groups. The image of the exact triangle (\ref{z-r-rmodz}) under the tangent space functor $T_\infty$ is the exact triangle (\ref{tangenttri2}).
\end{rem}

\begin{defn} Define $R\Gamma_{\Ar}(\X_\infty,\bz(n))$ to be the mapping fibre of the composite map
\[ \tilde{R\Gamma}_\cd(\X_{/\br},\bz(n))\xrightarrow{\beta_5} R\Gamma_\cd(\X_{/\br},\br(n))\xrightarrow{\sigma}\tau^{\leq 2n-1}R\Gamma_\cd(\X_{/\br},\br(n)).\]
\end{defn}

\begin{prop} There is an exact triangle
\[ R\Gamma_{\Ar}(\X_\infty,\bz(n))\to R\Gamma_{\Ar}(\X_\infty,\tr(n))\to R\Gamma_{\Ar}(\X_\infty,\tr/\bz(n))\to.\]
\end{prop}

\begin{proof} The mapping fibre of $\beta_5$ identifies with $R\Gamma_W(\mathcal{X}_{\infty},\br/\mathbb{Z}(n))[-1]$ and that  of $\sigma$ with $\tau^{\geq{2n}}R\Gamma_\cd(\X_{/\br},\br(n))$. The definition of $R\Gamma_{\Ar}(\X_\infty,\bz(n))$ as the mapping fibre of the composite $\sigma\circ\beta_5$ gives the central horizontal exact triangle in the diagram
\[\minCDarrowwidth1em\begin{CD} R\Gamma_\Ar(\mathcal{X}_{\infty},\tr/\mathbb{Z}(n))[-1]@>>>\left(\tau^{\geq{2n}}R\Gamma_\cd(\X_{/\br},\br(n))\right)[-2]@>0>> R\Gamma_W(\mathcal{X}_{\infty},\br/\mathbb{Z}(n))@>>>\\
@VVV @VV 0 V \Vert@.\\
R\Gamma_{\Ar}(\X_\infty,\bz(n))@>>>\tau^{\geq 2n}R\Gamma_\cd(\X_{/\br},\br(n))@>>> R\Gamma_W(\mathcal{X}_{\infty},\br/\mathbb{Z}(n))@>>>\\
@VVV @VVV @VVV\\
R\Gamma_{\Ar}(\X_\infty,\tr(n))@={\begin{array}{ll}{}&\tau^{\geq 2n}R\Gamma_\cd(\X_{/\br},\br(n))\\ \oplus & \left(\tau^{\geq 2n}R\Gamma_\cd(\X_{/\br},\br(n))\right)[-1]\end{array}} @>>> 0\\
@VVV @VVV @VVV
\end{CD}\]
while (\ref{r/z-ar-w-dia}) gives the upper horizontal exact triangle. The diagram commutes and the middle and right column are exact, hence so is the left.
\end{proof}

The relation between Weil-\'etale and Weil-Arakelov cohomology with $\bz(n)$-coefficients can then be summarized in the following diagram. The corresponding diagram (\ref{r/z-ar-w-dia}) for $\tr/\bz(n)$-coefficients is simpler which is why we discussed $\tr/\bz(n)$-coefficients first. This can be traced to the fact that Deligne cohomology with $\br/\bz(n)$-coefficients coincides with singular cohomology with $\br/\bz(n)$-coefficients for any $n\in\bz$.

\begin{defn} Define $T(\X_\infty,n)$ to be the mapping cone of
\[ R\Gamma_{dR}(\mathcal{X}_\br/\br)/\mathrm{Fil}^n[-1]\to R\Gamma_\cd(\X_{/\br},\br(n))\xrightarrow{\sigma}\tau^{\leq 2n-1}R\Gamma_\cd(\X_{/\br},\br(n)).\]
\end{defn}

We leave it again as an exercise to show exactness of the rows and columns in the following diagram. Note that the middle, resp. right hand, column consists of perfect complexes of abelian groups, resp. $\br$-vector spaces.
\begin{equation}\minCDarrowwidth1em\begin{CD}R\Gamma_{\Ar,c}(\X,\bz(n))@>>>R\Gamma_{W,c}(\X,\bz(n))@>>>
R\Gamma_{dR}(\mathcal{X}_\br/\br)/\mathrm{Fil}^n[-1]@>>> \\
@VVV @VVV @VVV \\
R\Gamma_{\Ar}(\overline{\mathcal{X}},\mathbb{Z}(n))@>>> R\Gamma_W(\overline{\mathcal{X}},\mathbb{Z}(n))@>>>
\tau^{\leq{2n-1}}R\Gamma_\cd(\X_{/\br},\br(n))@>>>\\
@VVV @VVV @VVV \\
R\Gamma_{\Ar}(\mathcal{X}_\infty,\mathbb{Z}(n))@>>> R\Gamma_W(\mathcal{X}_\infty,\mathbb{Z}(n))@>>>
T(\X_\infty,n)@>>>\end{CD}\label{z-ar-w-dia}
\end{equation}

\section{Special values of zeta functions}\label{sec:zeta}

Throughout this section, $\mathcal{X}$ denotes a proper regular connected arithmetic scheme of dimension $d$. Additional assumptions on $\mathcal{X}$ will be given at the beginning of each subsection.

In the introduction we have already given a conjectural description of the vanishing order and leading Taylor coefficient of $\zeta(\X,s)$ at any integer argument $s=n\in\bz$ in terms of Weil-Arakelov groups, and we have seen how to reformulate these conjectures in terms of a fundamental line. In the following we shall exclusively work with the fundamental line and leave the Weil-Arakelov description as a suggestive reformulation that invites further exploration. The main thing that remains to be done is a precise definition of the correction factor $C(\X,n)\in\bq$ and a proof of the equivalence of our formulation with the Tamagawa number conjecture of Fontaine and Perrin-Riou \cite{fpr91}.

\subsection{De Rham cohomology}\label{sec:derived-de-rham}

Let $n\in\mathbb{Z}$ be an integer. We consider the derived de Rham complex modulo the Hodge filtration $L\Omega^*_{\mathcal{X}/\mathbb{Z}}/\mathrm{Fil}^n$ (see \cite{Illusie71} VIII.2.1) as a complex of abelian sheaves on the Zariski site of $\mathcal{X}$. Note that  $L\Omega^*_{\mathcal{X}/\mathbb{Z}}/\mathrm{Fil}^n=0$ for $n\leq 0$. We denote
$$R\Gamma_{dR}(\mathcal{X}/\mathbb{Z})/F^n:=R\Gamma(\mathcal{X}_{Zar},L\Omega^*_{\mathcal{X}/\mathbb{Z}}/\mathrm{Fil}^n).$$
We remark that $H^i_{dR}(X/\mathbb{Z})/F^n:=H^i(R\Gamma_{dR}(\mathcal{X}/\mathbb{Z})/F^n)$  is finitely generated for all $i$ and vanishes for  $i<0$ and $i\geq d+n$. Indeed, since $\mathcal{X}$ is regular, the map $\mathcal{X}\rightarrow \mathrm{Spec}(\mathbb{Z})$ is a local complete intersection, hence the cotangent complex $L_{\mathcal{X}/\mathbb{Z}}$ has perfect amplitude $\subset[-1,0]$ (see \cite{Illusie71} III.3.2.6).
It follows that $L\Lambda^p L_{\mathcal{X}/\mathbb{Z}}$  has perfect amplitude $\subseteq [-p,0]$ (see \cite{Illusie71} III.3.2.6). By (\cite{SGA6II} 2.2.7.1) and (\cite{SGA6II} 2.2.8),  $L\Lambda^p L_{\mathcal{X}/\mathbb{Z}}$ is globally isomorphic in $\mathcal{D}(\mathcal{O}_{\mathcal{X}})$ to a complex of locally free finitely generated $\mathcal{O}_{\mathcal{X}}$-modules put in degrees $[-p,0]$, where $\mathcal{D}(\mathcal{O}_{\mathcal{X}})$ denotes the derived category of $\mathcal{O}_{\mathcal{X}}$-modules. Since $\mathcal{X}$ is proper over $\mathbb{Z}$, $H^q(\mathcal{X}_{Zar},L\Lambda^{p} L_{\mathcal{X}/\mathbb{Z}})$ is a finitely generated abelian group for all $q$ and $0$ for almost all $q$. Then the spectral sequence
$$H^q(\mathcal{X}_{Zar},L\Lambda^{p<n} L_{X/\mathbb{Z}})\Longrightarrow H^{p+q}_{dR}(\mathcal{X}/\mathbb{Z})/F^n$$
shows that $H^i_{dR}(X/\mathbb{Z})/F^n:=H^i(R\Gamma_{dR}(\mathcal{X}/\mathbb{Z})/F^n)$  is finitely generated for all $i$ and vanishes for  $i<0$ and $i\geq d+n$. Here $L\Lambda^{p<n} L_{X/\mathbb{Z}}:= L\Lambda^{p} L_{X/\mathbb{Z}}$ for $p<n$ and $L\Lambda^{p<n} L_{X/\mathbb{Z}}:=0$ for $p\geq n$.

For any flat $\mathbb{Z}$-algebra $A$, we have a canonical isomorphism
$$R\Gamma_{dR}(\mathcal{X}/\mathbb{Z})/F^n\otimes_{\mathbb{Z}}A\simeq R\Gamma_{dR}(\mathcal{X}_A/A)/F^n:=R\Gamma(\mathcal{X}_{A,Zar},L\Omega^*_{\mathcal{X}_A/A}/\mathrm{Fil}^n)$$
where $\mathcal{X}_A:=\mathcal{X}\otimes_{\mathbb{Z}}A$. Moreover, if $\mathcal{X}_A/A$ is smooth then we have a quasi-isomorphism
$$L\Omega^*_{\mathcal{X}_A/A}/\mathrm{Fil}^n\stackrel{\sim}{\rightarrow}\Omega^{*<n}_{\mathcal{X}_A/A}.$$

\subsection{The fundamental line}
We suppose that $\mathcal{X}$ satisfies Conjectures $\textbf{L}(\overline{\mathcal{X}}_{et},n)$, $\textbf{L}(\overline{\mathcal{X}}_{et},d-n)$ and $\textbf{AV}(\overline{\mathcal{X}}_{et},n)$.
\begin{defn}\label{deltadef} The fundamental line is
$$\Delta(\mathcal{X}/\mathbb{Z},n):=\mathrm{det}_{\mathbb{Z}}R\Gamma_{W,c}(\mathcal{X},\mathbb{Z}(n))
\otimes_{\mathbb{Z}}\mathrm{det}_{\mathbb{Z}}R\Gamma_{dR}(\mathcal{X}/\mathbb{Z})/F^n.$$
\end{defn}

\begin{prop}\label{prop-lambda-infty}
If $\mathcal{X}$ satisfies Conjecture ${\bf B}(\mathcal{X},n)$, then there is a canonical trivialization
$$\lambda_{\infty}(\mathcal{X},n):\mathbb{R} \stackrel{\sim}{\longrightarrow}
\mathrm{det}_{\mathbb{R}}  R\Gamma_{\mathrm{\Ar,c}}(\mathcal{X},\tr(n))
\stackrel{\sim}{\longrightarrow} \Delta(\mathcal{X}/\mathbb{Z},n)\otimes_{\mathbb{Z}}\mathbb{R}.$$
\end{prop}
\begin{proof}
The first isomorphism is induced by the long exact sequence (\ref{cupseq}) and the second by the exact triangle (\ref{tangenttri2}).
\end{proof}

\subsection{The complex $R\Gamma'_{eh}(\mathcal{X}_{\mathbb{F}_p},\mathbb{Z}_p(n))$ and Milne's correcting factor}
Let $p$ be a prime number. Recall that, if $Y$ is a smooth scheme over $\mathbb{F}_p$, one has $$\mathbb{Z}/p^{r}(n)\stackrel{\sim}{\rightarrow} \nu^n_r[-n]:=W_r\Omega^n_{Y,\mathrm{log}}[-n],$$ where $W_r\Omega^n_{Y,\mathrm{log}}$ is the \'etale subsheaf of the de Rham-Witt sheaf $W_r\Omega^n_{Y}$ locally generated by the sections of the form $d\mathrm{log}(\underline{f_1})\wedge\cdots \wedge d\mathrm{log}(\underline{f_n})$. Here $\underline{f_i}$ denotes the Teichmuller representative of the unit $f_i\in\mathcal{O}^{\times}_Y$. It follows that  $$R\Gamma_{et}(Y,\mathbb{Z}_p(n)):=\mathrm{holim}\  R\Gamma(Y_{et},\mathbb{Z}(n)/p^{\bullet})$$ is a perfect complex of $\mathbb{Z}_p$-modules if $Y$ is a smooth projective variety.

To treat arbitrary separated schemes of finite type over $\mathbb{F}_p$, we consider the $eh$-topos over $\mathbb{F}_p$  \cite{Geisser06} and we denote
$$R\Gamma_{eh}(Y,\mathbb{Z}_p(n)):=\mathrm{holim}\  R\Gamma(Y_{eh},\mathbb{Z}(n)/p^{\bullet}),$$ see \cite{Geisser06} Section 4. We also denote by $\mathbf{R}(\mathbb{F}_p,c)$ the strong form of resolution of singularities given in (\cite{Geisser06} Definition 2.4) for varieties over $\mathbb{F}_p$ of dimension $\leq c$. If $\mathbf{R}(\mathbb{F}_p,\mathrm{dim}(Y))$ holds, then $R\Gamma_{eh}(Y,\mathbb{Z}_p(n))$ is perfect  for $Y$ proper over $\mathbb{F}_p$ (see \cite{Geisser06} Corollary 4.4) and the canonical map $R\Gamma_{et}(Y,\mathbb{Z}_p(n))\rightarrow R\Gamma_{eh}(Y,\mathbb{Z}_p(n))$ is a quasi-isomorphism for $Y$ smooth (see \cite{Geisser06} Theorem 4.3).

\begin{notation}\label{ehprime}
Let $\mathcal{X}$ be a proper regular arithmetic scheme. We set
$$R\Gamma'_{eh}(\mathcal{X}_{\mathbb{F}_p},\mathbb{Z}_p(n)):= R\Gamma_{et}(\mathcal{X}^{\mathrm{red}}_{\mathbb{F}_p},\mathbb{Z}_p(n))$$ if $\mathcal{X}^{\mathrm{red}}_{\mathbb{F}_p}$ is smooth, and
$$R\Gamma'_{eh}(\mathcal{X}_{\mathbb{F}_p},\mathbb{Z}_p(n)):= R\Gamma_{eh}(\mathcal{X}_{\mathbb{F}_p},\mathbb{Z}_p(n))$$
otherwise. Here $\mathcal{X}^{\mathrm{red}}_{\mathbb{F}_p}$ denotes the maximal reduced closed subscheme of $\mathcal{X}_{\mathbb{F}_p}$.
\end{notation}
Notice that, under $\mathbf{R}(\mathbb{F}_p,\mathrm{dim}(\mathcal{X}_{\mathbb{F}_p}))$, one has $R\Gamma'_{eh}(\mathcal{X}_{\mathbb{F}_p},\mathbb{Z}_p(n))\simeq R\Gamma_{eh}(\mathcal{X}_{\mathbb{F}_p},\mathbb{Z}_p(n))$. Indeed, the map
 $\mathcal{X}^{\mathrm{red}}_{\mathbb{F}_p}\rightarrow \mathcal{X}_{\mathbb{F}_p}$ induces an isomorphism in the $eh$-topos (since this map is both a monomorphism and an $eh$-covering), so that $\mathbf{R}(\mathbb{F}_p,\mathrm{dim}(\mathcal{X}_{\mathbb{F}_p}))$ yields
 $$R\Gamma_{et}(\mathcal{X}^{\mathrm{red}}_{\mathbb{F}_p},\mathbb{Z}_p(n))\stackrel{\sim}{\rightarrow }R\Gamma_{eh}(\mathcal{X}^{\mathrm{red}}_{\mathbb{F}_p},\mathbb{Z}_p(n))\stackrel{\sim}{\rightarrow }R\Gamma_{eh}(\mathcal{X}_{\mathbb{F}_p},\mathbb{Z}_p(n))$$
 whenever $\mathcal{X}^{\mathrm{red}}_{\mathbb{F}_p}$ is smooth. We introduce $R\Gamma'_{eh}(\mathcal{X}_{\mathbb{F}_p},\mathbb{Z}_p(n))$ in order to avoid the systematic use of $\mathbf{R}(\mathbb{F}_p,\mathrm{dim}(\mathcal{X}_{\mathbb{F}_p}))$. We proceed similarly for Milne's correcting factor, and we refer to (\cite{Geisser06} Section 4.1) for the definition of $H^j_{eh}(\mathcal{X}_{\mathbb{F}_p},\Omega^i)$.
\begin{defn}
Let $\mathcal{X}$ be a proper regular arithmetic scheme. We set
$$\chi(\mathcal{X}_{\mathbb{F}_p},\mathcal{O},n):= \sum_{i\leq n, j}(-1)^{i+j} \cdot (n-i)\cdot \mathrm{dim}_{\mathbb{F}_p}H^j_{Zar}(\mathcal{X}^{\mathrm{red}}_{\mathbb{F}_p},\Omega^i)$$
if $\mathcal{X}^{\mathrm{red}}_{\mathbb{F}_p}$ is smooth, and
$$\chi(\mathcal{X}_{\mathbb{F}_p},\mathcal{O},n):= \sum_{i\leq n, j}(-1)^{i+j} \cdot (n-i)\cdot \mathrm{dim}_{\mathbb{F}_p}H^j_{eh}(\mathcal{X}_{\mathbb{F}_p},\Omega^i)$$
if $\mathcal{X}^{\mathrm{red}}_{\mathbb{F}_p}$ is singular and $\mathbf{R}(\mathbb{F}_p,\mathrm{dim}(\mathcal{X}_{\mathbb{F}_p}))$ holds.
\end{defn}

\subsection{The local factor $c_p(\mathcal{X},n)$}\label{sec:localfactor}

The following conjecture is a $p$-adic analogue of the fundamental exact triangle
$$R\Gamma_{dR}(\mathcal{X}_{\mathbb{R}}/\mathbb{R})/F^n[-1]\rightarrow R\Gamma_{\mathcal{D}}(\mathcal{X}_{/\mathbb{R}},\mathbb{R}(n))\rightarrow R\Gamma(G_{\mathbb{R}},\mathcal{X}(\mathbb{C}),(2\pi i)^n\mathbb{R})$$
for Deligne cohomology.

\begin{conj}\label{conjD_p} ${\bf D}_p(\mathcal{X},n)$
There is an exact triangle of complexes of $\mathbb{Q}_p$-vector spaces
$$R\Gamma_{dR}(\mathcal{X}_{\mathbb{Q}_p}/\mathbb{Q}_p)/F^n[-1]\rightarrow R\Gamma_{et}(\mathcal{X}_{\mathbb{Z}_p},\mathbb{Q}_p(n))\rightarrow R\Gamma'_{eh}(\mathcal{X}_{\mathbb{F}_p},\mathbb{Q}_p(n)).$$
\end{conj}
The triangle of Conjecture \ref{conjD_p} must be compatible with the fundamental triangle of \cite{Bloch-Esnault-Kerz-14}[Thm. 5.4] in the following sense.  If $\mathcal{X}_{\mathbb{Z}_p}/\mathbb{Z}_p$ is smooth and $n<p-1$, then we have an isomorphism of triangles
\[ \xymatrix{
R\Gamma_{dR}(\mathcal{X}_{\mathbb{Q}_p}/\mathbb{Q}_p)/F^n[-1] \ar[d]\ar[r]& R\Gamma_{et}(\mathcal{X}_{\mathbb{Z}_p},\mathbb{Q}_p(n))\ar[d]\ar[r]& R\Gamma_{et}(\mathcal{X}_{\mathbb{F}_p},\mathbb{Q}_p(n))\ar[d] \\
R\Gamma(\mathcal{X}_{\mathbb{F}_p},p(n)\cdot\Omega^{<n}_{X_{\bullet}})_{\mathbb{Q}_p}[-1]\ar[r]& R\Gamma(\mathcal{X}_{\mathbb{F}_p},\mathfrak{S}_{X_{\bullet}}(n))_{\mathbb{Q}_p} \ar[r]&
R\Gamma(\mathcal{X}_{\mathbb{F}_p},W_{\bullet}\Omega^n_{\mathcal{X}_{\mathbb{F}_p},\mathrm{log}}[-n])_{\mathbb{Q}_p}
}
\]
where the left vertical isomorphism follows from $L\Omega^*_{\mathcal{X}_{\mathbb{Q}_p}/\mathbb{Q}_p}/F^n\simeq \Omega^{<n}_{\mathcal{X}_{\mathbb{Q}_p}/\mathbb{Q}_p}$ since $\mathcal{X}_{\mathbb{Q}_p}/\mathbb{Q}_p$ is smooth,
the middle vertical isomorphism is given by (\cite{Geisser04a} Theorem 1.3) and the right vertical isomorphism is given by (\cite{Geisser06} Theorem 4.3) together with the quasi-isomorphism $\mathbb{Z}/p^r(n)\simeq \nu_r^n[-n]$ over $\mathcal{X}_{\mathbb{F}_p,et}$. Finally, $p(n)\cdot\Omega^{<n}_{X_{\bullet}}$ is defined as in \cite{Bloch-Esnault-Kerz-14}.

Conjecture \ref{conjD_p} gives an isomorphism $\lambda_{p}(\mathcal{X},n):$
\begin{eqnarray*}
\left(\mathrm{det}_{\mathbb{Z}_p}  R\Gamma_{et}(\mathcal{X}_{\mathbb{Z}_p},\mathbb{Z}_p(n))\right)_{\mathbb{Q}_p}&\stackrel{\sim}{\longrightarrow}&\mathrm{det}_{\mathbb{Q}_p}  R\Gamma_{et}(\mathcal{X}_{\mathbb{Z}_p},\mathbb{Q}_p(n))\\\
&\stackrel{\sim}{\longrightarrow}&\mathrm{det}_{\mathbb{Q}_p}  R\Gamma'_{eh}(\mathcal{X}_{\mathbb{F}_p},\mathbb{Q}_p(n))
\otimes_{\mathbb{Q}_p}\mathrm{det}^{-1}_{\mathbb{Q}_p}R\Gamma_{dR}(\mathcal{X}_{\mathbb{Q}_p}/\mathbb{Q}_p)/F^n\\
&\stackrel{\sim}{\longrightarrow}& \left(\mathrm{det}_{\mathbb{Z}_p}  R\Gamma'_{eh}(\mathcal{X}_{\mathbb{F}_p},\mathbb{Z}_p(n))
\otimes_{\mathbb{Z}_p}\mathrm{det}^{-1}_{\mathbb{Z}_p}R\Gamma_{dR}(\mathcal{X}_{\mathbb{Z}_p}/\mathbb{Z}_p)/F^n\right)_{\mathbb{Q}_p}.
\end{eqnarray*}

\begin{defn}
We define $$d_{p}(\mathcal{X},n):=\mathrm{det}(\lambda_{p}(\mathcal{X},n))\in \mathbb{Q}_p^\times/\mathbb{Z}_p^\times \hspace{0.5cm} \mbox{and} \hspace{0.5cm} c_{p}(\mathcal{X},n):=p^{\chi(\mathcal{X}_{\mathbb{F}_p},\mathcal{O},n)}\cdot d_{p}(\mathcal{X},n).$$
\label{ddef}\end{defn}
Here the determinant of $\lambda_{p}(\mathcal{X},n)$ is computed with the given integral structures, i.e. one has
\begin{align}&\lambda_{p}\left(d_{p}(\mathcal{X},n)^{-1}\cdot \mathrm{det}_{\mathbb{Z}_p}  R\Gamma_{et}(\mathcal{X}_{\mathbb{Z}_p},\mathbb{Z}_p(n))\right)\notag\\
=\ &\mathrm{det}_{\mathbb{Z}_p}  R\Gamma_{eh}(\mathcal{X}_{\mathbb{F}_p},\mathbb{Z}_p(n))
\otimes_{\mathbb{Z}_p}\mathrm{det}^{-1}_{\mathbb{Z}_p}R\Gamma_{dR}(\mathcal{X}_{\mathbb{Z}_p}/\mathbb{Z}_p)/F^n.\label{lambda-int}\end{align}

%\begin{rem} It is shown in App. B that conjecture ${\bf D}_p(\mathcal{X},n)$ is a consequence of Cor. \ref{pcor}. In order to define the isomorphism $\lambda_{p}(\mathcal{X},n)$, and hence the factor $d_{p}(\mathcal{X},n)$, one does not need the full strength of conjecture ${\bf D}_p(\mathcal{X},n)$ or of Cor. \ref{pcor}. It suffices to know the ingredients of Cor. \ref{pcor}, i.e. the localization triangle of Cor. \ref{absloc}, Conj. \ref{ploc} and Conj. \ref{pdescent} but not necessarily the commutativity of the diagram in Cor. \ref{pcor}. All these ingredients hold true if $\X_{\bz_p}$ is smooth over a local integer ring (without assuming $n<p-1$ as in Prop. \ref{psmooth}) and hence one can define the factor $d_{p}(\mathcal{X},n)$ unconditionally in this case.
%\end{rem}

\begin{prop}\label{prop-cp-n=0}
For $n\leq 0$, ${\bf D}_p(\mathcal{X},n)$ holds and $c_{p}(\mathcal{X},n)\equiv 1 \,\mathrm{mod}\, \mathbb{Z}_p^\times$ for all $p$.
\end{prop}
\begin{proof}
By definition, we have $R\Gamma_{dR}(\mathcal{X}_{\mathbb{Z}_p}/\mathbb{Q}_p)/F^n=0$ and
the map
$$R\Gamma_{et}(\mathcal{X}_{\mathbb{Z}_p},\mathbb{Z}_p(n))\rightarrow R\Gamma_{et}(\mathcal{X}_{\mathbb{F}_p},\mathbb{Z}_p(n))\rightarrow R\Gamma_{eh}(\mathcal{X}_{\mathbb{F}_p},\mathbb{Z}_p(n))$$
is an isomorphism by (\cite{Geisser06} Theorem 3.6).
\end{proof}

\begin{prop}\label{prop-cp-charp}
Assume that $\mathcal{X}$ has characteristic $p$. Then ${\bf D}_l(\mathcal{X},n)$ holds and $c_{l}(\mathcal{X},n)\equiv 1 \,\mathrm{mod}\, \mathbb{Z}_l^\times$ for all primes $l$.
\end{prop}
\begin{proof}
We have
$R\Gamma_{et}(\mathcal{X}_{\mathbb{Z}_p},\mathbb{Z}_p(n))= R\Gamma_{et}(\mathcal{X},\mathbb{Z}_p(n))$
and
$$R\Gamma_{dR}(\mathcal{X}_{\mathbb{Z}_p}/\mathbb{Z}_p)/F^n\simeq R\Gamma_{dR}(\mathcal{X}/\mathbb{Z})/F^n\otimes_{\mathbb{Z}}\mathbb{Z}_p.$$
The cohomology groups of  the complex $R\Gamma_{dR}(\mathcal{X}/\mathbb{Z})/F^n$ are finite, and the alternate product of their orders is
$p^{\chi(\mathcal{X}_{\mathbb{F}_p},\mathcal{O},n)}$ by \cite{Morin15}, so that $d_p(\mathcal{X},n)\equiv p^{-\chi(\mathcal{X}_{\mathbb{F}_p},\mathcal{O},n)} \,\mathrm{mod}\,\mathbb{Z}_p^\times$. For $l\neq p$, we have $\mathcal{X}_{\mathbb{Z}_l}=\mathcal{X}_{\mathbb{F}_l}=\emptyset$, hence
$c_{l}(\mathcal{X},n)\equiv d_{l}(\mathcal{X},n)\equiv 1\,\mathrm{mod}\,\mathbb{Z}_l^\times$.
\end{proof}

\begin{prop}\label{prop-cp=1}
Let  $\mathcal{X}$ be a regular proper arithmetic scheme. We have
 $c_{p}(\mathcal{X},n)\equiv 1 \,\mathrm{mod}\, \mathbb{Z}_p^\times$ for almost all $p$.
\end{prop}
\begin{proof}
By Proposition \ref{prop-cp-charp}, we may assume that $\mathcal{X}$ is flat over $\mathbb{Z}$. We may also assume $n<p-1$ and $\mathcal{X}_{\mathbb{Z}_p}/\mathbb{Z}_p$ smooth. By Remark \ref{rem-BEKcompatible} and by (\cite{Bloch-Esnault-Kerz-14} Theorem 5.4) the triangle
$$R\Gamma(\mathcal{X}_{\mathbb{Z}_p}, p(n)\cdot \Omega^{<n}_{\mathcal{X}_{\mathbb{Z}_p}/\mathbb{Z}_p})[-1]\rightarrow R\Gamma_{et}(\mathcal{X}_{\mathbb{Z}_p},\mathbb{Z}_p(n))\rightarrow R\Gamma_{et}(\mathcal{X}_{\mathbb{F}_p},\mathbb{Z}_p(n))$$
is exact,
where $$p(n)\cdot \Omega^{<n}_{\mathcal{X}_{\mathbb{Z}_p}/\mathbb{Z}_p}:=[p^{n}\cdot  \Omega^{0}_{\mathcal{X}_{\mathbb{Z}_p}/\mathbb{Z}_p}\rightarrow  p^{n-1}\cdot  \Omega^{1}_{\mathcal{X}_{\mathbb{Z}_p}/\mathbb{Z}_p} \rightarrow \cdots \rightarrow p \cdot  \Omega^{n-1}_{\mathcal{X}_{\mathbb{Z}_p}/\mathbb{Z}_p}]$$
sits in degrees $[0,n-1]$. The local factor $d_p(\mathcal{X},n)$ measures the difference between two different $\mathbb{Z}_p$-structures on $\mathrm{det}_{\mathbb{Q}_p}  R\Gamma_{et}(\mathcal{X}_{\mathbb{Z}_p},\mathbb{Q}_p(n))$. The first of those $\mathbb{Z}_p$-structures is given by
\begin{eqnarray*}
&&\mathrm{det}_{\mathbb{Q}_p}  R\Gamma_{et}(\mathcal{X}_{\mathbb{Z}_p},\mathbb{Q}_p(n))\\
&\stackrel{\sim}{\longrightarrow}&\mathrm{det}_{\mathbb{Q}_p}  R\Gamma_{et}(\mathcal{X}_{\mathbb{F}_p},\mathbb{Q}_p(n))
\otimes_{\mathbb{Q}_p}\mathrm{det}^{-1}_{\mathbb{Q}_p}R\Gamma_{dR}(\mathcal{X}_{\mathbb{Q}_p}/\mathbb{Q}_p)/F^n\\
&\stackrel{\sim}{\longrightarrow}&\left(\mathrm{det}_{\mathbb{Z}_p}  R\Gamma_{et}(\mathcal{X}_{\mathbb{F}_p},\mathbb{Z}_p(n))
\otimes_{\mathbb{Z}_p}\mathrm{det}^{-1}_{\mathbb{Z}_p}R\Gamma_{dR}(\mathcal{X}_{\mathbb{Z}_p}/\mathbb{Z}_p)/F^n\right)\otimes_{\mathbb{Z}_p} \mathbb{Q}_p\\
&\stackrel{\sim}{\longrightarrow}& \left(\mathrm{det}_{\mathbb{Z}_p}  R\Gamma_{et}(\mathcal{X}_{\mathbb{F}_p},\mathbb{Z}_p(n))
\otimes_{\mathbb{Z}_p}\mathrm{det}^{-1}_{\mathbb{Z}_p}R\Gamma(\mathcal{X}_{\mathbb{Z}_p},\Omega^{<n}_{\mathcal{X}_{\mathbb{Z}_p}/\mathbb{Z}_p})\right)_{\mathbb{Q}_p},
\end{eqnarray*}
where the last isomorphism follows from $L\Omega^*_{\mathcal{X}_{\mathbb{Z}_p}/\mathbb{Z}_p}/F^n\simeq \Omega^{<n}_{\mathcal{X}_{\mathbb{Z}_p}/\mathbb{Z}_p}$ since $\mathcal{X}_{\mathbb{Z}_p}$ is smooth. The second $\mathbb{Z}_p$-structure is
\begin{eqnarray*}
\mathrm{det}_{\mathbb{Q}_p}  R\Gamma_{et}(\mathcal{X}_{\mathbb{Z}_p},\mathbb{Q}_p(n))
&\stackrel{\sim}{\longrightarrow}&\left(\mathrm{det}_{\mathbb{Z}_p}  R\Gamma_{et}(\mathcal{X}_{\mathbb{Z}_p},\mathbb{Z}_p(n))\right)_{\mathbb{Q}_p}\\
&\stackrel{\sim}{\longrightarrow}&\left(\mathrm{det}_{\mathbb{Z}_p}  R\Gamma_{et}(\mathcal{X}_{\mathbb{F}_p},\mathbb{Z}_p(n))
\otimes_{\mathbb{Z}_p}\mathrm{det}^{-1}_{\mathbb{Z}_p}R\Gamma(\mathcal{X}_{\mathbb{Z}_p},p(n)\cdot\Omega^{<n}_{\mathcal{X}_{\mathbb{Z}_p}/\mathbb{Z}_p})
\right)_{\mathbb{Q}_p}.
\end{eqnarray*}
Hence $d_p(\mathcal{X},n)$ is defined as the determinant of the isomorphism
$$\left(\mathrm{det}^{-1}_{\mathbb{Z}_p}R\Gamma(\mathcal{X}_{\mathbb{Z}_p},p(n)\cdot\Omega^{<n}_{\mathcal{X}_{\mathbb{Z}_p}/\mathbb{Z}_p})
\right)_{\mathbb{Q}_p}\stackrel{\sim}{\rightarrow} \left(\mathrm{det}^{-1}_{\mathbb{Z}_p}R\Gamma(\mathcal{X}_{\mathbb{Z}_p},\Omega^{<n}_{\mathcal{X}_{\mathbb{Z}_p}/\mathbb{Z}_p})
\right)_{\mathbb{Q}_p}$$
with respect to the given integral structures. The Hodge to de Rham spectral sequence gives the commutative square of isomorphisms
\[ \xymatrix{
\left(\mathrm{det}^{-1}_{\mathbb{Z}_p}R\Gamma(\mathcal{X}_{\mathbb{Z}_p},p(n)\cdot\Omega^{<n}_{\mathcal{X}_{\mathbb{Z}_p}/\mathbb{Z}_p})
\right)_{\mathbb{Q}_p}\ar[r]\ar[d]&\left(\mathrm{det}^{-1}_{\mathbb{Z}_p}R\Gamma(\mathcal{X}_{\mathbb{Z}_p},\Omega^{<n}_{\mathcal{X}_{\mathbb{Z}_p}/\mathbb{Z}_p})
\right)_{\mathbb{Q}_p} \ar[d]
\\
\left(\bigotimes_{i<n;j} \mathrm{det}^{{(-1)}^{i+j+1}}_{\mathbb{Z}_p}H^j(\mathcal{X}_{\mathbb{Z}_p},p^{n-i}\cdot \Omega^{i}_{\mathcal{X}_{\mathbb{Z}_p}/\mathbb{Z}_p}) \right)_{\mathbb{Q}_p}\ar[r]& \left(\bigotimes_{i<n;j} \mathrm{det}^{{(-1)}^{i+j+1}}_{\mathbb{Z}_p}H^j(\mathcal{X}_{\mathbb{Z}_p},\Omega^{i}_{\mathcal{X}_{\mathbb{Z}_p}/\mathbb{Z}_p}) \right)_{\mathbb{Q}_p}
}
\]
where the vertical maps identify the given latices and the horizontal maps are the obvious identifications. Hence  $d_p(\mathcal{X},n)$ is the determinant of the lower horizontal map.

In view of $$H^j(\mathcal{X}_{\mathbb{Z}_p},\Omega^i_{\mathcal{X}_{\mathbb{Z}_p}/\mathbb{Z}_p})\simeq H^j(\mathcal{X},\Omega^i_{\mathcal{X}/\mathbb{Z}})\otimes_{\mathbb{Z}}\mathbb{Z}_p$$ and  since $H^j(\mathcal{X},\Omega^i_{\mathcal{X}/\mathbb{Z}})$ is a finitely generated $\mathbb{Z}$-module, one may suppose
that $H^j(\mathcal{X}_{\mathbb{Z}_p},\Omega^i_{\mathcal{X}_{\mathbb{Z}_p}/\mathbb{Z}_p})$ is a finitely generated free $\mathbb{Z}_p$-module for any $i,j$. It follows that $d_p(\mathcal{X},n)^{-1}$ is the determinant of the isomorphism
$$\left(\bigotimes_{i<n;j} \mathrm{det}^{{(-1)}^{i+j}}_{\mathbb{Z}_p} p^{n-i}\cdot H^j(\mathcal{X}_{\mathbb{Z}_p},\Omega^{i}_{\mathcal{X}_{\mathbb{Z}_p}/\mathbb{Z}_p}) \right)_{\mathbb{Q}_p}\stackrel{\sim}{\rightarrow}\left(\bigotimes_{i<n;j} \mathrm{det}^{{(-1)}^{i+j}}_{\mathbb{Z}_p} H^j(\mathcal{X}_{\mathbb{Z}_p},\Omega^{i}_{\mathcal{X}_{\mathbb{Z}_p}/\mathbb{Z}_p})\right)_{\mathbb{Q}_p}$$
and we obtain
$$d_p(\mathcal{X},n)=p^{-\sum_{i\leq n, j}(-1)^{i+j} \cdot (n-i)\cdot \mathrm{rank}_{\mathbb{Z}_p}H^j(\mathcal{X}_{\mathbb{Z}_p},\Omega^i_{\mathcal{X}_{\mathbb{Z}_p}/\mathbb{Z}_p})}.$$
Moreover, since $H^j(\mathcal{X}_{\mathbb{Z}_p},\Omega^i_{\mathcal{X}_{\mathbb{Z}_p}/\mathbb{Z}_p})$ is a free $\mathbb{Z}_p$-module for any $i,j$,  we have (see for example \cite{Illusie96} Proposition 6.6)
 $$H^j(\mathcal{X}_{\mathbb{Z}_p},\Omega^i_{\mathcal{X}_{\mathbb{Z}_p}/\mathbb{Z}_p})\otimes_{\mathbb{Z}_p}\mathbb{F}_p\simeq H^j(\mathcal{X}_{\mathbb{F}_p},\Omega^i_{\mathcal{X}_{\mathbb{F}_p}/\mathbb{F}_p}).$$
The result follows.
\end{proof}

\subsection{The main conjecture} Let $\mathcal{X}$ be a regular proper arithmetic scheme. We assume that $\mathcal{X}$ satisfies $\mathbf{L}(\overline{\mathcal{X}}_{et},n)$, $\mathbf{L}(\overline{\mathcal{X}}_{et},d-n)$, $\mathbf{AV}(\overline{\mathcal{X}}_{et},n)$, $\mathbf{B}(\mathcal{X},n)$ and $\mathbf{D}_p(\mathcal{X},n)$ for any prime number $p$. Moreover, we assume that $\mathbf{R}(\mathbb{F}_p,\mathrm{dim}(\mathcal{X}_{\mathbb{F}_p}))$ holds at the primes $p$ where $\mathcal{X}^{\mathrm{red}}_{\mathbb{F}_p}$ is singular. We suppose that
$$\zeta(\mathcal{X},s)=\Prod_{x\in\mathcal{X}_0}\dfrac{1}{1-N(x)^{-s}},$$
which converges for $\mathrm{Re}(s)>\mathrm{dim}(\mathcal{X})$,
has a meromorphic contination to the whole complex plane. We denote by $\mathrm{ord}_{s=n}\zeta(\mathcal{X},s)\in\bz$ its vanishing order and by $\zeta^*(\mathcal{X},n)\in\br$ its leading Taylor coefficient at $s=n$.
\begin{conj}\label{conj-vanishingorder} For any $n\in\bz$
$$ \mathrm{ord}_{s=n}\zeta(\mathcal{X},s)=\sum_{i\in\mathbb{Z}}(-1)^i\cdot i\cdot\mbox{\mbox{\emph{dim}}}_{\mathbb{R}}H^i_{\mathrm{\Ar,c}}(\mathcal{X},\tr(n)).$$
\end{conj}
We consider the rational number
$$C(\mathcal{X},n):=\Prod_{p<\infty}\mid c_p(\mathcal{X},n)\mid_p
\,\,:=\Prod_{p<\infty}p^{-v_p(c_p(\mathcal{X},n))}$$
where $v_p$ denotes the $p$-adic valuation. Recall from Proposition \ref{prop-lambda-infty} that we have a trivialization
$$\lambda_{\infty}=\lambda_{\infty}(\mathcal{X},n):\mathbb{R} \stackrel{\sim}{\longrightarrow}
\mathrm{det}_{\mathbb{R}}  R\Gamma_{\mathrm{\Ar,c}}(\mathcal{X},\tr(n))
\stackrel{\sim}{\longrightarrow} \Delta(\mathcal{X}/\mathbb{Z},n)\otimes_{\mathbb{Z}}\mathbb{R}$$
induced by cup-product with the fundamental class.

\begin{conj}\label{conjmain}
$$\lambda_{\infty}(\zeta^*(\mathcal{X},n)^{-1}\cdot C(\mathcal{X},n)\cdot\mathbb{Z})= \Delta(\mathcal{X}/\mathbb{Z},n).
$$

\end{conj}

We draw the following immediate consequence of Conjecture \ref{conj-vanishingorder}. The definition of $\zeta(\X_\infty,s)$ will be recalled in the proof.

\begin{prop} Conjecture \ref{conj-vanishingorder} implies that
\begin{equation} \ord_{s=n}\zeta(\overline{\X},s)=\sum_{i\in\bz}(-1)^i\cdot i\cdot\dim_\br H^{i}_{\Ar}(\overline{\mathcal{X}},\tr(n))\notag\end{equation}
where
$$\zeta(\overline{\X},s)=\zeta(\X,s)\zeta(\X_\infty,s)$$
is the {\em completed} Zeta-function of $\X$.
\end{prop}

\begin{proof} In view of definition \ref{r-rgc-def} the middle horizontal triangle in (\ref{rgc2}) gives an exact triangle
\[R\Gamma_{\Ar,c}(\X,\tr(n)) \to R\Gamma_\Ar(\overline{\X},\tr(n)) \to R\Gamma_\Ar(\X_\infty,\tr(n)) \to\]
and it suffices to show that
\begin{align*}\mathrm{ord}_{s=n}\zeta(\mathcal{X}_\infty,s)=&\sum_{i\in\mathbb{Z}}(-1)^i\cdot i\cdot \dim_\br H^i_{\mathrm{\Ar}}(\mathcal{X}_\infty,\tr(n))=\sum_{i\in\bz}(-1)^{i+1}\cdot \dim_\br H^i(\mathcal{X}_\infty,\br(n))\\
=&\sum_{i\geq 2n}(-1)^{i+1}\cdot \dim_\br H^i_\cd(\mathcal{X}_{/\br},\br(n))\\
=&\sum_{i\geq 2n}(-1)^{i+1}\cdot \left(\dim_\br H^i(\X(\bc),\br(n))^{G_\br} - \dim_\br (H^i(\X(\bc),\bc)/F^n)^{G_\br}\right) \end{align*}
where we have used the definition (\ref{rgxinftydef}) and the exact sequence (\ref{delseq>2n}). Denoting by $H^i(\X(\bc),\bc)\cong\bigoplus_{p+q=i}H^{p,q}$ the Hodge
decomposition and by
\[h^{p,q}=\dim_{\bc}H^{p,q};\quad h^{p,\pm}=\dim_{\bc}(H^{p,p})^{F_\infty=\pm
(-1)^p}\]
the Hodge numbers we have
\[ \dim_\br (H^i(\X(\bc),\bc)/F^n)^{G_\br}=\dim_\bc H^i(\X(\bc),\bc)/F^n= \sum_{p<n}h^{p,q}\]
and
\begin{align*}\dim_\br H^i(\X(\bc),\br(n))^{G_\br}=&\dim_\br H^i(\X(\bc),\br)^{F_\infty=(-1)^n}=\dim_\bc H^i(\X(\bc),\bc)^{F_\infty=(-1)^n}\\
=&\sum_{p<q}h^{p,q}+h^{\frac{i}{2},(-1)^{n-i/2}}  \end{align*}
since $F_\infty(H^{p,q})=H^{q,p}$. Here $h^{\frac{i}{2},\pm}=0$ if $\frac{i}{2}\notin\bz$. For $i\geq 2n$ we have that $p<n$ implies $p<q=i-p$ and hence we must show that
\begin{align*} \mathrm{ord}_{s=n}\zeta(\mathcal{X}_\infty,s)= &\sum_{i\geq 2n}(-1)^{i+1} \left(\sum_{n\leq p<q}h^{p,q}+h^{\frac{i}{2},(-1)^{n-i/2}}\right)\\
=&\sum_{i\in\bz}(-1)^{i+1} \left(\sum_{n\leq p<q}h^{p,q}+\sum_{n\leq \frac{i}{2}}h^{\frac{i}{2},(-1)^{n-i/2}}\right)\end{align*}
where the last identity holds since the sums are empty for $i<2n$. By definition
\begin{equation} \zeta(\mathcal{X}_\infty,s)=\prod_{i\in\bz} L_\infty(h^i(X),s)^{(-1)^i}\label{zetaxinftydef}\end{equation}
where
\[L_\infty(h^i(X),s)=\prod_{p<q}\Gamma_{\bc}(s-p)^{h^{p,q}}\cdot\prod_{p=\frac{i}{2}}
\Gamma_{\br}(s-p)^{h^{p,+}}\Gamma_{\br}(s-p+1)^{h^{p,-}},\]
\[\Gamma_\br(s)=\pi^{-s/2}\Gamma(\frac{s}{2});\quad \Gamma_{\bc}(s)=2(2\pi)^{-s}\Gamma(s)\]
and $X=\X_\bq$ is the generic fibre. Since
\[\ord_{s=n}\Gamma(s)=\begin{cases} -1 & n\leq 0\\ 0 & n\geq 1\end{cases}\]
we find
\[ \ord_{s=n}L_\infty(h^i(X),s)= -\left(\sum_{n\leq p<q}h^{p,q}+\sum_{n\leq \frac{i}{2}}h^{\frac{i}{2},(-1)^{n-i/2}}\right) \]
which proves the proposition.
\end{proof}

\subsection{Compatibility with the Tamagawa Number Conjecture}\label{sec:compatibility}

Let $F$ be a number field and  $$\pi:\mathcal{X}\rightarrow\textrm{Spec}(\mathcal{O}_F)$$ a smooth projective scheme, connected  of dimension $d$. We assume that $\mathcal{X}$ satisfies Conjectures $\textbf{L}(\overline{\mathcal{X}}_{et},n)$, $\textbf{L}(\overline{\mathcal{X}}_{et},d-n)$ and $\textbf{B}(\mathcal{X},n)$ and note that Conjecture $\textbf{AV}(\overline{\mathcal{X}}_{et},n)$ holds by Corollary \ref{cor-AVsmooth} and Conjecture  ${\bf D}_p(\mathcal{X},n)$ by Prop. \ref{psmooth}, so that Conjecture \ref{conjmain} makes sense.

We write $X:=\mathcal{X}\otimes_{\mathcal{O}_F}F$ for the generic fiber of $\mathcal{X}$, a smooth projective variety over $F$ of dimension $d-1$ and we fix a prime number $p$.
By Lemma \ref{invertible} in App. B we have a quasi-isomorphism of complexes of sheaves on $\X[1/p]_{et}$
\[ \bz(n)/p^\bullet\cong \mu_{p^\bullet}^{\otimes n}[0]\]
for any $n\in\bz$. By \cite{deligne94} there is a decomposition in the derived category of $p$-adic sheaves on $\X[1/p]_{et}$
\begin{equation} R\pi'_*\bq_p\cong\bigoplus\limits_{i\in\bz}R^i\pi'_*\bq_p[-i]\label{directsum}\end{equation}
where $\pi':=\pi[1/p]$ and $R^i\pi'_*\bq_p$ is a local system whose generic fibre we denote by
$$V^i_p\cong H^i(\mathcal{X}_{\overline{F}},\mathbb{Q}_p).$$
For the Artin-Verdier etale topos $\overline{\X}_{et}$ studied in App. A and the open subtopos
\[ \psi:\mathcal{X}[1/p]_{et}\xrightarrow{j}\X_{et}\xrightarrow{\phi} \overline{\X}_{et}\]
we define compact support cohomology $R\Gamma_c(\X[1/p],\F):=R\Gamma(\overline{\X}_{et},\psi_!\F)$ in the usual way. We denote by
$$u:\X_{\bF_p,et} \coprod \X_\infty\to\overline{\X}_{et}$$
the complementary closed embedding with components $u_p$ and $u_\infty$. We define the morphism of topoi $\alpha'$ by the factorization $$\alpha:Sh(G_{\mathbb{R}},\mathcal{X}(\mathbb{C}))\xrightarrow{\alpha'} \mathcal{X}[1/p]_{et}\xrightarrow{j}\mathcal{X}_{et}$$
where $\alpha$ was defined in section \ref{sect-AVTopos}.

\begin{lem} For a complex of sheaves $\mathcal{A}$ on $\overline{\X}_{et}$ with torsion cohomology there is a commutative diagram of exact triangles
\[\begin{CD} R\Gamma_c(\X[1/p],\mathcal{A}) @>>> R\Gamma(\X[1/p],\psi^*\mathcal{A}) @>>> R\Gamma(G_\br,\X(\bc),(\alpha')^*\psi^*\mathcal{A})\oplus R\Gamma(\X_{\bq_p},\mathcal{A})\\
\Vert@. @AAA @AAA\\
R\Gamma_c(\X[1/p],\mathcal{A}) @>>> R\Gamma(\overline{\X},\mathcal{A}) @>>> R\Gamma(\X_\infty,u_\infty^*\mathcal{A})\oplus R\Gamma(\X_{\bF_p},u_p^*\mathcal{A})\\
@. @AAA @AAA\\
{} @. R\Gamma_{\X_{\bF_p}\coprod\X_\infty}(\overline{\X},\mathcal{A}) @>\sim >> R\Gamma(\X_\infty,Ru_\infty^!\mathcal{A})\oplus R\Gamma(\X_{\bF_p},Ru_p^!\mathcal{A})
\end{CD}\]
\label{locxbar}\end{lem}

\begin{proof} This follows from the diagram of complexes of sheaves on $\overline{\X}_{et}$
\[\begin{CD}
\psi_!\psi^*\mathcal{A}@>>> R\psi_*\psi^*\mathcal{A} @>>> u_{\infty,*}u^*_\infty R\psi_*\psi^*\mathcal{A}\oplus u_{p,*}u^*_p R\psi_*\psi^*\mathcal{A}\\
\Vert@. @AAA @AAA\\
\psi_!\psi^*\mathcal{A}@>>> \mathcal{A} @>>> u_{\infty,*}u^*_\infty\mathcal{A}\oplus u_{p,*}u^*_p\mathcal{A}\\
@. @AAA @AAA\\
{} @. u_* Ru^!\mathcal{A} @>\sim>> u_{\infty,*}Ru^!_\infty\mathcal{A}\oplus u_{p,*}Ru^!_p\mathcal{A}
\end{CD}\]
together with the isomorphism
\begin{equation}u^*_\infty R\psi_* \cong u^*_\infty R\phi_*Rj_*\cong R\pi_*\alpha^*Rj_*\cong R\pi_*(\alpha')^*j^*Rj_*\cong R\pi_*(\alpha')^*\label{psicompute}\end{equation}
of Lemma \ref{lem-here0}, the isomorphism
\begin{align*} u^*_p R\psi_*\psi^*\mathcal{A}\cong &\hat{i}^*g^*\phi^* R\psi_*\psi^*\mathcal{A}\cong  \hat{i}^*g^*\phi^* R\phi_*Rj_*\psi^*\mathcal{A}\\
\cong & \hat{i}^*g^*Rj_*\psi^*\mathcal{A}\\
\cong & \hat{i}^*R\hat{j}_*\tilde{g}^*\psi^*\mathcal{A}
\end{align*}
of Lemma \ref{gloclem2} (the notation of which we use) and the proper base change theorem
\[ R\Gamma(\X_{\bF_p},u^*_p R\psi_*\psi^*\mathcal{A})\cong R\Gamma(\X_{\bF_p},\hat{i}^*R\hat{j}_*\tilde{g}^*\psi^*\mathcal{A})
\cong R\Gamma(\X_{\bz_p},R\hat{j}_*\tilde{g}^*\psi^*\mathcal{A})\cong R\Gamma(\X_{\bq_p},\mathcal{A}).\]
\end{proof}

For $n\in\bz$ we set
\[ R\Gamma(\overline{\X},\mathbb{Z}_p(n)):=\mathrm{holim}_\bullet R\Gamma(\overline{\X}_{et},\bz(n)^{\overline{\mathcal{X}}}/p^\bullet),\]
\[ R\Gamma(\X_\infty,\mathbb{Z}_p(n)):=\mathrm{holim}_\bullet R\Gamma(\X_\infty,u_\infty^*\bz(n)^{\overline{\mathcal{X}}}/p^\bullet)\]
and
\[ R\Gamma_{\X_\infty}(\overline{\X}, \bz_p(n)):= \mathrm{holim}_\bullet R\Gamma(\X_\infty, Ru^!_\infty\bz(n)^{\overline{\mathcal{X}}}/p^\bullet)\]
where $\bz(n)^{\overline{\mathcal{X}}}$ is defined in Def. \ref{znxbardef} in App. A. If $Z$ is a scheme we set
\[ R\Gamma(Z,\mathbb{Z}_p(n)):=\mathrm{holim}_\bullet R\Gamma(Z_{et},\bz(n)/p^\bullet)\]
where $\bz(n)$ was defined in section \ref{sect-emc}, i.e. is given by the higher Chow complex for $n\geq 0$.
If $R\Gamma_?(Z,\mathbb{Z}_p(n))$ is any of the complexes just defined we set
\[ R\Gamma_?(Z,\mathbb{Q}_p(n)):=R\Gamma_?(Z,\mathbb{Z}_p(n))_\bq.\]

\begin{lem} There is a commutative diagram with exact rows and columns
\[\minCDarrowwidth1em\begin{CD} R\Gamma_c(\X[1/p],\bz_p(n)) @>>> R\Gamma(\X[1/p],\bz_p(n)) @>>> R\Gamma(\X_{\br},\bz_p(n))\oplus R\Gamma(\X_{\bq_p},\bz_p(n))\\
\Vert@. @AAA @AAA\\
R\Gamma_c(\X[1/p],\bz_p(n)) @>>> R\Gamma(\overline{\X},\bz_p(n)) @>>> R\Gamma(\X_\infty,\bz_p(n))\oplus R\Gamma(\X_{\bz_p},\bz_p(n))\\
@. @AAA @AAA\\
{} @. R\Gamma_{\X_{\bF_p}\coprod\X_\infty}(\overline{\X},\bz_p(n)) @>\sim >> R\Gamma_{\X_\infty}(\overline{\X}, \bz_p(n))\oplus R\Gamma_{\X_{\bF_p}}(\X_{\bz_p},\bz_p(n))
\end{CD}\]
\label{locdiagram}\end{lem}

\begin{proof} We apply Lemma \ref{locxbar} with $\mathcal{A}=\bz(n)^{\overline{\mathcal{X}}}/p^\bullet$. By Prop. \ref{propstalk} there is an isomorphism
\begin{equation} \psi^*\bz(n)^{\overline{\mathcal{X}}}/p^\bullet\cong j^*\phi^*\bz(n)^{\overline{\mathcal{X}}}/p^\bullet\cong \bz(n)/p^\bullet\cong \mu_{p^\bullet}^{\otimes n}.\label{pullbackcompute}\end{equation}
By Artin's comparison isomorphism we have
\begin{align*} R\Gamma(G_\br,\X(\bc),(\alpha')^*\mu_{p^\bullet}^{\otimes n})\cong &R\Gamma(G_\br,R\Gamma(\X(\bc),(\alpha')^*\mu_{p^\bullet}^{\otimes n}))\cong R\Gamma(G_\br,R\Gamma(\X_{\bc,et},\mu_{p^\bullet}^{\otimes n}))\\\cong &R\Gamma(\X_{\br,et},\mu_{p^\bullet}^{\otimes n}).\end{align*}
By proper base change and Lemma \ref{gloclem} we get an isomorphism
\begin{align*} R\Gamma(\X_{\bF_p},u^*_p\bz(n)^{\overline{\mathcal{X}}}/p^\bullet )\cong &R\Gamma(\X_{\bF_p},\hat{i}^*g^*\phi^*\bz(n)^{\overline{\mathcal{X}}}/p^\bullet)
\cong R\Gamma(\X_{\bz_p},g^*\bz(n)/p^\bullet)\\\cong &R\Gamma(\X_{\bz_p},\bz(n)/p^\bullet)\end{align*}
where we use the notation of Lemma \ref{gloclem2} with $D=\bz$.
\end{proof}

\begin{prop} \label{lem-Gammaf-1} There is an isomorphism of exact triangles
\[\minCDarrowwidth1em\begin{CD}
R\Gamma_c(\X[1/p],\bq_p(n)) @>>> R\Gamma(\overline{\X},\bq_p(n)) @>>> R\Gamma(\X_\infty,\bq_p(n))\oplus R\Gamma(\X_{\bz_p},\bq_p(n))\\
@VVV @VVV @VVV\\
\bigoplus\limits_{i\in\bz} R\Gamma_c(\mathcal{O}_F[\frac{1}{p}], V^i_p(n))[-i]@>>> \bigoplus\limits_{i\in\bz} R\Gamma_{f}(F, V^i_p(n))[-i]@>>>
\bigoplus\limits_{i\in\bz} (R\Gamma_f(F_\br,V^i_p(n))\oplus R\Gamma_f(F_{\bq_p}, V^i_p(n)))[-i]
\end{CD}\]
where the upper row is the middle row in Lemma \ref{locdiagram} tensored with $\bq$, the lower exact triangles are defined as in \cite{fon92} for the $p$-adic representation $V_p^i(n)$ and the outer vertical isomorphisms are induced by the decomposition (\ref{directsum}).
\end{prop}

\begin{proof} The left vertical map is clearly an isomorphism.  The complex $R\Gamma(\X_\infty, Ru^!_\infty\bz(n)^{\overline{\mathcal{X}}}/p^\bullet)$ has $2$-torsion cohomology by (\ref{shriek}) in App. A. Hence
$$R\Gamma_{\X_\infty}(\overline{\X}, \bq_p(n)):=(\mathrm{holim}_\bullet R\Gamma(\X_\infty, Ru^!_\infty\bz(n)^{\overline{\mathcal{X}}}/p^\bullet))_\bq=0$$
and
\begin{equation}R\Gamma(\X_\infty,\bq_p(n))\cong R\Gamma(\X_{\br},\bq_p(n))\cong \bigoplus\limits_{i\in\bz} R\Gamma(F_\br,V^i_p(n))[-i]\cong \bigoplus\limits_{i\in\bz} R\Gamma_f(F_\br,V^i_p(n))[-i]\notag\end{equation}
is an isomorphism, where $R\Gamma_f(F_\br,V):=R\Gamma(F_\br,V)$ holds by definition \cite{fpr91}.
The isomorphism
\[ R\Gamma(\X_{\bz_p},\bq_p(n))\cong \bigoplus\limits_{i\in\bz} R\Gamma_f(F_{\bq_p}, V^i_p(n))[-i]\]
is the statement of Cor. \ref{pdecompcor} where we take the local decomposition  (\ref{pmotdec}) to be induced by (\ref{directsum}).

The middle vertical map will be an isomorphism if it exists, and existence will follow from commutativity of
 \[\begin{CD}R\Gamma(\X_\infty,\bq_p(n))\oplus R\Gamma(\X_{\bz_p},\bq_p(n))@>>> R\Gamma_c(\X[1/p],\bq_p(n))[1]\\
@V\beta VV @VVV \\
 \bigoplus\limits_{i\in\bz} (R\Gamma_f(F_\br,V^i_p(n))\oplus R\Gamma_f(F_{\bq_p}, V^i_p(n)))[-i]@>>>  \bigoplus\limits_{i\in\bz} R\Gamma_c(\mathcal{O}_F[\frac{1}{p}], V^i_p(n))[-i+1]
\end{CD}.\]
The following diagram commutes
 \[\begin{CD}\bigoplus\limits_{i\in\bz} (R\Gamma(F_\br,V^i_p(n))\oplus R\Gamma(F_{\bq_p}, V^i_p(n)))[-i]@>>>  \bigoplus\limits_{i\in\bz} R\Gamma_c(\mathcal{O}_F[\frac{1}{p}], V^i_p(n))[-i+1]\\
@AAA @AAA\\
 R\Gamma(\X_{\br},\bq_p(n))\oplus R\Gamma(\X_{\bq_p},\bq_p(n))@>>>R\Gamma_c(\X[1/p],\bq_p(n))[1] \\
@AAA \Vert@.\\
R\Gamma(\X_\infty,\bq_p(n))\oplus R\Gamma(\X_{\bz_p},\bq_p(n))@>>> R\Gamma_c(\X[1/p],\bq_p(n))[1]
\end{CD}\]
 since the bottom square is a shift of the commutative diagram in Lemma \ref{locdiagram} and the top square is induced by the decomposition (\ref{directsum}). But the left vertical map factors through $\beta$ which concludes the proof.
\end{proof}

\begin{lem} There are natural isomorphisms
\begin{equation}R\Gamma_W(\overline{\mathcal{X}},\mathbb{Z}(n))_{\mathbb{Z}_p}\simeq \mathrm{holim}_\bullet R\Gamma(\overline{\mathcal{X}}_{et},\bz(n)/p^\bullet) = R\Gamma(\overline{\mathcal{X}},\mathbb{Z}_p(n))\label{xbarzp}\end{equation}
where $R\Gamma_W(\overline{\mathcal{X}},\mathbb{Z}(n))$ was defined in Def. \ref{defn-fg-cohomology},
and
\begin{equation}R\Gamma_W(\mathcal{X}_\infty,\mathbb{Z}(n))_{\mathbb{Z}_p}\simeq \mathrm{holim}_\bullet R\Gamma(\mathcal{X}_\infty,u_\infty^*\bz(n)^{\overline{\X}}/p^\bullet)=R\Gamma(\mathcal{X}_\infty,\mathbb{Z}_p(n))\label{xinftyzp}\end{equation}
where $R\Gamma_W(\X_\infty,\mathbb{Z}(n))$ was defined in Def. \ref{iinftydef}.
\label{zplemma}\end{lem}

\begin{proof} The first isomorphism is clear from Lemma \ref{lem-finitecoef} and perfectness of $R\Gamma_W(\overline{\mathcal{X}},\mathbb{Z}(n))$. One has an isomorphism of exact triangles in the derived category of sheaves on $\X_\infty$
\[\begin{CD} i^*_{\infty}\mathbb{Z}(n)/p^\bullet @>>> R\pi_*(2\pi i)^n\mathbb{Z}/p^\bullet @>>> \tau^{>n}R\widehat{\pi}_*(2\pi i)^n\mathbb{Z}/p^\bullet@>>>{}\\
@V\simeq VV @V\simeq VV@V\simeq VV@.\\
u^*_{\infty}\mathbb{Z}(n)^{\overline{\mathcal{X}}}/p^\bullet @>>> u^*_\infty R\psi_*\psi^*\mathbb{Z}(n)^{\overline{\mathcal{X}}} /p^\bullet @>>> Ru_{\infty}^!\mathbb{Z}(n)^{\overline{\mathcal{X}}}/p^\bullet[1]@>>>{}
\end{CD}\]
where the first row is the defining triangle of $i_\infty^*\bz(n)$ modulo $p^\bullet$ and the second row is the localization triangle for $\mathbb{Z}(n)^{\overline{\mathcal{X}}}/p^\bullet$. The right vertical map is an isomorphism by (\ref{shriek}) and the middle by (\ref{psicompute}) and (\ref{pullbackcompute}). The statement then follows again from perfectness of $R\Gamma_W(\X_\infty,\mathbb{Z}(n)):=R\Gamma(\X_\infty,i^*_{\infty}\mathbb{Z}(n))$.
\end{proof}

\begin{defn} For each prime $\mathfrak{p}\mid p$ of $F$ define the two-term complex
\[C_{cris,\mathfrak{p}}(V^i_p(n)):=[ D_{cris,\mathfrak{p}}(V^i_p(n))\stackrel{1-\phi}{\longrightarrow} D_{cris,\mathfrak{p}}(V^i_p(n)) ]\]
\end{defn}

\begin{lem} \label{lem-Gammaf-2}In the situation of Prop. \ref{lem-Gammaf-1}  there is an isomorphism of exact triangles.
\[ \xymatrix{
R\Gamma_{dR}(\mathcal{X}_{\mathbb{Q}_p}/\mathbb{Q}_p)/F^n[-1]\ar[r]\ar[d]
& \bigoplus_i \bigoplus_{\mathfrak{p}\mid p} D_{dR,\mathfrak{p}}(V^i_p(n))/F^0[-i-1] \ar[d]
\\
R\Gamma(\mathcal{X}_{\mathbb{Z}_p}, \mathbb{Q}_p(n))\ar[r]\ar[d]& \bigoplus_i \bigoplus_{\mathfrak{p}\mid p} R\Gamma_{f}(F_{\mathfrak{p}}, V^i_p(n))[-i]\ar[d]
\\
R\Gamma'_{eh}(\mathcal{X}_{\mathbb{F}_p}, \mathbb{Q}_p(n))\ar[r]&\bigoplus_i \bigoplus_{\mathfrak{p}\mid p} C_{cris,\mathfrak{p}}(V^i_p(n))[-i]
}
\]
\end{lem}

\begin{proof} This is clear in view of Prop. \ref{psmooth}, the isomorphism
$$R\Gamma_f(F_{\bq_p}, V^i_p(n)))\cong \bigoplus_{\mathfrak{p}\mid p} R\Gamma_{f}(F_{\mathfrak{p}}, V^i_p(n))$$
and the fact that $\mathcal{X}^{\mathrm{red}}_{\mathbb{F}_p}=\Coprod_{\mathfrak{p}\mid p}\mathcal{X}_{\mathfrak{p}}$ is smooth projective over $\bF_p$ where $\mathcal{X}_{\mathfrak{p}}:=\mathcal{X}\otimes_{\mathcal{O}_F}\mathbb{F}_{\mathfrak{p}}$. By definition we have $$R\Gamma'_{eh}(\mathcal{X}_{\mathbb{F}_p},\mathbb{Z}_p(n))=R\Gamma_{et}(\mathcal{X}^{\mathrm{red}}_{\mathbb{F}_p},\mathbb{Z}_p(n)).$$ \end{proof}

Recall that an endomorphism $D\xrightarrow{\psi} D$ of a vector space over $\bq_p$, say, is called semisimple at zero if the map
\[ \bar{\psi}:\ker(\psi)\subseteq D\to\coker(\psi)\]
is an isomorphism. In this case one has a commutative diagram of isomorphisms
\begin{equation}\begin{CD} \det[D\xrightarrow{\psi} D] @>\id_{D,triv}>>\bq_p\\
@VVV @VV\mydet^*(\psi)V\\
\det(\ker(\psi))\otimes\det^{-1}(\coker(\psi)) @>\bar{\psi}_{triv}>>\bq_p
\end{CD}\label{ssdiagram}\end{equation}
where for any isomorphism $f:V\to W$ we denote by $f_{triv}$ the induced isomorphism $\det(V)\otimes\det^{-1}(W)\cong\bq_p$ and
$\mydet^*(\psi)\in\bq_p^\times$ is the determinant of $\psi$ on a complement of $\ker(\psi)$.

\begin{lem} \label{conj-ss} In the situation of Prop. \ref{lem-Gammaf-1}, assume in addition that the complex $C_{cris,\mathfrak{p}}(V^i_p(n))$
is semi-simple at $0$ for any $i$ and any $\mathfrak{p}\mid p$. Then cup-product with the fundamental class $e\in H^1(W_{\mathbb{F}_p},\mathbb{Z})$ gives an acyclic complex
$$\cdots\stackrel{\cup e}{\longrightarrow} H'^{\,i}_{eh}(\mathcal{X}_{\mathbb{F}_p},\mathbb{Q}_p(n))\stackrel{\cup e}{\longrightarrow} H'^{\,i+1}_{eh}(\mathcal{X}_{\mathbb{F}_p},\mathbb{Q}_p(n))\stackrel{\cup e}{\longrightarrow}\cdots$$
and hence a trivialization
$$\cup e:\mathrm{det}_{\mathbb{Q}_p}R\Gamma'_{eh}(\mathcal{X}_{\mathbb{F}_p}, \mathbb{Q}_p(n))\stackrel{\sim}{\longrightarrow} \bq_p.$$
Moreover the square of isomorphisms
\[\begin{CD} \mathrm{det}_{\mathbb{Q}_p} R\Gamma'_{eh}(\mathcal{X}_{\mathbb{F}_p}, \mathbb{Q}_p(n))@>\cup e>> \bq_p\\
@VV\beta_p(\X,n)V \Vert@.\\
\bigotimes\limits_{i, \mathfrak{p}\mid p}\mathrm{det}_{\mathbb{Q}_p}^{(-1)^i}C_{cris,\mathfrak{p}}(V^i_p(n))@>\sigma>> \bq_p
\end{CD}\]
commutes, where $\beta_p(\X,n)$ is induced by the bottom isomorphism in Lemma \ref{lem-Gammaf-2}, and
\[ \sigma=\bigotimes_{i, \mathfrak{p}\mid p}(\overline{(1-\phi)|D_{cris,\mathfrak{p}}(V^i_p(n))})_{triv}^{(-1)^i}.\]
\end{lem}

\begin{proof} For any complex of $W_{\bF_p}$-modules $C$ we have
$$R\Gamma(W_{\bF_p},C)\cong \mathrm{holim} (C\xrightarrow{1-\phi} C)=[C\xrightarrow{1-\phi} C]$$
and the discussion before \cite{Geisser04b}[Prop. 4.4] shows that there is a commutative diagram
\[\begin{CD}R\Gamma(W_{\bF_p},C)@>\sim>> [0 @>>> C@>{1-\phi}>> C]\\
@VV\cup eV @. @V\id_C VV @.\\
R\Gamma(W_{\bF_p},C)[1] @>\sim>> [C@>{1-\phi}>> C@>>> 0]\end{CD}\]
and hence a commutative diagram with exact rows for each $i$
\[\minCDarrowwidth1em\begin{CD}0@>>> \coker(1-\phi\vert H^{i-1}(C))@>>> H^i(W_{\bF_p},C)@>>> \ker(1-\phi\vert H^i(C))@>>> 0\\
@. @. @VV\cup eV @V\id_{H^i(C)} VV\\
0@<<<\ker(1-\phi\vert H^{i+1}(C)) @<<< H^{i+1}(W_{\bF_p},C)@<<< \coker(1-\phi\vert H^{i}(C)) @<<< 0.
\end{CD}\]
So if $1-\phi$ is semisimple at zero on each $H^i(C)$ we obtain a long exact sequence
\[ \cdots\stackrel{\cup e}{\longrightarrow} H^i(C)\stackrel{\cup e}{\longrightarrow} H^{i+1}(C) \stackrel{\cup e}{\longrightarrow}\cdots\]
It then suffices to remark that
$$R\Gamma'_{eh}(\mathcal{X}_{\mathbb{F}_p},\mathbb{Q}_p(n))
=R\Gamma_{et}(\mathcal{X}^{\mathrm{red}}_{\mathbb{F}_p},\mathbb{Q}_p(n))\cong R\Gamma(W_{\bF_p},R\Gamma_{cris}(\mathcal{X}^{\mathrm{red}}_{\mathbb{F}_p}/\bq_p)) $$
where $\phi\in W_{\bF_p}$ acts on $R\Gamma_{cris}(\mathcal{X}^{\mathrm{red}}_{\mathbb{F}_p}/\bq_p)$ by $\phi_n=\phi p^{-n}$
and that for this action we have an isomorphism of $\phi$-modules
\[ H^i_{cris}(\mathcal{X}^{\mathrm{red}}_{\mathbb{F}_p}/\bq_p)\cong \bigoplus_{\mathfrak{p}\mid p}D_{cris,\mathfrak{p}}(V^i_p(n)).\]
\end{proof}

\begin{prop} Under the assumptions of this section, Conjecture \ref{inj} holds for $\X$. In fact,
\[ H^i(\X,\bq(n))\to H^i(X,\bq(n)) \]
is injective for all $n$ and $i$.
\label{injprop}\end{prop}

\begin{proof} For $i>2n$ both groups are zero. For $i\leq 2n$ we have a commutative diagram
\[\xymatrix{H^i(\X,\bq(n))\ar[d]\ar[r]^{\sim} & H^i_W(\overline{\X},\bz(n))_\bq\ar[r] & H^i_W(\overline{\X},\bz(n))_{\bq_p}\ar[r]^{(\ref{xbarzp})\otimes\bq} &H^i(\overline{\X},\bq_p(n))\ar[d]\\
H^i(X,\bq(n))\ar[rrr]& & & H^i(X,\bq_p(n))}
\]
where the first map in the top row is the isomorphism of Cor. \ref{prop-rational-decompo}, the second map is clearly injective, and the right vertical map is injective since by Prop. \ref{lem-Gammaf-1} it is isomorphic to the injective map
\[ H^0_f(F, V^i_p(n))\oplus H^1_f(F, V^{i-1}_p(n)) \to  H^0(F, V^i_p(n))\oplus H^1(F, V^{i-1}_p(n)).\]
It follows that the left vertical map is injective. Note here that for $j=2,3$ we have
\[ H^j_f(F, V^{i-j}_p(n))\cong H^{3-j}_f(F, V^{i-j}_p(n)^*(1))^*\cong H^{3-j}_f(F, V^{2d-2-i+j}_p(d-n))^*=0\]
since $$H^{3-j+2d-2-i+j}(\X,\bq(d-n))=H^{2(d-n)+2n-i+1}(\X,\bq(d-n))=0$$ for $2n-i\geq 0$.
\end{proof}

\begin{rem} \label{injrem}
Suppose $\X$ is a regular scheme satisfying Conjectures $\textbf{L}(\overline{\mathcal{X}}_{et},n)$, $\textbf{L}(\overline{\mathcal{X}}_{et},d-n)$ and $\textbf{AV}(\overline{\mathcal{X}}_{et},n)$. Suppose $p$ is a prime number so that the conclusion of Cor. \ref{pdecompcor} holds for $\X_{\bz_p}$ and the conclusion of Cor. \ref{decompcor} (with the roles of $p$ and $l$ switched) for $\X_{\bz_l}$ if $p\neq l$. Then one can prove an analogue of Prop. \ref{lem-Gammaf-1} with $\X[1/p]/\co_F[1/p]$ replaced by $\X[1/Np]/\bz[1/Np]$ where $N$ is divisible by all $l$ where $\X_{\bz_l}/\bz_l$ is not smooth, and one can deduce Conjecture \ref{inj} for $\X$ and $i<2n$ following the proof of Prop. \ref{injprop}. So essentially, Conjecture \ref{inj} is a consequence of finite generation of motivic cohomology,  the monodromy weight conjecture for all $\X_{\bq_l}$ and the syntomic description of $p$-adically completed motivic cohomology of $\X_{\bz_p}$ (which holds if one simply chooses $p$ to be a good reduction prime).
\end{rem}

We recall the formulation of the Tamagawa number conjecture from \cite{fpr91}.
We consider the pure motive $h^i(X)(n)$ of weight $i-2n$ which we imagine as a pure object of $\mathcal{MM}_F$. Note that
$$\left(h^i(X)(n)\right)^*(1)\simeq h^{2(d-1)-i}(X)(d-1-n)(1)=h^{2d-2-i}(X)(d-n).$$
Consider the $\mathbb{Q}$-fundamental line
\begin{eqnarray*}
\Delta_f(h^i(X)(n))&:= &\mathrm{det}_{\mathbb{Q}}H^0_f(F,h^i(X)(n))\otimes_{\mathbb{Q}}\mathrm{det}^{-1}_{\mathbb{Q}}H^1_f(F,h^i(X)(n))\\
& & \otimes_{\mathbb{Q}}\mathrm{det}_{\mathbb{Q}}H^0_f(F,(h^i(X)(n))^*(1))\otimes_{\mathbb{Q}}\mathrm{det}^{-1}_{\mathbb{Q}}H^1_f(F,(h^i(X)(n))^*(1))\\
& & \otimes_{\mathbb{Q}}\mathrm{det}^{-1}_{\mathbb{Q}} \left(h^i(X)(n)_B^+\right) \otimes_{\mathbb{Q}}\mathrm{det}_{\mathbb{Q}} \left(t_{h^i(X)(n)}\right).
\end{eqnarray*}
Here $t_{h^i(X)(n)}:=(h^i(X)(n))_{dR}/\mathrm{Fil}^0$ is the tangent space. The period isomorphism induces the map
$$\alpha_{h^i(X)(n)}:\left(h^i(X)(n)_B^+\right)_{\mathbb{R}}\longrightarrow \left(t_{h^i(X)(n)}\right)_{\mathbb{R}}.$$

\begin{conj}\label{conj-6terms}(Fontaine-Perrin-Riou)
There is a canonical exact sequence of finite dimensional $\mathbb{R}$-vector spaces
$$0\rightarrow  H^0_f(F,h^i(X)(n))_{\mathbb{R}} \rightarrow\mathrm{Ker}(\alpha_{h^i(X)(n)}) \rightarrow H^1_f(F,(h^i(X)(n))^*(1))_{\mathbb{R}}^*$$
$$\rightarrow H^1_f(F,h^i(X)(n))_{\mathbb{R}}
\rightarrow \mathrm{Coker}(\alpha_{h^i(X)(n)}) \rightarrow H^0_f(F,(h^i(X)(n))^*(1))_{\mathbb{R}}^*\rightarrow 0.$$
\end{conj}
Conjecture \ref{conj-6terms} gives a trivialization
$$\vartheta^{i,n}_{\infty}:\mathbb{R}\stackrel{\sim}{\longrightarrow} \Delta_f(h^i(X)(n))_{\mathbb{R}}.$$
Beilinson's conjecture on special values, in the formulation Fontaine-Perrin-Riou,  is the following
\begin{conj}\label{conj-Beilinson} (Beilinson)
$$\vartheta^{i,n}_{\infty}(L^*(h^i(X),n)^{-1})\in \Delta_f(h^i(X)(n)).$$
\end{conj}

For a prime number $p$ one defines
\begin{eqnarray*}
\Delta_f(V^i_p(n))&:= &\mathrm{det}_{\mathbb{Q}_p}H^0_f(F,V^i_p(n))\otimes_{\mathbb{Q}_p}\mathrm{det}^{-1}_{\mathbb{Q}}H^1_f(F,V^i_p(n))\\
& & \otimes_{\mathbb{Q}_p}\mathrm{det}_{\mathbb{Q}_p}H^0_f(F,(V^i_p(n))^*(1))\otimes_{\mathbb{Q}_p}\mathrm{det}^{-1}_{\mathbb{Q}_p}H^1_f(F,(V^i_p(n))^*(1))\\
& & \otimes_{\mathbb{Q}_p}\mathrm{det}^{-1}_{\mathbb{Q}_p} \left(V^i_p(n)\right)^+ \otimes_{\mathbb{Q}_p}\mathrm{det}_{\mathbb{Q}_p} \left(t_{V^i_p(n)}\right).
\end{eqnarray*}
\begin{conj}\label{conj-p-adic-realization}(Bloch-Kato)
For $M=h^i(X)(n)$ and $M=h^i(X)(n)^*(1)$ the $p$-adic realization induces isomorphisms
$$H^j_f(F,M)\otimes_{\mathbb{Q}}\mathbb{Q}_p\stackrel{\sim}{\longrightarrow}H^j_f(F,M_p)$$
for $j=0,1$.
\end{conj}

One has an isomorphism
\begin{equation}
\vartheta_p^{i,n}:\Delta_f(h^i(X)(n))_{\mathbb{Q}_p}\simeq \Delta_f(V^i_p(n))\cong \mathrm{det}_{\mathbb{Q}_p}R\Gamma_{c}(\mathcal{O}_F[1/p],V^i_p(n))
\notag\end{equation}
where the first isomorphism is obtained by Conjecture \ref{conj-p-adic-realization} and Artin's comparison theorem and the second isomorphism by the lower exact triangle in Prop. \ref{lem-Gammaf-1}, the exact triangle
\[ D_{dR,\mathfrak{p}}(V^i_p(n))/F^0[-1]\to R\Gamma_f(F_\mathfrak{p},V^i_p(n))\to C_{cris,\mathfrak{p}}(V^i_p(n))\to \]
arising from the definition of $R\Gamma_f(F_\mathfrak{p},V^i_p(n))$ and the isomorphism
\begin{equation}\tau_\mathfrak{p}^{i,n}=\id_{D_{cris,\mathfrak{p}}(V^i_p(n)),triv}:
\mathrm{det}_{\mathbb{Q}_p}C_{cris,\mathfrak{p}}(V^i_p(n))\simeq\mathbb{Q}_p\label{taudef}\end{equation}
in the notation of diagram (\ref{ssdiagram}).
Any locally constant $\mathbb{Z}_p$-sheaf $T^i_p(n)$ on $\mathrm{Spec}(\mathcal{O}_F[1/p])$
together with an isomorphism $T_p^i(n) \otimes_{\mathbb{Z}_p}\mathbb{Q}_p\simeq V^i_p(n)$ gives an integral structure
$$\mathrm{det}_{\mathbb{Q}_p}R\Gamma_{c}(\mathcal{O}_F[1/p],V^i_p(n))\simeq \mathrm{det}_{\mathbb{Z}_p}R\Gamma_{c}(\mathcal{O}_F[1/p],T_p^i(n))\otimes_{\mathbb{Z}_p}\mathbb{Q}_p$$
which does not dependent on the choice of $T_p^i(n)$.

\begin{conj}\label{conj-tnc}(Bloch-Kato, Fontaine-Perrin-Riou) There is an identity of invertible $\bz_p$-submodules
$$\vartheta_{p}^{i,n}\vartheta^{i,n}_{\infty}(L^*(h^i(X),n)^{-1})\cdot\bz_p=\mathrm{det}_{\mathbb{Z}_p}R\Gamma_{c}(\mathcal{O}_F[1/p],T^i_p(n)).$$
of $\mathrm{det}_{\mathbb{Q}_p}R\Gamma_{c}(\mathcal{O}_F[1/p],V^i_p(n))$.
\end{conj}

In order to compare this statement with Conjecture \ref{conjmain} we shall consider the total "motive" $h(X)(n)\in\mathcal{D}^+(\mathcal{MM}_F)$ such that $\mathcal{H}^i(h(X)(n))=h^i(X)(n)$. One expects a (non-canonical) direct sum decomposition
\begin{equation}\label{motivic-directsum-decomposition}
h(X)(n)\simeq \bigoplus_{0\leq i\leq 2(d-1)} h^i(X)(n)[-i].
\end{equation}
In any case, the fundamental line of $h(X)(n)$ is
$$\Delta_f(h(X)(n)):=\bigotimes_{0\leq i\leq 2(d-1)} \Delta_f(h^i(X)(n))^{(-1)^i}$$
and similarly, we set
$$\Delta_f(h(X)(n)_p):=\bigotimes_{0\leq i\leq 2(d-1)} \Delta_f(V^i_p(n))^{(-1)^i}.$$
In view of Prop. \ref{injprop} we can take the higher Chow groups $H^j(\mathcal{X},\mathbb{Q}(n))$ as our definition for the $f$-motivic cohomology $H^j_f(F,h(X)(n))$. In view of the isomorphism
$$R\Gamma_W(\overline{\mathcal{X}},\mathbb{Z}(n))_{\mathbb{Q}}\simeq R\Gamma(\mathcal{X},\mathbb{Q}(n))\oplus
R\mathrm{Hom}(R\Gamma(\mathcal{X},\mathbb{Q}(d-n)),\mathbb{Q}[-2d-1])$$
of Cor. \ref{prop-rational-decompo}, the definition of $R\Gamma_{W,c}(\mathcal{X},\mathbb{Z}(n))$ in Def. \ref{z-rgc-def} and the definition of $\Delta(\mathcal{X}/\mathbb{Z},n)$ in Def. \ref{deltadef} we then have an isomorphism
\begin{align*}&\Delta_f(h(X)(n))\\\cong\  &\mathrm{det}_\bq R\Gamma_W(\overline{\mathcal{X}},\mathbb{Z}(n))_{\mathbb{Q}}\otimes_\bq\mathrm{det}^{-1}_\bq R\Gamma(\mathcal{X}(\bc),\mathbb{Q}(n))^{G_\br}\otimes_\bq\mathrm{det}_\bq R\Gamma_{dR}(\X_\bq/\bq)/F^n\\
\cong\  &\mathrm{det}_\bq R\Gamma_W(\overline{\mathcal{X}},\mathbb{Z}(n))_{\mathbb{Q}}\otimes_\bq\mathrm{det}^{-1}_\bq R\Gamma_W(\mathcal{X}_{\infty},\mathbb{Z}(n))_\bq\otimes_\bq\mathrm{det}_{\mathbb{Q}}(R\Gamma_{dR}(\mathcal{X}/\mathbb{Z})/F^n)_\bq
\\
\cong\ &\mathrm{det}_\bq R\Gamma_{W,c}(\mathcal{X},\mathbb{Z}(n))_\bq\otimes_\bq\mathrm{det}_{\mathbb{Q}}(R\Gamma_{dR}(\mathcal{X}/\mathbb{Z})/F^n)_\bq\\
=\ &\Delta(\X/\bz,n)_\bq.
\end{align*}
The isomorphism $\vartheta_{\infty}:=\bigotimes_i(\vartheta_{\infty}^{i,n})^{(-1)^i}$ becomes the isomorphism (\ref{ourtheta})
\begin{align*}\vartheta_{\infty}:\mathbb{R}\cong\ \Delta(\X_\bq,n)_\br=\Delta(\X/\bz,n)_\br\cong\Delta_f(h(X)(n))_{\mathbb{R}}
\end{align*}
explained in the introduction, based on Prop. \ref{tangentprop}. The isomorphism
of Conjecture \ref{conj-p-adic-realization} is the composite of the isomorphisms
$$R\Gamma_W(\overline{\mathcal{X}},\mathbb{Z}(n))_{\mathbb{Q}_p}\simeq R\Gamma(\overline{\mathcal{X}},\mathbb{Q}_p(n))\simeq
\bigoplus\limits_{i\in\bz} R\Gamma_{f}(F, V^i_p(n))[-i]$$
arising from Lemma \ref{zplemma} and  Prop. \ref{lem-Gammaf-1}, and the isomorphism
$\vartheta_{p}:=\bigotimes_i(\vartheta_{p}^{i,n})^{(-1)^i}$ becomes an isomorphism
\begin{equation}\vartheta_{p}:\Delta_f(h(X)(n))_{\bq_p}\cong \Delta_f(h(X)(n)_p)\cong \mathrm{det}_{\mathbb{Q}_p}R\Gamma_{c}(\X[1/p],\bq_p(n))\notag\end{equation}
using the left vertical isomorphism in Prop. \ref{lem-Gammaf-1}. Since $\X$ is smooth projective over $\Spec(\co_F)$ we have
\[ \zeta(\X,s)=\prod_{i\in\bz}L(h^i(X),s)^{(-1)^i}\]
and Conjecture \ref{conj-tnc} therefore predicts that
\begin{equation}\vartheta_{p}\vartheta_{\infty}(\zeta^*(\X,n)^{-1})\cdot\bz_p=\mathrm{det}_{\mathbb{Z}_p}R\Gamma_{c}(\X[1/p],\bz_p(n))
\label{totaltnc}\end{equation}
inside $\mathrm{det}_{\mathbb{Q}_p}R\Gamma_{c}(\X[1/p],\bq_p(n))$. Note here that we have an isomorphism
\begin{eqnarray*}
\mathrm{det}_{\mathbb{Z}_p}R\Gamma_{c}(\mathcal{X}[1/p],\mathbb{Z}_p(n))&\simeq &\mathrm{det}_{\mathbb{Z}_p}R\Gamma_{c}(\mathcal{O}_F[1/p],R \pi'_*\mathbb{Z}_p(n))\\
&\simeq & \bigotimes_{i}\mathrm{det}_{\mathbb{Z}_p}^{(-1)^i}R\Gamma_{c}(\mathcal{O}_F[1/p],R^i\pi'_*\mathbb{Z}_p(n))
\end{eqnarray*}
and we can choose $T_p^i(n):=R^i\pi'_*\mathbb{Z}_p(n)$.

\bigskip

\begin{thm} \label{equiv-main-tnc}
Let $\mathcal{X}/\mathcal{O}_F$ be smooth projective and $n\in\mathbb{Z}$ so that Conjectures $\textbf{L}(\overline{\mathcal{X}}_{et},n)$, $\textbf{L}(\overline{\mathcal{X}}_{et},d-n)$ and $\textbf{B}(\mathcal{X},n)$ hold for $\X$ and $n$. Assume that the complex $C_{cris,\mathfrak{p}}(V^i_p(n))$ is semisimple at zero for all $i$ and all primes $\mathfrak{p}$ of $F$. Then Conjecture \ref{conjmain} for $(\mathcal{X},n)$ is equivalent to the conjunction of the Tamagawa number conjecture (\ref{totaltnc}) for the motive $h(X)(n)$ over all primes $p$.
\end{thm}

\begin{proof} First note that the isomorphism $\vartheta_\infty$ and the isomorphism $\lambda_\infty$ defined in Prop. \ref{prop-lambda-infty} coincide by definition. So Conjecture \ref{conjmain} is equivalent to the identity
\begin{equation}\vartheta_{\infty}(\zeta^*(\X,n)^{-1})\cdot C(\X,n)\cdot\bz_p=\Delta(\mathcal{X}/\mathbb{Z},n)\otimes_{\mathbb{Z}}\mathbb{Z}_p\label{pconjmain}\end{equation}
for all primes $p$. Lemma \ref{zplemma} and Lemma \ref{locdiagram} induce an isomorphism
\begin{eqnarray*}
&&\Delta(\mathcal{X}/\mathbb{Z},n)\otimes_{\mathbb{Z}}\mathbb{Z}_p\\&=& \left(\mathrm{det}_{\mathbb{Z}}R\Gamma_{W,c}(\mathcal{X},\mathbb{Z}(n))\otimes_{\mathbb{Z}}\mathrm{det}_{\mathbb{Z}}
R\Gamma_{dR}(\mathcal{X}/\mathbb{Z})/F^n\right)_{\mathbb{Z}_p} \\
&\simeq &
\left(\mathrm{det}_\bz R\Gamma_W(\overline{\mathcal{X}},\mathbb{Z}(n))\otimes_\bz\mathrm{det}^{-1}_\bz R\Gamma_W(\mathcal{X}_{\infty},\mathbb{Z}(n))\otimes_\bz\mathrm{det}_{\mathbb{Z}}(R\Gamma_{dR}(\mathcal{X}/\mathbb{Z})/F^n)\right)_{\mathbb{Z}_p}\\
&\simeq &
\mathrm{det}_{\bz_p} R\Gamma(\overline{\mathcal{X}},\mathbb{Z}_p(n))\otimes_{\bz_p}\mathrm{det}_{\bz_p}^{-1} R\Gamma(\mathcal{X}_{\infty},\mathbb{Z}_p(n))\otimes_{\bz_p}\mathrm{det}_{\mathbb{Z}_p}R\Gamma_{dR}(\mathcal{X}_{\mathbb{Z}_p}/\mathbb{Z}_p)/F^n\\
&\simeq & \mathrm{det}_{\mathbb{Z}_p}R\Gamma_{c}(\mathcal{X}[1/p],\mathbb{Z}_p(n))\otimes_{\mathbb{Z}_p}\mathrm{det}_{\mathbb{Z}_p}
R\Gamma(\mathcal{X}_{\mathbb{Z}_p},\mathbb{Z}_p(n))\otimes_{\mathbb{Z}_p}\mathrm{det}_{\mathbb{Z}_p} R\Gamma_{dR}(\mathcal{X}_{\mathbb{Z}_p}/\mathbb{Z}_p)/F^n
\end{eqnarray*}
which we denote by $\vartheta_p^{\bz_p}$. Lemma \ref{lem-Gammaf-2} induces a commutative diagram of isomorphisms
\begin{equation}
\begin{CD} \mathrm{det}_{\mathbb{Q}_p}R\Gamma(\mathcal{X}_{\mathbb{Z}_p}, \mathbb{Q}_p(n)) \otimes_{\mathbb{Q}_p}\mathrm{det}_{\mathbb{Q}_p} R\Gamma_{dR}(\mathcal{X}_{\mathbb{Q}_p}/\mathbb{Q}_p)/F^n @>\lambda_p(\X,n)>>
\mathrm{det}_{\mathbb{Q}_p} R\Gamma'_{eh}(\mathcal{X}_{\mathbb{F}_p}, \mathbb{Q}_p(n))\\
@VVV @VV\beta_p(\X,n)V\\
\bigotimes\limits_{i, \mathfrak{p}\mid p} \mathrm{det}_{\mathbb{Q}_p}^{(-1)^i} R\Gamma_{f}(F_{\mathfrak{p}}, V^i_p(n)) \otimes \mathrm{det}_{\mathbb{Q}_p}^{(-1)^i}D_{dR,\mathfrak{p}}(V^i_p(n))/F^0 @>\bigotimes \lambda^{i,n}_{\mathfrak{p}}>> \bigotimes\limits_{i, \mathfrak{p}\mid p}\mathrm{det}_{\mathbb{Q}_p}^{(-1)^i}C_{cris,\mathfrak{p}}(V^i_p(n))
\end{CD}
\notag\end{equation}
where $\lambda_p(\mathcal{X},n)$ is the map defined using conjecture $\mathbf{D}_p(\mathcal{X},n)$.
By definition we have
\[ \vartheta_p= \left(\mathrm{id}_{R\Gamma_c}\otimes \gamma_p\right)  \circ(\vartheta_p^{\bz_p})_\bq\]
where $\mathrm{id}_{R\Gamma_c}$ is the identity map of
$\mathrm{det}_{\mathbb{Q}_p}R\Gamma_{c}(\mathcal{X}[1/p],\mathbb{Q}_p(n))$ and
\[\gamma_p:= \left(\bigotimes(\tau^{i,n}_{\mathfrak{p}})^{(-1)^i}\circ\beta_p(\X,n)\circ\lambda_p(\X,n) \right)\]
where the trivializations $\tau^{i,n}_{\mathfrak{p}}$ were defined in (\ref{taudef}). Comparing (\ref{pconjmain}) and (\ref{totaltnc}) we see that the two statements are equivalent if and only if
\[ C(\X,n)\cdot \gamma_p\left(\mathrm{det}_{\mathbb{Z}_p}
R\Gamma(\mathcal{X}_{\mathbb{Z}_p},\mathbb{Z}_p(n))\otimes_{\mathbb{Z}_p}\mathrm{det}_{\mathbb{Z}_p} R\Gamma_{dR}(\mathcal{X}_{\mathbb{Z}_p}/\mathbb{Z}_p)/F^n\right)=\bz_p\]
and from (\ref{lambda-int}) and the definition of $C(\X,n)$ this identity holds if and only  if
\[ p^{-\chi(\mathcal{X}_{\mathbb{F}_p},\mathcal{O},n)} \cdot \left(\bigotimes(\tau^{i,n}_{\mathfrak{p}})^{(-1)^i}\circ\beta_p(\X,n)\right)\left(\mathrm{det}_{\mathbb{Z}_p}R\Gamma'_{eh}(\mathcal{X}_{\mathbb{F}_p},\mathbb{Z}_p(n))\right)=\bz_p.\]
Consider the rational function $Z(\mathcal{X}_{\mathbb{F}_p},t)$ such that $\zeta(\mathcal{X}_{\mathbb{F}_p},s)=Z(\mathcal{X}_{\mathbb{F}_p},p^{-s})$ and its special value
$$Z^*(\mathcal{X}_{\mathbb{F}_p},p^{-n})=\mathrm{lim}_{t\rightarrow p^{-n}} (1-p^{n}t)^{\rho_n}Z(\mathcal{X}_{\mathbb{F}_p},t)$$
where $\rho_n:=-\mathrm{ord}_{t=p^{-n}}Z(\mathcal{X}_{\mathbb{F}_p},t)$ is the order of the pole of $Z(\mathcal{X}_{\mathbb{F}_p},t)$ at $t=p^{-n}$. Note that we have
$$Z(\mathcal{X}_{\mathbb{F}_p},t)=Z(\mathcal{X}^{\mathrm{red}}_{\mathbb{F}_p},t)$$
and
$$Z(\mathcal{X}_{\mathbb{F}_p},p^{-n} t)=\prod_{\mathfrak{p}\mid p} Z(\mathcal{X}_{\mathfrak{p}},p^{-n}t)=\prod_{\mathfrak{p}\mid p} \prod_{i}\mathrm{det}_{\mathbb{Q}_p}\left(1- t\phi\mid D_{cris,\mathfrak{p}}(V^i_p(n))\right)^{(-1)^{i+1}}.$$
 If $C_{cris,\mathfrak{p}}(V^i_p(n))$ is semisimple at zero then diagram (\ref{ssdiagram}) implies
\[ \bigotimes (\tau^{i,n}_{\mathfrak{p}})^{(-1)^i} = Z^*(\mathcal{X}_{\mathbb{F}_p},p^{-n})\cdot\sigma\]
where $\sigma$ is the map in Lemma \ref{conj-ss}. Hence by Lemma \ref{conj-ss} we are reduced to showing
\begin{align*} p^{-\chi(\mathcal{X}_{\mathbb{F}_p},\mathcal{O},n)} \cdot Z^*(\mathcal{X}_{\mathbb{F}_p},p^{-n})\cdot \left(\sigma\circ\beta_p(\X,n)\right)&\left(\mathrm{det}_{\mathbb{Z}_p}R\Gamma'_{eh}(\mathcal{X}_{\mathbb{F}_p},\mathbb{Z}_p(n))\right)\\
=p^{-\chi(\mathcal{X}_{\mathbb{F}_p},\mathcal{O},n)} \cdot Z^*(\mathcal{X}_{\mathbb{F}_p},p^{-n})\cdot \left(\cup e\right)&\left(\mathrm{det}_{\mathbb{Z}_p}R\Gamma'_{eh}(\mathcal{X}_{\mathbb{F}_p},\mathbb{Z}_p(n))\right)=\bz_p.
\end{align*}
But this is just a rewriting of the leading term formula for $Z(\mathcal{X}_{\mathbb{F}_p},t)$ due to Milne \cite{Milne86}[Thm. 0.1]. Note here that if $C$ denotes the perfect complex of $\bz_p$-modules with finite cohomology groups
$$\cdots\stackrel{\cup e}{\longrightarrow} H'^{\,i}_{eh}(\mathcal{X}_{\mathbb{F}_p},\mathbb{Z}_p(n))\stackrel{\cup e}{\longrightarrow} H'^{\,i+1}_{eh}(\mathcal{X}_{\mathbb{F}_p},\mathbb{Z}_p(n))\stackrel{\cup e}{\longrightarrow}\cdots$$
then the image of $\det_{\bz_p}(C)$ under the isomorphism $a:\det_{\bz_p}(C)_{\bq_p}\cong\bq_p$ arising from acyclicity of $C_{\bq_p}$ is $\chi(C)^{-1}\cdot\bz_p$ where $\chi(C)\in p^\bz$ is the multiplicative Euler characteristic of $C$. But $a$ coincides with the isomorphism denoted $\cup e$ above. Hence $$\left(\cup e\right)\left(\mathrm{det}_{\mathbb{Z}_p}R\Gamma'_{eh}(\mathcal{X}_{\mathbb{F}_p},\mathbb{Z}_p(n))\right)=\chi(\X_{\bF_p},\hat{\bz}(n))^{-1}\cdot\bz_p $$ where $\chi(\X_{\bF_p},\hat{\bz}(n))$ is the quantity appearing in \cite{Milne86}[Thm. 0.1].

\end{proof}

\subsection{Relationship with the functional equation}
We fix a regular scheme $\mathcal{X}$ of pure dimension $d$, which is flat and proper over $\mathbb{Z}$. We assume that $\mathcal{X}$ satisfies $\mathbf{L}(\overline{\mathcal{X}}_{et},n)$, $\mathbf{L}(\overline{\mathcal{X}}_{et},d-n)$, $\mathbf{AV}(\overline{\mathcal{X}}_{et},n)$, $\mathbf{B}(\mathcal{X},n)$ and $\mathbf{D}_p(\mathcal{X},n)$ for any prime number $p$. Moreover, we assume that $\mathbf{R}(\mathbb{F}_p,\mathrm{dim}(\mathcal{X}_{\mathbb{F}_p}))$ holds at the primes $p$ where $\mathcal{X}^{\mathrm{red}}_{\mathbb{F}_p}$ is singular. Recall that we denote the fundamental line by
$$\Delta(\mathcal{X}/\bz,n):=\mathrm{det}_{\mathbb{Z}}R\Gamma_{W,c}(\mathcal{X},\mathbb{Z}(n))
\otimes_{\mathbb{Z}}\mathrm{det}_{\mathbb{Z}} R\Gamma_{dR}(\mathcal{X}/\mathbb{Z})/F^n.$$

\begin{defn}
We set
\begin{eqnarray*}
\Xi_{\infty}(\mathcal{X}/\mathbb{Z},n)&:=& \mathrm{det}_{\mathbb{Z}}R\Gamma_{W}(\mathcal{X}_{\infty},\mathbb{Z}(n))\otimes\mathrm{det}^{-1}_{\mathbb{Z}} R\Gamma_{dR}(\mathcal{X}/\mathbb{Z})/F^n \\
& &\otimes  \mathrm{det}^{-1}_{\mathbb{Z}}R\Gamma_{W}(\mathcal{X}_{\infty},\mathbb{Z}(d-n))
\otimes\mathrm{det}_{\mathbb{Z}}R\Gamma_{dR}(\mathcal{X}/\mathbb{Z})/F^{d-n}.
\end{eqnarray*}
\end{defn}

\begin{prop}\label{prop-forfunctionalequation}
Duality for Deligne cohomology and duality for Weil-\'etale cohomology induce isomorphisms
$$\xi_{\infty}:\mathbb{R}\stackrel{\sim}{\longrightarrow} \Xi_{\infty}\otimes\mathbb{R}$$
and
$$\Delta(\mathcal{X}/\mathbb{Z},n)\otimes \Xi_{\infty}(\mathcal{X}/\mathbb{Z},n)\stackrel{\sim}{\longrightarrow}\Delta(\mathcal{X}/\mathbb{Z},d-n)$$
respectively, such that the square
\[ \xymatrix{
\Delta(\mathcal{X}/\mathbb{Z},n)\otimes \Xi_{\infty}(\mathcal{X}/\mathbb{Z},n)\otimes\mathbb{R}\ar[r]&\Delta(\mathcal{X}/\mathbb{Z},d-n)\otimes\mathbb{R}\\
\mathbb{R}\otimes\mathbb{R}\ar[u]^{\lambda_{\infty}(\mathcal{X},n)\otimes\xi_{\infty}(\mathcal{X},n)}\ar[r]^=&\mathbb{R}\ar[u]^{\lambda_{\infty}(\mathcal{X},d-n)}
}
\]
commutes.
\end{prop}
\begin{proof}
Recall from Section \ref{sect-compactsupport} the definition of $R\Gamma_W(\mathcal{X}_{\infty},\mathbb{Z}(n))$. We have an isomorphism
$$R\Gamma_W(\mathcal{X}_{\infty},\mathbb{Z}(n))\otimes\mathbb{R}\simeq R\Gamma(G_{\mathbb{R}},\mathcal{X}(\mathbb{C}),(2\pi i)^n\mathbb{R}).$$
Duality for Deligne cohomology
$$R\Gamma_{\mathcal{D}}(\mathcal{X}/\mathbb{R},\mathbb{R}(n))\simeq \mbox{RHom}(R\Gamma_{\mathcal{D}}(\mathcal{X}/\mathbb{R},\mathbb{R}(d-n)),\mathbb{R}[-2d+1])$$
yields
\begin{eqnarray*}
&&\left(\mathrm{det}_{\mathbb{Z}}R\Gamma_{W}(\mathcal{X}_{\infty},\mathbb{Z}(n))\otimes_{\mathbb{Z}}\mathrm{det}^{-1}_{\mathbb{Z}} R\Gamma_{dR}(\mathcal{X}/\mathbb{Z})/F^n\right) \otimes\mathbb{R}\\
&\simeq&
\mathrm{det}_{\mathbb{R}}R\Gamma_{\mathcal{D}}(\mathcal{X}/\mathbb{R},\mathbb{R}(n))\\
&\simeq&
\mathrm{det}_{\mathbb{R}}\mbox{RHom}(R\Gamma_{\mathcal{D}}(\mathcal{X}/\mathbb{R},\mathbb{R}(d-n)),\mathbb{R}[-2d+1])\\
&\simeq&
\mathrm{det}_{\mathbb{R}}R\Gamma_{\mathcal{D}}(\mathcal{X}/\mathbb{R},\mathbb{R}(d-n))\\
&\simeq&\left(\mathrm{det}_{\mathbb{Z}}R\Gamma_{W}(\mathcal{X}_{\infty},\mathbb{Z}(d-n))\otimes_{\mathbb{Z}}\mathrm{det}_{\mathbb{Z}}^{-1} R\Gamma_{dR}(\mathcal{X}/\mathbb{Z})/F^{d-n}\right)\otimes\mathbb{R}
\end{eqnarray*}
We obtain
$$\xi_{\infty}:\mathbb{R}\stackrel{\sim}{\longrightarrow} \Xi_{\infty}\otimes\mathbb{R}.$$
The exact triangle
$$R\Gamma_{W,c}(\mathcal{X},\mathbb{Z}(n))\rightarrow R\Gamma_{W}(\overline{\mathcal{X}},\mathbb{Z}(n))\rightarrow R\Gamma_W(\mathcal{X}_{\infty},\mathbb{Z}(n))$$
gives
\begin{equation}\label{ke}
\mathrm{det}_{\mathbb{Z}}R\Gamma_{W,c}(\mathcal{X},\mathbb{Z}(n))\otimes \mathrm{det}_{\mathbb{Z}}R\Gamma_{W}(\mathcal{X}_{\infty},\mathbb{Z}(n))\simeq \mathrm{det}_{\mathbb{Z}}R\Gamma_{W}(\overline{\mathcal{X}},\mathbb{Z}(n)).
\end{equation}
Similarly, we have
\begin{equation}\label{kee}
\mathrm{det}_{\mathbb{Z}}R\Gamma_{W,c}(\mathcal{X},\mathbb{Z}(d-n))\otimes \mathrm{det}_{\mathbb{Z}}R\Gamma_{W}(\mathcal{X}_{\infty},\mathbb{Z}(d-n))\simeq \mathrm{det}_{\mathbb{Z}}R\Gamma_{W}(\overline{\mathcal{X}},\mathbb{Z}(d-n)).
\end{equation}
Moreover, duality for Weil-\'etale cohomology gives
\begin{equation}\label{keee}
\mathrm{det}_{\mathbb{Z}}R\Gamma_{W}(\overline{\mathcal{X}},\mathbb{Z}(n))\simeq
\mathrm{det}_{\mathbb{Z}}R\mathrm{Hom}(R\Gamma_{W}(\overline{\mathcal{X}},\mathbb{Z}(n)),\mathbb{Z}[-2d-1])\simeq
\mathrm{det}_{\mathbb{Z}}R\Gamma_{W}(\overline{\mathcal{X}},\mathbb{Z}(d-n)).
\end{equation}
Then (\ref{ke}), (\ref{kee}) and (\ref{keee}) induce
$$\Delta(\mathcal{X}/\mathbb{Z},n)\otimes \Xi_{\infty}(\mathcal{X}/\mathbb{Z},n)\stackrel{\sim}{\longrightarrow}\Delta(\mathcal{X}/\mathbb{Z},d-n).$$

We have canonical isomorphism (see Corollary \ref{prop-rational-decompo})
$$ R\Gamma_{W}(\overline{\mathcal{X}},\mathbb{Z}(n))_{\mathbb{R}}\simeq R\Gamma(\mathcal{X},\mathbb{R}(n))\oplus R\mathrm{Hom}(R\Gamma(\mathcal{X},\mathbb{R}(d-n)),\mathbb{R}[-2d-1])$$
and the pairing induced by Weil-\'etale duality, after $(-)\otimes \mathbb{R}$, is the evident one (see the proof of Theorem \ref{thm-duality-fg}). Moreover, the maps $\lambda_{\infty}(\mathcal{X},n)$ and $\lambda_{\infty}(\mathcal{X},d-n)$ are induced by the pairing between motivic cohomology with $\mathbb{R}$-coefficients and  motivic cohomology with $\mathbb{R}$-coefficients and compact support (see Conjecture  $\mathbf{B}(\mathcal{X},n)$), which is compatible with duality for Deligne cohomology in the sense that we have an isomorphism of long exact sequences
\[ \xymatrix{
\ar[r]& H^i_c(\mathcal{X},\mathbb{R}(n))\ar[d]\ar[r]
&H^i(\mathcal{X},\mathbb{R}(n))\ar[d]_{}^{}\ar[r]
&H^i_{\mathcal{D}}(\mathcal{X}/\mathbb{R},\mathbb{R}(n))\ar[d]\ar[r]&
\\
\ar[r]&  H^{2d-i}(\mathcal{X},\mathbb{R}(d-n))^*\ar[r]
&H^{2d-i}_c(\mathcal{X},\mathbb{R}(d-n))^*\ar[r]
&H^{2d-1-i}_{\mathcal{D}}(\mathcal{X}/\mathbb{R},\mathbb{R}(d-n))^*\ar[r]&
}
\]
If follows \cite{Knudsen-Mumford-76} that the induced square of isomorphisms
\[ \xymatrix{
\mathrm{det}_{\mathbb{R}} R\Gamma_c(\mathcal{X},\mathbb{R}(n))\otimes  \mathrm{det}_{\mathbb{R}} R\Gamma_{\mathcal{D}}(\mathcal{X}/\mathbb{R},\mathbb{R}(n))\ar[r]\ar[d]& \mathrm{det}_{\mathbb{R}} R\Gamma(\mathcal{X},\mathbb{R}(n))\ar[d]\\
\mathrm{det}_{\mathbb{R}} R\Gamma(\mathcal{X},\mathbb{R}(d-n))^*\otimes \mathrm{det}^{-1}_{\mathbb{R}} R\Gamma_{\mathcal{D}}(\mathcal{X}/\mathbb{R},\mathbb{R}(d-n))^* \ar[r]& \mathrm{det}_{\mathbb{R}} R\Gamma(\mathcal{X},\mathbb{R}(d-n))^*
}
\]
commutes, where $(-)^*:=R\mathrm{Hom}(-,\mathbb{R}[-2d])$. From there, we easily obtain the commutativity of the square of the proposition.

\end{proof}

We denote by $$x_{\infty}^2(\mathcal{X},n):=\mathrm{det}(\xi_{\infty}(\mathcal{X},n))\in\mathbb{R}_{>0}$$
the determinant of the isomorphism $\xi_{\infty}(\mathcal{X},n)$ with respect to the given integral structures, i.e. the strictly positive real number such that $$\xi_{\infty}(\mathcal{X},n)(x_{\infty}^2(\mathcal{X},n)^{-1}\cdot \mathbb{Z})=\Xi_{\infty}(\mathcal{X},n).$$
Recall the definition of
$$\zeta(\overline{\mathcal{X}},s):=\zeta(\mathcal{X},s)\cdot \zeta(\mathcal{X}_\infty,s)$$
from (\ref{zetaxinftydef}). 
\begin{conj}\label{conj-funct-equation} We have
$$A(\mathcal{X})^{(d-s)/2}\cdot \zeta(\overline{\mathcal{X}},d-s)=\pm A(\mathcal{X})^{s/2} \cdot \zeta(\overline{\mathcal{X}},s)$$
where the conductor $A(\mathcal{X})>0$ is a positive real number.
\end{conj}
For $\X$ of dimension $d=2$, the conductor $A(\X)$ is the square root of the discriminant of the pairing
$$R\Gamma_{dR}(\X/\bz)/F^2\otimes_\bz^L R\Gamma_{dR}(\X/\bz)/F^2\rightarrow \bz[-2]$$
induced by Poincar\'e duality  (see \cite{Bloch87}).

\begin{cor}\label{cor:fe}
Assume that $\zeta(\overline{\mathcal{X}},d-s)$ satisfies the functional equation of Conjecture \ref{conj-funct-equation}.
If two of the following assertions are true, then so is the third:
\begin{itemize}
\item[i)] We have \begin{equation}\label{toward}
\begin{split}
 & A(\mathcal{X})^{n/2} \cdot \zeta^*(\mathcal{X}_\infty,n)\cdot x_{\infty}(\mathcal{X},n)^{-1}\cdot C(\mathcal{X},n)^{-1}
 \\
  =&\pm   A(\mathcal{X})^{(d-n)/2}\cdot \zeta^*(\mathcal{X}_\infty,d-n)\cdot x_{\infty}(\mathcal{X},d-n)^{-1}\cdot C(\mathcal{X},d-n)^{-1}.
 \end{split}
\end{equation}
\item[ii)] Conjecture \ref{conjmain} for $(\mathcal{X},n)$ holds.
\item[iii)] Conjecture \ref{conjmain} for $(\mathcal{X},d-n)$ holds.
\end{itemize}
\end{cor}

\begin{proof}
We have $x_{\infty}(\mathcal{X},d-n)=\pm x_{\infty}(\mathcal{X},n)^{-1}$, hence
Proposition \ref{prop-forfunctionalequation} gives an equality
$$\mathrm{det}(\lambda_{\infty}(\mathcal{X},n))\cdot x_{\infty}(\mathcal{X},n)=\pm \mathrm{det}(\lambda_{\infty}(\mathcal{X},d-n))\cdot x_{\infty}(\mathcal{X},d-n).$$
Moreover, Conjecture \ref{conjmain} for $(\mathcal{X},n)$ is the following:
$$\mathrm{det}(\lambda_{\infty}(\mathcal{X},n))=\pm \zeta^*(\mathcal{X},n)\cdot C(\mathcal{X},n)^{-1}$$
and similarly for  $(\mathcal{X},d-n)$. The Corollary now easily follows from Conjecture \ref{conj-funct-equation}.
\end{proof}

\subsection{Proven cases and examples}\label{sec:examples}

\subsubsection{Varieties over finite fields} Let $\mathcal{X}$ be regular proper arithmetic scheme of pure dimension $d$ defined over the finite field $\mathbb{F}_p$. We assume that $\mathcal{X}$ satisfies $\mathbf{L}(\mathcal{X}_{et},n)$, $\mathbf{L}(\mathcal{X}_{et},d-n)$ and $\mathbf{B}(\mathcal{X},d)$.
\begin{prop}
Conjecture \ref{conjmain} holds for $\mathcal{X}$ and any $n\in\mathbb{Z}$.
\end{prop}
\begin{proof}
By Proposition \ref{prop-cp-charp}, we have $C(\mathcal{X},n)=1$. By Corollary \ref{cor-comparison-conject-Let-vs-LW} and since
$\mathbf{B}(\mathcal{X},d)$ implies $\mathbf{P}(\mathcal{X},d)$, $\mathcal{X}$ satisfies $\mathbf{L}(\mathcal{X}_{W},n)$. The result follows from Theorem \ref{thm-comparison-char-p} and \cite{Morin15}.
\end{proof}

\subsubsection{The case $n=0$.} Let $\mathcal{X}$ be regular proper arithmetic scheme of pure dimension $d$. We assume that $\mathcal{X}$ satisfies $\mathbf{L}(\overline{\mathcal{X}}_{et},d)$ and $\mathbf{B}(\mathcal{X},d)$.

\begin{prop}
Conjecture \ref{conjmain} for $\mathcal{X}$ and $n=0$ is equivalent to (\cite{Morin14} Conjecture 4.2 (b)).
\end{prop}
\begin{proof}
By Proposition \ref{prop-cp-n=0}, we have $C(\mathcal{X},0)=1$. The result follows.
\end{proof}

\subsubsection{Number rings} Let $F$ be a number field, set $$\X=\Spec(\co_F)$$
and let
\[ \rho_n:=\ord_{s=n}\zeta_F(s)=\begin{cases} r_2 & \text{$n<0$ odd}\\r_1+r_2 & \text{$n<0$ even}\\r_1+r_2-1 & \text{$n=0$}\\-1 & \text{$n=1$} \\0 & \text{$n>1$}\end{cases}\]
be the well known vanishing order of the Dedekind Zeta function
$$ \zeta(\X,s)=\zeta_F(s)$$
where $r_1$, resp. $r_2$, is the number of real, resp. complex, places of $F$.
Conjectures ${\bf L}(\overline{\mathcal{X}}_{et},n)$, ${\bf B}(\mathcal{X},n)$ and ${\bf AV}(\overline{\mathcal{X}}_{et},n)$
are known for any $n\in\bz$.  For $n\geq 1$ define
\[ H^{i,n}:=H^i(\X_\et,\bz(n)).\]
We have $H^{i,n}=0$ for $i\leq 0$, the group $H^{1,n}$ is finitely generated and $H^{2,n}$ is finite. This follows from the isomorphism \cite{Geisser04a}[Thm.1.2]
\[ H^i(\X_{Zar},\bz(n))\xrightarrow{\sim}H^i(\X_\et,\bz(n)),\quad  i\leq n+1,\]
the analysis of the spectral sequence from motivic cohomology to algebraic K-theory (see \cite{Levine99}[14.4]) and the known properties (finite generation, ranks) of the algebraic K-groups of $\co_F$. In degrees $i\geq 3$ we have
\[ H^{i,n}\cong \begin{cases}{\mathrm Br}(\co_F)\cong (\bz/2\bz)^{r_1,\Sigma=0},\bq/\bz, (\bz/2\bz)^{r_1\cdot\delta_{i,1}} & n=1\\(\bz/2\bz)^{r_1\cdot\delta_{i,n}} & n\geq 2\end{cases}\]
where
\[ \delta_{i,n}=\begin{cases} 1  &  n\equiv i \mod 2\\ 0 & n\not\equiv i\mod 2.  \end{cases}\]
The Beilinson regulator map
$$H^1(\X_\et,\bz(n))\xrightarrow{r_n}H^1_{\cd}(\X_{/\br},\br(n))\cong \prod_{v\mid\infty}F_v/H^0(F_v,(2\pi i)^n\br)\cong \prod_{v\mid\infty}H^0(F_v,(2\pi i)^{n-1}\br)$$
induces isomorphisms
$$r_{n,\br}:H^1(\X_\et,\bz(n))_\br\xrightarrow{\sim}\prod_{v\mid\infty}H^0(F_v,(2\pi i)^{n-1}\br)$$
for $n>1$ and
$$r_{1,\br}:H^1(\X_\et,\bz(1))_\br\cong\bigl(\prod_{v\mid\infty}\br\bigr)^{\Sigma=0}$$
for $n=1$.  For $n\geq 1$ we set
\begin{align}  h_n:=& |H^2(\X_\et,\bz(n))|\notag\\ w_n:=&|H^1(\X_\et,\bz(n))_{tor}|\notag\\R_n:=&\vol(\coker(r_n)) \notag\end{align}
where the volume is taken with respect to the $\bz$-structure $\prod_{v\mid\infty}H^0(F_v,(2\pi i)^{n-1}\bz)$, resp. $(\prod_{v\mid\infty}\bz)^{\Sigma=0}$, of the target.
We have
$$H^i_W(\overline{\mathcal{X}},\mathbb{Z}(n))\cong\begin{cases}
0,0,H^{1,n},H^{2,n},(\bz/2\bz)^{r_1\cdot\epsilon_{3,n}},(\bz/2\bz)^{r_1\cdot\epsilon_{i,n}} & n>1\\
0,0,\mathcal{O}_F^{\times},\Cl(\co_F),\mathbb{Z},0 & n=1\\
0,\bz,0,(\co_F^\times)^*\oplus \Cl(\co_F)^D,(\co_F^\times)_{tor}^D,0 & n=0\\
(\bz/2\bz)^{r_1\cdot\epsilon_{i-1,n}},0,0, (H^{1,1-n})^*\oplus (H^{2,1-n})^D,(H^{1,1-n}_{tor})^D,0 & n<0\end{cases} $$
in degrees $i<0$, $i=0,1,2,3$ and $i>3$ respectively. Here $A^*=\Hom_\bz(A,\bz)$, $D$ is the Pontryagin dual and
\[ \epsilon_{i,n}=\begin{cases} \delta_{i,n}  &  1\leq i\leq n \text{  or  } n<i<0\\ 0 & \text{else.}  \end{cases}\]
The long exact sequence induced by (\ref{z-ar-w-triangle}) gives
$$H^i_\Ar(\overline{\mathcal{X}},\mathbb{Z}(n))\cong\begin{cases}
0,0,H^{1,n}_{tor},\coker(r_n)\oplus H^{2,n},(\bz/2\bz)^{r_1\cdot\epsilon_{3,n}},(\bz/2\bz)^{r_1\cdot\epsilon_{i,n}} & n>1\\
0,0,(\mathcal{O}_F^{\times})_{tor},Cl(\overline{\X}),\mathbb{Z},0 & n=1\\
0,\bz,0,(\co_F^\times)^*\oplus \Cl(\co_F)^D,(\co_F^\times)_{tor}^D,0 & n=0\\
(\bz/2\bz)^{r_1\cdot\epsilon_{i-1,n}},0,0, (H^{1,1-n})^*\oplus (H^{2,1-n})^D,(H^{1,1-n}_{tor})^D,0 & n<0\end{cases} $$
where $\Cl(\overline{\X})\cong\Pic(\overline{\X})$ is the usual Arakelov class group of $\co_F$.
For $n\leq 0$ there are isomorphisms
\[ H^i_{\Ar,c}(\X,\bz(n))\cong H^i_{W,c}(\X,\bz(n))\cong 0, \bz^{\rho_n}, (H^{1,1-n})^*\oplus (H^{2,1-n})^D,
(H^{1,1-n}_{tor})^D,0 \]
in degrees $i<1$, $i=1,2,3$ and $i>3$ respectively. Hence for any $i\in\bz$ the groups $H^i_{\Ar,c}(\X,\bz(n))$ are finitely generated,
\[ H^i_{\Ar,c}(\X,\tr(n))\cong H^i_{\Ar,c}(\X,\bz(n))\otimes_\bz\br\]
and $H^i_{\Ar,c}(\X,\tr/\bz(n))$ is compact. More precisely, for $n<0$ we have isomorphisms
\[ H^1_{\Ar,c}(\X,\tr(n))=H^1_c(\X,\br(n))\cong H^0_{\cd}(\X_{/\br},\br(n))\cong \prod_{v\mid\infty} H^0(F_v,(2\pi i)^n\br)    \]
and
\begin{align} H^2_{\Ar,c}(\X,\tr(n))=&H^1_c(\X,\br(n))\cong H^1(\X,\br(1-n))^*\cong H^1_{\cd}(\X_{/\br},\br(1-n))^*\label{h2ccompute}\\
\cong &\prod_{v\mid\infty}(F_v/H^0(F_v,(2\pi i)^{1-n}\br))^*\cong
\prod_{v\mid\infty}H^0(F_v,(2\pi i)^{-n}\br)^* \notag\end{align}
and similarly for $n=0$ (taking the quotient by the diagonally embedded $\br$ in the target). We note that
\[ \dim_\br H^1_{\Ar,c}(\X,\tr(n))=\dim_\br H^2_{\Ar,c}(\X,\tr(n))=\rho_n\]
and $H^i_{\Ar,c}(\X,\tr(n))=0$ for $i\neq 1,2$, verifying Conjecture \ref{conj-vanishingorder}. For $n<0$ recall that $R_{1-n}$ was formed with respect to the $\bz$-structures $H^{1,1-n}/tor$ and
$\prod_{v\mid\infty}H^0(F_v,(2\pi i)^{-n}\bz)$ in the dual of (\ref{h2ccompute}) and therefore also agrees with the regulator formed with respect to $(H^{1,1-n})^*$ and the natural $\bz$-structure $$(\prod_{v\mid\infty}H^0(F_v,(2\pi i)^{-n}\bz))^*\cong \prod_{v\mid\infty}H^0(F_v,(2\pi i)^{n}\bz)\cong H^1_{W,c}(\X,\bz(n))$$
of $H^1_{\Ar,c}(\X,\tr(n))$. Similar considerations apply to $n=0$. For $n\leq 0$ we have $C(\X,n)=1$ by Prop. \ref{prop-cp-n=0} and $R\Gamma_{dR}(\mathcal{X}/\mathbb{Z})/F^n=0$ and hence
Conjecture \ref{conjmain} reduces to the statement
\begin{equation} \zeta^*_F(n) = \pm\frac{h_{1-n}\cdot R_{1-n}}{w_{1-n}}. \label{nless0}\end{equation}
For $n\geq 1$ the groups $H^i_{\Ar,c}(\X,\bz(n))$ vanish except for
$i=2,3$. There is an exact sequence
\[ 0\to \left(\prod_{v\mid\infty}F_v^\times\right)/\co_F^\times\to
H^2_{\Ar,c}(\X,\bz(1)) \to \Cl(\co_F)\to 0\] showing that
the group $H^2_{\Ar,c}(\X,\bz(1))$ is an extension
$$0\to H^1_{\Ar,c}(\X,\tr/\bz(1))\to H^2_{\Ar,c}(\X,\bz(1))\to H^2_{\Ar,c}(\X,\tr(1))\cong\br\to 0$$
of $\br$ by a compact group. The continuation of this long exact sequence looks like
\[\begin{CD}0@>>>H^3_{\Ar,c}(\X,\bz(1))@>>> H^3_{\Ar,c}(\X,\tr(1))@>>> H^3_{\Ar,c}(\X,\tr/\bz(1))@>>> 0\\
@. \Vert@. \Vert@.\Vert@. @.\\
0 @>>> \bz @>>> \br @>>> \br/\bz @>>> 0.\end{CD}\]
Conjecture \ref{conj-vanishingorder} follows since $H^i_{\Ar,c}(\X,\tr(1))=0$ for $i\neq 2,3$.
For $n\geq 2$ there is an exact sequence
\begin{multline*}0\to \left(\prod_{v\mid\infty}H^0(F_v,\bc/(2\pi i)^n\bz)\right)/H^1(\X,\bz(n))\to H^2_{\Ar,c}(\X,\bz(n))\to H^2(\X,\bz(n))\to\\
\to (\bz/2\bz)^{r_1\cdot\epsilon_{2,n}}\to H^3_{\Ar,c}(\X,\bz(n))\to 0 \end{multline*}
showing that the groups $$H^{i-1}_{\Ar,c}(\X,\tr/\bz(n))\cong H^i_{\Ar,c}(\X,\bz(n))$$ are compact for $i=2,3$. We have
\[H^i_{\Ar,c}(\X,\tr(n))=0\]
for all $i$, verifying Conjecture \ref{conj-vanishingorder}.

\begin{prop} Assume $\X=\Spec(\co_F)$ and $n\geq 1$. If Conjecture $C_{EP}(\bq_p(n))$ of \cite{pr95}[App.C2] holds over all local fields $K=F_v$, in particular if all $F_v$ are abelian extensions of $\bq_p$ \cite{bb05}, then
\[ C(\X,n)= (n-1)!^{-[F:\bq]}.\]
\label{Ccompute}\end{prop}

\begin{proof} We first explicate Conjecture $C_{EP}(V)$ of \cite{pr95}[App.C2] for $K=F_v$ and $V=\bq_p(n)$. We have $D_{dR}(V)=K\cdot t^{-n}$ and the period isomorphism
\[ B_{dR}\otimes_{\bq_p}K\cdot t^{-n}\cong B_{dR}\otimes_{\bq_p}\Ind_{K}^{\bq_p}\bq_p(n)    \]
is already induced by the isomorphism
\[ \qbar_p\otimes_{\bq_p}K\cong \qbar_p\otimes_{\bq_p}\Ind_{K}^{\bq_p}\bq_p\cong \prod_{\Sigma}\qbar_p\]
sending $1\otimes x$ to $(\sigma(x))_{\sigma\in\Sigma}$ where $\Sigma=\Hom(K,\qbar_p)$. Denote by $D_K\in\bq_p$ the discriminant of $K/\bq_p$, well defined up to $\bz_p^\times$. For $$\omega=\omega_1^{-1}\otimes\omega_T$$ where $\omega_1$, resp. $\omega_T$, is a $\bz_p$-basis of $\Det_{\bz_p}\co_K$, resp. $\Det_{\bz_p}\Ind_{K}^{\bq_p}\bz_p$ we have in the notation of \cite{pr95}[Lemme C.2.8]
\[ \tilde{\xi}_V(\omega)= (\sqrt{D_K})^{-1}\in\qbar_p\]
and
\[ \eta_V(\omega):=\left\vert\frac{\sqrt{D_K}\cdot\tilde{\xi}_V(\omega)}{\epsilon(V,\psi_{o,K},\mu_{o,K})} \right\vert_p^{-1}=
\left\vert\frac{1}{\epsilon(V,\psi_{o,K},\mu_{o,K})} \right\vert_p^{-1}\]
where $\epsilon(V,\psi_{o,K},\mu_{o,K})$ is the $\epsilon$-factor associated by the theory of \cite{deligne73} to the representation of the Weil Group of $K$ on $D_{pst}(V)$ defined in \cite{pr95}[C.1.4]. Here the additive character $\psi_{o,K}$ and the Haar measure $\mu_{o,K}$ are defined in \cite{pr95}[C.2.7]. In particular $\mu_{o,K}(\co_K)=1$ and the $n(\psi_{o,K})$ of \cite{deligne73}[3.4] is the valuation of the different of $K/\bq_p$. By \cite{deligne73}[5.9] we have
\[ \eta_V(\omega)=|\epsilon(\bq_p(n),\psi_{o,K},\mu_{o,K})|_p=|D_K^{1-n}|_p\sim D_K^{n-1} \mod \bz_p^\times.\]
Furthermore, the $\Gamma$-factor of \cite{pr95}[C.2.9] is
\[ \prod_{j\in\bz}\Gamma^*(-j)^{-h_j(V)[K:\bq_p]}=(n-1)!^{-[K:\bq_p]}\]
and so Conjecture $C_{EP}(\bq_p(n))$ becomes
\begin{equation}\Det_{\bz_p}R\Gamma(K,\bz_p(n))=(1-q^{-n})\cdot(n-1)!^{-[K:\bq_p]}\cdot D_K^{n-1}\cdot\omega_1^{-1}\label{cepeq}\end{equation}
where $q$ is the cardinality of the residue field of $K$. This is an identity of invertible $\bz_p$-submodules of the invertible $\bq_p$-module $\Det_{\bq_p}R\Gamma(K,\bq_p(n))\cong\Det_{\bq_p}(K)^{-1}$ where this last isomorphism is given by the Bloch-Kato exponential map
\[ K\xrightarrow{\exp} H^1(K,\bq_p(n))\]
for $n\geq 2$ and the exact sequence
\[ 0\to K\xrightarrow{\exp} H^1(K,\bq_p(1))\cong \hat{K}^\times\otimes_{\bz_p}\bq_p\xrightarrow{\val}\bq_p\cong H^2(K,\bq_p(1))\to 0 \]
for $n=1$.
Coming back to the computation of $C(\X,n)$ we have by Prop. \ref{prop:derived-de-rham} below
\begin{equation}{\det}_\bz R\Gamma_{dR}(\mathcal{X}/\mathbb{Z})/F^n=|D_F|^{n-1}\cdot {\det}_\bz\mathcal{O}_F\subseteq {\det}_{\bq_p} F_{\bq_p}={\det}_{\bq_p} R\Gamma_{dR}(\mathcal{X}_{\bq_p}/\bq_p)/F^n\label{der1}\end{equation}
where $D_F$ is the discriminant of $F$. Moreover
$$R\Gamma'_{eh}(\mathcal{X}_{\mathbb{F}_p},\mathbb{Z}_p(n))\cong \prod_{v\mid p}R\Gamma(\kappa(v)_\et,\mathbb{Z}_p(n))=0$$
and
$$ R\Gamma(\mathcal{X}_{\mathbb{Z}_p,{et}},\mathbb{Z}_p(n))\cong \prod_{v\mid p}  R\Gamma(\co_{F_v,\et},\bz_p(n)).$$
By the localization triangle (\ref{etloc}) for $\X=\Spec(\co_{F_v})$, the fact that $\tau^{\leq n}Rj_*\mu_{p^\bullet}^{\otimes n}=Rj_*\mu_{p^\bullet}^{\otimes n}$ for $n\geq 1$ and the vanishing of $R\Gamma(\kappa(v)_\et,\mathbb{Z}_p(n-1))$ for $n\geq 2$ we have
\begin{equation} R\Gamma(\co_{F_v,\et},\bz_p(n))\cong R\Gamma(F_{v,\et},\bz_p(n))\label{rgfcompute}\end{equation}
for $n\geq 2$ and an exact triangle
\begin{equation} R\Gamma(\co_{F_v,\et},\bz_p(1))\to R\Gamma(F_{v,\et},\bz_p(1))\to R\Gamma(\kappa(v)_\et,\bz_p)[-1]\to\label{rgfcompute2}\end{equation}
for $n=1$. The exact triangle in Conjecture ${\bf D}_p(\mathcal{X},n)$ is the product over $v\mid p$ of the isomorphism
$$ F_v[-1]\xrightarrow{\sim}H^1_f(F_v,\bq_p(n))[-1]\cong R\Gamma_f(F_v,\bq_p(n))\cong R\Gamma(\co_{F_v,\et},\bq_p(n))$$
where the first map is the Bloch-Kato exponential.
For $v\mid p$ define $d_{v}(\mathcal{X},n)\in \mathbb{Q}_p^\times/\mathbb{Z}_p^\times$
such that
$$d_{v}(\mathcal{X},n)^{-1}\cdot \mathrm{det}_{\mathbb{Z}_p}  R\Gamma(\co_{F_v,\et},\bz_p(n))
=\mathrm{det}^{-1}_{\mathbb{Z}_p}R\Gamma_{dR}(\mathcal{X}_{\mathbb{Z}_p}/\mathbb{Z}_p)/F^n.$$
From (\ref{cepeq}), (\ref{der1}), (\ref{rgfcompute}) and (\ref{rgfcompute2}) we deduce
\begin{align*}d_{v}(\mathcal{X},n)=&(1-q_v^{-n})\cdot(n-1)!^{-[F_v:\bq_p]}\cdot D_{F_v}^{n-1}\cdot D_F^{1-n}\\=&(1-q_v^{-n})\cdot(n-1)!^{-[F_v:\bq_p]}\end{align*}
where $q_v:=|\kappa(v)|$.
With $d_{p}(\mathcal{X},n)$ and $c_{p}(\mathcal{X},n)$ defined in Definition \ref{ddef} we have
$$d_{p}(\mathcal{X},n)=\prod_{v\mid p}d_{v}(\mathcal{X},n),\quad c_{p}(\mathcal{X},n)=\prod_{v\mid p}q_v^n\cdot d_{v}(\mathcal{X},n)$$
and hence
\begin{align*}C(\mathcal{X},n):=&\Prod_{p<\infty}\mid c_p(\mathcal{X},n)\mid_p
=\Prod_{p<\infty}\prod_{v\mid p}|(n-1)!|_p^{[F_v:\bq_p]}
=(n-1)!^{-[F:\bq]}.\end{align*}
\end{proof}
For $n\geq 2$ we have $H_{W,c}^i({\mathcal{X}},\mathbb{Z}(n))=0$ for $i\neq 1,2,3$ and an exact sequence
\begin{multline*}0\to \prod_{v\mid\infty} H^0(F_v,(2\pi i)^n\bz)\rightarrow H_{W,c}^1({\mathcal{X}},\mathbb{Z}(n))\rightarrow H^{1,n}\xrightarrow{\alpha} (\mathbb{Z}/2\mathbb{Z})^{r_1\cdot\epsilon_{1,n}}\to\\ \to H_{W,c}^2({\mathcal{X}},\mathbb{Z}(n))\to H^{2,n} \to (\mathbb{Z}/2\mathbb{Z})^{r_1\cdot\epsilon_{2,n}}\to
H_{W,c}^3({\mathcal{X}},\mathbb{Z}(n))\to 0.\end{multline*}
The exact triangle (\ref{tangenttri2}) reduces to an isomorphism
\begin{equation} H_{W,c}^1({\mathcal{X}},\mathbb{Z}(n))_\br\cong F_\br\label{hwcr}.\end{equation}
The determinant of this isomorphism for the $\bz$-structures $H_{W,c}^1({\mathcal{X}},\mathbb{Z}(n))$ and
\begin{equation} \prod_{v\mid\infty} H^0(F_v,(2\pi i)^n\bz)\oplus H^0(F_v,(2\pi i)^{n-1}\bz)\label{zstruct}\end{equation}
on
\[F_\br\cong\prod_{v\mid\infty} F_v\cong \prod_{v\mid\infty} H^0(F_v,(2\pi i)^{n}\br)\oplus H^0(F_v,(2\pi i)^{n-1}\br)\]
equals $R_n/w_n\cdot|\im(\alpha)|$ and hence the determinant between $R\Gamma_{W,c}(\mathcal{X},\mathbb{Z}(n))$ and (\ref{zstruct}) equals
\[ R_n/w_n\cdot|\im(\alpha)|\cdot \frac{|H_{W,c}^2({\mathcal{X}},\mathbb{Z}(n))|}{|H_{W,c}^3({\mathcal{X}},\mathbb{Z}(n))|}
=\frac{R_n\cdot h_n\cdot 2^{r_1\cdot\delta_{1,n}}}{w_n\cdot 2^{r_1\cdot\delta_{2,n}}}\]
noting that $\delta_{i,n}=\epsilon_{i,n}$ for $i=1,2$ and $n\geq 1$. The determinant between (\ref{zstruct}) and $\co_F$ is
\[(2\pi i)^{r_2(2n-1)+r_1(n-\delta_{1,n})}/\sqrt{D_F}=\pm(2\pi )^{[F:\bq]\cdot n-r_2-r_1\cdot\delta_{1,n}}/\sqrt{|D_F|} \]
using that $\sqrt{D_F}=\pm i^{r_2}\sqrt{|D_F|}$.
Hence, also using Prop. \ref{prop:derived-de-rham} below, the isomorphism
\begin{equation}\lambda_\infty:\br\cong \Delta(\mathcal{X}/\mathbb{Z},n)_\br\cong {\det}_\br R\Gamma_{W,c}(\mathcal{X},\mathbb{Z}(n))_\br\otimes_\br {\det}_\br (R\Gamma_{dR}(\mathcal{X}/\bz)/\mathrm{Fil}^n)_\br \notag\end{equation}
induced by (\ref{tangenttri2}) satisfies
$$\lambda_{\infty}\left(|D_F|^{n-1}\cdot\frac{w_n\sqrt{|D_F|}}{2^{r_1\cdot(\delta_{1,n}-\delta_{2,n})}(2\pi)^{[F:\bq]\cdot n-r_2-r_1\cdot\delta_{1,n}}h_nR_n}\cdot\mathbb{Z}\right)= \Delta(\mathcal{X}/\mathbb{Z},n).$$
Finally, using Prop. \ref{Ccompute}, the identity of Conjecture \ref{conjmain}
$$\lambda_{\infty}(\zeta^*(\mathcal{X},n)^{-1}\cdot C(\mathcal{X},n)\cdot\mathbb{Z})= \Delta(\mathcal{X}/\mathbb{Z},n)$$
is equivalent to
\begin{equation}\zeta_F^*(n)=\zeta_F(n)=|D_F|^{1-n}\cdot(n-1)!^{-[F:\bq]}\cdot \frac{2^{r_1\cdot(\delta_{1,n}-\delta_{2,n})}(2\pi)^{[F:\bq]\cdot n-r_2-r_1\cdot\delta_{1,n}}h_nR_n}{w_n\sqrt{|D_F|}},\label{cnfn}\end{equation}
at least if we also assume conjecture $C_{EP}(\bq_p(n))$ for all local fields $F_v$. For $n=1$  the triangle (\ref{tangenttri2}) gives an exact sequence
\[ 0\to H_{W,c}^1({\mathcal{X}},\mathbb{Z}(1))_\br\to F_\br\xrightarrow{\mathrm{Tr}_{F_\br/\br}}\br\to 0\to 0\to \br\xrightarrow{\sim} H_{W,c}^3({\mathcal{X}},\mathbb{Z}(1))_\br\to 0\]
instead of the isomorphism (\ref{hwcr}) but otherwise the computation is the same, showing that for any number field $F$ Conjecture \ref{conjmain} is equivalent to (\ref{cnfn}) for $n=1$, i.e. to the analytic class number formula. One easily proves that (\ref{cnfn}) is equivalent to (\ref{nless0}) by verifying identity (\ref{toward}) in Corollary \ref{cor:fe} i).

\begin{prop} Assume $\X=\Spec(\co_F)$. Then Conjecture \ref{conjmain} holds for $n=0,1$ if $F$ is arbitrary and for any $n\in\bz$ if $F/\bq$ is abelian.
\end{prop}

\begin{proof} By Theorem \ref{equiv-main-tnc} this reduces to known cases of the Tamagawa number conjecture
 \cite{bugr00,huki00,flach03,flach06-3}.
\end{proof}

We conclude this section with the computation of derived de Rham cohomology of $\X=\Spec(\co_F)$. Let $\mathcal{D}_{F}\subseteq \co_F$ be the different ideal so that $$\vert D_F\vert=\vert N(\mathcal{D}_{F})\vert=[\co_F:\mathcal{D}_F]$$ is the absolute discriminant of $F$. Recall from \cite{serre95}[III. Prop. 14] that the module of K\"ahler differentials $\Omega^1_{\mathcal{O}_F/\mathbb{Z}}$ is a cyclic $\mathcal{O}_F$-module with annihilator $\mathcal{D}_{F}$.  A generator $\omega$ of $\Omega^1_{\mathcal{O}_F/\mathbb{Z}}$ therefore gives an exact sequence of $\mathcal{O}_F$-modules
\begin{equation}0\longrightarrow \mathcal{D}_{F}\longrightarrow\mathcal{O}_{F}\stackrel{\omega}{\longrightarrow}\Omega^1_{\mathcal{O}_F/\mathbb{Z}}\longrightarrow 0\label{kaehler}\end{equation}
and we have $|\Omega^1_{\mathcal{O}_F/\mathbb{Z}}|=|D_F|$. Recall that $L\Omega_{\mathcal{O}_F/\mathbb{Z}}/F^n:=\mathrm{Tot}(\Omega^{<n}_{P_{\bullet}/\mathbb{Z}})$ where $P_{\bullet}\rightarrow \mathcal{O}_F$ is an augmented simplicial $\mathbb{Z}$-algebra which is a free resolution of $\mathcal{O}_F$ and that
$$R\Gamma_{dR}(\X/\bz)/F^n:=R\Gamma(\X_{Zar}, L\Omega_{\mathcal{O}_F/\mathbb{Z}}/F^n)=\Gamma(\X_{Zar}, L\Omega_{\mathcal{O}_F/\mathbb{Z}}/F^n)$$ is the derived de Rham cohomology modulo the $n$-th step of the Hodge filtration introduced in section \ref{sec:derived-de-rham}. Here the last isomorphism holds since $\X$ is affine.
We denote by $L\Omega^{\hat{}}_{\mathcal{O}_F/\mathbb{Z}}:="\mathrm{lim}"L\Omega_{\mathcal{O}_F/\mathbb{Z}}/F^n$ the Hodge-completed derived de Rham complex. Unless stated otherwise, a complex of the form $[M\rightarrow N]$ denotes a cochain complex placed in degrees $[0,1]$.

\begin{prop}\label{prop:derived-de-rham}
We have $\mathrm{gr}^0(L\Omega^{\hat{}}_{\mathcal{O}_F/\mathbb{Z}})\simeq \mathcal{O}_F[0]$. For any $i\geq 1$, a generator of the cyclic $\mathcal{O}_F$-module $\Omega^1_{\mathcal{O}_F/\mathbb{Z}}$ gives a quasi-isomorphism of complexes of $\mathcal{O}_F$-modules
\begin{equation}\mathrm{gr}^i(L\Omega^{\hat{}}_{\mathcal{O}_F/\mathbb{Z}})\simeq \Gamma^{i-1}_{\mathcal{O}_F}\mathcal{D}_{F}\otimes_{\mathcal{O}_F} \Omega^1_{\mathcal{O}_F/\mathbb{Z}}[-1]\label{griso}\end{equation}
where $\Gamma^{i-1}_{\mathcal{O}_F}\mathcal{D}_{F}$ is an invertible $\mathcal{O}_F$-module.
In particular, for any $n\geq 1$, we have an exact sequence of complexes
$$0\rightarrow F^1/F^n\rightarrow R\Gamma_{dR}(\X/\bz)/F^n\rightarrow \mathcal{O}_F[0]\rightarrow 0$$
where $F^1/F^n$ is concentrated in degree $1$, and $H^1(F^1/F^n)$ is finite with order
$$\vert H^1(F^1/F^n)\vert =|D_{F}|^{n-1}.$$
\end{prop}

\begin{proof}
The first assertion is obvious, since we have
$$\mathrm{gr}^0(L\Omega^{\hat{}}_{\mathcal{O}_F/\mathbb{Z}})\simeq[\cdots \rightarrow P_1\rightarrow P_0]\stackrel{\sim}{\rightarrow}\mathcal{O}_F[0].$$
Next we claim that the canonical morphism $$L_{\mathcal{O}_F/\mathbb{Z}}\rightarrow H^0(L_{\mathcal{O}_F/\mathbb{Z}})[0]\simeq \Omega^1_{\mathcal{O}_F/\mathbb{Z}}[0]$$ is a quasi-isomorphism. The map $\mathrm{Spec}(\mathcal{O}_F)\rightarrow \mathrm{Spec}(\mathbb{Z})$ is a local complete intersection, hence we have an isomorphism
$L_{\mathcal{O}_F/\mathbb{Z}}\simeq [M\rightarrow N][1]$
in the derived category of $\mathcal{O}_F$-modules,
where $M$ and $N$ are finitely generated and locally free $\mathcal{O}_F$-modules \cite{Illusie71}[Prop. 3.2.6]. In particular $H^{-1}(L_{\mathcal{O}_F/\mathbb{Z}})$ is a torsion-free $\mathcal{O}_F$-module. It is therefore enough to show that $H^{-1}(L_{\mathcal{O}_F/\mathbb{Z}})\otimes_{\mathcal{O}_F}F=H^{-1}(L_{\mathcal{O}_F/\mathbb{Z}}\otimes_{\mathcal{O}_F}F)=0$,
which in turn follows from
$$L_{\mathcal{O}_F/\mathbb{Z}}\otimes_{\mathcal{O}_F} F\simeq L_{F/\mathbb{Q}}\simeq 0.$$
Recall from \cite{Illusie72}[Ch. VIII (2.1.1.5)] that $\mathrm{gr}^i(L\Omega^{\hat{}}_{\mathcal{O}_F/\mathbb{Z}})\simeq L\Lambda^i(L_{\mathcal{O}_F/\mathbb{Z}})[-i]$
and assume $i\geq 1$. We obtain using (\ref{kaehler}), \cite{Illusie72}[Ch. VIII, Lemme 2.1.2.1] and \cite{Illusie71}[Ch. I, Prop. 4.3.2.1 (ii)]
\begin{eqnarray*}
L\Lambda^i(L_{\mathcal{O}_F/\mathbb{Z}})&\simeq & L\Lambda^i(\Omega^1_{\mathcal{O}_F/\mathbb{Z}})\\
&\simeq & L\Lambda^i\left([\mathcal{D}_{F}\rightarrow\mathcal{O}_{F}][1]\right)\\
&\simeq &[\Gamma^i \mathcal{D}_{F}\rightarrow \Gamma^{i-1} \mathcal{D}_{F}\otimes \Lambda^1\mathcal{O}_F\rightarrow\cdots \rightarrow \Lambda^i\mathcal{O}_F][i]\\
&\simeq &[\Gamma^i \mathcal{D}_{F}\rightarrow \Gamma^{i-1} \mathcal{D}_{F}\otimes \mathcal{O}_F][i]
\end{eqnarray*}
where all tensor products and exterior and divided power algebras are taken over $\co_F$. The map $\kappa:\Gamma^i \mathcal{D}_{F}\rightarrow \Gamma^{i-1} \mathcal{D}_{F}\otimes \mathcal{O}_F$ corresponds to the canonical "application $i$-ique" $\mathcal{D}_{F}\rightarrow \Gamma^{i-1} \mathcal{D}_{F}\otimes \mathcal{O}_F$ sending $x$ to $\gamma_{i-1}(x)\otimes x$. Thus $\kappa$ is given by
$$\fonc{\kappa}{\Gamma^i\mathcal{D}_{F}}{\Gamma^{i-1} \mathcal{D}_{F}\otimes \mathcal{O}_F}{\gamma_{i}(x)}{\gamma_{i-1}(x)\otimes x}.
$$
Consider the sequence of $\mathcal{O}_F$-modules
\begin{equation}\label{exact-sequence}
0\longrightarrow \Gamma^i\mathcal{D}_{F}\stackrel{\kappa}{\longrightarrow} \Gamma^{i-1}\mathcal{D}_{F}\otimes \mathcal{O}_{F}\stackrel{1\otimes \omega}{\longrightarrow} \Gamma^{i-1}\mathcal{D}_{F}\otimes\Omega^1_{\mathcal{O}_F/\mathbb{Z}}\longrightarrow 0.
\end{equation}
To see that (\ref{exact-sequence}) is exact, one may localize for the Zariski topology since the functor $\Gamma^i$ is compatible with extension of scalars, i.e. $\Gamma^i_A(M)\otimes_AB\simeq \Gamma^i_B(M\otimes_A B)$ (see \cite{Roby63} Th\'eor\`eme III.3). Hence one may assume that $\mathcal{D}_{F}$ is a free module of rank one with generator $x$, in which case $\Gamma^i\mathcal{D}_{F}$ is also free of rank one and generated by $\gamma_i(x)$, and similarly for $\Gamma^{i-1}\mathcal{D}_{F}$. The exactness of the sequence (\ref{exact-sequence}) follows. We obtain
$$L\Lambda^i(L_{\mathcal{O}_F/\mathbb{Z}})\simeq [\Gamma^i \mathcal{D}_{F}\rightarrow \Gamma^{i-1} \mathcal{D}_{F}\otimes \mathcal{O}_F][i]\simeq \Gamma^{i-1}\mathcal{D}_{F}\otimes\Omega^1_{\mathcal{O}_F/\mathbb{Z}}[i-1]$$
hence
$$
\mathrm{gr}^i(L\Omega^{\hat{}}_{\mathcal{O}_F/\mathbb{Z}})\simeq L\Lambda^i(L_{\mathcal{O}_F/\mathbb{Z}})[-i]\simeq \Gamma^{i-1}_{\mathcal{O}_F}\mathcal{D}_{F}\otimes_{\mathcal{O}_F} \Omega^1_{\mathcal{O}_F/\mathbb{Z}}[-1].
$$
As observed above, $\Gamma^{i-1}_{\mathcal{O}_F}\mathcal{D}_{F}$ is an invertible $\mathcal{O}_F$-module, hence $\Gamma^{i-1}_{\mathcal{O}_F}\mathcal{D}_{F}\otimes_{\mathcal{O}_F} \Omega^1_{\mathcal{O}_F/\mathbb{Z}}$ is a torsion $\mathbb{Z}$-module whose order is given by
$$\vert \Gamma^{i-1}_{\mathcal{O}_F}\mathcal{D}_{F}\otimes_{\mathcal{O}_F} \Omega^1_{\mathcal{O}_F/\mathbb{Z}} \vert=\vert \Omega^1_{\mathcal{O}_F/\mathbb{Z}}\vert=\vert D_{F} \vert.$$
In order to show the last assertion, it remains to prove that $F^1/F^n$ is concentrated in degree $1$ and that $H^1(F^1/F^n)$ is finite with order
$\vert H^1(F^1/F^n)\vert = |D_{F}|^{n-1}$, where $F^*$ denotes the Hodge filtration. This is obvious for $n=1$. We conclude by induction on $n$ using the exact triangle
\[ F^n/F^{n+1}\to F^1/F^{n+1}\to F^1/F^n\to\]
and the isomorphism (\ref{griso})
\[ F^n/F^{n+1}\cong \Gamma(\X,\mathrm{gr}^n(L\Omega^{\hat{}}_{\mathcal{O}_F/\mathbb{Z}}))\cong  \Gamma^{n-1}_{\mathcal{O}_F}\mathcal{D}_{F}\otimes_{\mathcal{O}_F} \Omega^1_{\mathcal{O}_F/\mathbb{Z}}[-1].\]
\end{proof}

\section{Appendix A: Artin-Verdier duality}\label{sectAVD}
In this section $\mathcal{X}$ denotes a regular connected scheme, which is proper over $\mathbb{Z}$. We assume that $\mathcal{X}$ is of pure (absolute) dimension $d$. Unless specified otherwise, a scheme is always endowed with the \'etale topology.

\subsection{Introduction}
Artin-Verdier duality for the cycle complex over $\mathcal{X}$ is known in certain  cases by \cite{Geisser04a}, \cite{Geisser10} and \cite{Sato07}. In order to deal with $2$-torsion, these duality theorems relate \'etale cohomology of $\mathcal{X}$ with \'etale cohomology with compact support of $\mathcal{X}$ in the sense of Milne \cite{Milne-duality}. The aim of this appendix is to define complexes $\mathbb{Z}(n)^{\overline{\mathcal{X}}}$ over the Artin-Verdier \'etale topos $\overline{\mathcal{X}}_{et}$ and to show duality over $\overline{\mathcal{X}}_{et}$. In particular, we obtain Conjecture ${\bf AV}(\overline{\mathcal{X}}_{et},n)$ for smooth proper schemes over number rings and any $n\in\mathbb{Z}$  (see Corollary \ref{cor-AVsmooth}), as well as for arbitrary regular proper arithmetic schemes if $n\geq d$ or $n\leq 0$ (see Corollary \ref{corAVn=0}). These results are essentially due to Geisser and Sato (see \cite{Geisser04a}, \cite{Geisser10} and \cite{Sato07}); we only
  treat the $2$-torsion in order to restore duality over $\overline{\mathcal{X}}_{et}$. This appendix may be seen as a generalization of \cite{Bienenfeld87} to higher dimensional arithmetic schemes.

We now explain our definition for $\mathbb{Z}(n)^{\overline{\mathcal{X}}}$ and describe the contents of this appendix. The topos $\overline{\mathcal{X}}_{et}$ is defined so that there is an open-closed decomposition (see Section \ref{sect-AVTopos})
$$\phi:\mathcal{X}_{et}\longrightarrow \overline{\mathcal{X}}_{et}\longleftarrow \mathcal{X}_{\infty}:u_{\infty}$$
where $\phi:\mathcal{X}_{et}\rightarrow \overline{\mathcal{X}}_{et}$ is the open embedding and $\mathcal{X}_{\infty}$ its closed complement. The right definition for $\mathbb{Z}(n)^{\overline{\mathcal{X}}}$ in the range $0\leq n\leq d$ is
\begin{equation}\label{rightdef}
\mathbb{Z}(n)^{\overline{\mathcal{X}}}:=\tau^{\leq n}(R\phi_*\mathbb{Z}(n)^{\mathcal{X}}).
\end{equation}
Since we need $\phi^*\mathbb{Z}(n)^{\overline{\mathcal{X}}}\simeq\mathbb{Z}(n)^{\mathcal{X}}$, this definition requires $\mathcal{H}^i(\mathbb{Z}(n)^{\mathcal{X}})=0$ for any $i>n$, which is not known for general regular $\mathcal{X}$. However (\ref{rightdef}) gives
$$Ru_{\infty}^!\mathbb{Z}(n)^{\overline{\mathcal{X}}}= \left( \tau^{>n} u_{\infty,*}R\pi_*(2i\pi)^n\mathbb{Z}\right)[-1]\simeq  \left(\tau^{>n} u_{\infty,*}R\widehat{\pi}_*(2i\pi)^n\mathbb{Z}\right)[-1].$$
where $R\widehat{\pi}_*(2i\pi)^n\mathbb{Z}$ is the complex of $2$-torsion sheaves on $\mathcal{X}_{\infty}$ defined in Section \ref{sectTate}. Therefore, we define $\mathbb{Z}(n)^{\overline{\mathcal{X}}}$ such that there is an exact triangle
$$\mathbb{Z}(n)^{\overline{\mathcal{X}}}\longrightarrow R\phi_*\mathbb{Z}(n)\longrightarrow \tau^{>n} \left(u_{\infty,*}R\widehat{\pi}_*(2i\pi)^n\mathbb{Z}\right)$$
for any $n\in\mathbb{Z}$ (see Corollary \ref{cor-u^!}). We then show that the resulting complexes $\mathbb{Z}(n)^{\overline{\mathcal{X}}}$ for any $n\in\mathbb{Z}$ satisfy the expected Artin-Verdier duality. This fact relies on Poincar\'e duality for the cohomology of the possibly non-orientable manifold $\mathcal{X}(\mathbb{R})$ with $\mathbb{Z}/2\mathbb{Z}$-coefficients. Moreover, our definition of $\mathbb{Z}(n)^{\overline{\mathcal{X}}}$  coincides with (\ref{rightdef}) whenever $n\geq0$ and $\mathcal{H}^i(\mathbb{Z}(n)^{\mathcal{X}})=0$ for any $i>n$ (Proposition \ref{prop-trunc}). In particular, we have $\mathbb{Z}(0)^{\overline{\mathcal{X}}}\simeq \mathbb{Z}$ and
$\mathbb{Z}(1)^{\overline{\mathcal{X}}}\simeq \phi_*\mathbb{G}_m[-1]$, where $\mathbb{Z}$ denotes the constant sheaf  on $\overline{\mathcal{X}}_{et}$, and $\mathbb{G}_m$ denotes the multiplicative group on $\mathcal{X}_{et}$ (Proposition \ref{prop-description-atn=0,1}). We notice that, even though the complex $R\Gamma(\overline{\mathcal{X}}_{et},\mathbb{Z}(n))$ has bounded cohomology, it may have non-trivial cohomology in negative degrees for $n<0$. We observe in Proposition \ref{prop-pbf} that this surprising fact is forced by the projective bundle formula.

\subsection{The motivic complex $\mathbb{Z}(n)^{\mathcal{X}}$}
For any $n\geq0$, we consider Bloch's cycle complex $\mathbb{Z}(n)=z^n(-,2n-*)$ as a complex of sheaves
on the small \'etale site $\mathcal{X}_{et}$ of the scheme $\mathcal{X}$. We have $\mathbb{Z}(0)\simeq\mathbb{Z}$ and $\mathbb{Z}(1)\simeq\mathbb{G}_m[-1]$ (see \cite{Levine01}, \cite{Levine99}).
We write
$\mathbb{Z}/m\mathbb{Z}(n):=\mathbb{Z}(n)\otimes^L\mathbb{Z}/m\mathbb{Z}$
and
$\mathbb{Q}/\mathbb{Z}(n):=\underrightarrow{ \mathrm{lim}}\,\mathbb{Z}/m\mathbb{Z}(n)$.
We have an exact triangle
$$\mathbb{Z}(n)\rightarrow \mathbb{Q}(n)\rightarrow \mathbb{Q}/\mathbb{Z}(n).$$
The Beilinson-Soul\'e vanishing conjecture states that $\mathbb{Z}(n)$ is acyclic in negative degrees.  In order to unconditionally define hypercohomology (or higher direct images) with coefficients in $\mathbb{Z}(n)$, we use $K$-injective resolutions (see \cite{Spaltenstein88} and \cite{Serpe03}). By \cite{Geisser04a}, if $\mathcal{X}$ is smooth over a number ring then $\mathbb{Z}/p^r\mathbb{Z}(n)$ is isomorphic, in the derived category, to Sato's complex \cite{Sato07}. Finally, recall from Section \ref{sect-emc} that for $n<0$ we define
$$\mathbb{Z}(n):=\bigoplus_{p}j_{p,!}(\mu_{p^{\infty}}^{\otimes n})[-1].$$

\subsection{The Artin-Verdier \'etale topos $\overline{\mathcal{X}}_{et}$}\label{sect-AVTopos} For a scheme $\mathcal{X}$ as above, we consider $\mathcal{X}(\mathbb{C})$ as a topological space with respect to the complex topology. The space $\mathcal{X}(\mathbb{C})$ carries an action of $G_{\mathbb{R}}:=\mathrm{Gal}(\mathbb{C}/\mathbb{R})$, and we consider the quotient topological space $\mathcal{X}_{\infty}:=\mathcal{X}(\mathbb{C})/G_{\mathbb{R}}$. Consider the canonical morphisms of topoi
$$\alpha:Sh(G_{\mathbb{R}},\mathcal{X}(\mathbb{C}))\longrightarrow \mathcal{X}_{et}$$
and
$$\pi:Sh(G_{\mathbb{R}},\mathcal{X}(\mathbb{C}))\longrightarrow Sh(\mathcal{X}_{\infty}).$$
Here $Sh(G_{\mathbb{R}},\mathcal{X}(\mathbb{C}))$ and $Sh(\mathcal{X}_{\infty})$ denote the topos of $G_{\mathbb{R}}$-equivariant sheaves (of sets) on $\mathcal{X}(\mathbb{C})$ and the topos of sheaves on $\mathcal{X}_{\infty}$ respectively. Recall that there is a canonical equivalence between $Sh(G_{\mathbb{R}},\mathcal{X}(\mathbb{C}))$ and the category of $G_{\mathbb{R}}$-equivariant \'etal\'e spaces over $\mathcal{X}(\mathbb{C})$, i.e. the category of topological spaces $E$ given with a $G_{\mathbb{R}}$-action and a $G_{\mathbb{R}}$-equivariant local homeomorphism $E\rightarrow \mathcal{X}(\mathbb{C})$. The morphisms $\pi$ and $\alpha$ are defined as follows. In order to construct $\alpha$, we need to define a functor
$\alpha^*$ from the \'etale site of $\mathcal{X}$ to the category $Sh(G_{\mathbb{R}},\mathcal{X}(\mathbb{C}))$ such that $\alpha^*$ is both continuous and left exact (i.e. such that $\alpha^*$ preserves the final object, fiber products and covering families). This functor simply takes an \'etale $\mathcal{X}$-scheme $\mathcal{U}\rightarrow\mathcal{X}$ to the $G_{\mathbb{R}}$-equivariant \'etal\'e space $\mathcal{U}(\mathbb{C})\rightarrow \mathcal{X}(\mathbb{C})$. We shall also consider the topos $Sh(G_{\mathbb{R}},\mathcal{X}_{\infty})$ of $G_{\mathbb{R}}$-equivariant sheaves on $\mathcal{X}_{\infty}$, where $G_{\mathbb{R}}$ acts trivially on $\mathcal{X}_{\infty}$. Notice that an abelian sheaf on $Sh(G_{\mathbb{R}},\mathcal{X}_{\infty})$ is simply a sheaf of $\mathbb{Z}[G_{\mathbb{R}}]$-modules on $\mathcal{X}_{\infty}$. In order to define $\pi$, we consider the map $p:Sh(G_{\mathbb{R}},\mathcal{X}(\mathbb{C}))\rightarrow Sh(G_{\mathbb{R}},\mathcal{X}_{\infty})$ induced by t
 he $G_{\mathbb{R}}$-equivariant continuous map $\mathcal{X}(\mathbb{C})\rightarrow\mathcal{X}_{\infty}$. Given an \'etal\'e space $E\rightarrow \mathcal{X}_{\infty}$, we define $$\pi^*(E\rightarrow \mathcal{X}_{\infty}):=E\times_{\mathcal{X}_{\infty}}\mathcal{X}(\mathbb{C})\rightarrow \mathcal{X}(\mathbb{C})$$ where $G_{\mathbb{R}}$ acts on the second factor. Given a $G_{\mathbb{R}}$-sheaf $F$ on $\mathcal{X}(\mathbb{C})$ and an open $U\subset \mathcal{X}_{\infty}$, we have a canonical $G_{\mathbb{R}}$-action on the set $p_*F(U)$, and we set $\pi_*F(U):=(p_*F(U))^{G_{\mathbb{R}}}$. The following result is well-known.
\begin{lem}\label{lem00}
Let $\mathcal{A}$ be an abelian object of $Sh(G_\mathbb{R},\mathcal{X}(\mathbb{C}))$. For any point $x\in\mathcal{X}_{\infty}$ and $y\in\mathcal{X}(\mathbb{C})$ lying over $x$, we have
\begin{equation}\label{Rpi}
(R\pi_*\mathcal{A})_x\simeq R\Gamma(G_y,\mathcal{A}_y)
\end{equation}
where $G_y\subseteq G_\mathbb{R}$ is the stabilizer of $y$.
\end{lem}
In particular the sheaf $R^i\pi_*\mathcal{A}$ is concentrated on the closed subset $\mathcal{X}(\mathbb{R})\subset \mathcal{X}_{\infty}$ as long as $i>0$. The topos $\overline{\mathcal{X}}_{et}$ is the category of sheaves of sets on the Artin-Verdier \'etale site of $\overline{\mathcal{X}}$ (see \cite{Flach-Morin-12}). There is an open-closed decomposition
$$\phi:\mathcal{X}_{et}\longrightarrow \overline{\mathcal{X}}_{et}\longleftarrow \mathcal{X}_{\infty}:u_{\infty}$$
whose gluing functor can be described as follows:
\begin{equation}\label{glueing}
u_{\infty}^*\phi_*\simeq \pi_*\alpha^*: \mathcal{X}_{et} \longrightarrow Sh(G_{\mathbb{R}},\mathcal{X}(\mathbb{C}))
\longrightarrow Sh(\mathcal{X}_{\infty})
\end{equation}
Those  two properties characterize  $\overline{\mathcal{X}}_{et}$ up to equivalence:
$\overline{\mathcal{X}}_{et}$ is canonically equivalent to the category of triples $(F_{\mathcal{X}},F_{\infty},f)$ where $F_{\mathcal{X}}$ is an object of $\mathcal{X}_{et}$ (i.e. an \'etale sheaf on $\mathcal{X}$), $F_{\infty}$ an object of $Sh(\mathcal{X}_{\infty})$ and
$f:F_{\infty}\rightarrow \pi_*\alpha^*F_{\mathcal{X}}$ is a map in $Sh(\mathcal{X}_{\infty})$. This description of the topos $\overline{\mathcal{X}}_{et}$ gives as usual two triples of adjoint functors $(\phi_!,\phi^*,\phi_*)$ and $(u_{\infty}^*,u_{\infty,*},u_{\infty}^!)$ between the corresponding categories of abelian sheaves, which satisfy the classical formalism. In particular, we have
$$u^*_{\infty}u_{\infty,*}\stackrel{\sim}{\rightarrow}\mathrm{Id},\hspace{0.5cm} \phi^* \phi_*\stackrel{\sim}{\rightarrow}\mathrm{Id},\hspace{0.5cm} \phi^* u_{\infty,*}= 0.$$
Moreover, for any complex of abelian sheaves $\mathcal{A}$ on $\overline{\mathcal{X}}_{et}$, we have an exact sequence
$$0\rightarrow \phi_! \phi^*\mathcal{A}\rightarrow \mathcal{A}\rightarrow u_{\infty,*}u^*_{\infty}\mathcal{A} \rightarrow 0$$
and an exact triangle
$$u_{\infty,*}Ru_{\infty}^!\mathcal{A}\rightarrow \mathcal{A} \rightarrow R\phi_* \phi^*\mathcal{A}$$
where the maps are given by adjunction.
The following derived version of (\ref{glueing}) will be useful. We denote by $\mathcal{D}(\mathcal{X}_{et})$, $\mathcal{D}(\overline{\mathcal{X}}_{et})$ and $\mathcal{D}(\mathcal{X}_{\infty})$ the derived categories of the abelian categories of abelian sheaves on $\mathcal{X}_{et}$, $\overline{\mathcal{X}}_{et} $ and $\mathcal{X}_{\infty}$ respectively. Moreover we denote by  $\mathcal{D}^+(\mathcal{X}_{et})$, $\mathcal{D}^+(\overline{\mathcal{X}}_{et})$ and $\mathcal{D}^+(\mathcal{X}_{\infty})$ the corresponding derived categories of bounded below complexes.
\begin{lem}\label{lem-here0} The functor $\alpha^*$ sends injective objects to $\pi_*$-acyclic objects, hence
the natural transformation of functors from $\mathcal{D}^+(\mathcal{X}_{et})$ to $\mathcal{D}(\mathcal{X}_{\infty})$
$$u_{\infty}^*\circ R\phi_*\stackrel{\sim}{\longrightarrow} R\pi_*\circ\alpha^*$$
is an isomorphism.
\end{lem}
\begin{proof}
In view of the canonical isomorphisms $$u_{\infty}^*R\phi_*\simeq R(u_{\infty}^*\phi_*)\simeq R(\pi_*\alpha^*),$$
one is reduced to showing that $\alpha^*$ sends injective objects to $\pi_*$-acyclic objects. Indeed, if this is the case then the spectral sequence for the composite functor $\pi_*\circ\alpha^*$ together with the exactness of $\alpha^*$ yield
$$R(\pi_*\alpha^*)\simeq R\pi_*R\alpha^*\simeq R(\pi_*)\alpha^*.$$
Let $I$ be an injective abelian sheaf on
$\mathcal{X}_{et}$. By Lemma \ref{lem00}, we have $(R^i\pi_*(\alpha^*I))_{y}=0$ for any $i\geq1$ and any $y\in\mathcal{X}_\infty-\mathcal{X}(\mathbb{R})$. Let $y\in \mathcal{X}(\mathbb{R})\subset \mathcal{X}_\infty$. The point $y$ is a morphism $y: \mathrm{Spec}(\mathbb{R})\rightarrow\mathcal{X}$ and we denote by
$$x: \mathrm{Spec}(\mathbb{C})\rightarrow \mathrm{Spec}(\mathbb{R})\rightarrow\mathcal{X}$$
the (unique) point $x\in\mathcal{X}(\mathbb{C})$ such that $\pi(x)=y$. Then we have
$$(R^i\pi_*(\alpha^*I))_{y}=H^i(G_\mathbb{R},(\alpha^*I)_x)=H^i(G_\mathbb{R},x^*I)=H^i( \mathrm{Spec}(\mathbb{R})_{et},y^*I)$$
for any $i\geq 1$. Moreover, we have
\begin{equation}\label{lalala}
H^i( \mathrm{Spec}(\mathbb{R})_{et},y^*I)=H^i(\underleftarrow{ \mathrm{lim}}\,\mathcal{U},I_{\mid \underleftarrow{ \mathrm{lim}}\mathcal{U}})=\underrightarrow{ \mathrm{lim}}H^i(\mathcal{U},I_{\mid \mathcal{U}})
\end{equation}
where $\mathcal{U}$ runs over the filtered set of pointed \'etale neighborhoods of $(\mathcal{X},y)$, i.e. the family of pairs $(\mathcal{U}\rightarrow\mathcal{X}, \mathrm{Spec}(\mathbb{R})\rightarrow\mathcal{U})$ such that $y$ coincides with the composition $ \mathrm{Spec}(\mathbb{R})\rightarrow\mathcal{U}\rightarrow\mathcal{X}$. Notice that (\ref{lalala}) holds since \'etale cohomology commutes with filtered projective limits of schemes and because $\underleftarrow{ \mathrm{lim}}\,\mathcal{U}$ is an henselian local ring with residue field of Galois group $G_{\mathbb{R}}$ (however the residue field of $\underleftarrow{ \mathrm{lim}}\,\mathcal{U}$ is not $\mathbb{R}$ but rather an henselian real field algebraic over $\mathbb{Q}$). We obtain
$$H^i( \mathrm{Spec}(\mathbb{R})_{et},y^{*}I)=\underrightarrow{ \mathrm{lim}}H^i(\mathcal{U},I_{\mid \mathcal{U}})=0$$
since $I_{\mid \mathcal{U}}$ is injective on $\mathcal{U}$, hence
$$R^i\pi_*(\alpha^*I)=0\mbox{ for any }i\geq1.$$
The result follows.

\end{proof}

\subsection{Tate cohomology and the functor $R{\widehat{\pi}}_*$  }\label{sectTate}

We choose a resolution $P_{\geq0}\rightarrow \mathbb{Z}$ of the $\mathbb{Z}[G_{\mathbb{R}}]$-module $\mathbb{Z}$ by finitely generated free $\mathbb{Z}[G_{\mathbb{R}}]$-modules, and we extend it into a complete resolution $P_*$. We have morphisms of complexes of $\mathbb{Z}[G_{\mathbb{R}}]$-modules:
$$P_*\rightarrow P_{\geq 0}\rightarrow \mathbb{Z}.$$
If $A$ is a  bounded below complex of $G_{\mathbb{R}}$-modules, Tate hypercohomology is defined as
$$R\widehat{\Gamma}(G_{\mathbb{R}},A):=\int \mathrm{Hom}_{G_{\mathbb{R}}-\mathrm{Mod}}(P_*,A)$$
where $\int$ denotes the totalization of the double complex $\mathrm{Hom}$ with respect to the \emph{direct sum} on diagonals.
The spectral sequence
\begin{equation*}
\widehat{H}^i(G_{\mathbb{R}},H^j(A))\Longrightarrow \widehat{H}^{i+j}(G_{\mathbb{R}},A)
\end{equation*}
converges. It follows that $R\widehat{\Gamma}(G_{\mathbb{R}},-)$ preserves quasi-isomorphisms. We keep the notations of the previous section and we consider the $G_{\mathbb{R}}$-action on the topological space $\mathcal{X}(\mathbb{C})$. We define Tate equivariant cohomology as follows:
$$R\widehat{\Gamma}(G_{\mathbb{R}},\mathcal{X}(\mathbb{C}),\mathcal{A}):=R\widehat{\Gamma}(G_{\mathbb{R}},R\Gamma(\mathcal{X}(\mathbb{C}),\mathcal{A})).$$
where $\mathcal{A}$ is a bounded below  complex of $G_{\mathbb{R}}$-equivariant abelian sheaves on $\mathcal{X}(\mathbb{C})$. We have
\begin{eqnarray*}
R\pi_*\mathcal{A}&\simeq & \mathrm{Ou}\int \underline{\mathrm{Hom}}_{Sh(G_{\mathbb{R}},\mathcal{X}_{\infty})}(\mathbb{Z},p_*I^*)\\
&\stackrel{\sim}{\rightarrow} & \mathrm{Ou}\int \underline{\mathrm{Hom}}_{Sh(G_{\mathbb{R}},\mathcal{X}_{\infty})}(\Gamma^*P_{\geq0},p_*I^*)\\
&\stackrel{\sim}{\leftarrow} & \mathrm{Ou}\int \underline{\mathrm{Hom}}_{Sh(G_{\mathbb{R}},\mathcal{X}_{\infty})}(\Gamma^*P_{\geq0},p_*\mathcal{A})
\end{eqnarray*}
where $\mathcal{A}\rightarrow I^*$ is an injective resolution (by injective equivariant sheaves), $\Gamma^*P_{\geq0}$ is the complex of equivariant sheaves associated with $P_{\geq0}$, $\underline{\mathrm{Hom}}_{Sh(G_{\mathbb{R}},\mathcal{X}_{\infty})}$ denotes the internal Hom inside the category of abelian sheaves on $Sh(G_{\mathbb{R}},\mathcal{X}_{\infty})$, $\int$ refers to the total complex associated with a double complex  and $\mathrm{Ou}:Sh(G_{\mathbb{R}},\mathcal{X}_{\infty})\rightarrow Sh(\mathcal{X}_{\infty})$ denotes the forgetful functor. Finally  $p_*$ is the direct image of the morphism $Sh(G_{\mathbb{R}},\mathcal{X}(\mathbb{C}))\rightarrow Sh(G_{\mathbb{R}},\mathcal{X}_{\infty})$. The functor $p_*$ is exact by proper base change.  Similarly, we define
$$\widehat{\pi}_*\mathcal{A}:=\mathrm{Ou}\int \underline{\mathrm{Hom}}_{Sh(G_{\mathbb{R}},\mathcal{X}_{\infty})}(\Gamma^*P_*,p_*\mathcal{A})$$
where $\int$  now denotes totalization with respect to the \emph{direct sum on diagonals}. We have
a convergent spectral sequence
\begin{equation}\label{hss0}
\underline{\widehat{H}}^i(G_{\mathbb{R}},p_*\mathcal{H}^j(\mathcal{A}))\Longrightarrow \mathcal{H}^{i+j}(\widehat{\pi}_*\mathcal{A}).
\end{equation}
Here, given a sheaf of $\mathbb{Z}[G_{\mathbb{R}}]$-modules $\mathcal{F}$ on $\mathcal{X}_{\infty}$, we denote by $\underline{\widehat{H}}^i(G_{\mathbb{R}},\mathcal{F})$ the sheaf associated with the presheaf $U\rightarrow \widehat{H}^i(G_{\mathbb{R}},\mathcal{F}(U))$. The spectral sequence (\ref{hss0}) shows that $\widehat{\pi}_*$ preserves quasi-isomorphisms. Hence $\widehat{\pi}_*$ induces  a functor
$$R\widehat{\pi}_*:\mathcal{D}^+(G_{\mathbb{R}},\mathcal{X}(\mathbb{C}))\longrightarrow \mathcal{D}(\mathcal{X}_{\infty}).$$
The spectral sequence above reads as follows:
\begin{equation}\label{hss}
\underline{\widehat{H}}^i(G_{\mathbb{R}},p_*\mathcal{H}^j(\mathcal{A}))\Longrightarrow \mathcal{H}^{i+j}(R\widehat{\pi}_*\mathcal{A}).
\end{equation}
Note that the map $P_*\rightarrow P_{\geq 0}$ gives a natural transformation $R\pi_*\rightarrow R\widehat{\pi}_*$. We shall need the following
\begin{lem}\label{lemTate}
Let $\mathcal{A}$ be a bounded below complex of abelian sheaves on $Sh(G_{\mathbb{R}},\mathcal{X}(\mathbb{C}))$. There are canonical isomorphisms
$$R\widehat{\Gamma}(G_{\mathbb{R}},\mathcal{X}(\mathbb{C}),\mathcal{A})\simeq R\Gamma(\mathcal{X}_{\infty},R\widehat{\pi}_*\mathcal{A})\simeq R\Gamma(\mathcal{X}(\mathbb{R}),R\widehat{\pi}_*\mathcal{A}_{\mid \mathcal{X}(\mathbb{R})}).$$

\end{lem}
\begin{proof}
Let $\mathcal{A}\rightarrow I^*$ be an injective resolution. The first isomorphism follows from the following canonical identifications:
\begin{eqnarray}
\label{first}R\Gamma(\mathcal{X}_{\infty},R\widehat{\pi}_*\mathcal{A})&\simeq & R\Gamma(\mathcal{X}_{\infty},\int \underline{\mathrm{Hom}}_{Sh(G_{\mathbb{R}},\mathcal{X}_{\infty})}(\Gamma^*P_*,p_*I^*))\\
\label{sec} &\simeq & \Gamma(\mathcal{X}_{\infty},\int \underline{\mathrm{Hom}}_{Sh(G_{\mathbb{R}},\mathcal{X}_{\infty})}(\Gamma^*P_*,p_*I^*))\\
\label{third}&\simeq &\int \Gamma(\mathcal{X}_{\infty},\underline{\mathrm{Hom}}_{Sh(G_{\mathbb{R}},\mathcal{X}_{\infty})}(\Gamma^*P_*,p_*I^*))\\
\label{fou}&\simeq  &\int \mathrm{Hom}_{G_{\mathbb{R}}-\mathrm{Mod}}(P_*,\Gamma(\mathcal{X}_{\infty},p_*I^*))\\
&\simeq &\int \mathrm{Hom}_{G_{\mathbb{R}}-\mathrm{Mod}}(P_*,\Gamma(\mathcal{X}(\mathbb{C}),I^*))\\
&\simeq & R\widehat{\Gamma}(G_{\mathbb{R}},\mathcal{X}(\mathbb{C}),\mathcal{A}).
\end{eqnarray}
Here (\ref{first}) follows from the spectral sequence (\ref{hss}); (\ref{sec}) and (\ref{third}) follow from the fact that $\mathcal{X}_{\infty}$ is compact and finite dimensional. Indeed, the sheaves $\underline{\mathrm{Hom}}_{Sh(G_{\mathbb{R}},\mathcal{X}_{\infty})}(\Gamma^*P_k,p_*I^l)$ are injective abelian equivariant sheaves (hence in particular injective abelian sheaves) since $P_k$ is a finitely generated free $\mathbb{Z}[G_{\mathbb{R}}]$-module. Moreover, a direct sum of injective sheaves is acyclic for the global sections functor $\Gamma(\mathcal{X}_{\infty},-)$, since $H^i(\mathcal{X}_{\infty},-)$ commutes with direct sums. It follows that $\int \underline{\mathrm{Hom}}_{Sh(G_{\mathbb{R}},\mathcal{X}_{\infty})}(\Gamma^*P_*,p_*I^*))$ is a complex of acyclic sheaves, hence (\ref{sec}) follows from the fact that $\Gamma(\mathcal{X}_{\infty},-)$ has finite cohomological dimension. The identification (\ref{third}) is valid since $\Gamma(\mathcal{X}_{\infty},-)$ commutes
 with direct sums, and (\ref{fou}) is given by adjunction.

The second isomorphism of the Lemma follows from the fact that  $\mathcal{H}^i(R\widehat{\pi}_*\mathcal{A})$ is concentrated on the closed subset $\mathcal{X}(\mathbb{R})\subseteq \mathcal{X}_{\infty}$ for any $i\in\mathbb{Z}$.
\end{proof}

\subsection{The motivic complex $\mathbb{Z}(n)^{\overline{\mathcal{X}}}$}
The construction of $\mathbb{Z}(n)^{\overline{\mathcal{X}}}$ requires the following lemma.
\begin{lem}\label{lem-here10}
Let $\mathbb{Z}(n)\rightarrow I(n)$ be  a $K$-injective resolution. There is a canonical morphism of complexes of abelian sheaves on $\overline{\mathcal{X}}_{et}$
$$\sigma_{\overline{\mathcal{X}},\mathbb{Z}(n)}:
\phi_*I(n)\longrightarrow u_{\infty,*} \tau^{>n} \widehat{\pi}_*\alpha^*(\tau^{\geq 0}I(n))$$
such that $\sigma_{\overline{\mathcal{X}},\mathbb{Z}(n)}$ induces the following morphism in $\mathcal{D}(\overline{\mathcal{X}}_{et})$
\begin{eqnarray*}
R\phi_*\mathbb{Z}(n)&\longrightarrow  &u_{\infty,*}u_{\infty}^*R\phi_*\mathbb{Z}(n)\\
&\stackrel{\sim}{\longrightarrow}& u_{\infty,*}R(\pi_*\alpha^*)\mathbb{Z}(n)\\
&\longrightarrow & u_{\infty,*}R\pi_*\alpha^* (\tau^{\geq0}\mathbb{Z}(n))\\
&\longrightarrow &u_{\infty,*}R\widehat{\pi}_* \alpha^*(\tau^{\geq0}\mathbb{Z}(n))\\
&\longrightarrow &u_{\infty,*}\tau^{>n}R\widehat{\pi}_*\alpha^*(\tau^{\geq0}\mathbb{Z}(n)).
\end{eqnarray*}
\end{lem}
\begin{proof}
We consider the morphism
\begin{eqnarray*}
\phi_*I(n)&\longrightarrow  &u_{\infty,*}u_{\infty}^*\phi_*I(n)\\
&\stackrel{\sim}{\longrightarrow}&  u_{\infty,*}\pi_*\alpha^* I(n) \\
&\longrightarrow &u_{\infty,*}\pi_*\alpha^* (\tau^{\geq 0} I(n))\\
&\longrightarrow & u_{\infty,*}\widehat{\pi}_*\alpha^* (\tau^{\geq 0} I(n))\\
&\longrightarrow & u_{\infty,*}\tau^{>n} \widehat{\pi}_*\alpha^* (\tau^{\geq 0} I(n)).
\end{eqnarray*}
where the first map is given by adjunction and the second map is given by (\ref{glueing}). The fourth map is given by the natural transformation $\pi_*\rightarrow \widehat{\pi}_*$ induced by $P_*\rightarrow \mathbb{Z}$.
\end{proof}

\begin{defn}\label{znxbardef}
We consider the following morphisms of complexes:
\begin{eqnarray*}
\sigma_{\overline{\mathcal{X}},\mathbb{Z}(n)}:
\phi_*I(n)& \longrightarrow & u_{\infty,*} \tau^{>n} \widehat{\pi}_*\alpha^*(\tau^{\geq 0} I(n))\\
\sigma^!_{\overline{\mathcal{X}},\mathbb{Z}(n)}:
\phi_*I(n)& \longrightarrow &  u_{\infty,*} \widehat{\pi}_*\alpha^*(\tau^{\geq 0} I(n))
\end{eqnarray*}
and we define
\begin{eqnarray*}
\mathbb{Z}(n)^{\overline{\mathcal{X}}}&:=&  \mathrm{\emph{Cone}}(\sigma_{\overline{\mathcal{X}},\mathbb{Z}(n)})[-1]\\
R\widehat{\phi}_!\mathbb{Z}(n)&:=& \mathrm{\emph{Cone}}(\sigma^!_{\overline{\mathcal{X}},\mathbb{Z}(n)})[-1].
\end{eqnarray*}

\end{defn}

Notice that $\mathbb{Z}(n)^{\overline{\mathcal{X}}}$ and $R\widehat{\phi}_!\mathbb{Z}(n)$ are well defined in the derived category $\mathcal{D}(\overline{\mathcal{X}}_{et})$, i.e. they do not depend on the choice of $I(n)$ up to a canonical isomorphism in $\mathcal{D}(\overline{\mathcal{X}}_{et})$.

\begin{prop}\label{prop-comp}
We have canonical morphisms of complexes
$$R\widehat{\phi}_!\mathbb{Z}(n)\longrightarrow \mathbb{Z}(n)^{\overline{\mathcal{X}}} \longrightarrow  R\phi_*\mathbb{Z}(n)$$
If $\mathcal{X}(\mathbb{R})=\emptyset$, these two maps are quasi-isomorphisms.
\end{prop}
\begin{proof}
The maps are the obvious ones. The cohomology sheaves of the complex $R\widehat{\pi}_*\alpha^* (\tau^{\geq 0} \mathbb{Z}(n))$ are concentrated on $\mathcal{X}(\mathbb{R})$ (see Lemma \ref{lemcomp} below). In particular, $\mathcal{X}(\mathbb{R})=\emptyset$ implies that $R\widehat{\pi}_*\alpha^* (\tau^{\geq 0} \mathbb{Z}(n))\simeq 0$.
\end{proof}

\begin{lem}\label{lemcomp}
For any $n\in\mathbb{Z}$, we have a canonical isomorphism
$$R\widehat{\pi}_*(\tau^{\geq 0}\alpha^*\mathbb{Z}(n))\simeq R\widehat{\pi}_*((2i\pi)^n\mathbb{Z})$$
in the derived category $\mathcal{D}(\mathcal{X}_{\infty})$. For $n\geq 0$ the natural map of complexes
$$\tau^{>n}R(\pi_*\alpha^*)\mathbb{Z}(n)\longrightarrow \tau^{>n} R\widehat{\pi}_*(\tau^{\geq0} \alpha^* \mathbb{Z}(n))$$
is a quasi-isomorphism.
\end{lem}
\begin{proof}
The exact triangle
$$\alpha^*\mathbb{Z}(n)\rightarrow \alpha^*\mathbb{Q}(n)\rightarrow \alpha^*\mathbb{Q}/\mathbb{Z}(n)$$
induces
$$\tau^{\geq0}\alpha^*\mathbb{Z}(n)\rightarrow \tau^{\geq0}\alpha^*\mathbb{Q}(n)\rightarrow \tau^{\geq0}\alpha^*\mathbb{Q}/\mathbb{Z}(n)$$
since $\alpha^*\mathbb{Q}/\mathbb{Z}(n)$ is concentrated in degree $0$.
This gives another exact triangle
$$R\widehat{\pi}_*(\tau^{\geq0}\alpha^*\mathbb{Z}(n))\rightarrow R\widehat{\pi}_*(\tau^{\geq0}\alpha^*\mathbb{Q}(n))\rightarrow R\widehat{\pi}_*\alpha^*\mathbb{Q}/\mathbb{Z}(n).$$
The spectral sequence (\ref{hss}) and the fact that $\tau^{\geq0}\alpha^*\mathbb{Q}(n)$ is bounded show that the cohomology sheaves of $R\widehat{\pi}_*(\tau^{\geq0}\alpha^*\mathbb{Q}(n))$ are $2$-primary torsion. Since they are also divisible, they  vanish. We obtain
$$R\widehat{\pi}_*(\tau^{\geq0}\alpha^*\mathbb{Z}(n))\simeq R\widehat{\pi}_*\alpha^*\mathbb{Q}/\mathbb{Z}(n)[-1].$$
We have isomorphisms
$$\alpha^*\mathbb{Q}/\mathbb{Z}(n)\simeq \underrightarrow{ \mathrm{lim}}\,\mu_m^{\otimes n}(\mathbb{C})[0]\simeq (2i\pi)^n\mathbb{Q}/\mathbb{Z}[0]$$
in the derived category of $G_\mathbb{R}$-equivariant abelian sheaves on $\mathcal{X}(\mathbb{C})$. We obtain
$$R\widehat{\pi}_*(\tau^{\geq0}\alpha^*\mathbb{Z}(n))\simeq R\widehat{\pi}_*\alpha^*\mathbb{Q}/\mathbb{Z}(n)[-1]\simeq
R\widehat{\pi}_*((2i\pi)^n\mathbb{Q}/\mathbb{Z})[-1]\simeq R\widehat{\pi}_*((2i\pi)^n\mathbb{Z})$$
where the last isomorphisms follows from the exact sequence
$$0\rightarrow (2i\pi)^n\mathbb{Z}\rightarrow (2i\pi)^n\mathbb{Q}\rightarrow (2i\pi)^n\mathbb{Q}/\mathbb{Z}\rightarrow 0$$
of $G_\mathbb{R}$-equivariant abelian sheaves and from the fact that $R\widehat{\pi}_*(2i\pi)^n\mathbb{Q}\simeq 0$ as above.

We prove the second assertion. Let $Sh(\mathcal{X}_{et},\mathbb{Q})$ and $Sh(\mathcal{X}_{\infty},\mathbb{Q})$ (respectively $Sh(\mathcal{X}_{et},\mathbb{Z})$ and $Sh(\mathcal{X}_{\infty},\mathbb{Z})$) be the  categories of sheaves of $\mathbb{Q}$-vector spaces (respectively of abelian groups) on $\mathcal{X}_{et}$ and $\mathcal{X}_{\infty}$ respectively. The inclusion functor $i:Sh(\mathcal{X}_{et},\mathbb{Q})\rightarrow Sh(\mathcal{X}_{et},\mathbb{Z})$ is exact and preserves $K$-injective complexes. Moreover the functor
$\pi_*\alpha^* \circ i:Sh(\mathcal{X}_{et},\mathbb{Q}) \rightarrow Sh(\mathcal{X}_{\infty},\mathbb{Q})$ is exact. We obtain $$\pi_*\alpha^* \circ i\simeq R(\pi_*\alpha^* \circ i)\simeq R(\pi_*\alpha^*) Ri\simeq R(\pi_*\alpha^*) i$$
hence
$$R(\pi_{*}\alpha^*)\mathbb{Q}(n)= \pi_{*}\alpha^*\mathbb{Q}(n).$$
In particular we have
$$\mathcal{H}^{p}(R(\pi_*\alpha^*)\mathbb{Q}(n))\simeq \mathcal{H}^p(\pi_*\alpha^*\mathbb{Q}(n))=0$$
for $p>n$, since $\alpha^*\mathbb{Q}(n)$ is acyclic in degrees $p>n$. Then we consider the long exact sequence
$$\mathcal{H}^{n}(R(\pi_*\alpha^*)\mathbb{Q}(n))\rightarrow \mathcal{H}^{n}(R(\pi_*\alpha^*)\mathbb{Q}/\mathbb{Z}(n))\rightarrow \mathcal{H}^{n+1}(R(\pi_*\alpha^*)\mathbb{Z}(n))$$
$$\rightarrow 0\rightarrow \mathcal{H}^{n+1}(R(\pi_*\alpha^*)\mathbb{Q}/\mathbb{Z}(n))\rightarrow \mathcal{H}^{n+2}(R(\pi_*\alpha^*)\mathbb{Z}(n))\rightarrow 0\rightarrow\cdots $$
The complex
$\alpha^*\mathbb{Q}/\mathbb{Z}(n)\simeq (2i\pi)^n\mathbb{Q}/\mathbb{Z}[0]$ is concentrated in degree zero, hence the abelian sheaf
$$\mathcal{H}^{n}(R(\pi_*\alpha^*)\mathbb{Q}/\mathbb{Z}(n))=\mathcal{H}^{n}(R\pi_*(\alpha^*\mathbb{Q}/\mathbb{Z}(n)))=R^n\pi_*(\alpha^*\mathbb{Q}/\mathbb{Z}(n))\simeq R^n\pi_*((2i\pi)^n\mathbb{Q}/\mathbb{Z})$$
is killed by two for any $n>0$ (see Lemma \ref{lem00}). It follows that the map
$$\mathcal{H}^{n}(R(\pi_*\alpha^*)\mathbb{Q}(n))\rightarrow \mathcal{H}^{n}(R(\pi_*\alpha^*)\mathbb{Q}/\mathbb{Z}(n))$$
is the zero map for $n>0$, so that the long exact sequence  above  gives isomorphisms
$$\mathcal{H}^{i-1}(R\pi_*(\alpha^*\mathbb{Q}/\mathbb{Z}(n)))\stackrel{\sim}{\rightarrow} \mathcal{H}^{i}(R(\pi_*\alpha^*)\mathbb{Z}(n))\mbox{ for any }i>n.$$
We obtain
\begin{eqnarray*}
\tau^{>n}(R(\pi_*\alpha^*)\mathbb{Z}(n))&\simeq&\tau^{>n}(R(\pi_*\alpha^*)\mathbb{Q}/\mathbb{Z}(n)[-1])\\
&\simeq&\tau^{>n}(R\pi_*(\alpha^*\mathbb{Q}/\mathbb{Z}(n))[-1])\\
&\simeq&\tau^{>n}(R\pi_*(2i\pi)^n\mathbb{Q}/\mathbb{Z}[-1])\\
&\simeq&\tau^{>n}(R\pi_*(2i\pi)^n\mathbb{Z})
\end{eqnarray*}
for any $n>0$. Note that
\begin{equation}\label{la}
\tau^{>n}(R(\pi_*\alpha^*)\mathbb{Z}(n))\simeq\tau^{>n}(R\pi_*(2i\pi)^n\mathbb{Z})
\end{equation}
also holds for $n=0$ since $\mathbb{Z}(0)\simeq\mathbb{Z}[0]$. It follows that the composite map
$$\tau^{>n}R(\pi_*\alpha^*)\mathbb{Z}(n)\longrightarrow \tau^{>n}R\widehat{\pi}_*(\tau^{\geq0}\alpha^*\mathbb{Z}(n))\stackrel{\sim}{\longrightarrow} \tau^{>n}R\widehat{\pi}_*((2i\pi)^n\mathbb{Z})$$
is an isomorphism in the derived category. The result follows.

\end{proof}

\begin{cor}\label{cor-u^!}
For any $n\in\mathbb{Z}$, we have exact triangles
$$\mathbb{Z}(n)^{\overline{\mathcal{X}}}\rightarrow R\phi_*\mathbb{Z}(n)\rightarrow  u_{\infty,*}\tau^{>n}R\widehat{\pi}_*(2i\pi)^n\mathbb{Z}$$
and
$$R\widehat{\phi}_!\mathbb{Z}(n)\rightarrow R\phi_*\mathbb{Z}(n)\rightarrow u_{\infty,*}R\widehat{\pi}_*(2i\pi)^n\mathbb{Z}.$$
\end{cor}
\begin{proof}
This follows from Lemma \ref{lemcomp} together with the definition of $\mathbb{Z}(n)^{\overline{\mathcal{X}}}$ and  $R\widehat{\phi}_!\mathbb{Z}(n)$.
\end{proof}
In particular we have isomorphisms
\begin{equation}Ru_{\infty}^!\mathbb{Z}(n)^{\overline{\mathcal{X}}}\simeq \left(\tau^{>n}R\widehat{\pi}_*((2i\pi)^n\mathbb{Z})\right)[-1]\label{shriek}\end{equation}
and
$$Ru_{\infty}^!(R\widehat{\phi}_!\mathbb{Z}(n))\simeq R\widehat{\pi}_*((2i\pi)^n\mathbb{Z})[-1].$$

\begin{prop}\label{propstalk}
For any $n\in\mathbb{Z}$, we have an isomorphism in $\mathcal{D}(\mathcal{X}_{et})$:
$$\phi^*\mathbb{Z}(n)^{\overline{\mathcal{X}}}\simeq\mathbb{Z}(n)^{\mathcal{X}}.$$
For $n\geq 0$, we have an isomorphism in $\mathcal{D}(\mathcal{X}_{\infty})$:
$$u_{\infty}^*\mathbb{Z}(n)^{\overline{\mathcal{X}}}\simeq \tau^{\leq n}u_{\infty}^*R\phi_*\mathbb{Z}(n).$$
\end{prop}
\begin{proof}
Applying $\phi^*$ to the first exact triangle of Corollary \ref{cor-u^!}, we obtain $\phi^*\mathbb{Z}(n)^{\overline{\mathcal{X}}}\simeq\mathbb{Z}(n)^{\mathcal{X}}$
since $\phi^*R\phi_*=R(\phi^*\phi_*)=\mathrm{Id}$ and $\phi^*u_{\infty,*}=0$. For $n\geq 0$, we have an exact triangle
$$\mathbb{Z}(n)^{\overline{\mathcal{X}}}\rightarrow R\phi_*\mathbb{Z}(n)^{\mathcal{X}}\rightarrow u_{\infty,*}(\tau^{>n}R(\pi_*\alpha^*)\mathbb{Z}(n))$$
by Lemma \ref{lemcomp}. Applying $u_{\infty}^*$ to this exact triangle, we obtain
$$u_{\infty}^*\mathbb{Z}(n)^{\overline{\mathcal{X}}}\rightarrow u_{\infty}^*R\phi_*\mathbb{Z}(n)^{\mathcal{X}}\rightarrow \tau^{>n} u_{\infty}^*R\phi_*\mathbb{Z}(n)$$
since $u_{\infty}^*u_{\infty,*}=Id$ and  $R(\pi_*\alpha^*)\simeq R(u^*_{\infty}\phi_*)\simeq u^*_{\infty}R\phi_*$.
\end{proof}
\begin{prop}\label{prop-trunc} Let $n\geq 0$. If $\mathcal{H}^i(\mathbb{Z}(n)^{\mathcal{X}})=0$ for any $i>n$ then the map $\mathbb{Z}(n)^{\overline{\mathcal{X}}}\rightarrow R\phi_*\mathbb{Z}(n)^{\mathcal{X}}$ induces a quasi-isomorphism
$$\mathbb{Z}(n)^{\overline{\mathcal{X}}}\stackrel{\sim}{\longrightarrow}\tau^{\leq n}(R\phi_*\mathbb{Z}(n)^{\mathcal{X}}).$$
This is for example the case for $\mathcal{X}$ smooth over a number ring.
\end{prop}
\begin{proof}
Recall that the family of exact functors $(\phi^*,u_{\infty}^*)$ is conservative. Moreover, the complex $\mathbb{Z}(n)^{\mathcal{X}}$ is acyclic in degree $>n$, hence so is the complex $\mathbb{Z}(n)^{\overline{\mathcal{X}}}$ by Proposition \ref{propstalk}. Hence the map
$\mathbb{Z}(n)^{\overline{\mathcal{X}}}\rightarrow R\phi_*\mathbb{Z}(n)^{\mathcal{X}}$ factors though
$\mathbb{Z}(n)^{\overline{\mathcal{X}}}\rightarrow\tau^{\leq n}(R\phi_*\mathbb{Z}(n)^{\mathcal{X}})$.
This map is a quasi-isomorphism if and only if it is a quasi-isomorphism after applying the functors $\phi^*$ and $u_{\infty}^*$. Hence the result follows from Proposition \ref{propstalk}.
\end{proof}

\begin{prop}\label{prop-description-atn=0,1}
We have $\mathbb{Z}(0)^{\overline{\mathcal{X}}}\simeq \mathbb{Z}$ and
$\mathbb{Z}(1)^{\overline{\mathcal{X}}}\simeq \phi_*\mathbb{G}_m[-1]$, where $\mathbb{Z}$ denotes the constant sheaf  on $\overline{\mathcal{X}}_{et}$, and $\mathbb{G}_m$ denotes the multiplicative group on $\mathcal{X}_{et}$.
\end{prop}

\begin{proof}
In view of $\mathbb{Z}(0)\simeq \mathbb{Z}[0]$ and $\mathbb{Z}(1)\simeq \mathbb{G}_m[-1]$, this follows from Proposition \ref{prop-trunc}.

\end{proof}

\subsection{Functoriality}

Let $f:\mathcal{X}\rightarrow \mathcal{Y}$ be an equidimensional flat map between (regular proper) arithmetic schemes.
Recall that flat pull-back of cycles induces a morphism $\mathbb{Z}(n)^{\mathcal{Y}}\rightarrow Rf_*\mathbb{Z}(n)^{\mathcal{X}}$. We consider the following commutative diagram of topoi:
\[ \xymatrix{
\mathcal{X}_{et}\ar[d]^{f_{\infty}}\ar[r]^{\phi^{\mathcal{X}}}& \overline{\mathcal{X}}_{et}\ar[d]^{\overline{f}}& \mathcal{X}_{\infty}\ar[d]^{f_{\infty}}\ar[l]_{u_{\infty}^{\mathcal{X}}}\\
\mathcal{Y}_{et}\ar[r]^{\phi^{\mathcal{X}}}& \overline{\mathcal{Y}}_{et}& \mathcal{Y}_{\infty}\ar[l]_{u_{\infty}^{\mathcal{X}}}
}
\]

\begin{prop}\label{prop-funct-barX}
For any $n\in\mathbb{Z}$, the pull-back map $\mathbb{Z}(n)^{\mathcal{Y}}\rightarrow Rf_*\mathbb{Z}(n)^{\mathcal{X}}$ induces   compatible maps $\mathbb{Z}(n)^{\overline{\mathcal{Y}}}\rightarrow R\overline{f}_*\mathbb{Z}(n)^{\overline{\mathcal{X}}}$ and $R\widehat{\phi}_!^{\mathcal{Y}}\mathbb{Z}(n)^{\mathcal{Y}}\rightarrow R\overline{f}_*R\widehat{\phi}_!^{\mathcal{Y}}\mathbb{Z}(n)^{\mathcal{X}}$ in $\mathcal{D}(\overline{\mathcal{Y}}_{et})$.
\end{prop}
\begin{proof} The following diagram
\[ \xymatrix{
R\phi_*^{\mathcal{Y}}(\mathbb{Z}(n))^{\mathcal{Y}}\ar[d]\ar[r]& R\phi_*^{\mathcal{Y}}Rf_*\mathbb{Z}(n)^{\mathcal{X}} \ar[d]\ar[r]^{\sim}& R\overline{f}_* R\phi_*^{\mathcal{X}}\mathbb{Z}(n)^{\mathcal{X}}\ar[d]\\
u^{\mathcal{Y}}_{\infty,*}\tau^{>n}R\widehat{\pi}_*^{\mathcal{Y}}((2i\pi)^n\mathbb{Z})\ar[r]& u^{\mathcal{Y}}_{\infty,*}Rf_{\infty,*}\tau^{>n}R\widehat{\pi}_*^{\mathcal{X}}((2i\pi)^n\mathbb{Z})\ar[r]^{\sim}&  R\overline{f}_{*}u^{\mathcal{X}}_{\infty,*}\tau^{>n}R\widehat{\pi}_*^{\mathcal{X}}((2i\pi)^n\mathbb{Z})
}
\]
commutes in $\mathcal{D}(\overline{\mathcal{Y}}_{et})$, hence there exists a map
$$\mathbb{Z}(n)^{\overline{\mathcal{Y}}}\longrightarrow R\overline{f}_* \mathbb{Z}(n)^{\overline{\mathcal{X}}}$$
sitting in a morphism of exact triangles. If $n\geq 0$, such a map is unique since $u_{\infty}^{*}\mathbb{Z}(n))^{\overline{\mathcal{Y}}}$ is concentrated in degrees $\leq n$ by Proposition \ref{propstalk}, whereas $Rf_{\infty,*}\tau^{>n}R\widehat{\pi}_*^{\mathcal{X}}((2i\pi)^n\mathbb{Z})$ is concentrated in degrees $\geq n+2$.

For arbitrary $n\in\mathbb{Z}$, we represent the previous diagram by a commutative diagram of actual complexes of sheaves (using Lemma \ref{lem-here10} and Lemma \ref{lemcomp}), and we use the fact that the cone is functorial  (note that the right horizontal maps are isomorphisms of complexes). We define similarly a canonical map
$$R\widehat{\phi}_!^{\mathcal{Y}}\mathbb{Z}(n)^{\mathcal{Y}}\longrightarrow R\overline{f}_*R\widehat{\phi}_!^{\mathcal{X}}\mathbb{Z}(n)^{\mathcal{X}}$$
such that the following diagram commutes
\[ \xymatrix{
R\widehat{\phi}_!^{\mathcal{Y}}\mathbb{Z}(n)^{\mathcal{Y}}\ar[d]\ar[r]& \mathbb{Z}(n)^{\overline{\mathcal{Y}}} \ar[d]\ar[r]^{\sim}& R\phi^{\mathcal{Y}}_* \mathbb{Z}(n)^{\mathcal{Y}}\ar[d]\\
R\overline{f}_*R\widehat{\phi}_!^{\mathcal{X}}\mathbb{Z}(n)^{\mathcal{X}}\ar[r]& R\overline{f}_*\mathbb{Z}(n)^{\overline{\mathcal{X}}} \ar[r]^{\sim}& R\overline{f}_* R\phi^{\mathcal{X}}_*\mathbb{Z}(n)^{\mathcal{X}}
}
\]

\end{proof}

\subsection{Relationship with Milne's cohomology with compact support}

We start with the definition of cohomology with compact support with respect to the Artin-Verdier compactification: we define
$$R\Gamma_c(\mathcal{X}_{et},\mathcal{A}):=R\Gamma(\overline{\mathcal{X}}_{et},\varphi_!\mathcal{A})$$
for any bounded below complex $\mathcal{A}$ of abelian sheaves on ${\mathcal{X}}_{et}$. We obtain immediately an exact triangle
$$R\Gamma_c(\mathcal{X}_{et},\mathcal{A})\rightarrow R\Gamma(\overline{\mathcal{X}}_{et},\mathcal{A})\rightarrow R\Gamma(\mathcal{X}_{\infty},u_{\infty}^*\mathcal{A}).$$
\begin{prop}
We have an exact triangle
$$R\Gamma_c(\mathcal{X}_{et},\mathcal{A})\rightarrow R\Gamma(\mathcal{X}_{et},\mathcal{A})\rightarrow R\Gamma(G_{\mathbb{R}},\mathcal{X}(\mathbb{C}),\alpha^*\mathcal{A}).$$
\label{rgc-prop}\end{prop}
\begin{proof}
Using Lemma \ref{lem-here0}, we obtain an exact triangle
$$\phi_!\mathcal{A}\rightarrow R\phi_*\mathcal{A}\rightarrow u_{\infty,*}R\pi_*(\alpha^*\mathcal{A}).$$
The result then follows from the canonical identifications
$$R\Gamma(\overline{\mathcal{X}}_{et},u_{\infty,*}R\pi_*(\alpha^*\mathcal{A}))
\simeq R\Gamma(\mathcal{X}_{\infty},R\pi_*(\alpha^*\mathcal{A}))\simeq
R\Gamma(G_{\mathbb{R}},\mathcal{X}(\mathbb{C}),\alpha^*\mathcal{A}).$$

\end{proof}

We now recall Milne's definition for \'etale cohomology with compact support \cite{Milne-duality}. Let $A$ be an abelian sheaf on $\mathrm{Spec}(\mathbb{Z})_{et}$. One defines a complex $R\widehat{\Gamma}_c(\mathrm{Spec}(\mathbb{Z})_{et},A)$ so that there is an exact triangle
$$R\widehat{\Gamma}_c(\mathrm{Spec}(\mathbb{Z})_{et},A)\rightarrow R\Gamma(\mathrm{Spec}(\mathbb{Z})_{et},A)\rightarrow R\widehat{\Gamma}(G_{\mathbb{R}},v^*A)$$
where $R\widehat{\Gamma}(G_{\mathbb{R}},-)$ denotes Tate cohomology of the finite group $G_{\mathbb{R}}$ and
$v: \mathrm{Spec}(\mathbb{C})\rightarrow  \mathrm{Spec}(\mathbb{Z})$
is the unique map. This definition generalizes as follows. Let $f:\mathcal{X}\rightarrow  \mathrm{Spec}(\mathbb{Z})$ be a proper scheme over $\mathrm{Spec}(\mathbb{Z})$, and let $\mathcal{A}$ be a complex of abelian sheaves on $\mathcal{X}_{et}$ whose restriction $\mathcal{A}_{\mathcal{X}[1/S]}$ to $\mathcal{X}[1/S]:=\mathcal{X}\otimes_{\mathbb{Z}}\mathbb{Z}[1/S]$ is bounded below, for some finite set $S$ of prime numbers.
Then $v^*Rf_*\mathcal{A}$ is bounded below, and one defines
$$R\widehat{\Gamma}_c(\mathcal{X}_{et},\mathcal{A}):=R\widehat{\Gamma}_c(\mathrm{Spec}(\mathbb{Z})_{et},Rf_*\mathcal{A})$$
so that we have an exact triangle
$$R\widehat{\Gamma}_c(\mathcal{X}_{et},\mathcal{A})\rightarrow R\Gamma(\mathcal{X}_{et},\mathcal{A})\rightarrow R\widehat{\Gamma}(G_{\mathbb{R}},v^*Rf_*\mathcal{A}).$$

\begin{lem}\label{lem-Milne-Tate} Let $\mathcal{A}$ be a complex of abelian sheaves such that $\mathcal{A}_{ \mathcal{X}[1/S]}$ is a bounded below complex of locally constant torsion abelian sheaves, for some finite set of primes $S$. Then we have an exact triangle
$$R\widehat{\Gamma}_c(\mathcal{X}_{et},\mathcal{A})\rightarrow R\Gamma(\mathcal{X}_{et},\mathcal{A})\rightarrow R\widehat{\Gamma}(G_{\mathbb{R}},\mathcal{X}(\mathbb{C}),\alpha^*\mathcal{A}).$$
\end{lem}
\begin{proof}
By proper base change we have an isomorphism
$$v^*Rf_*\mathcal{A}\simeq R\Gamma_{et}(\mathcal{X}\otimes_{\mathbb{Z}}{\mathbb{C}},\mathcal{A}_{\mathcal{X}\otimes_{\mathbb{Z}}{\mathbb{C}}})$$
where $\mathcal{A}_{\mathcal{X}\otimes_{\mathbb{Z}}{\mathbb{C}}}$ is the pull-back of $\mathcal{A}$ to $\mathcal{X}\otimes_{\mathbb{Z}}{\mathbb{C}}$. Artin's comparison theorem then gives
$$R\Gamma(\mathcal{X}(\mathbb{C}),\alpha^*\mathcal{A})\simeq  R\Gamma_{et}(\mathcal{X}\otimes_{\mathbb{Z}}{\mathbb{C}},\mathcal{A}_{\mathcal{X}\otimes_{\mathbb{Z}}{\mathbb{C}}})\simeq v^*Rf_*\mathcal{A}$$
hence
$$R\widehat{\Gamma}(G_{\mathbb{R}},\mathcal{X}(\mathbb{C}),\alpha^*\mathcal{A})
\simeq R\widehat{\Gamma}(G_{\mathbb{R}},v^*Rf_*\mathcal{A}).
$$
\end{proof}

\begin{prop}\label{prop-milnecomplex}
For any $n\in\mathbb{Z}$ and any $m\geq 1$, we have an isomorphism
$$R\Gamma(\overline{\mathcal{X}}_{et},R\widehat{\phi}_!\mathbb{Z}(n))\otimes^L\mathbb{Z}/m\mathbb{Z}\simeq R\widehat{\Gamma}_c(\mathcal{X}_{et},\mathbb{Z}/m(n)).$$
\end{prop}
\begin{proof}
By definition of $R\widehat{\phi}_!\mathbb{Z}(n)$, Lemma \ref{lemTate} and Lemma \ref{lemcomp}, we have an exact triangle
$$R\Gamma(\overline{\mathcal{X}}_{et},R\widehat{\phi}_!\mathbb{Z}(n))\otimes^L\mathbb{Z}/m\mathbb{Z}\rightarrow R\Gamma(\mathcal{X}_{et},\mathbb{Z}/m(n))\rightarrow R\widehat{\Gamma}(G_{\mathbb{R}},\mathcal{X}(\mathbb{C}),(2i\pi)^n \mathbb{Z}/m)
$$
which can be identified with the triangle of Lemma \ref{lem-Milne-Tate}.
\end{proof}
This suggests the following
\begin{notation}
We set $R\widehat{\Gamma}_c(\mathcal{X}_{et},\mathbb{Z}(n)):=R\Gamma(\overline{\mathcal{X}}_{et},R\widehat{\phi}_!\mathbb{Z}(n))$ and $\widehat{H}^i_c(\mathcal{X}_{et},\mathbb{Z}(n)):=H^i(\overline{\mathcal{X}}_{et},R\widehat{\phi}_!\mathbb{Z}(n))$.
\end{notation}

\begin{prop}\label{propcomp}
There are canonical exact triangles
$$R\widehat{\phi}_!\mathbb{Z}(n)\rightarrow \mathbb{Z}(n)^{\overline{\mathcal{X}}}\rightarrow u_{\infty,*}\tau^{\leq n}R\widehat{\pi}_*(2i\pi)^n\mathbb{Z}$$
and
$$R\widehat{\Gamma}_c(\mathcal{X}_{et},\mathbb{Z}(n))\rightarrow R\Gamma(\overline{\mathcal{X}}_{et},\mathbb{Z}(n))\rightarrow R\Gamma(\mathcal{X}(\mathbb{R}),\tau^{\leq n}R\widehat{\pi}_*(2i\pi)^n\mathbb{Z}).$$

\end{prop}

\begin{proof}
We have a commutative diagram with exact rows and colons:
\[ \xymatrix{
R\widehat{\phi}_!\mathbb{Z}(n)\ar[d]\ar[r]
&\ar[r]\ar[d]R\phi_*\mathbb{Z}(n)&u_{\infty,*}R\widehat{\pi}_*(2i\pi)^n\mathbb{Z}\ar[d]\\
\mathbb{Z}(n)^{\overline{\mathcal{X}}}\ar[d]\ar[r]
&\ar[r]\ar[d]R\phi_*\mathbb{Z}(n)
&u_{\infty,*}\tau^{>n}R\widehat{\pi}_*(2i\pi)^n\mathbb{Z}\ar[d]\\
u_{\infty,*}\tau^{\leq n}R\widehat{\pi}_*(2i\pi)^n\mathbb{Z}\ar[r]& 0 \ar[r]& u_{\infty,*}\tau^{\leq n}R\widehat{\pi}_*(2i\pi)^n\mathbb{Z}[1]
}
\]
The result follows.
\end{proof}

\begin{cor}\label{cor-comp-milnebarX}
For any $i> \mathrm{dim}(\mathcal{X})+n$,  there is a canonical isomorphism
$$H^i(\overline{\mathcal{X}}_{et},\mathbb{Z}(n))\simeq \widehat{H}^i_c(\mathcal{X}_{et},\mathbb{Z}(n)).$$
\end{cor}
\begin{proof}
This follows immediately from the exact triangle
$$ R\Gamma(\mathcal{X}(\mathbb{R}),\tau^{\leq n}R\widehat{\Gamma}(G_{\mathbb{R}},(2i\pi)^n\mathbb{Z}))[-1]\rightarrow R\widehat{\Gamma}_c(\mathcal{X}_{et},\mathbb{Z}(n))\rightarrow R\Gamma(\overline{\mathcal{X}}_{et},\mathbb{Z}(n))
$$
since $\mathcal{X}(\mathbb{R})$ has topological dimension $\leq \mathrm{dim}(\mathcal{X})-1$.
\end{proof}
\begin{cor}\label{cor-trace}
We have $H^i(\overline{\mathcal{X}}_{et},\mathbb{Z}(d))=0$ for any $i>2d+2$ and there is a canonical trace map
$H^{2d+2}(\overline{\mathcal{X}}_{et},\mathbb{Z}(d))\rightarrow \mathbb{Q}/\mathbb{Z}$.
\end{cor}
\begin{proof} We may assume $\mathcal{X}$ flat over $\mathbb{Z}$. By Corollary \ref{cor-comp-milnebarX}, we have $H^i(\overline{\mathcal{X}}_{et},\mathbb{Z}(d))\simeq \widehat{H}^i_c(\mathcal{X}_{et},\mathbb{Z}(d))$ for $i>2d$. We have $\widehat{H}^i_c(\mathcal{X}_{et},\mathbb{Q}(d))\simeq H^i(\mathcal{X}_{Zar},\mathbb{Q}(d))=0$ for $i\geq 2d$ (the case $i=2d$ follows from the fact that $CH^d(\mathcal{X})$ is finite). We get $\widehat{H}^i_c(\mathcal{X}_{et},\mathbb{Z}(d))\simeq \widehat{H}^{i-1}_c(\mathcal{X}_{et},\mathbb{Q}/\mathbb{Z}(d))$ for $i>2d$. Using  \cite{Geisser10}, we find  $\widehat{H}^{i}_c(\mathcal{X}_{et},\mathbb{Q}/\mathbb{Z}(d))=0$ for $i>2d+1$. The push-forward map
$Rf_*\mathbb{Z}(d)[-2d]\rightarrow \mathbb{Z}(1)[-2]$
induces
$$\widehat{H}^{2d+1}_c(\mathcal{X},\mathbb{Z}/m(d))=\widehat{H}^{2d+1}_c(\mathrm{Spec}(\mathbb{Z}),Rf_*\mathbb{Z}/m(d))\rightarrow \widehat{H}^3_c(\mathrm{Spec}(\mathbb{Z}),\mathbb{Z}/m(1))\simeq \mathbb{Z}/m$$
where the last identification follows from $\mathbb{Z}(1)\simeq\mathbb{G}_m[1]$ and classical Artin-Verdier duality.

\end{proof}

\subsection{Products}
We consider below product maps $\mathbb{Z}(n)\otimes^L \mathbb{Z}(m)\rightarrow \mathbb{Z}(n+m)$ in $\mathcal{D}(\mathcal{X}_{et})$. These product maps are always assumed to induce, after $(-)\otimes^L\mathbb{Z}/r\mathbb{Z}$, the natural maps $\mu^{\otimes n}_r\otimes \mu^{\otimes n}_r\rightarrow \mu^{\otimes n+m}_r$
over the \'etale site of $\mathcal{X}'[1/r]$, where $\mathcal{X}'\subset \mathcal{X}$ is an open subscheme which is smooth over $\mathbb{Z}$.

\begin{prop}\label{prop-canonicalproduct}
Let $n, m\geq 0$ be non-negative integers. A product map $\mathbb{Z}(n)\otimes^L \mathbb{Z}(m)\rightarrow \mathbb{Z}(n+m)$ as above over $\mathcal{X}$ induces a unique product map over $\overline{\mathcal{X}}$:
$$\mathbb{Z}(n)^{\overline{\mathcal{X}}}\otimes^L \mathbb{Z}(m)^{\overline{\mathcal{X}}}\rightarrow \mathbb{Z}(n+m)^{\overline{\mathcal{X}}}.$$
\end{prop}
\begin{proof}
The product map $\mathbb{Z}(n)\otimes \mathbb{Z}(m)\rightarrow \mathbb{Z}(n+m)$ in $\mathcal{D}(\mathcal{X}_{et})$ induces
\begin{equation}\label{zero}
R\phi_*\mathbb{Z}(n)\otimes^L R\phi_*\mathbb{Z}(m)\rightarrow R\phi_*\mathbb{Z}(n+m)
\end{equation}
in $\mathcal{D}(\overline{\mathcal{X}}_{et})$, and we consider
\begin{equation}\label{phiprod}
\mathbb{Z}(n)^{\overline{\mathcal{X}}}\otimes^L \mathbb{Z}(m)^{\overline{\mathcal{X}}}\rightarrow
R\phi_*\mathbb{Z}(n)\otimes^L R\phi_*\mathbb{Z}(m)\rightarrow R\phi_*\mathbb{Z}(n+m).
\end{equation}
We now remark that the composite map
$$\mathbb{Z}(n)^{\overline{\mathcal{X}}}\otimes^L \mathbb{Z}(m)^{\overline{\mathcal{X}}}\rightarrow R\phi_*\mathbb{Z}(n+m) \rightarrow u_{\infty,*}\tau^{>n+m}R\pi_*(2i\pi)^{n+m}\mathbb{Z}$$
is zero and that
$$\mathrm{Hom}_{\mathcal{D}(\overline{\mathcal{X}}_{et})}(\mathbb{Z}(n)^{\overline{\mathcal{X}}}\otimes^L \mathbb{Z}(m)^{\overline{\mathcal{X}}},u_{\infty,*}\tau^{>n+m}R\pi_*(2i\pi)^{n+m}\mathbb{Z}[-1])=0$$
simply because
$$u^*_{\infty}\mathbb{Z}(n)^{\overline{\mathcal{X}}}\otimes^L u^*_{\infty}\mathbb{Z}(m)^{\overline{\mathcal{X}}}
\simeq \tau^{\leq n}u^*_{\infty}R\phi_*\mathbb{Z}(n)\otimes^L \tau^{\leq m}u^*_{\infty}R\phi_*\mathbb{Z}(m)$$
is concentrated in degrees $\leq n+m$ by Proposition \ref{propstalk}. It follows that (\ref{phiprod}) induces a unique map
$$\mathbb{Z}(n)^{\overline{\mathcal{X}}}\otimes^L \mathbb{Z}(m)^{\overline{\mathcal{X}}}\rightarrow \mathbb{Z}(m)^{\overline{\mathcal{X}}}.$$

\end{proof}

\begin{rem}
If $\mathbb{Z}(n)$, $\mathbb{Z}(m)$ and $\mathbb{Z}(n+m)$
are acyclic in degrees $>n$, $>m$ and $>n+m$ respectively, then (\ref{zero}) induces
$$\tau^{\leq n}
R\phi_*\mathbb{Z}(n)\otimes^L \tau^{\leq m} R\phi_*\mathbb{Z}(n+m)\rightarrow \tau^{\leq n+m}R\phi_*\mathbb{Z}(n+m)$$
by adjunction. Proposition \ref{prop-canonicalproduct} then follows somewhat more directly from Proposition \ref{prop-trunc}. This applies for $\mathcal{X}$ smooth over a number ring.
\end{rem}

\begin{prop}\label{prop-noncanonicalprod}
Let $n,m\in\mathbb{Z}$ be arbitrary integers. A product map $\mathbb{Z}(n)\otimes \mathbb{Z}(m)\rightarrow \mathbb{Z}(n+m)$ over $\mathcal{X}$ induces in a non-canonical way product maps
\begin{equation}\label{prodgeneral}
\mathbb{Z}(n)^{\overline{\mathcal{X}}}\otimes^L \mathbb{Z}(m)^{\overline{\mathcal{X}}}\rightarrow \mathbb{Z}(n+m)^{\overline{\mathcal{X}}}
\end{equation}
and
\begin{equation}\label{prodcompact}
R\widehat{\phi}_!\mathbb{Z}(n)\otimes^L R\phi_*\mathbb{Z}(m)\rightarrow \mathbb{Z}(n+m)^{\overline{\mathcal{X}}}.
\end{equation}
If $n\geq 0$ these product maps can be chosen so that the induced square
\[ \xymatrix{
R\widehat{\phi}_!\mathbb{Z}(n)\ar[r]\ar[d]&R\underline{\mathrm{Hom}}(R\phi_*\mathbb{Z}(m),\mathbb{Z}(n+m)^{\overline{\mathcal{X}}})\ar[d]\\
\mathbb{Z}(n)^{\overline{\mathcal{X}}}\ar[r]& R\underline{\mathrm{Hom}}(\mathbb{Z}(m)^{\overline{\mathcal{X}}},\mathbb{Z}(n+m)^{\overline{\mathcal{X}}})
}
\]
commutes.

\end{prop}
\begin{proof}
We need to show that the composite map
$$\mathbb{Z}(n)^{\overline{\mathcal{X}}}\otimes^L \mathbb{Z}(m)^{\overline{\mathcal{X}}}\rightarrow R\phi_*\mathbb{Z}(n)\otimes^L R\phi_*\mathbb{Z}(m)\rightarrow R\phi_*\mathbb{Z}(n+m) \rightarrow u_{\infty,*}\tau^{>n+m}R\widehat{\pi}_*(2i\pi)^{n+m}\mathbb{Z}$$
is the zero map. This follows from the fact that the map $$R\phi_*\mathbb{Z}(n)\otimes^L R\phi_*\mathbb{Z}(m) \rightarrow u_{\infty,*}\tau^{>n+m}R\widehat{\pi}_*(2i\pi)^{n+m}\mathbb{Z}$$ factors through
$$\tau^{>n+m}\left(u_{\infty,*}R\widehat{\pi}_*(2i\pi)^{n}\mathbb{Z}\otimes ^L u_{\infty,*}R\widehat{\pi}_*(2i\pi)^{m}\mathbb{Z}\right)$$
on the one hand and that
$$\mathbb{Z}(n)^{\overline{\mathcal{X}}}\otimes^L \mathbb{Z}(m)^{\overline{\mathcal{X}}}\rightarrow
u_{\infty,*}R\widehat{\pi}_*(2i\pi)^{n}\mathbb{Z}\otimes ^L u_{\infty,*}R\widehat{\pi}_*(2i\pi)^{m}\mathbb{Z}$$
factors through
$$u_{\infty,*}\tau^{\leq n}R\widehat{\pi}_*(2i\pi)^{n}\mathbb{Z}\otimes ^L u_{\infty,*} \tau^{\leq m}R\widehat{\pi}_*(2i\pi)^{m}\mathbb{Z}$$
on the other. This gives the existence of (\ref{prodgeneral}), which however is non-unique in general because $$u^*_{\infty}\mathbb{Z}(n)^{\overline{\mathcal{X}}}\otimes^L u^*_{\infty}\mathbb{Z}(m)^{\overline{\mathcal{X}}}$$
fails to be concentrated in degrees $\leq n+m$ (e.g. take $n<<0$). We obtain by a similar argument a non-canonical map (\ref{prodcompact}).

We now show that these product maps (\ref{prodgeneral}) and (\ref{prodcompact}) may be chosen to be compatible, at least for $n\geq 0$.  In view of
$$u_{\infty,*} \tau^{> n+m}R\widehat{\pi}_*(2i\pi)^{n+m}\mathbb{Z}\simeq u_{\infty,*} \tau^{\geq n+m+2}R\widehat{\pi}_*(2i\pi)^{n+m}\mathbb{Z}$$
we see that
$$\mathrm{Hom}_{\mathcal{D}(\overline{\mathcal{X}})}(
u_{\infty,*}\tau^{\leq n}R\widehat{\pi}_*(2i\pi)^{n}\mathbb{Z}\otimes ^L u_{\infty,*} \tau^{\leq m}R\widehat{\pi}_*(2i\pi)^{m}\mathbb{Z} [\epsilon],
u_{\infty,*} \tau^{> n+m}R\widehat{\pi}_*(2i\pi)^{n+m}\mathbb{Z})=0$$
for $\epsilon=0,1$. It follows that the map
$$u_{\infty,*}\tau^{\leq n}R\widehat{\pi}_*(2i\pi)^{n}\mathbb{Z}\otimes ^L u_{\infty,*} R\widehat{\pi}_*(2i\pi)^{m}\mathbb{Z}\rightarrow
u_{\infty,*} \tau^{> n+m}R\widehat{\pi}_*(2i\pi)^{n+m}\mathbb{Z}$$
induces a unique map
\begin{equation}\label{jj}
u_{\infty,*}\tau^{\leq n}R\widehat{\pi}_*(2i\pi)^{n}\mathbb{Z}\otimes ^L u_{\infty,*} \tau^{> m}R\widehat{\pi}_*(2i\pi)^{m}\mathbb{Z}\rightarrow
u_{\infty,*} \tau^{> n+m}R\widehat{\pi}_*(2i\pi)^{n+m}\mathbb{Z}.
\end{equation}
We obtain a commutative diagram
\[ \xymatrix{
\mathbb{Z}(n)^{\overline{\mathcal{X}}}\otimes^L \mathbb{Z}(m)^{\overline{\mathcal{X}}}\ar[d]&\mathbb{Z}(n+m)^{\overline{\mathcal{X}}}\ar[d]\\
\mathbb{Z}(n)^{\overline{\mathcal{X}}}\otimes^L R\phi_*\mathbb{Z}(m)\ar[r]\ar[d]&R\phi_*\mathbb{Z}(n+m)\ar[d]\\
\mathbb{Z}(n)^{\overline{\mathcal{X}}}\otimes ^L u_{\infty,*}\tau^{> m}R\widehat{\pi}_*(2i\pi)^{m}\mathbb{Z}
\ar[r]&
u_{\infty,*}\tau^{> n+m} R\widehat{\pi}_*(2i\pi)^{n+m}\mathbb{Z}
}
\]
where the lower horizontal map is the composition of (\ref{jj}) with
$$\mathbb{Z}(n)^{\overline{\mathcal{X}}}\otimes ^L u_{\infty,*}\tau^{> m}R\widehat{\pi}_*(2i\pi)^{m}\mathbb{Z}\rightarrow
u_{\infty,*}\tau^{\leq n}R\widehat{\pi}_*(2i\pi)^{n}\mathbb{Z}\otimes ^L u_{\infty,*} \tau^{> m}R\widehat{\pi}_*(2i\pi)^{m}\mathbb{Z}.$$
The colons are exact triangles hence one may choose a map (\ref{prodgeneral}) which turns this diagram  into a morphism of exact triangles. It follows that the square
\[ \xymatrix{
\mathbb{Z}(n)^{\overline{\mathcal{X}}}\ar[d]\ar[r]^{(\ref{prodgeneral})\hspace{1cm}}& R\underline{\mathrm{Hom}}(\mathbb{Z}(m)^{\overline{\mathcal{X}}},\mathbb{Z}(n+m)^{\overline{\mathcal{X}}})\ar[d]\\
u_{\infty,*}\tau^{\leq n}R\widehat{\pi}_*(2i\pi)^{n}\mathbb{Z}\ar[r]& R\underline{\mathrm{Hom}}(u_{\infty,*}\tau^{>m}R\widehat{\pi}_*(2i\pi)^{m}\mathbb{Z}[-1],\mathbb{Z}(n+m)^{\overline{\mathcal{X}}})
}
\]
commutes, where the left vertical map is induced by adjunction
$$\mathbb{Z}(n)^{\overline{\mathcal{X}}}\rightarrow u_{\infty,*}u^*_{\infty}\mathbb{Z}(n)^{\overline{\mathcal{X}}}\simeq u_{\infty,*}\tau^{\leq n}R\widehat{\pi}_*(2i\pi)^{n}\mathbb{Z}$$
and the  right vertical map is induced by the map $u_{\infty,*}\tau^{>m}R\widehat{\pi}_*(2i\pi)^{m}\mathbb{Z}[-1]\rightarrow \mathbb{Z}(m)^{\overline{\mathcal{X}}}$ which is in turn given by the definition of $\mathbb{Z}(m)^{\overline{\mathcal{X}}}$ (see Corollary \ref{cor-u^!}). Hence there exists a product map (\ref{prodcompact}) inducing a morphism of exact triangles:
\[ \xymatrix{
R\widehat{\phi}_!\mathbb{Z}(n)\ar[r]^{(\ref{prodcompact})\hspace{1cm}}\ar[d]&R\underline{\mathrm{Hom}}(R\phi_*\mathbb{Z}(m),\mathbb{Z}(n+m)^{\overline{\mathcal{X}}})\ar[d]\\
\mathbb{Z}(n)^{\overline{\mathcal{X}}}\ar[d]\ar[r]^{(\ref{prodgeneral})\hspace{1cm}}& R\underline{\mathrm{Hom}}(\mathbb{Z}(m)^{\overline{\mathcal{X}}},\mathbb{Z}(n+m)^{\overline{\mathcal{X}}})\ar[d]\\
u_{\infty,*}\tau^{\leq n}R\widehat{\pi}_*(2i\pi)^{n}\mathbb{Z}\ar[r]& R\underline{\mathrm{Hom}}(u_{\infty,*}\tau^{>m}R\widehat{\pi}_*(2i\pi)^{m}\mathbb{Z}[-1],\mathbb{Z}(n+m)^{\overline{\mathcal{X}}})
}
\]
\end{proof}

\subsection{Artin-Verdier Duality}
The following conjecture is known for $\mathcal{X}$ smooth proper over a number ring, and for regular proper $\mathcal{X}$ as long as $n\leq 0$. It is expected to hold for arbitrary regular proper $\mathcal{X}$.
\begin{conj}$\mathbf{AV}(\mathcal{X},n)$ There is a symmetric product map
$$\mathbb{Z}(n)\otimes^L \mathbb{Z}(d-n)\rightarrow \mathbb{Z}(d)$$ in $\mathcal{D}(\mathcal{X}_{et})$
such that the induced pairing
$$\widehat{H}^{i}_c(\mathcal{X}_{et},\mathbb{Z}/m(n))\times H^{2d+1-i}(\mathcal{X}_{et},\mathbb{Z}/m(d-n))\rightarrow \widehat{H}^{2d+1}_c(\mathcal{X}_{et},\mathbb{Z}/m(d))\rightarrow\mathbb{Q}/\mathbb{Z}$$
is a perfect pairing of finite abelian groups for any $i\in\mathbb{Z}$ and any positive integer $m$.
\end{conj}

The aim of this section is to prove the following result.
\begin{thm}\label{thm-barX}
Let $n\in\mathbb{Z}$. Assume that $\mathcal{X}$ satisfies $\mathbf{AV}(\mathcal{X},n)$. Then  there is a  product map
$$\mathbb{Z}(n)^{\overline{\mathcal{X}}}\otimes^L \mathbb{Z}(d-n)^{\overline{\mathcal{X}}}\rightarrow \mathbb{Z}(d)^{\overline{\mathcal{X}}}$$
in $\mathcal{D}(\overline{\mathcal{X}}_{et})$ such that the induced pairing
$$H^{i}(\overline{\mathcal{X}}_{et},\mathbb{Z}/m(n))\times H^{2d+1-i}(\overline{\mathcal{X}}_{et},\mathbb{Z}/m(d-n))\rightarrow H^{2d+1}(\overline{\mathcal{X}}_{et},\mathbb{Z}/m(d))\rightarrow\mathbb{Q}/\mathbb{Z}$$
is a perfect pairing of finite abelian groups for any $i\in\mathbb{Z}$  and any positive integer $m$.
\end{thm}
\begin{proof}
By Proposition \ref{prop-noncanonicalprod} there exist product maps  (\ref{prodgeneral}) and (\ref{prodcompact}) inducing a morphism of exact triangles (see the last diagram in the proof of Proposition \ref{prop-noncanonicalprod}). Applying $R\Gamma(\overline{\mathcal{X}}_{et},-)$ and composing with the map
$$R\Gamma(\overline{\mathcal{X}}_{et},\mathbb{Z}(d))\longrightarrow \tau^{\geq 2d+2} R\Gamma(\overline{\mathcal{X}}_{et},\mathbb{Z}(d))\simeq \mathbb{Q}/\mathbb{Z}[-2d-2]$$
given by Corollary \ref{cor-trace},
we obtain a morphism of exact triangles:
\[ \xymatrix{
R\widehat{\Gamma}(\mathcal{X}_{et},\mathbb{Z}(n))\ar[r]\ar[d]&R\mathrm{Hom}(R\Gamma(\mathcal{X}_{et},\mathbb{Z}(d-n)),\mathbb{Q}/\mathbb{Z}[-2d-2])\ar[d]\\
R\Gamma(\overline{\mathcal{X}}_{et},\mathbb{Z}(n))\ar[r]\ar[d]&R\mathrm{Hom}(R\Gamma(\overline{\mathcal{X}}_{et},\mathbb{Z}(d-n)),\mathbb{Q}/\mathbb{Z}[-2d-2])\ar[d]\\
R\Gamma(\mathcal{X}(\mathbb{R}),\tau^{\leq n}R\widehat{\pi}_*(2i\pi)^{n}\mathbb{Z})\ar[r]^{\sim\hspace{2.7cm}}& R\mathrm{Hom}(R\Gamma(\mathcal{X}(\mathbb{R}),\tau^{>m}R\widehat{\pi}_*(2i\pi)^{d-n}\mathbb{Z}[-1]),\mathbb{Q}/\mathbb{Z}[-2d-2])
}
\]
where the bottom horizontal map is an isomorphism by Lemma \ref{lemforduality} below. Applying the functor $(-)\otimes^L\mathbb{Z}/m$ we obtain
a morphism of exact triangles:
\[ \xymatrix{
R\widehat{\Gamma}(\mathcal{X}_{et},\mathbb{Z}/m(n))\ar[r]^{\sim\hspace{2.7cm}}\ar[d]&R\mathrm{Hom}(R\Gamma(\mathcal{X}_{et},\mathbb{Z}/m(d-n)),\mathbb{Q}/\mathbb{Z}[-2d-1])\ar[d]\\
R\Gamma(\overline{\mathcal{X}}_{et},\mathbb{Z}/m(n))\ar[r]\ar[d]&R\mathrm{Hom}(R\Gamma(\overline{\mathcal{X}}_{et},\mathbb{Z}/m(d-n)),\mathbb{Q}/\mathbb{Z}[-2d-1])\ar[d]\\
R\Gamma(\mathcal{X}(\mathbb{R}),\tau^{\leq n}R\widehat{\pi}_*(2i\pi)^{n}\mathbb{Z}/m)\ar[r]^{\sim\hspace{2.7cm}}& R\mathrm{Hom}(R\Gamma(\mathcal{X}(\mathbb{R}),\tau^{>m}R\widehat{\pi}_*(2i\pi)^{d-n}\mathbb{Z}/m[-1]),\mathbb{Q}/\mathbb{Z}[-2d-1])
}
\]
where top horizontal map is an isomorphism by assumption. The theorem therefore follows from the following Lemma \ref{lemforduality}.
\end{proof}

\begin{lem}\label{lemforduality}
The product $(2i\pi)^{n}\mathbb{Z}\otimes (2i\pi)^{d-n}\mathbb{Z}\rightarrow (2i\pi)^{d}\mathbb{Z}$ induces a perfect pairing
$$H^{i-1}(\mathcal{X}(\mathbb{R}),\tau^{\leq n}R\widehat{\pi}_*(2i\pi)^{n}\mathbb{Z})\times
H^{2d+2-i}(\mathcal{X}(\mathbb{R}),\tau^{>d-n}R\widehat{\pi}_*(2i\pi)^{n}\mathbb{Z})$$
$$\rightarrow H^{d-1}(\mathcal{X}(\mathbb{R}),\widehat{H}^{d+2}(G_{\mathbb{R}},(2i\pi)^{d}\mathbb{Z}))\rightarrow\mathbb{Q}/\mathbb{Z}$$
of finite $2$-torsion abelian groups.
\end{lem}

\begin{proof} Notice first that the pairing lands in $H^{d-1}(\mathcal{X}(\mathbb{R}),\widehat{H}^{d+2}(G_{\mathbb{R}},(2i\pi)^{d}\mathbb{Z}))$ because the real manifold $\mathcal{X}(\mathbb{R})$ is $(d-1)$-dimensional (we may assume $\mathcal{X}/\mathbb{Z}$ flat) and
$$\tau^{> d-n}R\widehat{\pi}_*((2i\pi)^{d-n}\mathbb{Z})\simeq \tau^{\geq d+2-n}R\widehat{\pi}_*((2i\pi)^{d-n}\mathbb{Z})$$
since $\widehat{H}^{d+1-n}(G_{\mathbb{R}},(2i\pi)^{d-n}\mathbb{Z})=0$ regardless the parity of $d-n$.
We need to show that the pairing mentioned in the lemma induces isomorphisms
\begin{equation}\label{duality-iso}
H^{2d+2-i}(\mathcal{X}(\mathbb{R}),\tau^{> d-n}R\widehat{\pi}_*((2i\pi)^{d-n}\mathbb{Z}))\simeq  H^{i-1}(\mathcal{X}(\mathbb{R}),\tau^{\leq n}R\widehat{\pi}_*((2i\pi)^n\mathbb{Z}))^D
\end{equation}
where $(-)^D$ denotes the Pontryagin dual.
We have
$$\tau^{> d-n}R\widehat{\pi}_*((2i\pi)^{d-n}\mathbb{Z})\simeq \tau^{\geq d+2-n}R\widehat{\pi}_*((2i\pi)^{d-n}\mathbb{Z})\simeq \bigoplus_{k\geq0}\mathbb{Z}/2\mathbb{Z}[-(d+2-n)-2k].$$
Similarly, we have
$$\tau^{\leq n}R\widehat{\pi}_*((2i\pi)^n\mathbb{Z})\simeq \bigoplus_{k\geq0}\mathbb{Z}/2\mathbb{Z}[-n+2k]$$
hence
\begin{eqnarray}
H^{i-1}(\mathcal{X}(\mathbb{R}),\tau^{\leq n}R\widehat{\pi}_*((2i\pi)^n\mathbb{Z}))&\simeq&
\bigoplus_{k\geq0} H^{i-1}(\mathcal{X}(\mathbb{R}),\mathbb{Z}/2\mathbb{Z}[-n+2k])\\
\label{ter}&\simeq&\bigoplus_{k\geq0} H^{i-1-n+2k}(\mathcal{X}(\mathbb{R}),\mathbb{Z}/2\mathbb{Z}).
\end{eqnarray}
Poincar\'e duality for the $(d-1)$--dimensional real manifold $\mathcal{X}(\mathbb{R})$ with $\mathbb{Z}/2\mathbb{Z}$-coefficients yields
\begin{eqnarray}
H^{2d+2-i}(\mathcal{X}(\mathbb{R}),\tau^{> d-n}R\widehat{\pi}_*((2i\pi)^{d-n}\mathbb{Z}))&\simeq& \bigoplus_{k\geq0} H^{2d+2-i-(d+2-n)-2k}(\mathcal{X}(\mathbb{R}),\mathbb{Z}/2\mathbb{Z})\\
&\simeq&\bigoplus_{k\geq0} H^{d-i+n-2k}(\mathcal{X}(\mathbb{R}),\mathbb{Z}/2\mathbb{Z})\\
&\simeq& \bigoplus_{k\geq0} H^{(d-1)-(d-i+n-2k)}(\mathcal{X}(\mathbb{R}),\mathbb{Z}/2\mathbb{Z})^D\\
 \label{terter}&\simeq& \left(\bigoplus_{k\geq0} H^{i-1-n+2k}(\mathcal{X}(\mathbb{R}),\mathbb{Z}/2\mathbb{Z})\right)^D\\
 &\simeq& H^{i-1}(\mathcal{X}(\mathbb{R}),\tau^{\leq n}R\widehat{\pi}_*((2i\pi)^n\mathbb{Z}))^D
\end{eqnarray}
Note that the sums (\ref{ter}) and (\ref{terter}) are both finite. Note also that the manifold $\mathcal{X}(\mathbb{R})$ may very well be non-orientable (e.g. take $\mathcal{X}=\mathbb{P}^2_{\mathbb{Z}}$) but Poincar\'e duality still holds with $\mathbb{Z}/2\mathbb{Z}$-coefficients. The result follows.
\end{proof}

\begin{cor}\label{corAVn=0} Let $\mathcal{X}$ be a regular proper scheme of pure dimension $d$ and let $n\leq 0$. There is a product map
$\mathbb{Z}(n)^{\overline{\mathcal{X}}}\otimes^L \mathbb{Z}(d-n)^{\overline{\mathcal{X}}}\rightarrow \mathbb{Z}(d)^{\overline{\mathcal{X}}}$
such that
$$H^{2d+1-i}(\overline{\mathcal{X}}_{et},\mathbb{Z}/m(n))\times H^{i}(\overline{\mathcal{X}}_{et},\mathbb{Z}/m(d-n))\rightarrow H^{2d+1}(\overline{\mathcal{X}}_{et},\mathbb{Z}/m(d))\rightarrow \mathbb{Q}/\mathbb{Z}$$
is a perfect pairing of finite groups for any $i\in\mathbb{Z}$.
\end{cor}
\begin{proof}
The pairing $\mathbb{Z}(0)\otimes^L\mathbb{Z}(d)\simeq \mathbb{Z}\otimes^L\mathbb{Z}(d)\rightarrow \mathbb{Z}(d)$ is the obvious one. By (\cite{Geisser10} Theorem 7.8) the assumption of Theorem \ref{thm-barX} for $n=0$ is fulfilled. The case $n<0$ will follow from (\cite{Geisser10} Theorem 7.8) and from an isomorphism
$$\mathbb{Z}/m(d-n)\simeq R\underline{\mathrm{Hom}}_{\mathcal{X}[1/m]}(\mu_m^{\otimes n}[-1],\mathbb{Z}(d)).$$
Let $f:\mathcal{X}[1/m]\rightarrow\mathrm{Spec}(\mathbb{Z}[1/m])$ be the unique map. By (\cite{Geisser10} Theorem 7.10) and since $\mathbb{Z}(1)\simeq \mathbb{G}_m[-1]$, we have
\begin{eqnarray}
R\underline{\mathrm{Hom}}_{\mathcal{X}[1/m]}(\mu_m^{\otimes n}[-1],\mathbb{Z}(d))&\simeq &
R\underline{\mathrm{Hom}}_{\mathcal{X}[1/m]}(f^*\mu_m^{\otimes n}[-1],\mathbb{Z}(d))\\
&\simeq & Rf^!R\underline{\mathrm{Hom}}_{\mathbb{Z}[1/m]}(\mu_m^{\otimes n}[-1],\mathbb{Z}(1)[-2d+2])\\
&\simeq & Rf^!(\mathbb{Z}/m(1-n))[-2d+2]\\
&\simeq & \mathbb{Z}/m(d-n).
\end{eqnarray}
We obtain
$$\mathbb{Z}/p^{\nu}(d-n)\stackrel{\sim}{\rightarrow}
Rj_{p,*}j^*_{p}\mathbb{Z}/p^{\nu}(d-n)\simeq R\underline{\mathrm{Hom}}_{\mathcal{X}}(j_{p,!}\mu_{p^{\nu}}^{\otimes n}[-1],\mathbb{Z}(d))$$
where the first isomorphism follows from (\cite{Geisser10} Theorem 7.2(a)) and (\cite{Geisser10} Proposition 2.3). Taking the limit over $\nu$ and $p$, we obtain the product map
$$\mathbb{Z}(d-n)\rightarrow \mathrm{holim}\, \mathbb{Z}/m(d-n)\rightarrow R\underline{\mathrm{Hom}}_{\mathcal{X}}(\mathbb{Z}(n),\mathbb{Z}(d))$$
over $\mathcal{X}$. Finally, the induced map
$$\widehat{H}^{i}_c(\mathcal{X}_{et},\mathbb{Z}/m(n))\times H^{2d+1-i}(\mathcal{X}_{et},\mathbb{Z}/m(d-n))\rightarrow \widehat{H}^{2d+1}_c(\mathcal{X}_{et},\mathbb{Z}/m(d))\rightarrow\mathbb{Q}/\mathbb{Z}$$
is a perfect pairing of finite groups by (\cite{Geisser10} Theorem 7.8).

\end{proof}

\begin{cor}\label{cor-AVsmooth} Let $\mathcal{X}$ be a smooth proper scheme over a number ring and let $n\in\mathbb{Z}$ be an arbitrary integer. There is a product map
$\mathbb{Z}(n)^{\overline{\mathcal{X}}}\otimes^L \mathbb{Z}(d-n)^{\overline{\mathcal{X}}}\rightarrow \mathbb{Z}(d)^{\overline{\mathcal{X}}}$
such that
$$H^{2d+1-i}(\overline{\mathcal{X}}_{et},\mathbb{Z}/m\mathbb{Z}(n))\times H^{i}(\overline{\mathcal{X}}_{et},\mathbb{Z}/m\mathbb{Z}(d-n))\rightarrow H^{2d+1}(\overline{\mathcal{X}}_{et},\mathbb{Z}/m(d))\rightarrow \mathbb{Q}/\mathbb{Z}$$
is a perfect pairing of finite groups for any $i\in\mathbb{Z}$.
\end{cor}
\begin{proof} It remains to treat the case $0\leq n\leq d$. By \cite{Geisser04a}, the complex $\mathbb{Z}/p^{\nu}\mathbb{Z}(n)$ is isomorphic (in the derived category) to Sato's complex  (see \cite{Schneider94} and \cite{Sato07}): we have $\mathbb{Z}/p^{\nu}\mathbb{Z}(n)\simeq \mathfrak{T}_\nu(n)$. For general $m=p_1^{\nu_1}\cdots p_s^{\nu_s}$, we simply write $$\mathbb{Z}/m\mathbb{Z}(n)\simeq \mathbb{Z}/p_1^{\nu_1}\mathbb{Z}(n)\times \cdots \times \mathbb{Z}/p_i^{\nu_i}\mathbb{Z}(n).$$ Using this identification with Sato's complex, there is a canonical product map $$\mathbb{Z}/m\mathbb{Z}(n)\otimes \mathbb{Z}/m\mathbb{Z}(d-n)\rightarrow \mathbb{Z}/m\mathbb{Z}(d)$$
which is uniquely induced by $$\mu_{p_i^{\nu_i}}^{\otimes n}\otimes \mu_{p_i^{\nu_i}}^{\otimes (d-n)}\rightarrow \mu_{p_i^{\nu_i}}^{\otimes d}$$
over $\mathcal{X}[1/p_i]$ for $i=1,...,s$. By \cite{Spitzweck14}, this product map is defined integrally: Spitzweck defines motivic complexes $\mathbb{Z}(n)_S$ on $\mathcal{X}_{et}$ which are canonically isomorphic (in the derived category) to Bloch's cycle complexes $\mathbb{Z}(n)$ on $\mathcal{X}$ (since $\mathcal{X}$ is assumed to be smooth over a number ring), and product maps $\mathbb{Z}(n)_S\otimes^L\mathbb{Z}(d-n)_S\rightarrow \mathbb{Z}(d)_S$ inducing the product on Sato's complexes.
By \cite{Sato07} 10.1.3 the induced map
$$\widehat{H}^{i}_c(\mathcal{X}_{et},\mathbb{Z}/m(n))\times H^{2d+1-i}(\mathcal{X}_{et},\mathbb{Z}/m(d-n))\rightarrow \widehat{H}^{2d+1}_c(\mathcal{X}_{et},\mathbb{Z}/m(d))\rightarrow\mathbb{Q}/\mathbb{Z}$$
is a perfect pairing of finite abelian groups for any $i\in\mathbb{Z}$ and any positive integer $m$, so that Theorem \ref{thm-barX} applies.
\end{proof}

\subsection{The conjecture $\mathbf{AV}(f,n)$}\label{sectionAVf}

Let $f:\mathcal{X}\rightarrow \mathcal{Y}$ be a flat map of relative dimension $c$ between connected regular proper arithmetic schemes of dimension $d_{\mathcal{X}}$ and $d_{\mathcal{Y}}$ respectively. We have canonical maps
\begin{equation}\label{forhat-push-forward}
R\widehat{\Gamma}_c(\mathcal{Y}_{et},\mathbb{Q}/\mathbb{Z}(n))\longrightarrow R\widehat{\Gamma}_c(\mathcal{X}_{et},\mathbb{Q}/\mathbb{Z}(n)).
\end{equation}
and
\begin{equation}\label{forbar-push-forward-}
R\Gamma(\overline{\mathcal{Y}}_{et},\mathbb{Q}/\mathbb{Z}(n))\longrightarrow R\Gamma(\overline{\mathcal{X}}_{et},\mathbb{Q}/\mathbb{Z}(n)).
\end{equation}
Assume that  $\mathbf{AV}(\overline{\mathcal{X}}_{et},n)$ and $\mathbf{AV}(\overline{\mathcal{Y}}_{et},n)$ hold. This yields isomorphisms
$$R\Gamma(\mathcal{Y}_{et},\widehat{\mathbb{Z}}(d_{\mathcal{Y}}-n))\stackrel{\sim}{\rightarrow}R\mathrm{Hom}(R\widehat{\Gamma}_c(\mathcal{Y}_{et},\mathbb{Q}/\mathbb{Z}(n)),\mathbb{Q}/\mathbb{Z}[-2d_{\mathcal{Y}}-1])$$
and
$$R\Gamma(\mathcal{X}_{et},\widehat{\mathbb{Z}}(d_{\mathcal{X}}-n))\stackrel{\sim}{\rightarrow}R\mathrm{Hom}(R\widehat{\Gamma}_c(\mathcal{X}_{et},\mathbb{Q}/\mathbb{Z}(n)),\mathbb{Q}/\mathbb{Z}[-2d_{\mathcal{X}}-1])$$
in $\mathcal{D}$, where
$$R\Gamma(\mathcal{Y}_{et},\widehat{\mathbb{Z}}(d-n)):=\mathrm{holim}\,R\Gamma(\mathcal{Y}_{et},\mathbb{Z}/m(d-n)).$$
Hence (\ref{forhat-push-forward}) induces a morphism
\begin{equation}\label{hat-push-forward}
R\Gamma(\mathcal{X}_{et},\widehat{\mathbb{Z}}(d_{\mathcal{X}}-n))\longrightarrow R\Gamma(\mathcal{Y}_{et},\widehat{\mathbb{Z}}(d_{\mathcal{Y}}-n))[-2c].
\end{equation}
We obtain similarly a morphism
\begin{equation}\label{bar-push-forward}
R\Gamma(\overline{\mathcal{X}}_{et},\widehat{\mathbb{Z}}(d_{\mathcal{X}}-n))\longrightarrow R\Gamma(\overline{\mathcal{Y}}_{et},\widehat{\mathbb{Z}}(d_{\mathcal{Y}}-n))[-2c].
\end{equation}

\begin{conj} $\mathbf{AV}(f,n)$ The diagram
\[ \xymatrix{
R\Gamma(\mathcal{X}_{Zar},\mathbb{Z}(d_{\mathcal{X}}-n))\ar[d]\ar[r]& R\Gamma(\mathcal{X}_{et},\widehat{\mathbb{Z}}(d_{\mathcal{X}}-n))\ar[d]^{(\ref{hat-push-forward})}&R\Gamma(\overline{\mathcal{X}}_{et},\widehat{\mathbb{Z}}(d_{\mathcal{X}}-n))\ar[l]\ar[d]^{(\ref{bar-push-forward})}\\
R\Gamma(\mathcal{Y}_{Zar},\mathbb{Z}(d_{\mathcal{Y}}-n))[-2c]\ar[r]& R\Gamma(\mathcal{Y}_{et},\widehat{\mathbb{Z}}(d_{\mathcal{Y}}-n))[-2c] &R\Gamma(\overline{\mathcal{Y}}_{et},\widehat{\mathbb{Z}}(d_{\mathcal{Y}}-n))[-2c]\ar[l]
}
\]
 commutes in $\mathcal{D}$, where the horizontal maps are the evident ones and the left vertical map is induced by proper push-forward of cycles.
\end{conj}

\subsection{The projective bundle formula}

For $n<0$, the complex $R\Gamma(\overline{\mathcal{X}}_{et},\mathbb{Z}(n))$ may have non-trivial cohomology in negative degrees. The following proposition shows that this surprising fact is a consequence of the projective bundle formula. We only treat the simplest (but decisive) case $\mathcal{X}=\mathrm{Spec}(\mathbb{Z})$.
\begin{prop}\label{prop-pbf}
There is an isomorphism
$$R\Gamma(\overline{\mathbb{P}^m_{\mathbb{Z}}}_{,et},\mathbb{Z})\simeq \bigoplus_{0\leq n\leq m}R\Gamma(\overline{\mathrm{Spec}(\mathbb{Z})}_{et},\mathbb{Z}(-n))[-2n].$$
\end{prop}
\begin{proof}
By proper base change, one has
$$R\Gamma(\mathbb{P}^m_{\mathbb{Z},et},\mathbb{Z})\simeq \bigoplus_{0\leq n\leq m}R\Gamma( \mathrm{Spec}(\mathbb{Z})_{et},\mathbb{Z}(-n))[-2n].$$
Moreover, one has
$\tau^{> 0}R\widehat{\Gamma}(G_{\mathbb{R}},\mathbb{Z})\simeq\bigoplus_{k>0}
\mathbb{Z}/2\mathbb{Z}[-2k]$ hence
$$R\Gamma(\mathbb{P}^m(\mathbb{R}),\tau^{> 0}R\widehat{\pi}_*\mathbb{Z})=\bigoplus_{k>0}
R\Gamma(\mathbb{P}^m(\mathbb{R}),\mathbb{Z}/2\mathbb{Z})[-2k].$$
Let $\mathbf{S}^m\subset\mathbb{R}^{m+1}$ be the $m$-sphere endowed with its natural  (antipodal) action of $\{\pm 1\}$. A look at the spectral sequence for the Galois cover $$\mathbf{S}^m\longrightarrow \mathbf{S}^m/\{\pm 1\}\simeq \mathbb{P}^m(\mathbb{R})$$
shows that the canonical map
$$R\Gamma(\{\pm 1\},\mathbb{Z}/2\mathbb{Z})_{\leq m}\rightarrow
R\Gamma(\{\pm 1\},\mathbb{Z}/2\mathbb{Z})\rightarrow R\Gamma(\mathbb{P}^m(\mathbb{R}),\mathbb{Z}/2\mathbb{Z})$$
is an isomorphism. This
yields
$$R\Gamma(\mathbb{P}^m(\mathbb{R}),\mathbb{Z}/2\mathbb{Z})\simeq\bigoplus_{0\leq n\leq m}\mathbb{Z}/2\mathbb{Z}[-n].$$
We obtain
\begin{eqnarray*}
R\Gamma(\mathbb{P}^m(\mathbb{R}),\tau^{> 0}R\widehat{\pi}_*\mathbb{Z})&\simeq&\bigoplus_{k>0}
R\Gamma(\mathbb{P}^m(\mathbb{R}),\mathbb{Z}/2\mathbb{Z})[-2k]\\
&\simeq&\bigoplus_{k>0}\bigoplus_{0\leq n\leq m}\mathbb{Z}/2\mathbb{Z}[-n][-2k]\\
&\simeq&\bigoplus_{0\leq n\leq m}\bigoplus_{k>0}\mathbb{Z}/2\mathbb{Z}[-2k+n][-2n]\\
&\simeq&\bigoplus_{0\leq n\leq m}\tau^{> -n}R\widehat{\Gamma}(G_{\mathbb{R}},(2i\pi)^{-n}\mathbb{Z})[-2n]
\end{eqnarray*}
We obtain an exact triangle
$$R\Gamma(\overline{\mathbb{P}^m_{\mathcal{X}}}_{,et},\mathbb{Z})\rightarrow \bigoplus_{0\leq n\leq m}R\Gamma( \mathrm{Spec}(\mathbb{Z})_{et},\mathbb{Z}(-n))[-2n]\rightarrow \bigoplus_{0\leq n\leq m}\tau^{> -n}R\widehat{\Gamma}(G_{\mathbb{R}},(2i\pi)^{-n}\mathbb{Z})[-2n]$$
where the second map is the sum of the maps
$R\Gamma( \mathrm{Spec}(\mathbb{Z})_{et},\mathbb{Z}(-n))\rightarrow \tau^{> -n}R\widehat{\Gamma}(G_{\mathbb{R}},(2i\pi)^{-n}\mathbb{Z})$.
The result follows.
\end{proof}

\section{Appendix B: Motivic and syntomic cohomology}

The first purpose of this appendix is to formulate a conjectural relation between ($p$-adically completed cohomology of) higher Chow complexes and syntomic cohomology for arbitrary regular arithmetic schemes over local integer rings, extending results of Geisser \cite{Geisser04a} in the smooth case. Whereas \cite{Geisser04a} applies with integral coefficients under the assumption $0\leq n<p-1$ we shall only consider rational coefficients but any $n\in\bz$. The second purpose is to discuss the motivic decomposition of $p$-adically completed motivic cohomology which is necessary to compare our main conjecture to the Tamagawa Number Conjecture. This appendix is only needed in the main body of the text in section \ref{sec:compatibility}, and then only in the restricted setting of smooth schemes for which more complete results are known (see Prop. \ref{psmooth} and the remarks following it).

For any equidimensional scheme $Y$ and $n\in\bz$ we define the complex of etale sheaves $$\bz(n)=z^n(-,2n-*)$$ from Bloch's higher Chow complex (and we retain the cohomological indexing even if $Y$ is singular). For any prime number $l$ we set
 \begin{equation} R\Gamma(Y,\bz_l(n))=\mathrm{holim}_\bullet R\Gamma(Y,\bz(n)/l^\bullet);\quad  R\Gamma(Y,\bq_l(n))=R\Gamma(Y,\bz_l(n))_\bq\label{compdef}\end{equation}
where cohomology groups are always understood in the etale topology. For $n<0$ we have $\bz(n)=0$ which differs from the definition in section \ref{sect-emc}. However, the two definitions will lead to the same cohomology with $\bq_l(n)$-coefficients in Cor. \ref{decompcor} and Cor. \ref{pcor} below.

In the following a regular scheme will always assumed to be (essentially) of finite type over a field or a Dedekind ring.

\begin{conj} For a regular scheme $\X$ the complex $\bz(n)$ on $\X_{et}$ is (cohomologically) concentrated in degrees $\leq n$.
\label{zhongconj}\end{conj}

This conjecture is known if $\X$ is smooth over a field or Dedekind ring by \cite{Geisser04a}[Cor. 4.4]. Note that the Bloch complex of presheaves $\bz(n)$ is concentrated in degrees $\leq 2n$ and so the conjecture says that the sheafification (in the etale topology) of $\mathcal H^i(\bz(n))$ vanishes for $i=n+1,...,2n$. This should be true for the Zariski topology as well.

\begin{lem} If Conjecture \ref{zhongconj} holds, $l$ is invertible on $\X$ and $n\geq 0$ then $\bz(n)/l^\bullet\cong \mu_{l^\bullet}^{\otimes n}$ on $\X_{et}$.
\label{invertible}\end{lem}

\begin{proof} This is \cite{Zhong14}[Thm. 2.6], noting that truncation is unnecessary under Conjecture \ref{zhongconj}.
\end{proof}

We shall usually consider the following local setting. For a prime $p$ let $K/\bq_p$ be a finite extension with maximal unramified subextension $K_0/\bq_p$. We set $$s=\Spec(k),\quad S=\Spec(\co_K),\quad \eta=\Spec(K)$$ and $$ \etabar=\Spec(\etabar),\quad \sbar=\Spec(\bar{k}),\quad \shen=\Spec(\mathcal O_K^{ur}),\quad \etahen=\Spec(K^{ur})$$ where $k$ is the residue field of $K$ and $K^{ur}/K$ the maximal unramified extension. We denote by $\iota:s\to S$ and $j:\eta\to S$ the complementary immersions. Finally, we let
$$f:X\to S$$
be a flat, projective morphism of relative dimension $d-1$ and we assume throughout that $X$ is a {\em regular} scheme. We denote the base change of a map by indexing it with the source of the base change map to its target. For example, $\iota_X:X_s\to X$ is the closed immersion of the special fibre.

\subsection{$l$-adic cohomology}  In this section $l$ is a prime number different from $p$. We summarize here some facts from $l$-adic cohomology in order to motivate the conjectures of the next section.

\subsubsection{Localization triangles} There is a localisation triangle
\[ R\Gamma_{X_{\sbar}}(X_{\shen},\bq_l)\to R\Gamma(X_{\shen},\bq_l)\to R\Gamma(X_{\etahen},\bq_l)\]
where we can view the first group as (Borel-Moore or ordinary) homology and the second (via proper base change) as cohomology of the special fibre $X_{\sbar}$. The topological analogue of this situation is a tubular neighborhood, i.e. a closed embedding $X_{\sbar}\to X$ of a compact space $X_{\sbar}$ into a manifold $X$ which is moreover a homotopy equivalence.  This gives rise to a map from homology to cohomology of $X_{\sbar}$ by the same localization sequence. Using regularity of $X$ and $S$ we have $f^!\bq_l=\bq_l(d-1)[2d-2]$ and using regularity of $s$ and $S$ we have $R\iota^!\bq_l=\bq_l(-1)[-2]$. Since $\iota\circ f_s=f\circ\iota_X$ we obtain $$R\iota_X^!\bq_l=R\iota_X^!f^!\bq_l(-d+1)[-2d+2]=f_s^!\bq_l(-d)[-2d]$$ and we can rewrite the shifted localization triangle as
\begin{equation}  R\Gamma(X_{\sbar},\bq_l)\to R\Gamma(X_{\etahen},\bq_l)\to R\Gamma(X_{\sbar}, f_{\sbar}^!\bq_l(-d)[-2d+1])  \to\label{tri1}\end{equation}
which we view as a computation of the cohomology of $X_{\etahen}$. The cohomology of $X_{\etahen}$ can also be computed by Galois descent from the cohomology of $X_{\etabar}$. Setting $I:=\Gal(\bar{K}/K^{ur})$ one has
\[ R\Gamma(X_{\etahen},\bq_l)\cong R\Gamma(I,R\Gamma(X_{\etabar},\bq_l))\]
but in order to bring out the analogy with $p$-adic cohomology in the next section we rewrite this slightly using Weil-Deligne representations. If $(V,\rho)$ is a continuous $l$-adic representation of $G_K:=\Gal(\bar{K}/K)$, a theorem of Grothendieck guarantees that an open subgroup $I_1\subseteq I$ acts unipotently, i.e. for $\sigma\in I_1$
\[ \rho(\sigma)=\exp(t_l(\sigma)N)\]
where $t_l:I\to\bz_l(1)$ is the natural surjection and $N:V\to V(-1)$ is a nilpotent endomorphism. Following  \cite{deligne73}[8.4.2] one defines a representation $(V^\delta,\rho^\delta)$ of the Weil group $W_K\subseteq G_K$ on $V=V^\delta$ by $\rho^\delta(\phi^n\sigma)=\rho(\phi^n\sigma)\exp(-t_l(\sigma)N)$ for $\sigma\in I$ and $\phi$ any lift of Frobenius. By Grothendieck's theorem $\rho^\delta$ is trivial on the open subgroup $I_1$ of $W_K$, i.e. $\rho^\delta$ is discrete. One then has with $\sigma_t$ denoting a topological generator of $\bz_l(1)$ and $P=\ker(t_l)$
\[R\Gamma(I,V)\cong R\Gamma(\bz_l(1),V^P)\cong\left[V^P\xrightarrow{\sigma_t-1} V^P(-1)\right]\cong \left[V^{\delta,I}\xrightarrow{N}V^{\delta,I}(-1)\right]\]
where for a map of complexes $A\to B$ we write
$$[A\to B]:=\mathrm{holim}[A\to B]\cong \mathrm{Cone}(A\to B)[-1].$$
Applying these considerations to a decomposition \cite{deligne94}
\begin{equation} R\Gamma(X_{\etabar},\bq_l)\cong\bigoplus_{i\in\bz}H^i(X_{\etabar},\bq_l)[-i] \label{motdec}\end{equation}
we find
\begin{align*} R\Gamma(X_{\etahen},\bq_l)\cong &R\Gamma(I,R\Gamma(X_{\etabar},\bq_l))\cong \mathrm{holim} \left(R\Gamma(X_{\etabar},\bq_l)^{\delta,I}\xrightarrow{N}R\Gamma(X_{\etabar},\bq_l)^{\delta,I}(-1)\right)\end{align*}
for a certain nilpotent endomorphism $N$ of $R\Gamma(X_{\etabar},\bq_l)$ and
obtain the exact triangle
\begin{equation}  R\Gamma(X_{\sbar},\bq_l)\to  \left[R\Gamma(X_{\etabar},\bq_l)^{\delta,I}\xrightarrow{N}R\Gamma(X_{\etabar},\bq_l)^{\delta,I}(-1)\right]\to R\Gamma(X_{\sbar}, \bq_l)^*(-d)[-2d+1]  \to\label{tri3}\end{equation}
where we have used the duality
$R\Gamma(X_{\sbar}, f_{\sbar}^!\bq_l)\cong R\Gamma(X_{\sbar}, \bq_l)^*$
between homology and cohomology.

\subsubsection{Motivic decomposition} If $f$ is smooth then $I$ acts trivially, $N=0$, and the composite map
\[ \mathrm{sp}:R\Gamma(X_{\sbar},\bq_l)\to R\Gamma(X_{\etahen},\bq_l)\to R\Gamma(X_{\etabar},\bq_l)^{\delta,I}=R\Gamma(X_{\etabar},\bq_l)\]
is an isomorphism, i.e. gives a {\em splitting} of (\ref{tri1}) and (\ref{tri3}). For general regular $X$
it was shown in \cite{Flach-Morin-12}[Thm 10.1] that the monodromy weight conjecture \cite{illusie}[Conj. 3.9] implies that in each degree $i$ there is a short exact sequence
\[ 0\to Z^i\to H^i(X_{\sbar},\bq_l)\xrightarrow{sp} H^i(X_{\etabar},\bq_l)^I=H^i(X_{\etabar},\bq_l)^{\delta,I,N=0}\to 0\]
where $Z^i$ is pure of weight $i$. For each $i$ this gives a splitting of the short exact sequence
\[ 0\to H^{i-1}(X_{\etabar},\bq_l)(-1)^{\delta,I}/N\to H^i(X_{\etahen},\bq_l)\to H^i(X_{\etabar},\bq_l)^{\delta,I,N=0}\to 0\]
as well as a short exact sequence
\begin{equation} 0\to H^{i-1}(X_{\etabar},\bq_l)(-1)^{\delta,I}/N\to H^{2d-1-i}(X_{\sbar}, \bq_l(d))^*\to Z^{i+1}\to 0\label{dualpart}\end{equation}
using the long exact sequence induced by (\ref{tri3}).

\begin{prop} Assume the monodromy weight conjecture for the generic fibre of the regular scheme $X$ and set
\[V^i:=H^i(X_{\etabar},\bq_l).\]
Assume moreover that Frobenius acts semisimply on $H^i(X_{\sbar},\bq_l)$ for each $i$. Then the triangle (\ref{tri3}) is isomorphic to the direct sum over $i\in\bz$ of the $[-i]$-shift of the exact triangles
\begin{equation} Z^i[0]\oplus V^{i,\delta,I,N=0}[0]\to\left[V^{i,\delta,I}\xrightarrow{N} V^{i,\delta,I}(-1)\right]\to V^{i,\delta,I}(-1)/N[-1]\oplus Z^i[1]\to\label{tri5}\end{equation}
in the derived category of $W_k$-modules. Here we denote by $W_k\subset\Gal(\bar{k}/k)$ the Weil group of the finite field $k$.
\label{locmotivic}\end{prop}

\begin{proof} If $M$ denotes the monodromy filtration on $V:=V^{i,\delta,I}$ \cite{weilii}[1.6.1] we have an isomorphism of $W_k$-modules
\[ V\cong\bigoplus_{n\in\bz}\mathrm{Gr}^M_n V \]
since $\mathrm{Gr}^M_n V$ is pure of weight $n$ by the monodromy weight conjecture.  By \cite{weilii}[1.6.14.2, 1.6.14.3] there is an isomorphism of $W_k$-modules
\[ \mathrm{Gr}^M_n V \cong \bigoplus_{j\geq |n|\atop j\equiv n\,(2)} \mathrm{Gr}^M_{-j}(V^{N=0})(\frac{n+j}{2})\]
and $V^{N=0}$ is a quotient of $H^i(X_{\sbar},\bq_l)$, hence semisimple. We conclude that $W_k$ acts semisimply on $V=V^{i,\delta,I}$.  Since the cohomological dimension of the abelian category of $l$-adic sheaves on $s=\Spec(k)$ is equal to one, there exists a decomposition
\begin{equation} R\Gamma(X_{\sbar},\bq_l)\cong\bigoplus_{i\in\bz}H^i(X_{\sbar},\bq_l)[-i]\label{sdecomp}\end{equation}
in the derived category of $W_k$-modules. In the following diagram the unlabeled arrows form a commutative diagram induced by the truncation functors, $\sigma_0$ is the splitting given by (\ref{sdecomp}),
$\sigma_1$ is the splitting given by (\ref{motdec}), the (surjective) splitting $\sigma_2$ exists by semisimplicity of $V^{i,\delta,I}$ and the (injective) splitting $\sigma_3$ by semisimplicity of $H^i(X_{\sbar},\bq_l)$.
\[ \xymatrix{H^i(X_{\sbar},\bq_l)[-i]\ar[r]\ar[d] & V^{i,\delta,I,N=0}[-i]\ar[r]\ar@/_1pc/[l]_{\sigma_3} & V^{i,\delta,I}[-i]\ar[d]\ar@/_1pc/[l]_{\sigma_2}  \\
\tau^{\geq i}R\Gamma(X_{\sbar},\bq_l) \ar[rr]\ar@/_/[u]_{\sigma_0} && \tau^{\geq i}R\Gamma(X_{\etabar},\bq_l)^{\delta,I}\ar@/_/[u]_{\sigma_1}}.      \]
Write $\sigma_0=\sigma_0^Z\oplus\sigma_0^V$ corresponding to the decomposition of $W_k$-modules
\[ H^i(X_{\sbar},\bq_l)\cong Z^i\oplus V^{i,\delta,I,N=0}\]
 induced by $\sigma_3$. After replacing $\sigma_0$ by $\sigma_0^Z\oplus \sigma_3\circ\sigma_2\circ\sigma_1\circ\gamma$ where $\gamma=\tau^{\geq i}sp$ is the bottom horizontal arrow, the diagram of splittings commutes. A similar argument applies to (\ref{dualpart})$[-1]$. By an easy induction one finds that (\ref{tri3}) is isomorphic to the claimed direct sum of triangles.
\end{proof}

We record the following Corollary which is not needed in the rest of the paper.

\begin{cor} Under the assumptions of Prop. \ref{locmotivic} there exists a (noncanonical) decomposition
\begin{equation} Rf_*\bq_l\cong \bigoplus_{i\in\bz}R^if_*\bq_l[-i] \cong \bigoplus_{i\in\bz} j_*V^i[-i]\oplus \iota_*Z^i[-i]\notag\end{equation}
in the derived category of $l$-adic sheaves on $S$.
\label{sdecompo}\end{cor}

\begin{proof} The abelian category of $l$-adic sheaves on $S$ can be described as the category of diagrams $V_0\xrightarrow{sp}V_1$ where $V_0$ is an $l$-adic $G_k$-representation, $V_1$ a $G_K$-representation and $sp$ is $G_K$-equivariant. In this description, by proper base change, $R^if_*\bq_l$ is given by $H^i(X_{\sbar},\bq_l)\xrightarrow{sp}V^i$, $j_*V^i$ by $(V^i)^I\to V^i$ and $\iota_*Z^i$ by $Z^i\to 0$.
\end{proof}

\begin{rem} An alternative proof of Cor. \ref{sdecompo} might be obtained along the following lines. There is a perverse $t$-structure on the derived category of $l$-adic sheaves on separated, finite type $S$-schemes
\cite{illusie}[App.]. Assuming the monodromy weight conjecture it seems plausible that one can adapt the proof of the relative hard Lefschetz isomorphism
\[ \ell^i: {^p}H^{-i}Rf_*\bq_l[d]\xrightarrow{\sim}{^p}H^{i}Rf_*\bq_l[d](i) \]
from \cite{bbd}[Thm. 5.4.10].  As explained in \cite{deligne94} this implies a (noncanonical) direct sum decomposition
\[ Rf_*\bq_l[d]\xrightarrow{\sim}\bigoplus_{i\in\bz}({^p}H^{i}Rf_*\bq_l[d])[-i].\]
Since $\F:={^p}H^{i}Rf_*\bq_l[d]$ is a pure perverse sheaf one might be able to adapt the proof of the direct sum decomposition \cite{bbd}[5.3.11]
\[ \F\cong j_{*!}j^*\F\oplus \iota_*{^p}\iota^*\F\]
assuming semisimplicity of Frobenius. By proper base change $j^*\F\cong V^{i+d}$ and since $S$ is of dimension one we have $j_{*!}V^i\cong j_*V^i$. Moreover ${^p}\iota^*\F\cong Z^{i-1}[1]$.
\end{rem}

For any $l$-adic representation $V$ of $G_K:=\Gal(\bar{K}/K)$ recall the definition of $f$-cohomology of Bloch-Kato \cite{fon92}
\[ R\Gamma_f(K,V):=\left( V^I\xrightarrow{1-\phi}V^I\right)\cong R\Gamma(S, j_*V)\cong R\Gamma(s,\iota^*j_*V)\]
and the definition of $R\Gamma_{/f}(K,V)$ via the exact triangle
\[ R\Gamma_f(K,V)\to R\Gamma(K,V)\to R\Gamma_{/f}(K,V)\to\]
where $R\Gamma(K,V)=R\Gamma(\eta,V)$ is continuous Galois cohomology of $V$  and $\phi$ is the geometric Frobenius.

\begin{cor} Assume Conjecture \ref{zhongconj} and the assumptions of Prop. \ref{locmotivic} for the regular scheme $X$. Then for any $n\in\bz$ there is a (noncanonical) isomorphism of localization triangles
\[\minCDarrowwidth1em \begin{CD}
 R\Gamma(X,\bq_l(n)) @>\sim>> R\Gamma(s,Z^{2n}(n))[-2n]\oplus \bigoplus\limits_{i\in\bz} R\Gamma_f(K,V^i(n))[-i]\\
@VVV @VVV \\
 R\Gamma(X_{\eta},\bq_l(n))@>\sim >> \bigoplus\limits_{i\in\bz} R\Gamma(K,V^i(n))[-i]\\
@VVV @VVV\\
 R\Gamma(X_s, f_s^!\bq_l(n-d)[-2d+1])@>\sim >> R\Gamma(s,Z^{2n}(n))[-2n+1]\oplus\bigoplus\limits_{i\in\bz} R\Gamma_{/f}(K,V^i(n))[-i]\\
@VVV @VVV\end{CD}\]
where the Tate twist in the right hand column is defined in the usual way.
For $n<0$ all complexes in this diagram are acyclic.
\label{decompcor}\end{cor}

\begin{proof} The left vertical triangle is isomorphic to $R\Gamma(s,-)$ applied to the $(n)$-Tate-twist of (\ref{tri1}) which agrees with the higher Chow definition of the left hand column for $n\geq 0$ if we assume Conjecture \ref{zhongconj}. The statement now follows from our rewriting of (\ref{tri1}) as (\ref{tri3}) and $\bigoplus_{i\in\bz}(\ref{tri5})[-i]$ together with the fact $R\Gamma(s,Z^i(n))$ is acyclic for $i\neq 2n$ since $Z^i$ is pure of weight $i$. The acyclicity of the complexes in the right hand column for $n<0$ follows from an easy weight argument (see the Proof of Prop. \ref{nkleiner0} below). This then shows that the higher Chow definition of the left hand column also holds for $n<0$.

\end{proof}

The middle isomorphism in the diagram of Corollary \ref{decompcor} can be interpreted as a computation of $l$-adic motivic cohomology of the generic fibre from the geometric cohomology (fibre functor) $X\mapsto H^i(X_{\etabar},\bq_l)$ with its natural extra structure, i.e. the continuous $G_K$-action. Succinctly, one has the isomorphism
\[ R\Gamma(X_\eta,\bq_l(n))\cong R\Gamma(K,R\Gamma(X_{\etabar},\bq_l(n))).\]
Alternatively, one can compute $R\Gamma(X_\eta,\bq_l(n))$ from the slightly different fibre functor $X\mapsto H^i(X_{\etabar},\bq_l)^\delta$ with values in Weil-Deligne representations introduced above. One has
\begin{align} R\Gamma(X_{\eta},\bq_l(n))\cong &R\Gamma(s,R\Gamma(X_{\etahen},\bq_l(n)))\notag\\ \cong &\mathrm{holim} \left(\begin{CD}
R\Gamma(X_{\etabar},\bq_l)^{\delta,I} @>1-\phi_n>> R\Gamma(X_{\etabar},\bq_l)^{\delta,I}\\@V N VV @V N VV\\
R\Gamma(X_{\etabar},\bq_l)^{\delta,I}@>1-\phi_{n-1}>> R\Gamma(X_{\etabar},\bq_l)^{\delta,I}\end{CD}\right)\label{t1}\end{align}
where $\phi_r=\phi |k|^{-r}$.

\subsection{$p$-adic cohomology} The two computations of $l$-adic motivic cohomology of the generic fibre $X_\eta$ given at the end of the last section lead to different absolute cohomology theories for $l=p$. Since $p$ is invertible on $X_\eta$ and $X_\eta$ is smooth over a field, one still has the isomorphism $\bz(n)/p^\bullet\cong \mu_{p^\bullet}^{\otimes n}[0]$ on $X_{\eta,et}$. Hence the first
computation
\[ R\Gamma(X_\eta,\bq_p(n))\cong R\Gamma(G_K, R\Gamma(X_{\etabar},\bq_p(n)))\]
applies in the same way. The analogue of the fibre functor $X\mapsto H^i(X_{\etabar},\bq_l)^\delta$ with values in $l$-adic Weil-Deligne representation is the fibre functor
\[ X\mapsto  D_{pst}(H^i(X_{\etabar},\bq_p))\]
with values in (weakly admissible) filtered $(\phi,N,G_K)$-modules in the sense of \cite{fon2} (finite dimensional $K_0^{ur}$ vector spaces with operators $\phi$ and $N$ and a discrete $G_K$-action as well as a filtration on the scalar extension to $\bar{K}$. See also \cite{fon3} for a detailed discussion of these fibre functors both for $l=p$ and $l\neq p$). If $f$ is smooth then this fibre functor is isomorphic to crystalline cohomology of the special fibre $X_{\sbar}$, if $f$ is log-smooth it is isomorphic to log-crystalline, or Hyodo-Kato cohomology of $X_{\sbar}$ and for general $X_\eta$ (not necessarily smooth or proper) it was constructed by Beilinson \cite{beil13} from the log-smooth case by $h$-descent. Following \cite{nekniz13} we denote this functor by
\[ X\mapsto H_{HK}^{B,i}(X_{{\etabar},h})\cong D_{pst}(H^i(X_{\etabar},\bq_p)).\]
The corresponding absolute cohomology theory is log-syntomic cohomology of $X_\eta$ as defined by Niziol and Nekovar in
\cite{nekniz13}. It can be computed by a homotopy limit similar to (\ref{t1})
\begin{multline}R\Gamma_{syn}(X_\eta,n)\\\cong \mathrm{holim} \left(\begin{CD} R\Gamma_{HK}^B(X_{\etabar,h})^{G_K}
 @>(1-\phi_n,\iota_{dR})>>R\Gamma_{HK}^B(X_{\etabar,h})^{G_K}\oplus (R\Gamma_{dR}(X_{\etabar,h})/F^n)^{G_K} \\@V N VV @V (N,0) VV\\
R\Gamma_{HK}^B(X_{\etabar,h})^{G_K}@>1-\phi_{n-1}>>R\Gamma_{HK}^B(X_{\etabar,h})^{G_K}\end{CD}\right). \label{syndef}\end{multline}
There is a comparison map \cite{nekniz13}[Thm. A (4)]
\[R\Gamma_{syn}(X_\eta,n)\xrightarrow{\beta}R\Gamma(X_\eta,\bq_p(n))\]
but it only induces an isomorphism in degrees $i\leq n$.
\bigskip

\subsubsection{Localization triangles} We now discuss the localization exact triangles for both version of $p$-adic motivic cohomology.
We first establish a localization triangle for a fairly general regular scheme $\X$ which uses the definition of $\bz(n)$ as a cycle complex.

\begin{lem} Assume $\X$ is of finite type over a discrete valuation ring $D$ with perfect residue field $\kappa$ of characteristic $p$ and fraction field $F$ of characteristic $0$. Denote by
\[j:\X_{F}\to \X,\quad\quad i:\X_{\kappa}\to \X\]
the complementary open and closed immersions. If $\X$ is regular and satisfies Conjecture \ref{zhongconj}
then there is an exact triangle of complexes of sheaves on $\X_{et}$
\begin{equation} \tau^{\leq n-1}(i_*\bz(n-1)/p^\bullet)[-2]\to \bz(n)/p^\bullet\to \tau^{\leq n}Rj_*\mu_{p^\bullet}^{\otimes n}\to. \label{etloc}\end{equation}
If $\X_s$ is a normal crossing divisor, the truncation in front of the first term can be removed.
\label{loclem}\end{lem}

\begin{proof} We follow the argument in \cite{Geisser04a}[Proof of Thm. 1.2.1].  Since $\X$ is of finite type over a discrete valuation ring one has an exact localization triangle on $\X_{Zar}$ \cite{Geisser04a}[Cor. 3.3]
\begin{equation} i_*\bz(n-1)_{Zar}[-2]\to \bz(n)_{Zar}\to j_*\bz(n)_{Zar}\to \notag\end{equation}
and an isomorphism $\Gamma(V,\bz(n))\cong R\Gamma(V_{Zar},\bz(n))$ for any open subscheme $V\subseteq \X$ \cite{Geisser04a}[Thm.3.2 b)]. But this implies that
\[j_*\bz(n)_{Zar}\xrightarrow{\sim}Rj_*\bz(n)_{Zar}\]
since the map on stalks at $x\in \X$
\[ \varinjlim_{x\in V}\Gamma(V\cap \X_F,\bz(n))\to\varinjlim_{x\in V}R\Gamma((V\cap \X_F)_{Zar},\bz(n))\]
is an isomorphism. So we obtain a localization triangle in the Zariski topology
\begin{equation} i_*\bz(n-1)_{Zar}[-2]\to \bz(n)_{Zar}\to Rj_*\bz(n)_{Zar}\to \label{zarloc}\end{equation}
not only on $\X$ but, by the same argument, on any etale scheme $\X'\to \X$.
Let $\epsilon:\X_{et}\to \X_{Zar}$ be the morphism from the etale topos to the topos of Zariski sheaves on the category of etale schemes over $\X$ and use the same notation for $\X_F$ and $X_\kappa$. Note that $\epsilon_*$ is the inclusion of etale into Zariski sheaves and hence the identity map on objects whereas $\epsilon^*$ is etale sheafification. The identity $\epsilon^*\epsilon_*\F\cong \F$ for any etale sheaf $\F$ induces isomorphisms
\begin{equation}\epsilon^*\epsilon_*\F^\bullet\xrightarrow{\sim}\epsilon^*R\epsilon_*\F^\bullet\xrightarrow{\sim}\F^\bullet\label{adiso}\end{equation}
for any complex of etale sheaves $\F^\bullet$. In particular $\epsilon_*\bz(n)=\bz(n)_{Zar}$ and $\epsilon^*\bz(n)_{Zar}=\bz(n)$.
There is a commutative diagram of exact triangles on $\X_{et}$
\begin{equation}\begin{CD} \epsilon^*i_*\bz(n-1)_{Zar}[-2] @>>> \epsilon^*\bz(n)_{Zar} @>>>  \epsilon^*Rj_*\bz(n)_{Zar} @>>>{}\\
@VVV \Vert@. @VVV @.\\
 i_*Ri^!\bz(n) @>>> \bz(n) @>>>  Rj_*\bz(n) @>>> {}
\end{CD}\notag\end{equation}
where the top row is the pullback of (\ref{zarloc}) and the bottom row is the localization triangle in the etale topology. The vertical maps are induced by choosing a K-injective resolution $\bz(n)\to I(n)$. To see the right commutative diagram, start with the adjunction $\id\to j_*j^*$ in the category of complexes of sheaves and compose with $\epsilon^*\epsilon_*$. Applied to $\bz(n)$ we get
\[ \epsilon^*\bz(n)_{Zar}=\epsilon^*\epsilon_*\bz(n)\to \epsilon^*\epsilon_*j_*j^*\bz(n)=\epsilon^*j_*\epsilon_*j^*\bz(n)=\epsilon^*j_*\bz(n)_{Zar}\] and applied to $I(n)$ we get
\[\bz(n)\cong \epsilon^*\epsilon_*I(n)\to \epsilon^*\epsilon_*j_*j^*I(n)=\epsilon^*\epsilon_*Rj_*\bz(n)\cong Rj_*\bz(n).\]
The left commutative diagram is similarly obtained by applying $\epsilon^*\epsilon_*i_*i^!\to \epsilon^*\epsilon_*$ to $\bz(n)\to I(n)$. Taking mapping cones of multiplication by $p^\bullet$ we obtain the diagram
\begin{equation}\begin{CD} \epsilon^*i_*\bz(n-1)_{Zar}/p^\bullet[-2] @>>> \epsilon^*\bz(n)_{Zar}/p^\bullet @>>>  \epsilon^*j_*\bz(n)_{Zar}/p^\bullet @>>>{}\\
@VVV \Vert@. @VVV @.\\
 i_*Ri^!\bz(n)/p^\bullet @>>> \bz(n)/p^\bullet @>>>  Rj_*\bz(n)/p^\bullet @>>> {}.
\end{CD}\label{geisserdia}\end{equation}
By the Rost-Voevodsky theorem (previously Beilinson-Lichtenbaum conjecture, see e.g. \cite{Zhong14}[Thm. 2.5]) on $\X_F$, the adjunction
\[ \bz(n)_{Zar}/p^\bullet=\epsilon_*\bz(n)/p^\bullet\to \epsilon_*I(n)/p^\bullet=R\epsilon_*\bz(n)/p^\bullet\]
induces a quasi-isomorphism
\[ \bz(n)_{Zar}/p^\bullet\cong\tau^{\leq n}R\epsilon_*\bz(n)/p^\bullet.\]
By \cite{Zhong14}[Lemma 2.4] we obtain a quasi-isomorphism
\[ \tau^{\leq n}Rj_*\bz(n)_{Zar}/p^\bullet\cong\tau^{\leq n}Rj_*\tau^{\leq n}R\epsilon_*\bz(n)/p^\bullet\cong \tau^{\leq n}Rj_*R\epsilon_*\bz(n)/p^\bullet=\tau^{\leq n}R\epsilon_*Rj_*\bz(n)/p^\bullet\]
and hence an isomorphism
\[ \tau^{\leq n}\epsilon^*Rj_*\bz(n)_{Zar}/p^\bullet\cong\tau^{\leq n}\epsilon^*R\epsilon_*Rj_*\bz(n)/p^\bullet\cong\tau^{\leq n}Rj_*\bz(n)/p^\bullet,\]
i.e. the right vertical map in (\ref{geisserdia}) is an isomorphism in degrees $\leq n$.
From the Five Lemma and $\tau^{\leq n}\bz(n)/p^\bullet\cong\bz(n)/p^\bullet$ it follows that the truncation of the left vertical map in (\ref{geisserdia})
\[\tau^{\leq n+1}\epsilon^*i_*\bz(n-1)_{Zar}/p^\bullet[-2]\to \tau^{\leq n+1} i_*Ri^!\bz(n)/p^\bullet\]
is a quasi-isomorphism and that there is an exact triangle
\[ \tau^{\leq n+1}(i_*\bz(n-1)/p^\bullet[-2])\to\bz(n)/p^\bullet\to \tau^{\leq n}Rj_*\bz(n)/p^\bullet\to \]
using $\epsilon^*i_*\bz(n-1)_{Zar}/p^\bullet\cong i_*\epsilon^*\bz(n-1)_{Zar}/p^\bullet\cong i_*\bz(n-1)/p^\bullet$.
Using Lemma \ref{invertible} we have an isomorphism
\begin{equation}\bz(n)/p^\bullet\cong\mu_{p^\bullet}^{\otimes n}\label{generickummer}\end{equation}
on $\X_F$ and we get the exact triangle (\ref{etloc}).

Recall that if $Z$ is a separated, finite type scheme over a perfect field $\kappa$ of characteristic $p$ all of whose irreducible components are of dimension $d-1$ and $r\geq 0$ there is a quasi-isomorphism on $Z_{et}$ \cite{Zhong14}[Thm. 1.1]
\[ \bz(r)/p^\bullet=\bz^c(d-1-r)/p^\bullet[-2d+2]\cong\tilde{\nu}_{Z,\bullet}(d-1-r)[-2d+2]\]
where $\tilde{\nu}_{Z,\bullet}(d-1-r)[-2d+2]$ is the Gersten complex of logarithmic deRham-Witt sheaves (see \cite{sato07-2}[1.7])
\[ \bigoplus_{x\in Z^0}i_{x,*}W_\bullet\Omega^r_{x,log} \xrightarrow{(-1)^r\partial} \bigoplus_{x\in Z^1}i_{x,*}W_\bullet\Omega^{r-1}_{x,log} \xrightarrow{(-1)^r\partial}\cdots\]
concentrated in degrees $[r,2r]$. Note that these complexes are identical for $Z$ and $Z_{red}$. If $\X_s$ is a normal crossing divisor, i.e. $\X$ is semistable, then this complex is cohomologically concentrated in degree $r$ by \cite{sato07-2}[Cor. 2.2.5]. Hence the truncation in front of the first term in (\ref{etloc}) can be removed.
\end{proof}

The following consequence of Lemma \ref{loclem} is only needed in section \ref{sec:compatibility} in the main text.

\begin{lem} With notation and assumptions as in Lemma \ref{loclem} denote by $\hat{D}$ the $p$-adic completion of $D$ and by $g:\X_{\hat{D}}\to\X$ the natural (flat) morphism. Then the flat pullback on higher Chow complexes \cite{bloch86}
\[ g^*\bz(n)/p^\bullet \to \bz(n)/p^\bullet\]
is a quasi-isomorphism.
\label{gloclem}\end{lem}

\begin{proof}  We first prove the following general base change result for torsion sheaves.

\begin{lem} Let $\X$ be of finite type over a Dedekind $D$ ring with fraction field of characteristic zero. Let $p$ be a prime number and denote by $\hat{D}$ the $p$-adic completion of $D$. Consider the Cartesian diagram
\[\begin{CD} \X_{\hat{D}}[1/p] @>\hat{j}>> \X_{\hat{D}} @<\hat{i}<< \X_{\bF_p}\\ @V\tilde{g}VV @VgVV \Vert@.\\
\X[1/p] @>j>> \X @<i<< \X_{\bF_p}.     \end{CD}\]
Then for any complex of sheaves $\F$ on $\X_{et}$ with torsion cohomology, the base change morphism
\begin{equation} \beta: g^*Rj_*j^*\F\to R\hat{j}_*\tilde{g}^*j^*\F\notag\end{equation}
as wells the natural morphism
\begin{equation} \hat{i}^*\alpha:Ri^!\F\to R\hat{i}^!g^*\F\notag\end{equation}
are quasi-isomorphisms.
\label{gloclem2}\end{lem}

\begin{proof} There is a commutative diagram with exact rows
\[\begin{CD}
\hat{i}_*R\hat{i}^!g^*\F@>>> g^*\F@>>> R\hat{j}_*\hat{j}^*g^*\F @>>>{}\\
@AA\alpha A\Vert@.@AA\beta A @.\\
g^*i_*Ri^!\F @>>> g^*\F@>>> g^*Rj_*j^*\F@>>>{}
\end{CD}\]
where the top row is the localization triangle for $g^*\F$ on $\X_{\hat{D}}$ and the bottom row is the pullback of the localization triangle for $\F$ on $\X$. The right square commutes since both maps are adjoint to the same map and similarly for the left square. The stalk of $\beta$ at a geometric point $p:\Spec(\bar{x})\to\X_{\hat{D}}$ is an isomorphism if $x\in \X_{\hat{D}}[1/p]$, i.e. $p=\hat{j}p'$, in view of the isomorphism
\begin{align*} p^*g^*Rj_*j^*\F\cong &(p')^*\hat{j}^*g^*Rj_*j^*\F\cong (p')^*\tilde{g}^*j^*Rj_*j^*\F\\
\cong &(p')^*\tilde{g}^*j^*\F\cong (p')^*\hat{j}^*R\hat{j}_*\tilde{g}^*j^*\F\\\cong &p^*R\hat{j}_*\tilde{g}^*j^*\F.
\end{align*}
For $x\in \X_{\bF_p}$ the stalk of $\beta$ is the map
\begin{equation} R\Gamma(A[\frac{1}{p}]_{et},\F)\to R\Gamma(A'[\frac{1}{p}]_{et},\F)\label{stalk}\end{equation}
where $A$ (resp $A'$) is the strict Henselization of $\X$ (resp. $\X_{\hat{D}}$) at $\bar{x}$. By the definition and elementary properties of the notion of Henselian pair \cite{egaiv}[(18.5.5), (18.5.6)] it follows that $(A,(p))$ and $(A',(p))$ are Henselian pairs. By the Gabber-Fujiwara formal base change theorem \cite{fujiwara95}[Cor. 6.6.4] the restriction map
\[R\Gamma(A[\frac{1}{p}]_{et},\F)\to R\Gamma(\hat{A}[\frac{1}{p}]_{et},\F)\]
is a quasi-isomorphism, where $\hat{A}$ is the $p$-adic completion of $A$. The same holds for $A'$ and we have $\hat{A}\cong \widehat{A'}$. Hence (\ref{stalk}) and therefore $\beta$ are quasi-isomorphisms. This implies that $\alpha$ is a quasi-isomorphism which proves Lemma \ref{gloclem2}.
\end{proof}

We continue with the proof of Lemma \ref{gloclem}. There is a commutative diagram with exact rows where the top row is (\ref{etloc}) on $\X_{\hat{D}}$, the bottom row is the pullback of (\ref{etloc}) on $\X$ and the middle row is the truncated localization triangle for $g^*\bz(n)/p^\bullet$ on $\X_{\hat{D}}$.
\[\begin{CD}\tau^{\leq n-1}(\hat{i}_*\bz(n-1)/p^\bullet)[-2]@>>> \bz(n)/p^\bullet@>>> \tau^{\leq n}R\hat{j}_*\mu_{p^\bullet}^{\otimes n}@>>>{}\\
@AA\alpha' A@AA\gamma A@AA\beta' A @.\\
\tau^{\leq n+1}\hat{i}_*R\hat{i}^!g^*\bz(n)/p^\bullet@>>> g^*\bz(n)/p^\bullet@>>> \tau^{\leq n}R\hat{j}_*\hat{j}^*g^*\bz(n)/p^\bullet @>>>{}\\
@AA\tau^{\leq n+1}\alpha A\Vert@.@AA\tau^{\leq n}\beta A @.\\
g^*\tau^{\leq n-1}(i_*\bz(n-1)/p^\bullet)[-2]@>>> g^*\bz(n)/p^\bullet@>>> g^*\tau^{\leq n}Rj_*\mu_{p^\bullet}^{\otimes n}@>>>{}
\end{CD}\]
The maps from the middle to the top row form a commutative diagram by functoriality of the (truncated) localization triangle. By Lemma \ref{gloclem2} for $\F:=\bz(n)/p^\bullet$ the middle and bottom row are quasi-isomorphic. Since the base change of $g$ to $\kappa$ is an isomorphism, the map $\alpha'\circ\tau^{\leq n+1}\alpha$ is an isomorphism. It follows that $\alpha'$ is a quasi-isomorphism. Since
\begin{equation}\hat{j}^*g^*\bz(n)/p^\bullet\cong \tilde{g}^*j^*\bz(n)/p^\bullet\cong \tilde{g}^*\mu_{p^\bullet}^{\otimes n}\cong \mu_{p^\bullet}^{\otimes n}\notag\end{equation}
the map $\beta'$ is a quasi-isomorphism and we deduce the same for $\gamma$.
\end{proof}

For the discussion below we find it most convenient to isolate the following statement. Lemma \ref{loclem} shows that it holds in the semistable case under Conjecture \ref{zhongconj}. Unfortunately we cannot prove it in the general regular case even assuming Gersten's conjecture.

\begin{conj} Assume $\X$ is of finite type over a discrete valuation ring $D$ with perfect residue field $\kappa$ of characteristic $p$ and fraction field $F$ of characteristic $0$. Denote by
\[j:\X_{F}\to \X,\quad\quad i:\X_{\kappa}\to \X\]
the complementary open and closed immersions. If $\X$ is regular then there is an exact triangle of complexes of sheaves on $\X_{et}$
\begin{equation} i_*\bz(n-1)/p^\bullet[-2]\to \bz(n)/p^\bullet\to \tau^{\leq n}Rj_*\mu_{p^\bullet}^{\otimes n}\to. \label{etloc2}\end{equation}
\label{locconj}\end{conj}

We now return to the local setting. Assuming Conjecture \ref{locconj} we obtain a commutative diagram of exact localization triangles where the top row is induced by (\ref{etloc2}) and the bottom row is the usual localization triangle in the etale topology.
\[ \begin{CD} R\Gamma(X_s,\bq_p(n-1))[-2] @>>>  R\Gamma(X,\bq_p(n))@>>>  R\Gamma(X,\tau^{\leq n}Rj_*\bq_p(n)) \to\\@VVV \Vert  @. @VVV\\ R\Gamma(X_s, Ri^! \bq_p(n))  @>>>  R\Gamma(X,\bq_p(n))@>>>  R\Gamma(X_{\eta},\bq_p(n))\to\end{CD}\]

\begin{conj} For regular $X$ and $n\geq 0$ the period map
\[R\Gamma_{syn}(X_{\eta},n)\xrightarrow{\beta} R\Gamma(X_{\eta},\bq_p(n))\]
of \cite{nekniz13}[Thm. A (4)] factors through an isomorphism
\[R\Gamma_{syn}(X_{\eta},n) \cong R\Gamma(X,\tau^{\leq n}Rj_*\bq_p(n)).\]
\label{synvan}\end{conj}

\begin{prop} If $X$ is strictly semistable (in the sense of \cite{colniz15}) then Conjecture \ref{synvan} holds. \label{niziol}\end{prop}

\begin{proof} By \cite{colniz15}[Thm.1.1] there is a morphism
\[ \alpha_{n,\bullet}^{FM}:\mathscr S_\bullet(n)_X\to i^*\tau^{\leq n}Rj_*\mu_{p^\bullet}^{\otimes n}\]
whose kernel and cokernel are annihilated by a fixed power of $p$, and where $\mathscr S_\bullet(n)_X$ is the log-syntomic complex
of Fontaine-Messing-Kato \cite{kato94}. Hence
\[ (\mathrm{holim}_\bullet R\Gamma(X_s,\mathscr S_\bullet(n)_X))_\bq\cong (\mathrm{holim}_\bullet R\Gamma(X_s,i^*\tau^{\leq n}Rj_*\mu_{p^\bullet}^{\otimes n})_\bq\cong R\Gamma(X,\tau^{\leq n}Rj_*\bq_p(n)).\]
By \cite{nekniz13}[Thm. 3.8] there is an isomorphism
\[\alpha_{syn}:(\mathrm{holim}_\bullet R\Gamma(X,\mathscr S_\bullet(n)_X))_\bq \cong R\Gamma_{syn}(X_{\eta},n) \]
where $R\Gamma_{syn}(X_{\eta},n)$ is given by the homotopy limit (\ref{syndef}).
\end{proof}

\begin{cor} For $X$ semistable satisfying Conjecture \ref{zhongconj} and $n\geq 0$ there is an exact localization triangle
\[ R\Gamma(X_s,\bq_p(n-1))[-2] \to R\Gamma(X,\bq_p(n)) \to  R\Gamma_{syn}(X_{\eta},n)\to.\]
\label{absloc}\end{cor}

\begin{proof} Combine Lemma \ref{loclem} and Prop. \ref{niziol}. \end{proof}

\begin{rem} The natural map $\bz(n)\to \bz(n)/p^\bullet$ of (pro)-complexes of sheaves on $X_{et}$ and $X_{\eta,et}$ induces a commutative diagram
\[\begin{CD} H^i(X,\bz(n))_\bq @>\iota >> H^i(X_\eta,\bz(n))_\bq\\
@VVV @VVc V \\ H^i(X,\bq_p(n)) @>>> H^i(X_\eta,\bq_p(n)).\end{CD}\]
The Chern class maps from $K$-theory to motivic cohomology \cite{bloch86} induce an isomorphism
\[K_{2n-i}(X_\eta)^{(n)}_\bq\cong H^i(X_{\eta,Zar},\bz(n))_\bq\cong H^i(X_\eta,\bz(n))_\bq  \]
whose composite with $c$ is the etale Chern class map $c^{et}$. By \cite{nekniz13}[Thm. A (7)]
$c^{et}$ factors through $H^i_{syn}(X_\eta, n)$, hence so does the composite map $c\iota$. Corollary \ref{absloc} then gives another proof of the factorization
\[K_{2n-i}(X)^{(n)}_\bq\cong H^i(X,\bz(n))_\bq\to H^i_{syn}(X_\eta, n)\to H^i(X_\eta,\bq_p(n))  \]
in the semistable case.
\end{rem}

 Concerning a syntomic description of $R\Gamma(X,\bq_p(n))$ we expect the following. The geometric cohomology theory for arbitrary varieties $Y/k$ is rigid cohomology \cite{ber86,ber97}
 $$Y\mapsto H^i_{rig}(Y/K_0)$$
 taking values in the (Tannakian) category of $\phi$-modules (finite dimensional $K_0$-vector spaces with a Frobenius-semilinear endomorphism $\phi$ \cite{fon2}[4.2]). We expect the following $p$-adic analogue of (\ref{tri3}) relating the geometric cohomology of the special and the generic fibre.

\begin{conj} For regular $X$ there is an exact triangle in the derived category of $\phi$-modules
\begin{multline}R\Gamma_{rig}(X_s/K_0)\xrightarrow{sp}\left[R\Gamma_{HK}^B(X_{\etabar,h})^{G_K}\xrightarrow{N}R\Gamma_{HK}^B(X_{\etabar,h})(-1)^{G_K}\right]\xrightarrow{sp'} \\R\Gamma_{rig}(X_s/K_0)^*(-d)[-2d+1]\to\label{ploctri}\end{multline}
where $sp$ induces the specialization map constructed in \cite{wu2014} and $sp'$ is the composite of the Poincare duality isomorphism
\[R\Gamma_{HK}^B(X_{\etabar,h})(-1)\cong R\Gamma_{HK}^B(X_{\etabar,h})^*(-d)[-2d+2] \]
on $X_{\etabar}$ and $sp^*$.
\label{ploc}\end{conj}

We expect the following relation between rigid cohomology with compact support \cite{ber86} and $p$-adic motivic cohomology with compact support as defined in \cite{Geisser06} and between the dual of rigid cohomology with compact support and $p$-adic motivic Borel-Moore homology as defined in (\ref{compdef}) above.

\begin{conj} a) For a separated, finite type $k$-scheme $Y$ there exists an isomorphism
\[R\Gamma_c(Y_{eh},\bq_p(n))\xrightarrow{\sim}\left[R\Gamma_{rig,c}(Y/K_0) \xrightarrow{1-\phi_n}R\Gamma_{rig,c}(Y/K_0)\right].\]
b) For a separated, finite type $k$-scheme $Y$ there exists an isomorphism
\[R\Gamma(Y,\bq_p(n))\xrightarrow{\sim}\left[R\Gamma_{rig,c}(Y/K_0)^*
 \xrightarrow{1-\phi_{n-d+1}}R\Gamma_{rig,c}(Y/K_0)^*\right][-2d+2]. \]
\label{pdescent}\end{conj}

\begin{cor} Assume $X$ is regular and satisfies Conjectures \ref{locconj}, \ref{synvan}, \ref{ploc} and \ref{pdescent}b) for $Y=X_s$ and that the bottom square in the diagram below commutes. Then for $n\geq 0$ there is an isomorphism of localization triangles
\[\minCDarrowwidth1em\begin{CD}
R\Gamma(X,\bq_p(n))@>\sim>> \left[R\Gamma_{rig}(X_s/K_0)
 \xrightarrow{(1-\phi_n,sp')}R\Gamma_{rig}(X_s/K_0)\oplus R\Gamma_{dR}(X_\eta)/F^n\right]\\
 @VVV @VVV\\
R\Gamma_{syn}(X_\eta,n)@>\sim>> \left[\minCDarrowwidth1em\begin{CD} R\Gamma_{HK}^B(X_{\etabar,h})^{G_K}
 @>(1-\phi_n,\iota_{dR})>>R\Gamma_{HK}^B(X_{\etabar,h})^{G_K}\oplus R\Gamma_{dR}(X_\eta)/F^n\\@V N VV @V (N,0) VV\\
R\Gamma_{HK}^B(X_{\etabar,h})^{G_K}@>1-\phi_{n-1}>>R\Gamma_{HK}^B(X_{\etabar,h})^{G_K}\end{CD}\right] \\
 @VVV @VVV\\
 R\Gamma(X_s,\bq_p(n-1))[-1]@>\sim>>\left[R\Gamma_{rig}(X_s/K_0)^*
 \xrightarrow{1-\phi_{n-d}}R\Gamma_{rig}(X_s/K_0)^*\right][-2d+1] \end{CD}\]
where $sp'=\iota_{dR}\circ sp$.
\label{pcor}\end{cor}

 Corollary \ref{pcor} implies Conjecture  ${\bf D}_p(\mathcal{X},n)$  (Conj. \ref{conjD_p} in section \ref{sec:localfactor}) in the presence of Conjecture \ref{pdescent}a). This rather indirect way of obtaining a syntomic description of $R\Gamma(X,\bq_p(n))$ goes back to \cite{Geisser04a} for smooth $f$ and $0\leq n<p-1$.   A more natural way to obtain a syntomic description of $R\Gamma(X,\bq_p(n))$ would be to construct a cycle class map with values in syntomic cohomology, following the construction of the etale cycle class map in  \cite{geisserlevine01}.

\begin{prop} Corollary \ref{pcor} holds unconditionally for $n<0$. More precisely, all complexes in the right hand column are acyclic, whereas the complexes in the left hand column are acyclic by definition (for $R\Gamma(X,\bq_p(n))$ we can use the higher Chow definition or that of section \ref{sect-emc}).
\label{nkleiner0}\end{prop}

\begin{proof} Since $X_s$ is proper of dimension $d-1$ the eigenvalues of $\phi_n^{[k:\bF_p]}$ on any $H^i_{rig}(X_s/K_0)$ are Weil numbers of weight $w$ in the range $-2n\leq w\leq 2(d-1)-2n$ \cite{nakk}. Hence for $n<0$ the eigenvalue $1$ of weight $w=0$ cannot occur. Similarly, the eigenvalues of $\phi_{n-d}^{[k:\bF_p]}$ on any $H^i_{rig}(X_s/K_0)^*$ have weight $w$ in the range $-2(n-d)-2(d-1)\leq w\leq -2(n-d)$ and $w=0$ cannot occur if $n<0$. Together with the fact that  $R\Gamma_{dR}(X_\eta)=F^0 R\Gamma_{dR}(X_\eta)=F^nR\Gamma_{dR}(X_\eta)$ for $n<0$ this implies that two of the three complexes in the right hand column are acyclic, hence so is the third. It remains to show that $R\Gamma(X,\bq_p(n))$ as defined in section \ref{sect-emc} is acyclic. We have
\[ R\Gamma(X,\bz(n)/p^\bullet)=R\Gamma(X,j_{X,!}\mu_{p^\bullet}^{\otimes n})\cong R\Gamma(S,Rf_*j_{X,!}\mu_{p^\bullet}^{\otimes n})
\cong R\Gamma(S,j_!Rf_{\eta*}\mu_{p^\bullet}^{\otimes n})=0\]
using the fact that $f$ is proper, i.e. $Rf_*=Rf_!$, and the vanishing of $R\Gamma(S,j_!\F)$ for any sheaf $\F$ \cite{milduality}[Prop. II.1.1].
\end{proof}

\subsubsection{Motivic decomposition} Define the $\phi$-module
\[ Z^i:=\ker\left(H^i_{rig}(X_s/K_0)\xrightarrow{sp} H_{HK}^{B,i}(X_{\etabar,h})\right).\]

 \begin{prop} Assume the $p$-adic monodromy weight conjecture \cite{mok93} for the generic fibre of the regular scheme $X$ and
assume moreover that the $\phi$-module $H^i_{rig}(X_s/K_0)$ is semisimple for each $i$. Assume Conjecture \ref{ploc} holds for $X$. Then the triangle (\ref{ploctri}) is the direct sum over $i\in\bz$ of the $[-i]$-shift of the exact triangles
\begin{multline} Z^i[0]\oplus H_{HK}^{B,i}(X_{\etabar,h})^{G_K,N=0}[0]\to\left[H_{HK}^{B,i}(X_{\etabar,h})^{G_K}\xrightarrow{N} H_{HK}^{B,i}(X_{\etabar,h})(-1)^{G_K}\right]\to \\H_{HK}^{B,i}(X_{\etabar,h})(-1)^{G_K}/N[-1]\oplus Z^i[1]\to\label{ptri5}\end{multline}
in the derived category of $\phi$-modules. Moreover $Z^i$ is pure of weight $i$.
\label{plocmotivic}\end{prop}

\begin{proof} The proof of \cite{Flach-Morin-12}[Thm 10.1] consists in applying the exact weight filtration functor to the long exact sequence induced by (\ref{tri1}), using the fact that $H^i(X_{\sbar},\bq_l)$ has weights $\leq i$ since $X_s$ is proper, together with the monodromy weight conjecture. In view of \cite{nakk} these arguments are available to show surjectivity of
\[H^i_{rig}(X_s/K_0)\xrightarrow{sp} H_{HK}^{B,i}(X_{\etabar,h})^{G_K,N=0}\]
as well as the fact that $Z^i$ is pure of weight $i$, assuming the $p$-adic monodromy weight conjecture. We can then follow the proof of Prop. \ref{locmotivic}, using the fact that the category of $\phi$-modules has global dimension one, and that the motivic decomposition
\begin{equation} R\Gamma(X_{\etabar},\bq_p)\cong\bigoplus_{i\in\bz}H^i(X_{\etabar},\bq_p)[-i] \label{pmotdec}\end{equation}
induces a motivic decomposition of
\[R\Gamma_{HK}^B(X_{\etabar,h})\cong D_{pst}(R\Gamma(X_{\etabar},\bq_p)).\]
\end{proof}

For any $\phi$-module $D$ set
\[ R(\phi,D):=[D\xrightarrow{1-\phi} D]=\mathrm{holim}(D\xrightarrow{1-\phi} D).\]
For any $p$-adic representation $V$ of $G_K$ recall the definition of $f$-cohomology of Bloch-Kato and Fontaine Perrin-Riou \cite{fon92}
\[R\Gamma_f(K,V):=\left( D_{cris}(V)\xrightarrow{(1-\phi,\subseteq)}D_{cris}(V)\oplus D_{dR}(V)/F^0\right)\]
and the definition of $R\Gamma_{/f}(K,V)$ via the exact triangle
\[ R\Gamma_f(K,V)\to R\Gamma(K,V)\to R\Gamma_{/f}(K,V)\to.\]
Define
\[ R\Gamma_{st}(K,V):=\left[\minCDarrowwidth1em\begin{CD}D_{st}(V) @>(1-\phi,\subseteq) >>D_{st}(V)\oplus D_{dR}(V)/F^0\\
@V N VV @V N VV\\
D_{st}(V(-1)) @>1-\phi >>D_{st}(V(-1))\end{CD}
\right]\]
so that there is an exact triangle
\[ R\Gamma_f(K,V)\to R\Gamma_{st}(K,V)\to R(\phi,D_{st}(V(-1))/N)[-1]\to.\]
Setting
\[V^i:=H^i(X_{\etabar},\bq_p)\]
we have an isomorphism of $(\phi,N)$-modules
\[ D_{st}(V^i)\cong H_{HK}^{B,i}(X_{\etabar,h})^{G_K}\]
and of $\phi$-modules
\[ D_{cris}(V^i)\cong H_{HK}^{B,i}(X_{\etabar,h})^{G_K,N=0}.\]

\begin{cor} Let $X$ be a regular scheme satisfying the assumptions of Cor. \ref{pcor} and of Prop. \ref{plocmotivic}. Then for any $n\in\bz$  there is an isomorphism of localization triangles
\[\minCDarrowwidth1em \begin{CD}
 R\Gamma(X,\bq_p(n)) @>\sim>> R(\phi,Z^{2n}(n))[-2n]\oplus \bigoplus\limits_{i\in\bz} R\Gamma_f(K,V^i(n))[-i]\\
@VVV @VVV \\
 R\Gamma_{syn}(X_\eta,n)@>\sim >> \bigoplus\limits_{i\in\bz} R\Gamma_{st}(K,V^i(n))[-i]\\
@VVV @VVV\\
 R\Gamma(X_s,\bq_p(n-1))[-1]@>\sim >> R(\phi,Z^{2n}(n))[-2n+1]\oplus\bigoplus\limits_{i\in\bz}R(\phi, D_{st}(V^i(n-1))/N)[-i-1] \\
@VVV @VVV\end{CD}\]
and an isomorphism of localization triangles
\[\minCDarrowwidth1em \begin{CD}
 R\Gamma(X,\bq_p(n)) @>\sim>> R(\phi,Z^{2n}(n))[-2n]\oplus \bigoplus\limits_{i\in\bz} R\Gamma_f(K,V^i(n))[-i]\\
@VVV @VVV \\
 R\Gamma(X_{\eta},\bq_p(n))@>\sim >> \bigoplus\limits_{i\in\bz} R\Gamma(K,V^i(n))[-i]\\
@VVV @VVV\\
 R\Gamma(X_s, Ri^!\bq_p(n))[1]@>\sim >> R(\phi,Z^{2n}(n))[-2n+1]\oplus\bigoplus\limits_{i\in\bz} R\Gamma_{/f}(K,V^i(n))[-i]\\
@VVV @VVV\end{CD}\]
\label{pdecompcor}\end{cor}

\begin{prop} If $f:X\to S$ is smooth then Conjectures \ref{zhongconj}, \ref{synvan}, \ref{ploc}, \ref{pdescent}
hold true, if in \ref{pdescent}a) we replace $R\Gamma(X_{s,eh},\bq_p(n))$ by (see notation \ref{ehprime}) $$R\Gamma_{eh}'(X_{s},\bq_p(n)):=R\Gamma(X_{s,et},\bq_p(n)).$$
Moreover, the conclusions of Corollary \ref{pcor}, Prop. \ref{plocmotivic}, and Cor. \ref{pdecompcor} hold true, in particular there is an isomorphism
\begin{equation} R\Gamma(X,\bq_p(n))\cong\left[ R\Gamma_{cris}(X_s/K_0)\xrightarrow{(1-\phi_n,sp')}R\Gamma_{cris}(X_s/K_0)\oplus R\Gamma_{dR}(X_\eta)/F^n\right].\label{syn}\end{equation}
and Conjecture ${\bf D}_p(\mathcal{X},n)$ in section \ref{sec:localfactor} holds true. 
\label{psmooth}\end{prop}

\begin{proof} Conjecture \ref{zhongconj} holds by \cite{Geisser04a}[Cor. 4.4], Conjecture \ref{synvan} by Prop. \ref{niziol} and Conjecture \ref{ploc} is trivial since $sp$ is in this case the isomorphism
\[ R\Gamma_{rig}(X_s/K_0)\cong R\Gamma_{cris}(X_s/K_0)\cong R\Gamma_{HK}(X_s/K_0)\cong R\Gamma_{HK}^B(X_{\etabar,h})^{G_K} \]
and $N=0$. The theory of the deRham-Witt complex \cite{illusie79}[I.5.7.2] gives a short exact sequence
\[ 0\to W_\bullet\Omega^n_{X_s,log}\to W_\bullet\Omega^n_{X_s}\xrightarrow{1-F} W_\bullet\Omega^n_{X_s}\to 0\]
where $F$ is the Frobenius on the deRham-Witt complex, and an isomorphism
\[ R\Gamma_{cris}(X_s/K_0)\cong \bigoplus_{j=0}^{d-1} R\Gamma(X_s,W_\bullet\Omega_{X_s}^j)_\bq[-j] \]
where $\phi_j=\phi p^{-j}$ on the left hand side induces $F$ on $R\Gamma(X_s,W_\bullet\Omega_{X_s}^j)_\bq$.
In each degree $i$ the decomposition of the $\phi$-module
\[ H^i_{cris}(X_s/K_0)\cong \bigoplus_{j=0}^{d-1} H^{i-j}(X_s,W_\bullet\Omega_{X_s}^j)_\bq\]
is such that the slopes of $H^{i-j}(X_s,W_\bullet\Omega_{X_s}^j)_\bq$ lie in the interval $[j,j+1)$. In particular, $\phi$ is divisible by $p^j$, and $F=\phi p^{-j}$ has no eigenvalue $1$ for $n\neq j$. Hence
\begin{align} &\left[ R\Gamma_{cris}(X_s/K_0)\xrightarrow{1-\phi_n}R\Gamma_{cris}(X_s/K_0)\right]\label{pdescent1}\\
\cong &\left[ \left(\mathrm{holim}\  R\Gamma(X_s,W_\bullet\Omega_{X_s}^n)\right)_\bq[-n] \xrightarrow{1-F}\left(\mathrm{holim}\ R\Gamma(X_s,W_\bullet\Omega_{X_s}^n)\right)_\bq[-n]\right]\notag\\
\cong & \left(\mathrm{holim}\ R\Gamma(X_s,W_\bullet\Omega_{X_s,log}^n)\right)_\bq[-n]\notag\\
\cong & R\Gamma(X_s,\bq_p(n))\notag
\end{align}
where this last isomorphism follows from the isomorphism of \'etale sheaves on $X_s$ \cite{Geisser-Levine-00}[Thm. 8.5]
\begin{equation} \bz(n)/p^\bullet\cong W_\bullet\Omega^n_{X_s,log}[-n].\label{motcomp}\end{equation}
for $n\geq 0$. This gives Conjecture \ref{pdescent}a) with $R\Gamma(X_{s,eh},\bq_p(n))$ replaced by $R\Gamma(X_{s,et},\bq_p(n))$. Conjecture \ref{pdescent}b) follows from Conj. \ref{pdescent}a), Poincare duality for $R\Gamma_{cris}(X_s/K_0)$, Milne's duality \cite{Milne86}[Thm. 1.11] for the sheaves $W_\bullet\Omega^n_{X_s,log}$ and the isomorphism (\ref{motcomp}).

One has a commutative diagram of exact triangles of pro-complexes of \'etale sheaves on $X_s$
\begin{equation}\begin{CD}
  i^*\bz(n)/p^\bullet@>>> i^*\tau^{\leq n}Rj_*\mu_{p^\bullet}^{\otimes n}@>>> (\bz(n-1)/p^\bullet[-1])@>>>{}\\
 @V\cong VV \Vert@. @V\cong VV @.\\
\mathscr S''_\bullet(n)_X @>>> i^*\tau^{\leq n}Rj_*\mu_{p^\bullet}^{\otimes n} @>\kappa>> W_\bullet\Omega^{n-1}_{X_s,log}[-n]@>>>{}\\
 @AA\alpha_\bullet' A @AA\alpha_{n,\bullet}^{FM}A \Vert@. @.\\
 \mathscr S'_\bullet(n)_X @>>> \mathscr S_\bullet(n)_X @>>> W_\bullet\Omega^{n-1}_{X_s,log}[-n]@>>>{} 
\end{CD}\notag\end{equation}
where $\mathscr S''_\bullet(n)_X$ is defined as the mapping fibre of the map $\kappa$ and $\mathscr S'_\bullet(n)_X$ is the (non-logarithmic) syntomic complex of the smooth scheme $X$ as defined in \cite{kato87}. The isomorphism of the top two rows was shown in \cite{Geisser04a}[Sec. 6], the exactness of the lower triangle in \cite{ertlniziol16}[Thm. 3.2] and the commutativity of the lower two rows in \cite{ertlniziol16}[(3.10)].  If $n<p-1$ a result of Kurihara \cite{kurihara} shows that $\alpha_\bullet'$ is an isomorphism, and for any $n$ a result of Niziol and Colmez \cite{colniz15}[Thm.1.1] shows that $\alpha_{n,\bullet}^{FM}$ and hence $\alpha_\bullet'$ have bounded kernel and cokernel. In either case, by following the proof of \cite{nekniz13}[Thm. 3.8], one verifies that the composite map
\begin{multline*}(\mathrm{holim}_\bullet R\Gamma(X,\mathscr S'_\bullet(n)_X))_\bq\to (\mathrm{holim}_\bullet R\Gamma(X,\mathscr S_\bullet(n)_X))_\bq \xrightarrow{\alpha_{syn}} R\Gamma_{syn}(X_{\eta},n)\to \\
\left[ R\Gamma_{cris}(X_s/K_0)\xrightarrow{(1-\phi_n,sp')}R\Gamma_{cris}(X_s/K_0)\oplus R\Gamma_{dR}(X_\eta)/F^n\right]\ \end{multline*}
is an isomorphism. This implies the commutativity of the top square in Cor. \ref{pcor} and hence the conclusion of Cor. \ref{pcor}. Concerning Prop. \ref{plocmotivic}, the $p$-adic monodromy weight conjecture is trivially true in the smooth case and semisimplicity of
$H^i_{cris}(X_s/K_0)$ is not needed in the proof. The conclusion of Cor. \ref{pcor} and Prop. \ref{plocmotivic} then imply Corollary \ref{pdecompcor}.
\end{proof}

\begin{rem} Using the theory of cohomological descent one can show \cite{cisdeg13}[Prop. 5.3.3] that the natural maps
\[ R\Gamma(X_{s,et},\bq_p(n))\to R\Gamma(X_{s,eh},\alpha^*\bq_p(n))\to R\Gamma(X_{s,h},\alpha^*\bq_p(n))\]
are quasi-isomorphisms where $\alpha$ is the pullback from the {\em small} \'etale site of $X_s$. However, the $eh$-motivic cohomology defined by Geisser in \cite{Geisser06} and occurring in Conjecture \ref{pdescent}a) is defined by pulling back Voevodsky's complex $\bz(n)$ from the site $(\mathrm{Sm}/k)_{et}$ to the site $(\mathrm{Sch}/k)_{eh}$. As already remarked before notation \ref{ehprime}, one needs to assume resolution of singularities in order to prove the isomorphism $R\Gamma(X_{s,et},\bq_p(n))\cong R\Gamma(X_{s,eh},\bq_p(n))$ in this case (see \cite{Geisser06}[Thm. 4.3]).
\end{rem}

\begin{rem}\label{rem-BEKcompatible} If $f$ is smooth and and $n<p-1$ then one has an isomorphism \cite{Bloch-Esnault-Kerz-14}[Prop. 7.2 (3)]
\[ i^*\bz(n)/p^\nu\cong \mathscr S'_\nu(n)_X\cong \bz(n)_{X_\bullet}/p^\nu\]
where $\bz(n)_{X_\bullet}$ is the motivic pro-complex defined in \cite{Bloch-Esnault-Kerz-14}. Hence the compatibility requested after Conjecture ${\bf D}_p(\mathcal{X},n)$ (Conjecture \ref{conjD_p} in section  \ref{sec:localfactor}) is satisfied.
The exact triangle
\[ R\Gamma_{dR}(X_\eta)/F^n[-1]\to R\Gamma(X,\bq_p(n))\to \left[ R\Gamma_{cris}(X_s/K_0)\xrightarrow{1-\phi_n}R\Gamma_{cris}(X_s/K_0)\right]\]
arising from (\ref{syn}) can be written as the fundamental triangle
\[ R\Gamma_{dR}(X_\eta)/F^n[-1]\to R\Gamma(X,\bq_p(n))\to R\Gamma(X_s,\bq_p(n))\]
by (\ref{pdescent1}), and it agrees with the (rational cohomology of the) fundamental triangle of \cite{Bloch-Esnault-Kerz-14}[Thm. 5.4].
\end{rem}

%\begin{bibdiv}
%\begin{biblist}
%\bibselect{../all-my-references}
%\end{biblist}
%\end{bibdiv}
\end{document}